\documentclass{article}
\usepackage[utf8]{inputenc}
\usepackage{amssymb,amsmath,amsfonts,mathtools,mathrsfs}
\usepackage{amsthm,enumerate,enumitem}
\usepackage{authblk}
\setlist[enumerate]{label=(\arabic*), itemsep=0.7ex}
\usepackage[hidelinks]{hyperref}
\DeclareMathOperator{\pcom}{\mathsf{pow}-\mathsf{com}}

\DeclareMathOperator{\per}{\mathsf{per}}
\DeclareMathOperator{\mroot}{\mathsf{root}}
\DeclareMathOperator{\First}{\mathfrak{L}}
\DeclareMathOperator{\bsigma}{\boldsymbol{\sigma}}
\DeclareMathOperator{\btau}{\boldsymbol{\tau}}

\DeclareMathOperator{\RS}{\mathsf{RS}}

\DeclareMathOperator{\Fac}{\mathbf{F}}
\newcommand{\A}[2]{\mathcal{#1}_{\mathsf{#2}}}
\newcommand{\rQ}[1]{\mathrm{Q}_{\mathsf{#1}}}
\newcommand{\phii}[1]{\phi_{\mathsf{#1}}}
\newcommand{\psii}[1]{\psi_{\mathsf{#1}}}
\newtheorem{theo}{Theorem}

\newtheorem{prop}{Proposition}
\newtheorem{lem}{Lemma}

\newtheorem{defi}{Definition}
\newtheorem{rema}{Remark}
\newtheorem{ques}{Question}

\makeatletter % mathcal, mathbb
\@for\cs:={A,B,C,D,E,F,G,H,I,J,K,L,M,N,O,P,Q,R,S,T,U,V,W,X,Y,Z}\do{
  \expandafter\edef\csname c\cs\endcsname{\noexpand\mathcal{\cs}}
}
\@for\cs:={C,N,Q,R,Z}\do{
  \expandafter\edef\csname \cs\endcsname{\noexpand\mathbb{\cs}}
}
\def\@fnsymbol#1{\ensuremath{\ifcase#1\or \dagger\or \ddagger\or 
   \mathsection\or \mathparagraph\or \|\or **\or \dagger\dagger
   \or \ddagger\ddagger \else\@ctrerr\fi}}
\makeatother

\usepackage[style=alphabetic,maxalphanames=4]{biblatex}
\author[1,2]{Bastián Espinoza}
\affil[1]{Universidad de Chile}
\affil[2]{Université de Picardie Jules Verne}
\title{The structure of low complexity subshifts}

\begin{document}
\maketitle

\begin{abstract}
An idea that became unavoidable to study zero entropy symbolic dynamics is that the dynamical properties of a system induce in it a combinatorial structure.
An old problem addressing this intuition is finding a structure theorem for linear-growth complexity subshifts using the $\cS$-adic formalism.
It is known as the $\cS$-adic conjecture and motivated several influential results in the theory.
In this article, we completely solve the conjecture by providing an $\cS$-adic structure for this class.
Our methods extend to nonsuperlinear-complexity subshifts.
An important consequence of our main results is that these complexity classes gain access to the $\cS$-adic machinery.
We show how this provides a unified framework and simplified proofs of several known results, including the pioneering 1996 Cassaigne's Theorem.
\end{abstract}

% \tableofcontents

\section{Introduction}

\subsubsection*{Structure theorems}
An idea that became unavoidable to study zero entropy symbolic dynamics is that the dynamical properties of a system induce in it a combinatorial structure.
The first use of this approach was in the works of Morse and Hedlund \cite{MH38, MH40}, where Sturmian sequences were studied based on a structure given by what the authors called {\em derivative sequences}.
As the theory developed, more examples like this one emerged.
Relevant ones include substitutive and linearly recurrent subshifts \cite{DHS99}, Toeplitz systems \cite{toeplitz_are_bratteli}, (natural codings of) interval exchange transformations \cite{iet_are_bratteli}, dendric sequences \cite{GL22}  and general minimal subshifts \cite{HPS92}.
% Structure theorems have had a major impact on the development of the theory.%, see \cite{BKMS13, eigenvalues_jems, BST20}.

Let us briefly describe the modern standard for stating these theorems.
A substitution is a map $\tau\colon \cA^+\to\cB^+$ that substitutes the letters $a_i$ of a word $w = a_1 \dots a_\ell$ by $\tau(a_i)$.
Then, a sequence of substitutions $\btau$ of the form $(\tau_n\colon\cA_{n+1}\to\cA_n^+)_{n\geq0}$ is called an $\cS$-adic sequence and generates a subshift $X_{\btau} \subseteq \cA_0^\Z$ defined by requiring that $x \in X_{\btau}$ if and only if, for all $\ell \geq 0$, $x_{[-\ell, \ell)}$ occurs in $\tau_0\tau_1\dots\tau_{n-1}(a)$ for some $n \geq 1$ and $a \in \cA_n$.
In this setting, the structure theorems for the above-mentioned classes are an $\cS$-adic characterization, namely, of the form: a subshift $X$ belongs to the class $\cC$ if and only if $X$ is generated by an $\cS$-adic sequence satisfying certain property $\cP$.
The structure then appears as an infinite desubstitution process for the points of $X$.

In the context of structure theorems, an interesting intuition is that a subshift of low enough complexity should be very restricted, and thus hide a strong structure. 
Here, {\em low complexity} is a vague term referring to a slow growth of the complexity function $p_X(n)$, defined as the number of words of length $n$ that occur in some point of $X$.
This intuition dates back to the 70s, and matured in the 80s and 90s until it was finally established as the following more concrete question:
\begin{ques} \label{main_question}
Consider the class (L) of linear-growth complexity subshifts, defined by requiring that $p_X(n) \leq dn$ for some $d > 0$.
Is there an $\cS$-adic characterization of the class (L)?
\end{ques}
% The choice of linear-growth complexity subshifts is due to the fact that most of the known classes of subshifts have linear-growth complexity, and thus Question \ref{main_question} defines a relevant research horizon.

Question \ref{main_question} is known as the $\cS$-adic conjecture.
The first time it was explicitly stated was in \cite{ferenczi96}, where the author attributes the idea to B. Host, who, in turn, attributes the idea to the whole Marseille community.

The attempts to solve this conjecture have identified two major difficulties.
% There have been identified two major difficulties
% A non-exhaustive list of the attempts to solve the conjecture includes \cite{ferenczi96, DLR13_towards, Ler_1°diff, CFPZ18, Ler12_improvements}. 
The first one is that, in contrast to what happens with other structure theorems, there is no clear structure induced by the complexity.
For example, in the substitutive case, it was always clear that the substitution itself should produce a self-similar structure; the main obstruction was technical and referred to whether the desubstitution process was properly defined \cite{Mosse96}.
Similarly, in the Sturmian and IET cases, the known structure came from the geometric counterpart (more precisely, from the Rauzy induction).
The second challenge is that the condition $\cP$ we are looking for in Question \ref{main_question} is ill-defined.
To exemplify this point, observe that a corollary of \cite{Cassaigne_example} is the following $\cS$-adic characterization of (L): a subshift is in (L) if and only if there exist $\btau$ generating it and such that $X_{\btau}$ is in (L).
This tautological answer to Question \ref{main_question} does not provide information.
Certain restrictions on Question \ref{main_question} have been proposed to avoid this type of trivial answer, but none of them is considered satisfactory; we refer the reader to \cite{DLR13_towards} for a full discussion.

In this article, we completely solve the $\cS$-adic conjecture for minimal subshifts by proving the following theorem.
\begin{theo} \label{intro:main_sublinear}
A minimal subshift $X$ has linear-growth complexity, {\em i.e.}, $X$ satisfies
$$ \limsup_{n\to+\infty} p_X(n) / n < +\infty, $$
if and only if there exist $d > 0$ and an $\cS$-adic sequence $\bsigma = (\sigma_n\colon\cA_{n+1}\to\cA_n^+)_{n\geq0}$ generating $X$ such that, for every $n \geq 0$, the following holds:
\begin{enumerate}[label=$(\cP_{\arabic*})$]
    \item $\#(\mroot\sigma_{[0,n)}(\cA_n)) \leq d$
    \footnote{For a word $u$, $\mroot u$ denotes the shortest prefix $v$ of $u$ such that $u = v^k$ for some $k$; for a set of words $\cW$, $\mroot \cW = \{\mroot w : w \in \cW\}$.}.
    \item $|\sigma_{[0,n)}(a)| \leq d\cdot|\sigma_{[0,n)}(b)|$ for every $a, b \in \cA_n$.
    \item $|\sigma_{n-1}(a)| \leq d$ for every $a \in \cA_n$.
\end{enumerate}
\end{theo}

Our techniques extend to nonsuperlinear complexity subshifts (NSL).
\begin{theo} \label{intro:main_nonsuperlinear}
A minimal subshift $X$ has nonsuperlinear-growth complexity, {\em i.e.}, $X$ satisfies 
$$ \liminf_{n\to+\infty} p_X(n) / n < +\infty, $$
if and only if there exist $d > 0$ and an $\cS$-adic sequence $\bsigma = (\sigma_n\colon\cA_{n+1}\to\cA_n^+)_{n\geq0}$ generating $X$ such that, for every $n \geq 0$, the following holds:
\begin{enumerate}[label=$(\cP_{\arabic*})$]
    \item $\#(\mroot\sigma_{[0,n)}(\cA_n)) \leq d$.
    \item $|\sigma_{[0,n)}(a)| \leq d\cdot|\sigma_{[0,n)}(b)|$ for every $a, b \in \cA_n$.
\end{enumerate}
\end{theo}

The case of non-minimal subshifts does not pose additional intrinsic difficulties and follows from methods similar to those given here.
However, we decided not to include it in order to avoid over saturating an already technical presentation.

An important consequence of our main results is that the classes (L) and (NSL) gain access to the $\cS$-adic machinery.
We show in Section \ref{sec:apps} how this provides a unified framework and simplified proofs of several known results on (L) and (NSL), including the pioneering Cassaigne's Theorem \cite{Cassaigne95}.
Further applications of our main results, which include a new proof of partial rigidity for (NSL) \cite{creutz22} using the technique in \cite[Theorem 7.2]{BKMS13}, will be presented in a future paper.

We prove, in the more specialized Theorems \ref{theo:main_sublinear} and \ref{theo:main_nonsuperlinear}, that when $X$ is in (L) or in (NSL), then $\btau$ can be assumed to be recognizable.
Observe that the conditions $(\cP_i)$ in Theorems \ref{intro:main_sublinear} and \ref{intro:main_nonsuperlinear} are optimal in the sense that if we remove any of them then the corresponding theorem is false.
Conditions $(\cP_2)$ and $(\cP_3)$ also occur in the positive substitutive case 
\footnote{A substitution $\sigma\colon\cA\to\cB^+$ is {\em positive} if for all $a \in \cA$ and $b \in \cB$, $b$ occurs in $\sigma(a)$} 
and in linearly recurrent subshifts, but the behavior in our theorems is very different since we do not impose positiveness.

% Regarding $(\cP_1)$ and $(\cP_3)$, these were designed using as a reference two conditions that are present in most works that use $\cS$-adic sequences.
With regard to $(\cP_1)$ and $(\cP_3)$, these were designed on the basis of two conditions that are present in most works that involve $\cS$-adic sequences.
The first is having bounded alphabets (BA), which requires that $\#\cA_n$ is uniformly bounded, and the second is finitariness, which asks for the set $\{\tau_n : n\geq0\}$ to be finite.
Note that finitariness implies both (BA) and Conditions $(\cP_1)$ and $(\cP_3)$, that (BA) implies $(\cP_1)$, and that, under $(\cP_3)$, finitariness and (BA) are equivalent.
There are several papers in which a finitary $\cS$-adic sequence is looked for a subshift in (L) (see \cite{Ler_1°diff} and the references therein), and $\cS$-adic sequences with (BA) have shown to be closely connected with (L) and (NSL) \cite{ferenczi96, DDPM20_interplay}.
It is then natural to ask if we can replace, in Theorem \ref{intro:main_sublinear}, Conditions $(\cP_1)$ and $(\cP_3)$ by finitariness.
We show in Theorem \ref{theo:negative_result} that this is not possible.
More precisely, we build a minimal subshift with linear-growth complexity such that any $\btau$ generating it and satisfying $(\cP_1)$, $(\cP_2)$ and $(\cP_3)$ is not finitary (equivalently, (BA) does not hold).
However, in Theorems \ref{theo:main_sublinear} and \ref{theo:main_nonsuperlinear} we give a sufficient condition for $\btau$ being finitary.
Subshifts satisfying this sufficient condition include substitutive subshifts, codings of IETs and dendric subshifts.
% This opens new directions into negative versions of the $\cS$-adic conjecture.
% In particular, the analog of Theorem REF for nonsuperlinear complexity subshifts will be treated in a future paper.

\subsubsection*{Strategy of the main proofs}

The hard part of the proofs of Theorems \ref{intro:main_sublinear} and \ref{intro:main_nonsuperlinear} is constructing an $\cS$-adic sequence satisfying properties $(\cP_i)$ from the complexity hypothesis.
We detail our strategy in the case of Theorem \ref{intro:main_sublinear}; the proof of Theorem \ref{intro:main_nonsuperlinear} is similar.
%**

It is convenient to introduce the following terminology: a {\em coding} of a subshift $X \subseteq \cA^\Z$ is a pair $(Z,\sigma)$, where $Z \subseteq \cC^\Z$ is a subshift and $\sigma\colon\cC\to\cA^+$ a substitution such that $X = \bigcup_{k\in\Z} S^k \sigma(Z)$.
It is a standard fact  that if $\btau$ is an $\cS$-adic sequence, then there are subshifts $X_{\btau}^{(n)}$, with $X_{\btau}^{(0)} = X$, such that $(X_{\btau}^{(m)}, \tau_{[n,m)})$ is a coding of $X_{\btau}^{(n)}$ for any $n < m$, where $\tau_{[n,m)} = \tau_n \tau_{n+1} \dots \tau_{m-1}$.

Let $X$ be a linear-growth complexity subshift and $d = \sup_{n\geq1} p_X(n) / n$.
The typical method for building an $\cS$-adic sequence for a subshift $X$ is an inductive process: First, $X_0 \coloneqq X$; then, a coding $(X_{i+1}, \sigma_{i+1})$ of $X_i$ is defined.
In this way, $\bsigma \coloneqq (\sigma_n)_{n\geq0}$ is an $\cS$-adic sequence that, under mild conditions, generates $X$.
We, instead, take a more direct approach, similar to that in \cite[Theorem 4.3]{DDPM20_interplay} and \cite[Corollary 1.4]{E22}, but with additional technical details.
We consider an increasing sequence of positive integers $(\ell_n)_{n\geq0}$ with adequate growth and build codings $(X_n \subseteq \cC_n^\Z, \sigma_n\colon\cC_n\to\cA^+)$ of $X \subseteq \cA^\Z$ satisfying $(\cP_1)$, $\frac{1}{d'} \ell_n \leq |\sigma(a)| \leq d'\ell_n$ for all letters $a$ and with $d'$ depending on $d$, and such that certain technical properties hold. 
These technical properties allow us to define {\em connecting} substitutions $\tau_n\colon\cC_{n+1}\to\cC_n^+$ in such a way that $\sigma_n\tau_n(x)$ is, up to a shift, equal to $\sigma_{n+1}(x)$, for all $x \in X_{n+1}$.
Then, we can  prove that $\btau = (\sigma_0, \tau_0, \tau_1, \tau_2, \dots)$ generates $X$ and satisfies all the properties in Theorem \ref{intro:main_sublinear}.

The main idea for constructing the codings  $(X_n, \sigma_n)$ is that, thanks to a modification of the technique from \cite[Proposition 5]{ferenczi96}, we can build a coding $(X'_n, \sigma'_n)$ of $X$ (which is described in Proposition \ref{Z2:main_prop}) in such a way that the words $\sigma'(a)$ are either strongly aperiodic or strongly periodic.
The aperiodic words greatly contribute to the complexity, so we can efficiently control them using $d$.
For controlling the periodic words, we rely on tricks from combinatorics on words.
These two ideas are used to obtain, in Sections \ref{sec:Z3} to \ref{sec:Z4}, two variations of $(X'_n, \sigma'_n)$, with increasingly better properties, and where the last one is $(X_n, \sigma_n)$.

\subsubsection*{About the organization of the paper}
The paper has three parts.
The first one consists of Sections \ref{sec:preliminaries} to \ref{sec:Z0} and presents the preliminary theory and several lemmas for handling periodic words.
Then, in Sections \ref{sec:Z1} to \ref{sec:MT}, we carry out the proofs of Theorems \ref{intro:main_sublinear} and \ref{intro:main_nonsuperlinear}.
Finally, we prove Theorem \ref{theo:negative_result} and present applications of our main results in Sections \ref{sec:alph_rank} and \ref{sec:apps}.

\section{Preliminaries}
\label{sec:preliminaries}

In this paper, we will only use discrete intervals, that is, $[a,b) = \{k \in \Z : a \leq k < b\}$ for $a,b\in\R$.

%%%%%%%%%%%%%%%%%%%%%%%%%%%%%%%%%%%%%%%%%%%%%%%%%%%%%%%%%%%%%%%%%%%%%%%%%%%
%%%%%%%%%%%%%%%%%%%%%%%%%%%%%%%%%%%%%%%%%%%%%%%%%%%%%%%%%%%%%%%%%%%%%%%%%%%
\subsection{Words and subshifts} 
\label{subsec:words&subshifts}

Let $\cA$ be an {\em alphabet}, {\em i.e.}, a finite set. Elements in $\cA$ are called {\em letters} and concatenations $w = a_1\dots a_\ell$ of them are called {\em words}. 
The number $\ell$ is the length of $w$ and it is denoted by $|w|$, the set of all words in $\cA$ of length $\ell$ is $\cA^\ell$, and $\cA^+ = \bigcup_{\ell\geq1}\cA^\ell$.
It is convenient to introduce the empty word $1$, which has length zero. We set $\cA^* = \cA^+ \cup \{1\}$.
Given $w \in \cA^+$ and $n \geq 1$, $w^n$ denotes the concatenation of $n$ copies of $w$.
If $\cW \subseteq \cA^+$, we set
\begin{equation}
    \text{$\langle\cW\rangle = \min\{|w| : w \in \cW\}$ and $|\cW| = \max\{|w| : w \in \cW\}$.}
\end{equation}
The word $w\in \cA^+$ is $|u|$-{\em periodic}, with $u \in \cA^+$, if $w$ occurs in $u^n$ for some $n \geq 1$.
We denote by $\per(w)$ the least $p$ for which $w$ is $p$-periodic.
We say that $u, v \in \cA^+$ are {\em conjugate} if $u r = r v$ for some $r \in \cA^*$.

In order to describe more precisely the periodicity properties of $w$, we use the notion of {\em root}, which will play a key role throughout the paper.
\begin{defi}
The {\em minimal root}, or just {\em root} for short, of $w \in \cA^*$ is the shortest prefix $u$ of $w$ for which $w = u^k$ for some $k \geq 1$, and it is denoted by $\mroot w$.
\end{defi}
We remark that $\per(w)$ is an integer but that $\mroot w$ is a word, and that $\per(w)$ is in general different from $|\mroot w|$.
% Note also that if $u$ and $v$ are powers of a common word $r$, then $\mroot u = \mroot v = \mroot r$.
% This basic relation will be freely used through the paper.

We endow $\cA^\Z$ with the product topology of the discrete topology in $\cA$. 
Then, $\cA^\Z$ is a Cantor ({\em i.e.}, compact and totally disconnected) space and we call it the {\em full shift over $\cA$}.
Elements $x$ of $\cA^\Z$ are sometimes described using the notation $x = \dots x_{-2} x_{-1}. x_0 x_1\dots$, where the dot indicates the zero coordinate of $x$.
The {\em shift map} $S \colon \cA^\Z \to \cA^\Z$ is defined by $S ((x_n)_{n\in \mathbb{Z}}) = (x_{n+1})_{n\in \mathbb{Z}}$, and is a homeomorphism of $\cA^\Z$ onto itself.
For $x \in \cA^\Z$ and integers $i < j$, we denote by $x_{[i,j)}$ the word $x_i x_{i+1} \dots x_{j-1}$.
Analogous notation will be used when dealing with intervals of the form $[i,\infty)$, $(i,\infty)$, $(-\infty,i]$ and $(-\infty,i)$. 
A {\em subshift} is a closed subset $X$ of $\cA^\Z$ such that $S(X) = X$.
The {\em language} of a subshift $X\subseteq\cA^\Z$ is the set $\cL(X)$ of all words $w\in\cA^+$ that occur in at least one $x \in X$.

Let $X \subseteq \cA^\Z$ be a subshift and $v \in \cA^+$.
We will use the notation $v^\infty = vvv\dots \in \cA^\N$ and $b^\Z = \dots vv.vv\dots \in \cA^\Z$.
We denote by $\mathrm{Pow}_X(v)$ the set of words $v^k$, where $k \geq 1$, for which there exist $u,w \in \cA^+\setminus\{v\}$ of length $|v|$ such that $u v^k w \in \cL(X)$.
The {\em power complexity} of $X$ is the number $\pcom(X) = \sup_{v \in \cA^+} \#\mathrm{Pow}_X(v)$.
Remark that $\pcom(X)$ may be infinite.
Examples with finite power complexity include linearly recurrent subshifts and subshifts in which the extension graph of every long enough word is acyclic (in particular, Sturmian subshifts and codings of minimal interval exchange transformations).

%%%%%%%%%%%%%%%%%%%%%%%%%%%%%%%%%%%%%%%%%%%%%%%%%%%%%%%%%%%%%%%%%%%%%%%%%%%
%%%%%%%%%%%%%%%%%%%%%%%%%%%%%%%%%%%%%%%%%%%%%%%%%%%%%%%%%%%%%%%%%%%%%%%%%%%
\subsection{Morphisms and codings}
\label{subsec:morphisms&codings}
The set $\cA^*$ equipped with the operation of concatenation can be viewed as the free monoid on $\cA$. 
In particular, $\cA^+$ is the free semigroup on $\cA$.

Let $\cA$ and $\cB$ be finite alphabets and $\tau \colon \cA^+ \to \cB^+$ be a morphism between the free semigroups that they define. 
Then, $\tau$ extends naturally to maps from $\cA^\mathbb{N}$ into itself and from $\cA^\mathbb{Z}$ into itself by concatenation (in the case of a two-sided sequence we apply $\tau$ to positive and negative coordinates separately and we concatenate the results at coordinate zero).
We say that $\tau$ is {\em positive} if for every $a\in\cA$, all letters $b \in \cB$ occur in $\tau(a)$, that $\tau$ is {\em proper} if there exist letters $a,b \in \cB$ such that $\tau(c)$ starts with $a$ and ends with $b$ for any $c \in \cA$, and that $\tau$ is {\em injective on letters} if for all $a, b \in \cB$, $\tau(a) = \tau(b)$ implies $a = b$.
The minimum and maximum length of $\tau$ are, respectively, the numbers $\langle\tau\rangle \coloneqq \langle\tau(\cA)\rangle = \min_{a\in\cA}|\tau(a)|$ and $|\tau| \coloneqq |\tau(\cA)| = \max_{a\in\cA}|\tau(a)|$.

We observe that any map $\tau\colon \cA\to \cB^+$ can be naturally extended to a morphism (that we also denote by $\tau$) from $\cA^+$ to $\cB^+$ by concatenation, and that any morphism $\tau\colon \cA^+\to\cB^+$ is uniquely determined by its restriction to $\cA$. 
From now on, we will use the same notation for denoting a map $\tau\colon\cA\to\cB^+$ and its extension to a morphism $\tau\colon\cA^+\to\cB^+$.

\subsubsection{Factorizations and recognizability}

We now introduce {\em factorizations}, the {\em recognizability property} and the associated notation.

\begin{defi}
\label{defi:factorizations}
Let $Y\subseteq \cB^\Z$ be a subshift and $\tau\colon \cB^+\to \cA^+$ be a morphism. 
We say that $(k,y) \in \Z\times Y$ is a {\em $\tau$-factorization of $x \in \cA^\Z$ in $Y$} if $x = S^k\tau(y)$ and $0 \leq k < |\tau(y_0)|$.

The pair $(Y,\tau)$ is {\em recognizable} if every point $x \in \cA^\Z $ has at most one $\tau$-factorization in $Y$.
We say that $(Y,\tau)$ is {\em $d$-recognizable}, with $d \geq 1$, if whenever $(k,y)$ and $(\tilde{k},\tilde{y})$ are $\tau$-factorizations of $x, \tilde{x} \in \cA^\Z$ in $Y$, respectively, and $x_{[-d,d)} = \tilde{x}_{[-d,d)}$, we have that $k = \tilde{k}$ and $y_0 = \tilde{y}_0$.

The {\em cut function} $c\colon\Z\to\Z$ of the $\tau$-factorization $(k,y)$ of $x$ in $Y$ is defined by
\begin{equation}\label{eq:defi_cuts}
c_j = \begin{cases}
    -k + |\tau(y_{[0,j)})| &\text{if } j \geq 0, \\
    -k - |\tau(y_{[j,0)})| &\text{if } j < 0. \end{cases}
\end{equation}
When $(Y, \tau)$ is recognizable, we write $(c, y) = \Fac_{(Y,\tau)}(x)$ and $(c_0, y_0) = \Fac^0_{(Y,\tau)}(x)$.
\end{defi}

\begin{rema}
\label{rem:factorizations}
In the context of the previous definition:
\begin{enumerate}
\item\label{rem:factorizations:image_subshift} 
The point $x \in \cA^\Z$ has a $\tau$-factorization in $Y$ if and only if $x$ belongs to the subshift $X \coloneqq \bigcup_{n\in\Z} S^n\tau(Y)$. 
Hence, $(Y,\tau)$ is recognizable if and only if every $x \in X$ has a exactly {\em one} $\tau$-factorization in $Y$.

\item\label{rem:factorizations:centered} 
If $(k,y)$ is a $\tau$-factorization of $x \in \cA^\Z$ in $Y$ and $c$ is its cut function, then $x = S^{-c_j} \tau(S^j y)$ for any $j \in \Z$. 
However, we have $0 \leq -c_j < |\tau((S^jy)_0)|$ only if $j = 0$.
Therefore, $(-c_j, S^jy)$ is a $\tau$-factorization of $x$ if and only if $j = 0$.

\item If $(Y,\tau)$ is recognizable, then a compacity argument shows that it is $d$-recognizable for some $d \geq 1$.

\item \label{rem:factorizations:cover_symbol}
Suppose that  $(Y,\tau)$ is recognizable.
Let $x \in X$, $(c,y) = \Fac_{(Y,\tau)}(x)$ and $i \in \Z$.
Then, there exists a unique $j \in \Z$ such that $i \in [c_j, c_{j+1})$.
Note that the last condition is equivalent to $\Fac^0_{(Y,\tau)}(S^i x) = (c_j - i, y_j)$.
\end{enumerate}
\end{rema}

The behavior of the recognizability property under composition of morphisms is given by the following two lemmas.
\begin{lem}[\cite{BSTY19}, Lemma 3.5] \label{lem:recognizability_composition}
Let $\sigma\colon \cC\to \cB^+$ and $\tau\colon\cB\to \cC^+$ be morphisms, $Z\subseteq \cC^\Z$ be a subshift and $Y = \bigcup_{k\in\Z} S^k\sigma(Z)$. 
Then, $(Z, \tau\sigma)$ is recognizable if and only if $(Z,\sigma)$ and $(Y,\tau)$ are recognizable.
\end{lem}

\begin{lem} \label{lem:factorizations_and_compositions}
Let $\sigma\colon \cC\to \cB^+$ and $\tau\colon\cB\to \cC^+$ be morphisms and $Z\subseteq \cC^\Z$ be a subshift.
We set $Y = \bigcup_{k\in\Z} S^k \sigma(Z)$ and $X = \bigcup_{k\in\Z} S^k \tau(Y)$.
Suppose that $(Z, \tau \sigma)$ is recognizable.
Let $x \in X$, $(k, Y)$ be a $\tau$-factorization of $x$ in $Y$ and $(\ell, z)$ be a $\tau\sigma$-factorization of $x$ in $Z$.
Then, there exists $m \in [0, |\sigma(z_0)|)$ such that $y = S^m \sigma(z)$ and $k = |\sigma(z_{[-m, 0)})| + \ell$.
\end{lem}
\begin{proof}
Being $(\ell,z)$ a $\tau\sigma$-factorization of $x$, we have that $\ell \in [0, |\tau(\sigma(z_0))|)$.
Hence, there exists $m \in [0, |\sigma(z_0)|)$ such that
\begin{equation} \label{eq:lem:factorizations_and_compositions:1}
    |\sigma(z)_{[0, m)}| \leq \ell < |\sigma(z)_{[0,m]}|.
\end{equation}
Therefore, as $(\ell,z)$ is a $\tau\sigma$-factorization of $x$, we can write
\begin{equation*}
    S^{\ell- |\tau(\sigma(z)_{[0,m)})|} \tau(S^m \sigma(z)) = S^{\ell} \tau\sigma(z) = x.
\end{equation*}
This and \eqref{eq:lem:factorizations_and_compositions:1} ensure that $(\ell- |\tau(\sigma(z)_{[0,m)})|, S^m\sigma(z))$ is a $\tau$-factorization of $x$ in $Y$.
We conclude, using that $(Y, \tau)$ is recognizable by Lemma \ref{lem:recognizability_composition}, that $\ell- |\tau(\sigma(z)_{[0,m)})| = k$ and $S^m \sigma(z) = y$.
\end{proof}

\subsubsection{Codings of a subshift}

We fix a subshift $X \subseteq \cA^\Z$.
A {\em coding} of $X$ is a pair $(Y, \tau)$, where $Y \subseteq \cB^\Z$ is a subshift and $\tau \colon \cB^+\to\cA^+$ a morphism satisfying $X = \bigcup_{k\in\Z} S^k\tau(Y)$.
We present in Proposition \ref{prop:from_clopen_to_coding} a general method for building recognizable codings of a subshift $X$.
This idea occurs commonly in the literature under several different names and with different degrees of generality.
Our Proposition \ref{prop:from_clopen_to_coding} is inspired by the coding based on return words from \cite{durand98}.

Let $U \subseteq X$ be a clopen ({\em i.e.}, open and closed) set. 
We say that $U$ is
\begin{enumerate}
    \item {\em $\ell$-syndetic} if for all $x \in X$ there exists $k \in [0, \ell)$ such that $S^k x \in U$;
    \item of {\em radius $r$} if $U$ is an union of sets of the form $\{x \in X : x_{[-|u|,|v|)} = uv\}$, where $u,v \in \cA^r$.
\end{enumerate}
Remark that, since $X$ is minimal, any nonempty clopen set $U$ is $\ell$-syndetic and of radius $r$ for some $\ell$ and $r$.
% A {\em return word to $U$ in $X$} is a word of the form $u = x_{[i,j)}$, where $x \in X$, $S^j y \in U$, $S^i x \in U$ and $S^k x \not\in U$ for $k\in(i,j)$.

\begin{prop} \label{prop:from_clopen_to_coding}
Let $U \subseteq X$ be a nonempty clopen set. 
There exists a recognizable coding $(Y \subseteq \cB^\Z, \sigma\colon\cB\to\cA^+)$ of $X$, with $\sigma$ injective on letters, such that if $x \in X$, $(c, y) = \Fac_{(Y,\sigma)}(x)$ and $i \in \Z$, then $S^i x \in U$ if and only if $i = c_j$ for some $j \in \Z$.

If $U$ is $\ell$-syndetic and of radius $r$, then $(Y, \sigma)$ additionally satisfies that:
\begin{enumerate}
    \item $|\sigma(a)| \leq \ell$ for all $a \in \cB$.
    \item $(Y, \sigma)$ is $(\ell+r)$-recognizable.
\end{enumerate}
\end{prop}
% add proof; this uses defi {defi:return_words_to_clopen_set}

%%%%%%%%%%%%%%%%%%%%%%%%%%%%%%%%%%%%%%%%%%%%%%%%%%%%%%%%%%%%%%%%%%%%%%%%%%%
%%%%%%%%%%%%%%%%%%%%%%%%%%%%%%%%%%%%%%%%%%%%%%%%%%%%%%%%%%%%%%%%%%%%%%%%%%%
\subsection{\texorpdfstring{$\cS$}--adic subshifts}
\label{subsec:Sadic}

We recall the definition of an {\em $\cS$-adic subshift} as stated in \cite{BSTY19}.
A {\em directive sequence} $\btau = (\tau_n \colon \cA_{n+1}^+\to \cA_n^+)_{n \geq 0}$ is a sequence of morphisms. 
For $0 \leq n < m$, we denote by $\tau_{[n,m)}$ the morphism $\tau_n \circ \tau_{n+1} \circ \dots \circ \tau_{m-1}$. 
We say that $\btau$ is {\em everywhere growing} if
\begin{equation}\label{eq:defi_everywhere_growing}
\lim_{n\to +\infty}\langle\tau_{[0,n)}\rangle = +\infty,
\end{equation}
and {\em primitive} if for any $n \geq 0$ there exists $m  > n$ such that $\tau_{[n,m)}$ is positive. 
Observe that $\btau$ is everywhere growing if $\btau$ is primitive.
We remark that this notion is slightly different from the usual one used in the context of substitutional dynamical systems. 

For $n \geq 0$, we define
\begin{equation*}
X^{(n)}_{\btau} = \big\{x \in \cA_n^\Z :\
\mbox{$\forall \ell\in\N$, $x_{[-\ell,\ell]}$ occurs in $\tau_{[n,m)}(a)$ for some $m>n,a\in\cA_m$}\big\}.
\end{equation*}
This set clearly defines a subshift.
We set $X_{\btau} = X^{(0)}_{\btau}$ and call this set the {\em $\cS$-adic subshift generated by $\btau$}. 
If $\btau$ is everywhere growing, then every $X^{(n)}_{\btau}$, $n\geq0$, is nonempty; if $\btau$ is primitive, then $X^{(n)}_{\btau}$ is minimal for every $n\in\N$. 
There are non-everywhere growing directive sequences that generate minimal subshifts.

The relation between the $X_{\btau}^{(n)}$ is given by the following lemma.
\begin{lem}[\cite{BSTY19}, Lemma 4.2]\label{lem:desubstitution}
Let $\btau = (\tau_n \colon \mathcal{A}_{n+1}^+\to \mathcal{A}_n^+)_{n \in \N}$ be a directive sequence of morphisms.
For any $0 \leq n < m$, $(X_{\btau}^{(m)}, \tau_{[n,m)})$ is a coding of $X_{\btau}^{(n)}$.
\end{lem}

% We define the {\em alphabet rank} of a directive sequence $\btau$ as 
% \begin{equation*}
%     \mathrm{AR}(\btau) = \liminf_{n\to +\infty} \#\cA_n.
% \end{equation*}

% A {\em contraction} of $\btau$ is a sequence $\tilde{\btau} = (\tau_{[n_k,n_{k+1})}\colon \cA_{n_{k+1}}^+\to \cA_{n_k}^+)_{k\in\N}$, where $0 = n_0 < n_1 < n_2 < \dots$. 
% Observe that any contraction of $\btau$ generates the same $\cS$-adic subshift $X_{\btau}$. 
% If $\btau$ has finite alphabet rank, then there exists a contraction $\tilde{\btau} = (\tau_{[n_k,n_{k+1})}\colon \cA_{n_{k+1}}^+\to \cA_{n_k}^+)_{k\in\N}$ of $\btau$ in which $\cA_{n_k}$ has cardinality $\mathrm{AR}(\btau)$ for every $k\geq1$.

%%%%%%%%%%%%%%%%%%%%%%%%%%%%%%%%%%%%%%%%%%%%%%%%%%%%%%%%%%%%%%%%%%%%%%%%%%%
%%%%%%%%%%%%%%%%%%%%%%%%%%%%%%%%%%%%%%%%%%%%%%%%%%%%%%%%%%%%%%%%%%%%%%%%%%%
\subsection{The complexity function}
\label{subsec:complexity}

The {\em complexity function} $p_X\colon\Z_{\geq1} \to \Z_{\geq1}$ of a subshift $X$ is defined by $p_X(n) = \#\cL(X) \cap \cA^n$.
Equivalently, $p_X(n)$ counts the number of words of length $n$ that occur in at least one $x \in X$.

\begin{defi} \label{defi:complexity_classes} 
We say that $X$ has 
\begin{enumerate}
    \item {\em linear-growth complexity} if there exists $d > 0$ such that $p_X(n) \leq dn$ for all $n \geq 1$;
    \item {\em nonsuperlinear-growth complexity} if there exists $d > 0$ such that $p_X(n) \leq dn$ for infinitely many $n \geq 1$.
\end{enumerate}
\end{defi}

\begin{rema}
When $X$ is infinite, then a classic theorem of Morse and Hedlund \cite{MH38} ensures that $p_X(n) \geq n + 1$ for all $n \geq 0$.
Thus, an infinite subshift of linear-growth complexity satisfies $n \leq p_X(n) \leq dn$, and so $p_X$ grows linearly.
\end{rema}

The following theorem is classic.
\begin{theo}[\cite{Cassaigne95}] \label{theo:cassaigne95}
Let $X$ be a transitive linear-growth complexity subshift. 
Then, $p_X(n+1)-p_X(n)$ is uniformly bounded.
\end{theo}

For the proof of Theorems \ref{theo:main_sublinear} and \ref{theo:main_nonsuperlinear} in Section \ref{sec:MT}, we will need only the following two weaker versions of Cassaigne's Theorem.
\begin{lem} \label{preli:find_good_s_sublinear}
Let $X$  be a subshift and $d \geq 1$ be such that $p_X(n) \leq dn$ for all $n \geq 1$. 
Then, for every $n \geq 1$ there exists $m \in [n, 2n)$ such that $p_X(m+1) - p_X(m) \leq 2d$.
\end{lem}
\begin{proof}
Let $n \geq 1$. We observe that the average of $p_X(m+1) - p_X(m)$ for $m \in [n, 2n)$ can be bounded as follows by using that $p_X(2n) \leq 2dn$:
\begin{equation*}
\frac{1}{n}\sum_{m=n}^{2n-1} p_X(m+1) - p_X(m) =
\frac{1}{n}(p_X(2n) - p(n)) \leq 2d.
\end{equation*}
Thus, there exists $m \in [n, 2n)$ satisfying $p_X(m+1) - p_X(m) \leq 2d$.
\end{proof}

\begin{lem} \label{preli:find_good_s_nonsuperlinear}
Let $X \subseteq \cA^\Z$  be a subshift and $d \geq 1$ be such that $p_X(n) \leq dn$ for infinitely many $n \geq 1$. 
Then, there are infinitely many $m$ such that $p_X(m) \leq 3dm$ and $p_X(m+1) - p_X(m) \leq 2d$.
\end{lem}
\begin{proof}
Let $n \geq 1$ be arbitrary.
The hypothesis permits to find $k \geq 2n$ such that $p_X(k) \leq dk$.
We now observe that 
\begin{equation*}
    \frac{1}{\lceil k/2\rceil}\sum_{m = \lfloor k/2\rfloor}^k p_X(m+1) - p_X(m) \leq
    \frac{1}{\lceil k/2\rceil} p_X(k) \leq 2d.
\end{equation*}
Therefore, there exists $m$ such that $\lfloor k/2\rfloor \leq m \leq k$ and $p_X(m+1) - p_X(m) \leq 2d$.
The first condition ensures that $m \geq n$ and $p_X(m) \leq p_X(k) \leq dk \leq 3dm$.
\end{proof}
\section{Some combinatorial lemmas}
\label{sec:CL}

In order to prove our main results, we will need to extensively deal with strongly periodic words.
The objective of this section is to give the necessary tools for doing so.

A basic result on periodicity of words is the {\em Fine and Wilf Theorem}, which we state below.
\begin{theo}\label{theo:fine&wilf}
Let $u,v,w \in \cA^+$ and suppose that $w$ is a prefix of $u^\infty$ and $v^\infty$.
If $|w| \geq |u| + |v| - 1$, then there exists $t \in \cA^+$ such that $u$ and $v$ are powers of $t$.
\end{theo}
A proof of Theorem \ref{theo:fine&wilf} can be found in \cite[Chapter 6, Theorem 6.1]{handbook_words}.

\begin{lem} \label{CL:root=per}
Let $u$ be a word such that $|u| \geq 2 |\mroot u|$.
Then, $|\mroot u| = \per(u)$.
\end{lem}
\begin{proof}
Note that $u$ is a prefix of $(\mroot u)^\Z$ and thus that $\per(u) \leq |\mroot u|$.
It rests to prove the other inequality.

Let $t$ be the prefix of $u$ of length $\per(u)$.
Then $u$ is a prefix of both $t^\infty$ and $(\mroot u)^\infty$.
We deduce, as $|u| \geq 2\per(u) \geq |t| + |\mroot u|$, that the hypothesis of Lemma \ref{theo:fine&wilf} is complied.
Hence, $t$ and $\mroot u$ are powers of a common word $r$.
In particular, $u$ is a power of $r$, so we must have that $\mroot u = r$.
Therefore, $|\mroot u| = |r| \leq |t| = \per(u)$.
\end{proof}

\begin{rema}
The previous lemma ensures that if $u$ is a word and $k \geq 1$, then $\mroot u^k = \mroot u$.
In particular, if $u$ and $v$ are powers of a common word, then they have the same root.
These basic relations will be freely used throughout the paper.
\end{rema}

The next proposition will allow us to synchronize occurrences of strongly periodic words. 
\begin{prop} \label{CL:on_sZ&tZ}
Let $t, s \in \cA^+$.
\begin{enumerate}
    \item \label{item:CL:on_sZ&tZ:equal}
    Suppose that $\ell \geq |s| + |t| - 1$ and $i, j \in \Z$ are such that $t^\Z_{[i, i+\ell)} = s^\Z_{[j, j+\ell)}$
    Then, $S^i t = S^j s$.

    \item \label{item:CL:on_sZ&tZ:equal_mod}
    An integer $i$ satisfies $S^i t^\Z = t^\Z$ if and only if $i = 0 \pmod{|\mroot t|}$.
\end{enumerate}
\end{prop}
\begin{proof}
We first prove Item (1).
Let $t_0 = t^\Z_{[i, i +|t|)}$, $s_0 = s^\Z_{[j, j + |s|)}$ and $w = t^\Z_{[i, i+\ell)} = s^\Z_{[j, j+\ell)}$.
Then, $w$ is a prefix of both $t_0^\infty$ and $s_0^\infty$.
Since $|w| = \ell \geq |s| + |t| - 1 = |s_0| + |t_0| - 1$, we can use Theorem \ref{theo:fine&wilf} to deduce that $s_0$ and $t_0$ are powers of a common word $r$.
We then have $S^i s^\Z = s_0^\Z = r^\Z = t_0^\Z = S^j t^\Z$.

We now prove Item (2).
It is clear that if $i = 0 \pmod{|\mroot t|}$ then $S^i t^\Z = t^\Z$.
Let us suppose that $S^i t^\Z = t^\Z$.
We argue by contradiction and assume that $i \not= 0 \pmod{|\mroot t|}$.
We write $\mroot t = s s'$, where $|s| = i \pmod{|\mroot t|}$.
Then, $(s's)^\Z = S^i t^\Z = t^\Z = (ss')^\Z$, so Theorem \ref{theo:fine&wilf} implies that $s' s$ and $s s'$ are powers of a common word $r$.
In particular, $\mroot(s's) = \mroot(s s') = \mroot r$.
This implies that 
\begin{equation*}
    |\mroot(s's)| = |\mroot(s s')| =
    |\mroot\mroot t| = |\mroot t| = |s s'| = |s' s|,
\end{equation*}
so $\mroot(s' s) = s' s$.
Hence, $s' s = s s'$.
Now, since $i \not= 0 \pmod{|\mroot t|}$, $s$ and $s'$ are not the empty word.
This and the condition $s' s = s s'$ imply that $s's$ is a prefix of $s^\infty$ and of ${s'}^\infty$.
We can then use Theorem \ref{theo:fine&wilf} to deduce that $s$ and $s'$ are powers of a common word $r'$.
Therefore, as $\mroot = s s'$, $s s' = \mroot t = \mroot s = \mroot s'$.
This is possible only if $s = 1$ or $s' = 1$.
Consequently, $|s| \in \{0, |\mroot t|\}$ and $i = |s| = 0 \pmod{|\mroot t|}$, contradicting our assumptions.
\end{proof}

The rest of the section is devoted to prove Propositions \ref{CL:crit_fact_like} and \ref{CL:localize_per_aper}.
These results describe situations in which information about the global period of a word can be retrieved from small subwords of it.
We remark that Propositions \ref{CL:crit_fact_like} and \ref{CL:localize_per_aper} can be obtained as a direct consequence of the Critical Factorization Theorem, a fundamental result in combinatorics on words; here we give proofs that depend only on Theorem \ref{theo:fine&wilf} in order to maintain our presentation as self-contained as possible.

\begin{lem} \label{CL:various_sync}
Let $u,v,w,s$ and $t$ be words in $\cA$.
\begin{enumerate}
    \item \label{item:CL:various_sync:uv_t_vw_s=>uvw_t_s}
    Suppose that $uv$ occurs in $t^\infty$ and that $vw$ occurs in $s^\infty$.
    If, $|v| \geq |t| + |s| - 1$, then $uvw$ occurs both in $t^\infty$ and in $s^\infty$.
    \item \label{item:CL:various_sync:uv_t_vw_s=>uvw_t_s:bis} 
    Suppose that $uv$ is a prefix of $t^\infty$ and that $vw$ is a suffix of $t^\infty$.
    If $|v| \geq 2|t|$, then $uvw$ is a power of $\mroot t$.
    \item \label{item:CL:various_sync:F&W_per}
    If $|v| \geq \per(uv) + \per(vw)$, then $\per(uvw) = \per(uv) = \per(vw)$.
\end{enumerate}
\end{lem}
\begin{proof}
Assume that the hypothesis of Item (1) holds.
Then, $uv = t^\Z_{[i, i + |uv|)}$ and $vw = s^\Z_{[j, j + |vw|)}$ for some $i, j \in \Z$.
Hence, $t^\Z_{[i+|u|, i+|uv|)} = s^\Z_{[j, j+|v|)}$.
This and the inequality $|v| \geq |t| + |s| - 1$ allows us to use Item \ref{item:CL:on_sZ&tZ:equal} in Proposition \ref{CL:on_sZ&tZ} to get that $S^{i+|u|} t^\Z = S^j s^\Z$.
We conclude that 
\begin{equation*}
    uvw = t^\Z_{[i,i+|uv|)} s^\Z_{[j+|v|,j+|vw|)} = t^\Z_{[i,i+|uv|)} t^\Z_{[i+|u|+|v|,i+|u|+|vw|)} = t^\Z_{[i, i+|uvw|)}, 
\end{equation*}
and that $uvw$ occurs in $t^\infty$.
Similarly, $uvw$ occurs in $s^\infty$.

We now assume that the hypothesis of Item (2) holds.
Let $t_0 = \mroot t$
Then, $uv = (t_0^\Z)_{[0, |uv|)}$ and $vw = (t_0^\Z)_{[-|vw|, 0)}$.
This implies that $(t_0^\Z)_{[|u|, |uv|)} = (t_0^\Z)_{[-|vw|, -|w|)}$, and then, since $|v| \geq 2|t| \geq 2|t_0|$, Item \ref{item:CL:on_sZ&tZ:equal} in Proposition \ref{CL:on_sZ&tZ} ensures that 
\begin{equation} \label{eq:CL:various_sync:1}
    S^{|u|} t_0^\Z = S^{-|vw|} t_0^\Z
    \enspace\text{and}\enspace
    uvw = (t_0^\Z)_{[0, |uvw|)}.
\end{equation}
Now, from the first equation in \eqref{eq:CL:various_sync:1} and Item \ref{item:CL:on_sZ&tZ:equal_mod} in Proposition \ref{CL:on_sZ&tZ} we get that $|u| = -|vw| \pmod{|t_0|}$, that is, $|uvw| = 0 \pmod{|t_0|}$.
This and the second equation in \eqref{eq:CL:various_sync:1} give that $uvw = (t_0^\Z)_{[0, |uvw|)}$ is a power of $t_0 = \mroot t$.

We finally prove Item (3).
Clearly, $\per(uv) \leq \per(uvw)$ and $\per(vw) \leq \per(uvw)$.
Let $t_0$ be the prefix of $uv$ of length $\per(uv)$ and $s_0$ be the prefix of $vw$ of length $\per(vw)$.
Then, $uv$ occurs in $t_0$ and $vw$ occurs in $s_0$.
This and the inequality $|v| \geq |u| + |v| \geq |t_0| + |s_0|$ allow us to use Item (1) of this lemma to deduce that $uvw$ occurs in $t_0^\Z$ and $s_0^\Z$.
We deduce that $\per(uvw) \leq |t_0| = \per(uv)$ and $\per(uvw) \leq |s_0| = \per(vw)$.
Therefore, $\per(uvw) = \per(uv) = \per(vw)$.
\end{proof}

\begin{prop}\label{CL:crit_fact_like}
Let $\cV \subseteq \cA^+$ and $u \in \cA^+$ be such that $|u| \geq 2|\cV|$.
Suppose that for any subword $v$ of $u$ with length $|v| = 2|\cV|$ there exists $w_v \in \cV$ such that $v$ occurs in $w_v^\Z$.
Then, for any such word $v$, $u$ occurs in $w_v^\Z$.
In particular, $\per(u) \leq |\cV|$.
\end{prop}
\begin{proof}
The case $|u| = 2|\cV|$ follows directly from the hypothesis.
Suppose the lemma is true for words $u'$ of length $2|\cV| \leq |u'| < |u|$.
Let $v$ be a subword of $u$ with length $|v| = 2|\cV|$.
We have to prove that $w$ occurs in $w_v^\Z$.
Let us write $u = au' = u''b$ for certain letters $a,b$ and words $u',u''$.
There is no loss of generality in assuming that $v$ occurs in $u'$.
Since $|u| > 2|\cV|$, we can take a subword $v'$ of $u''$ with length $|v'| = 2|\cV|$.
Then, the inductive hypothesis can be used to deduce that $u'$ occurs in $w_v^\Z$ and that $u''$ occurs in $w_{v'}^\Z$.
Now, $u'$ and $u''$ have a common subword of length $|u|-2 \geq 2|\cW|-1 \geq |w_v| + |w_{v'}|-1$.
Therefore, Item \ref{item:CL:various_sync:uv_t_vw_s=>uvw_t_s} of Lemma \ref{CL:various_sync} can be applied and we deduce that $w$ occurs in $w_v^\Z$.
This proves the inductive step and thereby the proposition.
\end{proof}

\begin{prop} \label{CL:localize_per_aper}
Let $u$ be a word.
\begin{enumerate}
    \item \label{item:CL:localize_per_aper:loc_per}
    If $t$ is a word occurring in $u$ and $|t| \geq 2 \per(u)$, then $\per(t) = \per(u)$.
    \item \label{item:CL:localize_per_aper:loc_aper}
    Let $k \geq 1$. If $|u| \geq 2k$ and $\per(u) > k$, then there exists $t$ occurring in $u$ with $|t| = 2k$ and $\per(t) > k$.
\end{enumerate}
\end{prop}
\begin{proof}
We start with Item (1).
Note that $\per(t) \leq \per(u)$, so we only have to prove the other inequality.
Let $s$ (resp.\ $s'$) be the prefix of $t$ of length $\per(t)$ (resp.\ $\per(u)$).
Then, $t$ occurs in $s^\Z$ and ${s'}^\Z$.
Being $|t| \geq 2\per(u) \geq |s| + |s'|$, we can use Item \ref{item:CL:on_sZ&tZ:equal} in Proposition \ref{CL:on_sZ&tZ} to deduce that $s^\Z = S^\ell {s'}^\Z$ for some $\ell \in \Z$.
This implies, as $u$ occurs in ${s'}^\Z$, that $u$ occurs in $s^\Z$.
In particular, $\per(u) \leq |s| = \per(t)$.
\medskip

Next, we prove Item (2) by contradiction.
Assume that $k \geq 1$ and $u$ are such that $|u| \geq 2k$ and $\per(u) > k$, but that for all word $t$ occurring in $u$ of length $|t| = 2k$ we have that $\per(t) \leq k$.
We define, for all such $t$, $s_t$ as the prefix of $t$ of length $\per(t)$, and note that $t$ occurs in $s_t^\Z$ and that $|s_t| \leq k$.
Therefore, the set $\cV$ consisting of the words $s_t$ and the word $u$ comply with the hypothesis of Proposition \ref{CL:crit_fact_like}.
We conclude that $\per(u) \leq |\cV| \leq k$.
\end{proof}

%%%%%%%%%%%%%%%%%%%%%%%%%%%%%%%%%%%%%%%%%%%%%%%%%%%%%%%%%%%%%%%%%%%%%%%%%%%
%%%%%%%%%%%%%%%%%%%%%%%%%%%%%%%%%%%%%%%%%%%%%%%%%%%%%%%%%%%%%%%%%%%%%%%%%%%
\section{The classic coding based on special words}
\label{sec:Z0}

The notion of {\em right-special word} is an important concept for studying linear-growth complexity subshifts.
In this section, we present basic results on right-special words and the coding associated to them.
Most of these ideas are common to many works on the $\cS$-adic conjecture and related problems.
One of the new ingredients of this paper is Proposition \ref{Z0:manyshorts=>period}.

\begin{defi}
Let $X$ be a subshift. A word $w \in \cL(X)$ is called {\em right-special} if there exist two different letters $a$ and $b$ such that $wa, wb \in \cL(X)$.
We denote by $\RS_n(X)$ the set of all right special words of $X$ having length $n$.
\end{defi}
\begin{rema}
We can also define {\em left-special} words, which together with right-special words form the set of {\em special words} of $X$.
In this article, we will only use right-special words.
\end{rema}

The next proposition summarizes the facts about $\RS_n(X)$ and its relation to the complexity of $X$ that are important for us.
A {\em return word} to a clopen set $U$ is an element $w \in \cA^+$ such that there exists $x \in X$ satisfying $S^k x \in U$ if $k \in \{0, |w|\}$ and $S^k x \not\in U$ if $k \in (0,|w|)$.
\begin{prop} \label{prop:props_special_words}
Let $X \subseteq \cA^\Z$ be an aperiodic subshift and $U$ the clopen set $U = \{x \in X : x_{[0,n)} \in \RS_n(X)\}$.
\begin{enumerate}
    \item \label{item:prop:props_special_words:RS&pX}
    We have the following bounds on the number of right-special words:
    \begin{equation*}
        \frac{1}{\#\cA}(p_X(n+1) - p_X(n)) \leq \#\RS_n(X) \leq p_X(n+1) - p_X(n).
    \end{equation*}
    
    \item \label{item:prop:props_special_words:RS_is_recurrent}
    The set $U$ is $(p_X(n)+n)$-recurrent in $X$.
    
    \item The number of return words to $U$ is at most $\#\cA\cdot\#\RS_n(X)$.
\end{enumerate}
\end{prop}
\begin{proof}
A proof of Items (1), (2) and (3) can be found, with a different notation, in \cite{comb_proof_s_adicity}.
\end{proof}

We can combine Propositions \ref{prop:props_special_words} and \ref{prop:from_clopen_to_coding} to obtain the following proposition.

\begin{prop} \label{Z0:main_prop}
Suppose that $X \subseteq \cA^\Z$ is an aperiodic subshift and let $d$ be the maximum of $\lceil p_X(n)/n \rceil$, $p_X(n+1) - p_X(n)$ and $\#\cA$.
Let $(Z\subseteq \cC^\Z, \tau\colon\cC^+\to\cA^+)$ be the coding obtained from Proposition \ref{prop:from_clopen_to_coding} with $U = \{x \in X : x_{[0,n)} \in \RS_n(X)\}$.
Then:
\begin{enumerate}
    \item \label{item:Z0:main_prop:cardinalities}
    $\#\cC \leq d^3$.
    \item \label{item:Z0:main_prop:lengths}
    $|\tau(a)| \leq (d+1)n$ for all $a \in \cC$.
    \item \label{item:Z0:main_prop:reco}
    $(Z, \tau)$ is $(d+2)n$-recognizable.
    \item \label{item:Z0:main_prop:structure}
    If $x \in X$, $(c, z) = \Fac_{(Z,\tau)}(x)$ and $i \in \Z$, then $i = c_j$ for some $j \in \Z$ if and only if $x_{[i, i+n)} \in \RS(X)$.
\end{enumerate}
\end{prop}

Proposition \ref{Z0:main_prop} is the starting point of other works on the $\cS$-adic conjecture; see for example \cite{ferenczi96, Ler12_improvements}.

\begin{prop} \label{Z0:manyshorts=>period}
Let $(Z, \tau)$ be the coding in Proposition \ref{Z0:main_prop}.
Let $x \in X$ and $(c, z) = \Fac_{(Z,\tau)}(x)$ and suppose that $i,j \in \Z$ satisfy $i + d < j$ and $\ell \coloneqq \max\{ |\tau(z_k)| : k \in [i, j) \} \leq n/6d$.
Then, $\per(x_{[c_i - n/3, c_{j-d})}) \leq d\ell$.
\end{prop}
\begin{proof}
We start by noticing that, since $x_{[c_m - n, c_m)} \in \RS_n(X)$ for all $m \in \Z$ and since $\#\RS_n(X) \leq d$, we can use the Pigeonhole principle to obtain, for each $k \in [i, j-d)$, integers $p_k, q_k \in [k,k+d)$ such that $p_k < q_k$ and $x_{[c_{p_k} - n, c_{p_k})} = x_{[c_{q_k} - n, c_{q_k})}$.
These conditions imply that $\per(x_{[c_{p_k} - n, c_{q_k})}) \leq c_{q_k} - c_{p_k} \leq d\ell$.
Therefore, as $c_{p_k} - n \leq c_k + d\ell - n \leq c_k - 2n/3$ and $c_{q_k} \geq c_{k+1}$,
\begin{equation} \label{eq:Z0:manyshorts=>period:1}
    \text{$x_{[c_k - 2n/3, c_{k+1})}$ for all $k \in [i, j-d)$.}
\end{equation}
We will use \eqref{eq:Z0:manyshorts=>period:1} to prove the lemma by contradiction.
Assume that $\per(x_{[c_i - n/3, c_{j-d})}) > d\ell$.
Then, by Item \ref{item:CL:localize_per_aper:loc_aper} in Lemma \ref{CL:localize_per_aper}, there exists $m \in [c_i - n/3 + 2d\ell, c_{j-d})$ such that $\per(x_{[m - 2d\ell, m)}) > d\ell$.
Now, the condition $m \in [c_i - n/3 + 2d\ell, c_{j-d})$ allows us to find $k \in [i, j-d)$ such that $m \in [c_k-n/3, c_{k+1})$.
Hence, as $2d\ell \leq n/3$, $x_{[m-2d\ell, m)}$ occurs in $x_{[c_k-2n/3, c_{k+1})}$, which yields $\per(x_{[c_k-n/3, c_{k+1})}) \geq \per(x_{[m-2d\ell, m)}) > \varepsilon$.
This contradicts \eqref{eq:Z0:manyshorts=>period:1} and completes the proof.
\end{proof}

\section{The first coding}
\label{sec:Z1}

In this section, we begin the proof of the main results: Theorems \ref{theo:main_sublinear} and \ref{theo:main_nonsuperlinear}.
We start by constructing the codings described in Proposition \ref{Z1:main_prop}.
Then, in Sections \ref{sec:Z2}, \ref{sec:Z3}, and \ref{sec:Z4}, we will modify these codings to obtain new versions of them, each with better properties than the previous one.
We will show in Subsection \ref{subsec:Z4:two_levels} that the final codings can be connected with morphisms, and we will use this fact in Section \ref{sec:MT} to complete the proof of the main results.

\begin{prop} \label{Z1:main_prop}
Let $X$ be a minimal infinite subshift, $n \geq 1$ and let $d$ be the maximum of $\lceil p_X(n)/n \rceil$, $p_X(n+1) - p_X(n)$, $\#\cA$ and $10^4$.
Then, there exist a coding $(Z \subseteq \cC^\Z, \tau\colon\cC\to\cA^+)$ of $X$ and $\varepsilon \in [n/d^{2d^3+4}, n/d)$ satisfying the following conditions:
\begin{enumerate}
    \item \label{item:Z1:main_prop:cardinalities}
    $\cC$ has at most $d^3$ elements. % OK
    \item \label{item:Z1:main_prop:lengths}
    $|\tau(a)| \leq 3dn$ for all $a \in \cC$. % OK
    \item \label{item:Z1:main_prop:reco}
    $(Z, \tau)$ is $3dn$-recognizable. % OK
    \item \label{item:Z1:main_prop:per}
    The periodicity properties in Proposition \ref{Z1:per} are satisfied.
\end{enumerate}
\end{prop}

\begin{prop} \label{Z1:per}
Consider the coding described in Proposition \ref{Z3:main_prop}.
Let $z \in Z$, $x = \tau(z)$ and $(c,z) = \Fac_{(Z,\tau)}(x)$.
We define $\rQ{p}(z)$ as the set of integers $j \in \Z$ such that $|\mroot\tau(z_j)| \leq \varepsilon$ and $x_{[c_j-99\varepsilon, c_{j+1}+99\varepsilon)} = (\mroot\tau(z_j))^\Z_{[-99\varepsilon, |\tau(z_j)|+99\varepsilon)}$.
\begin{enumerate}
    \item \label{item:Z1:per:0ap&long=>per_int}
    $0 \not\in \rQ{p}(z)$ and $|\tau(z_0)| > 401\varepsilon$ implies that $\per(x_{[c_0+97\varepsilon, c_1-97\varepsilon)}) > \varepsilon$.
    \item\label{item:Z1:per:0ap&short&edge_p=>per_int}
    Suppose that $0 \not\in \rQ{p}(z)$ and $|\tau(z_0)| \leq 401\varepsilon$. If $-1 \in \rQ{p}(z)$ or $1 \in \rQ{p}(x)$, then $\per(x_{[c_0+97\varepsilon, c_1-97\varepsilon)}) > \varepsilon$.
    \item \label{item:Z1:per:many_short=>p}
    If $k > d$ and $|\tau(z_j)| \leq 401\varepsilon$ for all $j \in [0,k)$, then $[0,k) \subseteq \rQ{p}(z)$.
    \item \label{item:Z1:per:conj_roots}
    Let $z' \in Z$ and assume that $0 \in \rQ{p}(z)$, $0 \in \rQ{p}(z')$ and that $\mroot \tau(z_0)$ is conjugate to $\mroot \tau(z'_0)$. Then $\mroot \tau(z_0) = \mroot \tau(z'_0)$.
\end{enumerate}
\end{prop}

We fix, for the rest of the section, the following notation.
Let $X \subseteq \cA^\Z$ be a minimal infinite subshift, $n \geq 0$ and $d$ be the maximum of $p_X(n)/n$, $p_X(n+1) - p_X(n)$, $\#\cA$ and $10^4$.
We denote by $(Y \subseteq \cB^\Z, \sigma\colon\cB \to \cA^+)$ the coding given by Proposition \ref{Z0:main_prop} when it is used with $X$ and $n$.

\subsection{Construction of the first coding}

%%%%                                                                 %%%% LEMMA %%%%
\begin{lem} \label{Z1:find_a_gap}
Let $\cW$ be a finite set of words.
Then, there exists $\varepsilon \in [|\cW|/d^{2\#\cW+4}, |\cW|/d)$ such that for all $w \in \cW$, either $|w| > 10^4\varepsilon$ or $|w| \leq \varepsilon/d$.
\end{lem}
\begin{proof}
Let $d_0 = 10^4d$ and, for $\ell \in [1,\#\cW+1]$, $\cW_\ell = \{w \in \cW : |\cW|/d_0^{\ell+1} < |w| \leq |\cW|/d_0^\ell\}$.
The Pigeonhole principle ensures that $\cW_\ell$ is empty for some $\ell \in [1, \#\cW+1]$.
We set $\varepsilon = \lfloor d|\cW|/d_0^{\ell+1} \rfloor$ and note that for any $w \in \cW$, either $w \in \cup_{\ell' < \ell} \cW_{\ell'}$ and $|w| > 10^4\varepsilon$, or $w \in \cup_{\ell' > \ell} \cW_{\ell'}$ and $|w| \leq \varepsilon/d$.
Also, since $\ell \in [1, \#\cW+1]$, we have that $\varepsilon \in [|\cW|/d^{2\#\cW+4}, |\cW|/d)$.
\end{proof}

%%%%                                                                 %%%% DEFINITION %%%%
We use Lemma \ref{Z1:find_a_gap} with the set $\sigma(\cB)$ to obtain $\varepsilon \in [n/d^{2\#\cW+4}, n/d)$ such that
\begin{equation} \label{eq:Z1:gap}
    \text{for all $a \in \cB$, either $|\sigma(a)| > 10^4\varepsilon$ or $|\sigma(a)| \leq \varepsilon/d$.}
\end{equation}
Note that $\varepsilon \in [n/d^{2d^3+4}, n/d)$  as $d^3 \geq \#\sigma(\cB)$ by Item \ref{item:Z0:main_prop:cardinalities} in Proposition \ref{Z0:main_prop}.

We now define a set $\cW_\varepsilon \subseteq \cA^+$ that will be important for controlling the periodicity properties in Proposition \ref{Z1:main_prop}.
We start by introducing classic notions related to periodicity of words.
Recall that two words $u, v \in \cA^+$ are conjugate if $u r = r v$ for some $r \in \cA^*$.
The relation $u \sim_R v$ iff $u$ and $v$ are conjugate is an equivalence relation, and a $\sim_R$-equivalence class is called a {\em rotation class}.
A word $u \in \cA^+$ is {\em primitive} if $u = \mroot u $ \footnote{We recall the reader that $\mroot u$ is the shortest prefix $v$ of $u$ such that $u = v^k$ for some $k \geq 1$}.
We fix a set $\cW_{\varepsilon} \subseteq \cA^+$ consisting of one element of the rotation class of each primitive word $w \in \cA^+$ such that $|w| \leq \varepsilon$.

%%%%                                                                 %%%% LEMMA %%%%
\begin{lem} \label{Z1:loc_per=>exists_w_in_Weps}
Let $t \in \cA^+$ be such that $\per(t) \leq \varepsilon$ and $|t| \geq 198\varepsilon + \per(t)$.
Then, for some $s \in \cW_{\varepsilon}$, $s^\Z_{[-99\varepsilon, 99\varepsilon)}$ occurs in $t$.
\end{lem}
\begin{proof}
Let $u$ be the prefix of $t$ of length $\per(t)$.
Note that $u$ is primitive as otherwise $\per(t) \leq |\mroot u| < |u| = \per(t)$, which is a contradiction.
The primitiveness of $u$ and the inequality $|u| = \per(t) \leq \varepsilon$ imply that there exist $s \in \cW_\varepsilon$ and a suffix $u'$ of $s$ such that $|s| = |u|$ and $u's$ is a prefix of $uu$.
Being $\per(t) = |u| = |s|$, we then have that
\begin{equation} \label{eq:Z1:loc_per=>exists_w_in_Weps:1}
    \text{$t$ is a prefix of $u's^\infty$.}
\end{equation}
We set $k = \left\lceil \frac{99\varepsilon - |u'|} {|s|} \right\rceil$.
Observe that, since $|s| \leq \varepsilon$.
\begin{equation*}
    |u's^k| = |u'| + k|s| \leq
    \per(t) + \left\lceil \frac{99\varepsilon} {\varepsilon} \right\rceil \varepsilon =
    99\varepsilon + \per(t).
\end{equation*}
Hence, $|u's^k| + 99\varepsilon \leq |t|$.
From this and Equation \eqref{eq:Z1:loc_per=>exists_w_in_Weps:1} we deduce that if $v$ is the prefix of $s^\infty$ of length $99\varepsilon$, then $u' s^k v$ is a prefix of $t$.
Now, we have the bound
\begin{equation*}
    |u' s^k| =
    |u'| + \left\lceil \frac{99\varepsilon - |u'|} {|s|} \right\rceil |s| \geq 99\varepsilon.
\end{equation*}
Hence, $s^\Z_{[-99\varepsilon, 99\varepsilon)}$ is a suffix of $u' s^k v$.
We conclude that $s^\Z_{[-99\varepsilon, 99\varepsilon)}$ occurs in $t$.
\end{proof}

%%%%                                                                 %%%% LEMMA %%%%
\begin{lem} \label{Z1:find_pq_deco_for_w}
Let $w$ be a word of length $n$.
Then, there exists a decomposition $w = v u u' v'$ satisfying one of the following sets of conditions.
\begin{enumerate}[label=$(\alph*)$]
    \item $|u| = |u'| = 99\varepsilon$, $|v|, |v'| \geq n/2-500\varepsilon$, and $uu' = s^\Z_{[-99\varepsilon, 99\varepsilon)}$ for some $s \in \cW_\varepsilon$.
    \item $|u| = |u'| = 500\varepsilon$, $|v| = \lfloor n/2-500\varepsilon\rfloor$, $|v'| \geq n/2-500\varepsilon$, and $s^\Z_{[-99\varepsilon, 99\varepsilon)}$ does not occur in $u u'$ for all $s \in \cW_\varepsilon$.
\end{enumerate}
\end{lem}
\begin{proof}
Since $|w| \geq 2\cdot500\varepsilon$, there is a decomposition $w = v_0 t v'_0$, where $|v_0| = \lfloor n/2-500\varepsilon \rfloor$, $|v'_0| \geq n/2-500\varepsilon$ and $|t| = 2\cdot500\varepsilon$.
There are two cases:
\begin{enumerate}[label=(\roman*)]
    \item $s^\Z_{[-99\varepsilon, 99\varepsilon)}$ occurs in $t$ for some $s \in \cW_\varepsilon$.
    \item $s^\Z_{[-99\varepsilon, 99\varepsilon)}$ does not occur in $t$ for all $s \in \cW_\varepsilon$.
\end{enumerate}
Suppose first that case (i) occurs.
It is then possible to write $t = v_1 u u' v'_1$, where $uu' = s^\Z_{[-99\varepsilon, 99\varepsilon)}$ and $|u| = |u'| = 99\varepsilon$.
We set $v = v_0v_1$ and $v' = v'_1v'_0$ and note that $w = vuu'v'$ satisfies Condition $(a)$.

We now assume that (ii) holds.
Being the length of $t$ equal to $2\cdot500\varepsilon$, we can write $w_0 = u u'$, where $|u| = |u'| = 500\varepsilon$.
Then, the decomposition $w = v_0 u u v'_0$ satisfies Condition $(b)$.
\end{proof}

We now can define $(Z,\tau)$.
%%%%                                                                 %%%% DEFINITION %%%%
\begin{defi} \label{Z1:defi:descs&Ztau}
For $w \in \RS_n(X)$, we use Lemma \ref{Z1:find_pq_deco_for_w} to fix a decomposition $w = v_w u_w u'_w v'_w$ satisfying one of the following conditions:
\begin{enumerate} [label=$(\cP_\alph*)$]
    \item $|u_w| = |u'_w| = 99\varepsilon$, $|v_w|, |v'_w| \geq n/2-500\varepsilon$, and $u_w u'_w = s^\Z_{[-99\varepsilon, 99\varepsilon)}$ for some $s \in \cW_\varepsilon$.
    \item $|u_w| = |u'_w| = 500\varepsilon$, $|v_w| = \lfloor n/2-500\varepsilon\rfloor$, $|v'_w| \geq n/2-500\varepsilon$, and $s^\Z_{[-99\varepsilon, 99\varepsilon)}$ does not occur in $u_w u'_w$ for all $s \in \cW_\varepsilon$.
\end{enumerate}
Moreover, we choose this decomposition so that $|v_w u_w|$ is as small as possible.
We define $(Z \subseteq \cC^\Z, \tau\colon\cC\to\cA^+)$ as the coding of $X$ obtained from Proposition \ref{prop:from_clopen_to_coding} and the clopen set $U = \{x \in X : \exists w \in \RS_n(X), x_{[-|v_w u_w|, |u'_w v'_w|)} = w\}$.
\end{defi}

\subsection{Basic properties of the first coding}

%%%%                                                                 %%%% LEMMA %%%%
\begin{lem} \label{Z1:old_cuts&new_cuts} % NEW CUTS ALMOST PRESERVE THE ORDER
Let $x \in X$ and $i, j \in \Z$ with $i < j$.
Suppose that $x_{[i-|v_wu_w|, i+|u'_wv'_w|)} = w$ and $x_{[j-|v_{\tilde{w}}u_{\tilde{w}}|, j+|u'_{\tilde{w}}v'_{\tilde{w}}|)} = \tilde{w}$ for some $w, \tilde{w} \in \RS_n(X)$.
Then, $i + |u'_w v'_w| < j + |u'_{\tilde{w}} v'_{\tilde{w}}|$.
\end{lem}
\begin{proof}
We assume, with the aim of obtaining a contradiction, that $i + |u'_w v'_w| \geq j + |u'_{\tilde{w}} v'_{\tilde{w}}|$.

First, we consider the case $i + |u'_w v'_w| = j + |u'_{\tilde{w}} v'_{\tilde{w}}|$.
Then,
\begin{equation*}
    w = x_{[i+|u'_w v'_w|-n, i+|u'_w v'_w|)} =
    x_{[j+|u'_{\tilde{w}} v'_{\tilde{w}}|-n, j+|u'_{\tilde{w}} v'_{\tilde{w}}|)} = \tilde{w}.
\end{equation*}
Hence, $u'_wv'_w = u'_{\tilde{w}}v'_{\tilde{w}}$, and therefore
\begin{equation*}
    i = (i+|u'_wv'_w|) - |u'_wv'_w| =
    (j+|u'_{\tilde{w}}v'_{\tilde{w}}|) - |u'_{\tilde{w}}v'_{\tilde{w}}| = j.
\end{equation*}
This contradicts that $i < j$.
\medskip

Next, we assume that
\begin{equation} \label{eq:Z0:old_cuts&new_cuts:1}
    i + |u'_w v'_w| > j + |u'_{\tilde{w}} v'_{\tilde{w}}|.
\end{equation}
Note that this is equivalent to $i - |v_w u_w| > j - |v_{\tilde{w}} u_{\tilde{w}}|$.
This fact will be freely used through the proof.

We consider the following two cases:
\begin{enumerate}[label=(\roman*)]
    \item[(i)] $i + |u'_w| < j + |u'_{\tilde{w}}|$.
    \item[(ii)] $i + |u'_w| \geq j + |u'_{\tilde{w}}|$.
\end{enumerate}
Suppose first that case (i) occurs.
We are going to define a decomposition $\tilde{w} = v u u' v'$ as the one in Definition \ref{Z1:defi:descs&Ztau} and such that $|v u_w| < |v_{\tilde{w}} u_{\tilde{w}}|$.
This would contradict the minimality of $|v_{\tilde{w}} u_{\tilde{w}}|$.

We start by noting that, thanks to \eqref{eq:Z0:old_cuts&new_cuts:1}, if we set $v = x_{[j-|v_{\tilde{w}} u_{\tilde{w}}|, i-|u_w|)}$ and $v' = x_{[i+|u'_w|, j+|u'_{\tilde{w}} v'_{\tilde{w}}|)}$, then $\tilde{w} = v u_w u'_w v'$.
Note that from \eqref{eq:Z0:old_cuts&new_cuts:1} we have that
\begin{equation*}
    |v| = |x_{[j-|v_{\tilde{w}} u_{\tilde{w}}|, i-|v_w u_w|)}| + |v_w| \geq n/2 - 500\varepsilon.
\end{equation*}
Also, (i) implies that
\begin{equation*}
    |v'| = |x_{[i+|u'_w|, j+|u'_{\tilde{w}}v'_{\tilde{w}}|)}| \geq
    |x_{[j+|u'_{\tilde{w}}|, j+|u'_{\tilde{w}}v'_{\tilde{w}}|)}| = 
    |v'_{\tilde{w}}| \geq n/2 - 500\varepsilon.
\end{equation*}
We conclude, as $w = v_w u_w u'_w v'_w$ satisfies Condition $(\cP_a)$ or $(\cP_b)$ in Definition \ref{Z1:defi:descs&Ztau}, that $w = v u_w u'_w v'$ satisfies $(\cP_a)$ or $(\cP_b)$.
Moreover, since $|v u_w| = |x_{[j-|u_{\tilde{w}} v_{\tilde{w}}|, i)}|$ and $|v_{\tilde{w}} u_{\tilde{w}}| = |x_{[j-|v_{\tilde{w}} u_{\tilde{w}}|, j)}|$, we have that
\begin{equation*}
    |v_{\tilde{w}} u_{\tilde{w}}| = |v u_w| + j - i > |v u_w|.
\end{equation*}
Thus, $w = v u_w u'_w v'$ satisfies $(\cP_a)$ or $(\cP_b)$, and $|v u_w|$ is strictly smaller than $|v_{\tilde{w}} u_{\tilde{w}}|$.
This contradicts the minimality of $|v_{\tilde{w}} u_{\tilde{w}}|$.
\medskip

Next, we assume that $i + |u'_w| \geq j + |u'_{\tilde{w}}|$.
Then, as $i < j$, we have that $[j, j+|u'_{\tilde{w}}|) \subsetneq [i, i+|u'_w|)$.
This implies two things.
First, since $|u_{\tilde{w}}| = |u'_{\tilde{w}}|$ and $|u_w| = |u'_w|$, that
\begin{equation} \label{eq:Z0:old_cuts&new_cuts:2}
    [j-|u_{\tilde{w}}|, j+|u'_{\tilde{w}}|) \subsetneq [i-|u_w|, i+|u'_w|).
\end{equation}
Second, that $|u'_{\tilde{w}}| < |u'_w|$.
Being $|u'_w|, |u'_{\tilde{w}}| \in \{99\varepsilon, 500\varepsilon\}$, the last relation is possible only if
\begin{equation}
    |u'_{\tilde{w}}| = 99\varepsilon
    \enspace\text{and}\enspace
    |u'_w| = 500\varepsilon.
\end{equation}
Therefore, Condition $(\cP_a)$ holds for $\tilde{w} = v_{\tilde{w}} u_{\tilde{w}} u'_{\tilde{w}} v'_{\tilde{w}}$ and Condition $(\cP_b)$ holds for $w = v_w u_w u'_w v'_w$.
In particular, we can find $s \in \cV_{\varepsilon}$ such that $u_{\tilde{w}} u'_{\tilde{w}} = s^\Z_{[-99\varepsilon, 99\varepsilon)}$.
This implies, by \eqref{eq:Z0:old_cuts&new_cuts:2}, that $s^\Z_{[-99\varepsilon, 99\varepsilon)} = u_{\tilde{w}} u'_{\tilde{w}}$ occurs in $u_w u'_w$.
But then Condition $(\cP_b)$ cannot hold for $w = v_w u_w u'_w v'_w$, contradicting our assumptions.
\end{proof}

It is convenient to introduce some notation.
Let $x \in X$, $(c,y) = \Fac_{(Y, \sigma)}(x)$ and $(f, z) = \Fac_{(Z,\tau)}(x)$.
For $j \in \Z$, we define $w_j(x) = x_{[c_j-n, c_j)} \in \RS_n(X)$, $v_j(x) = v_{w_j(x)}$, $u_j(x) = u_{w_j(x)}$, $u'_j(x) = u'_{w_j(x)}$ and $v'_j(x) = v'_{w_j(x)}$.
Then,
\begin{equation*}
    x_{[c_j-n, c_j)} = w_j(x) = v_j(x) u_j(x) u'_j(x) v'_j(x).
\end{equation*}
Observe that if $j \in \Z$ then $x_{[f_j - |v_w u_w|, f_j + |u'_w v'_w|)} = w$ for some $w \in \RS_n(X)$, so there exists $i \in \Z$ such that $f_j + |u'_i(x) v'_i(x)| = c_i$.
We define $\phi_x(j)$ as the smallest integer such that
\begin{equation} \label{eq:Z0:def_phi}
    f_j + |u'_{\phi_x(j)}(x) v'_{\phi_x(j)}(x)| = c_{\phi_x(j)}.
\end{equation}
Then, by Lemma \ref{Z1:old_cuts&new_cuts},
\begin{equation} \label{eq:Z0:almost_preserve_order}
    \phi_x(i) < \phi_x(j)
    \enspace\text{for all $x \in X$ and $i < j$.}
\end{equation}

%%%%                                                                 %%%% LEMMA %%%%
\begin{lem} \label{Z1:middle_cuts}
Let $x \in X$, $(c,y) = \Fac_{(Y, \sigma)}(x)$ and $(f, z) = \Fac_{(Z,\tau)}(x)$.
If $i \in \Z$ and $k \in [\phi_x(i), \phi_x(i+1))$, then $f_i + |u'_k(x) v'_k(x)| = c_k$.
\end{lem}
\begin{proof}
Observe that, since $x_{[c_k-n, c_k)} = w_k(x)$, there exists $j \in \Z$ such that $f_j = c_k - |u'_k(x) v'_k(x)|$.
We are going to prove that $j = i$.

First, we note that, since $k \in [\phi_x(i), \phi_x(i+1))$ and $c_k = f_j + |u'_k(x) v'_k(x)|$,
\begin{multline} \label{eq:Z1:middle_cuts:1}
    f_i + |u'_{\phi_x(i)}(x) v'_{\phi_x(i)}(x)| =
    c_{\phi_x(i)} \leq f_j + |u'_k(x) v'_k(x)| \\ <
    c_{\phi_x(i+1)} = f_{i+1} + |u'_{\phi_x(i+1)}(x) v'_{\phi_x(i+1)}(x)|.
\end{multline}
This implies, by Lemma \ref{Z1:old_cuts&new_cuts}, that $i \leq j \leq i+1$.
Now, if $j = i+1$, then Equation \eqref{eq:Z1:middle_cuts:1} ensures that $f_{i+1} + |u'_k(x) v'_k(x)|$ is strictly smaller than $f_{i+1} + |u'_{\phi_x(i+1)}(x) v'_{\phi_x(i+1)}(x)|$, which contradicts the minimality of $\phi_x(i+1)$.
We conclude that $j = i$.
\end{proof}

%%%%                                                                 %%%% LEMMA %%%%
\begin{lem} \label{Z1:bound_cC}
    The set $\cC$ has at most $d^3$ elements.
\end{lem}
\begin{proof}
Let $x \in X$, $(c,y) = \Fac_{(Y,\sigma)}(x)$ and $(f, z) = \Fac_{(Z,\tau)}(x)$.
We drop the dependency on $x$ in $\phi_x$ and just write $\phi$.
The lemma follows from the following claim.
\begin{enumerate}
    \item[$(\bullet)$] For $j \in \Z$, let $\zeta(j) = (w_{{\phi(j+1)-1}}, x_{c_{\phi(j+1)-1}}) \in \RS_n(X) \times \cA$.
    Then, $\zeta(i) = \zeta(j)$ implies that $x_{[f_i, f_{i+1})} = x_{[f_j, f_{j+1})}$.
\end{enumerate}
Indeed, being $Z$ minimal (as $X$ is minimal and $(Z,\tau)$ is recognizable), $(\bullet)$ implies that
\begin{equation*}
    \#\tau(\cC) = \#\{x_{[f_j, f_{j+1})} : j \in \Z\} \leq
    \#\RS_n(X) \cdot \#\cA \leq d^3,
\end{equation*}
where we used that $\#\RS_n(X) \leq \#\cA\cdot(p_X(n+1)-p_X(n))$ by Item \ref{item:prop:props_special_words:RS&pX} in Proposition \ref{prop:props_special_words}.
This implies, as $\tau$ is injective on letters by Proposition \ref{prop:from_clopen_to_coding}, that $\#\cC = \#\tau(\cC) \leq d^3$.

Let us prove the claim.
Suppose that $i, j \in \Z$ satisfy $\zeta(i) = \zeta(j) = (w, a)$.
We start with some observations.
First, the condition $\zeta(i) = \zeta(j) = (w,a)$ implies that
\begin{enumerate}[label=$(\roman*)$]
    \item $w = x_{[c_{\phi(j+1)-1}-n, c_{\phi(j+1)-1})} = x_{[c_{\phi(i+1)-1}-n, c_{\phi(i+1)-1})}$; and
    \item $a = x_{c_{\phi(j+1)-1}} = x_{c_{\phi(i+1)-1}}$.
\end{enumerate}
Also, Equation \eqref{eq:Z0:almost_preserve_order} ensures that $\phi(i) < \phi(i+1)$ and $\phi(j) < \phi(j+1)$, so
\begin{equation} \label{eq:Z1:bound_cC:1}
    \phi(i) \leq \phi(i+1)-1 < \phi(i+1)
    \enspace\text{and}\enspace
    \phi(j) \leq \phi(j+1)-1 < \phi(j+1).
\end{equation}

We now prove the claim $(\bullet)$.
The definition of $c_{\phi(j+1)-1}$ and $c_{\phi(j+1)}$ guarantees that the words $x_{[k-n, k)}$, $k \in (c_{\phi(j+1)-1}, c_{\phi(j+1)})$, are not right-special.
Thus, $x_{[c_{\phi(j+1)-1}, c_{\phi(j+1)})}$ is determined by $x_{[c_{\phi(j+1)-1}-n, c_{\phi(j+1)-1})}$ and $x_{c_{\phi(j+1)-1}}$.
A similar observation holds for $x_{[c_{\phi(i+1)-1}, c_{\phi(i+1)})}$.
Combining these two things with (i) and (ii) yields that
\begin{equation} \label{eq:Z1:bound_cC:2}
    x_{[c_{\phi(j+1)-1}, c_{\phi(j+1)})} =
    x_{[c_{\phi(i+1)-1}, c_{\phi(i+1)})}.
\end{equation}
Then, by (i),
\begin{equation*}
    w_{\phi(j+1)}(x) = 
    x_{[c_{\phi(j+1)}-n, c_{\phi(j+1)})} = 
    x_{[c_{\phi(i+1)}-n, c_{\phi(i+1)})} = w_{\phi(i+1)}(x).
\end{equation*}
Let us write $\tilde{w} = w_{\phi(j+1)}(x) = w_{\phi(i+1)}(x)$.
With this notation, we have, by \eqref{eq:Z0:def_phi}, that
\begin{equation} \label{eq:Z1:bound_cC:3}
    x_{[f_{j+1}, c_{\phi(j+1)})} = 
    x_{[f_{i+1}, c_{\phi(i+1)})} = 
    u'_{\tilde{w}} v'_{\tilde{w}}.
\end{equation}
Now, Equation \eqref{eq:Z1:bound_cC:1} allows us to use Lemma \ref{Z1:middle_cuts} with $\phi(i+1)-1$ and $\phi(j+1)-1$; we deduce, as $w = w_{\phi(j+1)-1} = w_{\phi(i+1)-1}$, that
\begin{equation*}
    f_j + |u'_w v'_w| = c_{\phi(j+1)-1}
    \enspace\text{and}\enspace
    f_i + |u'_w v'_w| = c_{\phi(i+1)-1}.
\end{equation*}
In particular,
\begin{equation*}
    x_{[f_j, c_{\phi(j+1)-1})} =
    x_{[f_i, c_{\phi(i+1)-1})} = u'_w v'_w.
\end{equation*}
This and Equation \eqref{eq:Z1:bound_cC:2} then give that
\begin{equation*}
    x_{[f_j, c_{\phi(j+1)})} = 
    x_{[f_i, c_{\phi(i+1)})}.
\end{equation*}
We conclude using \eqref{eq:Z1:bound_cC:3} that $x_{[f_j, f_{j+1})} = x_{[f_i, f_{i+1})}$.
This completes the proof of the claim and thereby the proof of the lemma.
\end{proof}

%%%%                                                                 %%%% LEMMA %%%%
\begin{lem} \label{Z1:basic_bounds_lengths}
Let $x \in X$, $(c,y) = \Fac_{(Y,\sigma)}(x)$ and $(f, z) = \Fac_{(Z,\tau)}(x)$.
Then:
\begin{enumerate}
    \item $|\sigma(y_k)| \leq |\tau(z_0)| + 2\cdot401\varepsilon$ for any $k \in [\phi_x(0), \phi_x(1))$.
    \item $|\tau(z_0)| \leq |\sigma(y_{\phi(1)-1})| + 2\cdot401\varepsilon$.
\end{enumerate}
\end{lem}
\begin{proof}
We write, for simplicity, $\phi = \phi_x$.
Let $k \in [\phi(0), \phi(1))$.
Then, by Lemma \ref{Z1:middle_cuts}, $f_0 + |u'_k(x) v'_k(x)| = c_k$.
Hence,
\begin{multline*} 
    \tau(z_0) \cdot u'_{\phi(1)}(x) v'_{\phi(1)}(x) =
    x_{[f_0, f_1)} \cdot x_{[f_1, c_{\phi(1)})} \\ = 
    x_{[f_0, c_k)} \cdot x_{[c_k, c_{\phi(1)})} =
    u'_k(x) v'_k(x) \cdot \sigma(y_{[k, \phi(1))}).
\end{multline*} 
In particular,
\begin{equation} \label{eq:Z1:basic_bounds_lengths:0} 
    \big| |\tau(z_0)| - |\sigma(y_{[k,\phi(1))})| \big| =
    \big| |u'_{\phi(1)}(x) v'_{\phi(1)}(x)| - |u'_k(x) v'_k(x)| \big|.
\end{equation} 
Now, Conditions $(\cP_a)$ and $(\cP_b)$ in Definition \ref{Z1:defi:descs&Ztau} ensure that for any $w \in \RS_n(X)$ the inequalities $n/2-401\varepsilon \leq |u'_w v'_w| \leq n/2 + 401\varepsilon$ hold.
Putting this in \eqref{eq:Z1:basic_bounds_lengths:0} produces
\begin{equation} \label{eq:Z1:basic_bounds_lengths:1} 
    \big| |\tau(z_0)| - |\sigma(y_{[k,\phi(1))})| \big| \leq
    2\cdot401\varepsilon
    \enspace\text{for all $k \in [\phi(0), \phi(1))$.}
\end{equation} 
Item (1) of this lemma follows.
Moreover, since $y_{\phi(1)-1} = y_{[\phi(1)-1,\phi(1))}$, Item (2) is also a consequence of \eqref{eq:Z1:basic_bounds_lengths:1}.
\end{proof}

\subsection{Proof of Propositions \ref{Z1:main_prop} and \ref{Z1:per}}

%%%%                                                                 %%%% MAIN PROOF %%%%
We now prove Proposition \ref{Z1:main_prop}.
\begin{proof}[Proof of Proposition \ref{Z1:main_prop}]
Item (1) follows directly from Lemma \ref{Z1:bound_cC}.

Let us prove Items (2) and (3).
We define $U' = \{x \in X: x_{[0,n)} \in \RS_n(X)\}$ and $U = \{x \in X : \exists w \in \RS_n(X), x_{[-|u_wv_w|, |u'_wv'_w|)} = w\}$.
First, we recall that $(Z,\tau)$ is defined as the coding of $X$ obtained from $U$ as in Proposition \ref{prop:from_clopen_to_coding}.
Observe that, since $|v_wu_w| \leq |w| = n$ and $|u'_wv'_w| \leq |w| = n$ for all $w \in \RS_n(X)$, $U$ has radius $n$.
Also, Item \ref{item:prop:props_special_words:RS_is_recurrent} in Proposition \ref{prop:props_special_words} ensures that $U'$ is $(d+1)n$-syndetic, and thus that $U$ is $(d+3)n$-syndetic.
Therefore, Proposition \ref{prop:from_clopen_to_coding} ensures that $|\tau(a)| \leq (d+3)n$ for all $a \in \cC$ and that $(Z,\tau)$ is $(d+4)n$-recognizable.
Since $d \geq \#\cA \geq 2$, Items (2) and (3) follow.
\end{proof}

The rest of the section is devoted to prove Proposition \ref{Z1:per}.

%%%%                                                                 %%%% LEMMA %%%%
\begin{lem} \label{Z1:rQp<=>good_root}
Let $x \in X$ and $(c,z) = \Fac_{(Z,\tau)}(x)$.
We use the notation $w = w_{\phi_x(0)}$ and $\tilde{w} = w_{\phi_x(1)}$.
Then, the following are equivalent:
\begin{enumerate}
    \item $|\mroot \tau(z_0) | \leq \varepsilon$ and $x_{[c_0-99\varepsilon, c_1+99\varepsilon)} = (\mroot \tau(z_0))^\Z_{[-99\varepsilon, |\tau(z_0)| + 99\varepsilon)}$.
    \item The decompositions $w = v_w u_w u'_w v'_w$ and $\tilde{w} = v_{\tilde{w}} u_{\tilde{w}} u'_{\tilde{w}} v'_{\tilde{w}}$ satisfy Condition $(\cP_a)$ in Definition \ref{Z1:defi:descs&Ztau} and $\per(x_{[c_0+97\varepsilon, c_1-97\varepsilon)}) \leq \varepsilon$.
\end{enumerate}
Moreover, if any of the previous condition holds, then $\mroot \tau(z_0) \in \cW_\varepsilon$.
\end{lem}
\begin{proof}
We assume that Item (1) holds.
Let $s = \mroot \tau(z_0)$ and note that Item (1) ensures that 
\begin{equation} \label{eq:Z1:rQp&good_root:4}
    |s| \leq \varepsilon \enspace \text{and} \enspace
    x_{[c_0 - 99\varepsilon, c_1 + 99\varepsilon)} = 
    s^\Z_{[-99\varepsilon, |\tau(z_0)| + 99\varepsilon)}.
\end{equation}
This allows us to use Lemma \ref{Z1:loc_per=>exists_w_in_Weps} with $x_{[c_0-99\varepsilon, c_0+|s|+99\varepsilon)}$ and find $t \in \cW_\varepsilon$ such that $t^\Z_{[-99\varepsilon, 99\varepsilon)}$ occurs in $x_{[c_0-99\varepsilon, c_0+|s|+99\varepsilon)}$.
Since $|s| \leq \varepsilon$, we have in particular that $t^\Z_{[-99\varepsilon, 99\varepsilon)}$ occurs in $x_{[c_0-500\varepsilon, c_0+9\varepsilon)}$.
This is incompatible with the decomposition $w = v_w u_w u'_w v'_w$ satisfying Condition $(\cP_b)$ in Definition \ref{Z1:defi:descs&Ztau}; therefore, $w = v_w u_w u'_w v'_w$ satisfies Condition $(\cP_a)$.
A similar argument shows that $\tilde{w} = v_{\tilde{w}} u_{\tilde{w}} u'_{\tilde{w}} v'_{\tilde{w}}$ also satisfies Condition $(\cP_a)$.
Finally, it follows from \eqref{eq:Z1:rQp&good_root:4} that $\per(x_{[c_0 + 97\varepsilon, c_1 - 97\varepsilon)}) \leq |s| \leq \varepsilon$.
\medskip

We assume that Item (2) holds.
Then, by Definition \ref{Z1:defi:descs&Ztau}, there exist $s, \tilde{s} \in \cW_\varepsilon$ such that
\begin{equation} \label{eq:Z1:rQp&good_root:1}
    s^\Z_{[-99\varepsilon, 99\varepsilon)} = u_w u'_w = 
    x_{[c_0 - 99\varepsilon, c_0 + 99\varepsilon)}
    \enspace\text{and}\enspace
    \tilde{s}^\Z_{[-99\varepsilon, 99\varepsilon)} = u_{\tilde{w}} u'_{\tilde{w}} =
    x_{[c_1 - 99\varepsilon, c_1 + 99\varepsilon)}.
\end{equation}

We claim that
\begin{equation} \label{eq:Z1:rQp&good_root:2}
    \per(x_{[c_0 - 99\varepsilon, c_1 + 99\varepsilon)}) \leq \varepsilon.
\end{equation}
Assume, with the objective of obtaining a contradiction, that \eqref{eq:Z1:rQp&good_root:2} is not satisfied.
Then, Item \ref{item:CL:localize_per_aper:loc_aper} in Proposition \ref{CL:localize_per_aper} gives $i \in [c_0 - 98\varepsilon, c_1 + 98\varepsilon)$ such that $\per(x_{[i-\varepsilon, i+\varepsilon)}) > \varepsilon$.
We consider three cases.
If $i \in [c_0 - 98\varepsilon, c_0 + 98\varepsilon)$, then $x_{[i-\varepsilon, i+\varepsilon)}$ occurs in $u_w u'_w$.
Thus, by \eqref{eq:Z1:rQp&good_root:1} and since $s \in \cW_{\varepsilon}$ implies that $|s| \leq \varepsilon$, $\per(x_{[i-\varepsilon, i+\varepsilon)}) \leq |s| \leq \varepsilon$.
This contradicts our assumptions.
In the case $i \in [c_1 - 98\varepsilon, c_1 + 98\varepsilon)$, a similar argument gives a contradiction.
Finally, if $i \in [c_0 + 98\varepsilon, c_1 - 98\varepsilon)$, then $x_{[i-\varepsilon, i+\varepsilon)}$ occurs in $x_{[c_0 + 97\varepsilon, c_1 - 97\varepsilon)}$ and thus, by the hypothesis, $\per(x_{[i-\varepsilon, i+\varepsilon)}) \leq \varepsilon$.
This proves \eqref{eq:Z1:rQp&good_root:2}.

Our next objective is to use the claim for proving that 
\begin{equation} \label{eq:Z1:rQp&good_root:2.25}
    |s| = |\tilde{s}| = \per(x_{[c_0 - 99\varepsilon, c_1 + 99\varepsilon)}).
\end{equation}
We note that $|s| \leq \varepsilon$ as $s \in \cW_\varepsilon$.
Hence, Equations \eqref{eq:Z1:rQp&good_root:1} and \eqref{eq:Z1:rQp&good_root:2} allows us to use Item \ref{item:CL:localize_per_aper:loc_per} of Proposition \ref{CL:localize_per_aper} and deduce that $\per(x_{[c_0 - 99\varepsilon, c_1 + 99\varepsilon)})$ is equal to $\per(x_{[c_0 - 99\varepsilon, c_0 + 99\varepsilon)})$.
Also, \eqref{eq:Z1:rQp&good_root:2} ensures that $\per(x_{[c_0 - 99\varepsilon, c_0 + 99\varepsilon)}) \leq |s|$, so by Item \ref{item:CL:localize_per_aper:loc_per} in Proposition \ref{CL:localize_per_aper} we have that $\per(x_{[c_0 - 99\varepsilon, c_0 + 99\varepsilon)})$ is equal to $\per(s^2)$.
Moreover, by Lemma \ref{CL:root=per}, $\per(s^2) = |\mroot s^2| = |s|$.
Combining all these relations produces
\begin{equation*}
    \per(x_{[c_0 - 99\varepsilon, c_1 + 99\varepsilon)}) =
    \per(x_{[c_0 - 99\varepsilon, c_0 + 99\varepsilon)}) = 
    \per(s^2) = |s|.
\end{equation*}
Similarly, $\per(x_{[c_0 - 99\varepsilon, c_1 + 99\varepsilon)}) = |\tilde{s}|$.
Equation \eqref{eq:Z1:rQp&good_root:2.25} follows.
\medskip

We combine \eqref{eq:Z1:rQp&good_root:2.25} with \eqref{eq:Z1:rQp&good_root:1} to obtain that
\begin{equation} \label{eq:Z1:rQp&good_root:2.5}
    s^\Z_{[-99\varepsilon, |\tau(z_0)| + 99\varepsilon)} = 
    \tilde{s}^\Z_{[-|\tau(z_0)| - 99\varepsilon, 99\varepsilon)} = 
    x_{[c_0 - 99\varepsilon, c_1 + 99\varepsilon)}.
\end{equation}
Being $|s|$ equal to $|\tilde{s}|$, we get that $s$ and $\tilde{s}$ are conjugate.
Moreover, since $s, \tilde{s} \in \cW_\varepsilon$ and since $\cW_\varepsilon$ contains at most one element of a given rotational class, we have that
\begin{equation} \label{eq:Z1:rQp&good_root:3}
    s = \tilde{s}.
\end{equation}
We use \eqref{eq:Z1:rQp&good_root:3} to prove that Item (1) of the lemma holds.
Observe that Equation \eqref{eq:Z1:rQp&good_root:1} and \eqref{eq:Z1:rQp&good_root:3} imply that $s^\Z_{[-99\varepsilon, 99\varepsilon)} = (S^{|\tau(z_0)|} s)_{[-99\varepsilon, 99\varepsilon)}$.
This and the fact that $|s| \leq \varepsilon$ (as $s \in \cW_\varepsilon$) allow us to use Item \ref{item:CL:on_sZ&tZ:equal} in Proposition \ref{CL:on_sZ&tZ} and deduce that $s^\Z = S^{|\tau(z_0)|} s^\Z$.
Item \ref{item:CL:on_sZ&tZ:equal_mod} of Proposition \ref{CL:on_sZ&tZ} then gives that $|\tau(z_0)| = 0 \pmod{|s|}$.
We conclude that $x_{[c_0, c_1)}$ is a power of $s$ and that $\mroot\tau(z_0) = \mroot s = s$.
Item (1) of this lemma is a consequence of the last relation and \eqref{eq:Z1:rQp&good_root:2.5}.
This also shows that if Item (2) of the lemma holds, then $\mroot \tau(z_0) \in \cW_\varepsilon$.
\end{proof}

%%%%                                                                 %%%% LEMMA %%%%
\begin{lem} \label{Z1:many_shorts=>nice_period}
Let $x \in X$, $(f, z) = \Fac_{(Z,\tau)}(x)$ and $i, j \in \Z$ with $j > i + d$.
Suppose that $|\tau(z_k)| \leq 401\varepsilon$ for all $k \in [i, j)$.
Then:
\begin{enumerate}
    \item $\mroot\tau(z_k) = \mroot\tau(z_i)$ for all $k \in [i, j)$ and $|\mroot\tau(z_i)| \leq \varepsilon$.
    \item $x_{[f_i - 99\varepsilon, f_j + 99\varepsilon)} = (\mroot\tau(z_i))^\Z_{[-99\varepsilon, |\tau(z_{[i, j)})| + 99\varepsilon)}$.
\end{enumerate}
\end{lem}
\begin{proof}
Let $(c, y) = \Fac_{(Y,\sigma)}(x)$.
We will use Lemma \ref{Z0:manyshorts=>period} with $y$ and $[\phi_x(i), \phi_x(j))$ to prove the following:
\begin{equation} \label{eq:Z1:many_shorts=>nice_period:bullet}
    \text{$\per(x_{[f_i-500\varepsilon, f_j+500\varepsilon)})$ is at most $\varepsilon$.}
\end{equation}
Let us check the hypothesis of Lemma \ref{Z0:manyshorts=>period}.
Let $k \in [\phi_x(i), \phi_x(j))$ be arbitrary.
There exists $\ell \in [i, j)$ such that $k \in [\phi_x(\ell), \phi_x(\ell+1))$.
Putting the hypothesis $|\tau(z_k)| \leq 401\varepsilon$ in the inequality of Lemma \ref{Z1:basic_bounds_lengths} produces the bound $|\sigma(y_k)| \leq |\tau(z_\ell)| + 2\cdot401\varepsilon \leq 10^4 \varepsilon$.
Hence, by \eqref{eq:Z1:gap},
\begin{equation} \label{eq:Z1:many_shorts=>nice_period:0}
    |\sigma(y_k)| < \varepsilon/d
    \enspace\text{for all $k \in [\phi_x(i), \phi_x(j))$.}
\end{equation}
Since $\varepsilon \leq n/10^4$, Equation \eqref{eq:Z1:many_shorts=>nice_period:0} and the inequalities $\phi_x(j) \geq (j-i) + \phi_x(i) > d + \phi_x(i)$ allow us to use Lemma \ref{Z0:manyshorts=>period} and deduce that $\per(x_{[c_{\phi_x(i)}-n/3, c_{\phi_x(j)-d})}) \leq \varepsilon$.
Now, observe that, for any $k \in \Z$,
\begin{equation*}
    c_{\phi_x(k)} - f_k  = |u'_{\phi_x(k)} v'_{\phi_x(k)}| \in [n/2-401\varepsilon, n/2+401\varepsilon)
    % \enspace\text{for any $k \in \Z$.}
\end{equation*}
Hence, as \eqref{eq:Z1:many_shorts=>nice_period:0} ensures that $c_{\phi_x(j)-d} \geq c_{\phi(j)} - \varepsilon$ and since $\varepsilon \leq n/10^4$, we have that $x_{[f_i-500\varepsilon, f_j+500\varepsilon)}$ occurs in $x_{[c_{\phi_x(i)}-n/3, c_{\phi_x(j)-d})}$.
Therefore, \eqref{eq:Z1:many_shorts=>nice_period:bullet} holds.
\medskip

Next, we use \eqref{eq:Z1:many_shorts=>nice_period:bullet} to prove the following:
\begin{multline} \label{eq:Z1:many_shorts=>nice_period:bullet_p}
    \text{$\forall k \in [i, j]$, the decomposition $w_{\phi_x(k)} = v_{\phi_x(k)} u_{\phi_x(k)} u'_{\phi_x(k)} v'_{\phi_x(k)}$ }\\\text{ satisfies $(\cP_a)$ in Definition \ref{Z1:defi:descs&Ztau}.}
\end{multline}
Let $k \in [i, j]$.
We note that \eqref{eq:Z1:many_shorts=>nice_period:bullet} implies that $\per(x_{[c_k - 500\varepsilon, c_k + 500\varepsilon)}) \leq \varepsilon$.
Thus, by Lemma \ref{Z1:loc_per=>exists_w_in_Weps}, there exists $s \in \cW_\varepsilon$ such that $s^\Z_{[-99\varepsilon, 99\varepsilon)}$ occurs in $x_{[c_k - 500\varepsilon, c_k + 500\varepsilon)}$.
This implies that if $(\cP_b)$ in Definition \ref{Z1:defi:descs&Ztau} holds for the decomposition $w_{\phi_x(k)} = v_{\phi_x(k)} u_{\phi_x(k)} u'_{\phi_x(k)} v'_{\phi_x(k)}$, then $s^\Z_{[-99\varepsilon, 99\varepsilon)}$ occurs in $u_{\phi_x(k)} u'_{\phi_x(k)} = x_{[c_k - 500\varepsilon, c_k + 500\varepsilon)}$, contradicting $(\cP_b)$.
Therefore, $w_{\phi_x(k)} = v_{\phi_x(k)} u_{\phi_x(k)} u'_{\phi_x(k)} v'_{\phi_x(k)}$ satisfies $(\cP_a)$ and \eqref{eq:Z1:many_shorts=>nice_period:bullet_p} is proved.
\medskip 

We now prove the properties in the statement of the lemma.
Let $k \in [i, j)$.
Then, Equations \eqref{eq:Z1:many_shorts=>nice_period:bullet} and \eqref{eq:Z1:many_shorts=>nice_period:bullet_p} imply that Item (1) in Lemma \ref{Z1:rQp<=>good_root} is satisfied.
Hence, for all $ k \in [i,j)$, 
\begin{equation} \label{eq:Z1:many_shorts=>nice_period:1}
    |\mroot\tau(z_k)| \leq \varepsilon
    \enspace\text{and}\enspace
    x_{[c_k-99\varepsilon, c_{k+1}+99\varepsilon)} = (\mroot\tau(z_k))^\Z_{[-99\varepsilon, |\tau(z_k)| + 99\varepsilon)}.
\end{equation}
In particular, we have for every $k \in [i, j-1)$ that 
\begin{equation*}
    (\mroot \tau(z_k))^\Z_{[0, 99\varepsilon)} = 
    x_{[c_k, c_k + 99\varepsilon)} =
    x_{[c_{k+1}, c_{k+1} + 99\varepsilon)} =
    (\mroot \tau(z_{k+1}))^\Z_{[0, 99\varepsilon)}.
\end{equation*}
This and the inequalities $|\mroot \tau(z_k)| \leq \varepsilon$ and $|\mroot \tau(z_{k+1})| \leq \varepsilon$ allow us to use Theorem \ref{theo:fine&wilf} to deduce that $\mroot \tau(z_k)$ and $\mroot \tau(z_{k+1})$ are powers of a common word, and thus that $\mroot \tau(z_k) = \mroot \tau(z_{k+1})$.
And inductive argument then yields Item (1) of this lemma, and therefore, by \eqref{eq:Z1:many_shorts=>nice_period:1}, that Item (2) holds as well.
\end{proof}

We have all the necessary elements to prove Proposition \ref{Z1:per}.
%%%%                                                                 %%%% MAIN PROOF %%%%
\begin{proof}[Proof of Proposition \ref{Z1:per}]
We prove Item (1) by contradiction,
Suppose that $0 \not\in \rQ{p}(z)$, $|\tau(z_0)| > 401\varepsilon$ and that $\per(x_{[c_0+97\varepsilon, c_1-97\varepsilon)})$ is at most $\varepsilon$.
Let us write $w = w_{\phi_x(0)}$ and $\tilde{w} = w_{\phi_x(1)}$.
Then, the condition  $0 \not\in \rQ{p}(z)$ ensures that Item (1) in Lemma \ref{Z1:rQp<=>good_root} does not hold.
Hence, Item (2) does not hold either.
This implies, as $\per(x_{[c_0+97\varepsilon, c_1-97\varepsilon)})$ at most $\varepsilon$, that one of the decompositions $w = v_w u_w u'_w v'_w$ or $\tilde{w} = v_{\tilde{w}} u_{\tilde{w}} u'_{\tilde{w}} v'_{\tilde{w}}$ satisfies $(\cP_b)$ in Definition \ref{Z1:defi:descs&Ztau}.
We assume, without loss of generality, that  $w = v_w u_w u'_w v'_w$ satisfies $(\cP_b)$.
Then, for any $s \in \cW_\varepsilon$, $s^\Z_{[-99\varepsilon, 99\varepsilon)}$ does not occur in $u_w u'_w = x_{[c_0 - 401\varepsilon, c_0 + 401\varepsilon)}$.
In particular, $s^\Z_{[-99\varepsilon, 99\varepsilon)}$ does not occur in $x_{[c_0 + 97\varepsilon, c_0 + 304\varepsilon)}$.
Being this valid for all $s \in \cW_\varepsilon$ and since $x_{[c_0 + 97\varepsilon, c_0 + 304\varepsilon)}$ has length at least $2\varepsilon$, we deduce from Lemma \ref{Z1:loc_per=>exists_w_in_Weps} that $\per(x_{[c_0 + 97\varepsilon, c_0 + 304\varepsilon)}) > \varepsilon$.
But $c_1 - c_0 = |\tau(z_0)| > 401\varepsilon$ so $x_{[c_0 + 97\varepsilon, c_0 + 304\varepsilon)}$ occurs in $x_{[c_0 + 97\varepsilon, c_1 - 97\varepsilon)}$ and thus $\per(x_{[c_0 + 97\varepsilon, c_1 - 97\varepsilon)}) > \varepsilon$.
This contradicts our assumptions and thereby proves Item (1).
\medskip

We continue with Item (2).
The proof is by contradiction.
We assume that the hypothesis of Item (2) holds and that $\per(x_{[c_0+97\varepsilon, c_1-97\varepsilon)}) \leq \varepsilon$.
Let us further assume, without losing generality, that $1 \in \rQ{p}(z)$.
We will use the notation $w = w_{\phi_x(0)}$ and $\tilde{w} = w_{\phi_x(1)}$.
Then, the condition $1 \in \rQ{p}(z)$ is equivalent to Item (1) of Lemma \ref{Z1:rQp<=>good_root} being satisfied by $Sz$; hence, Item (2) of that lemma holds with $Sz$.
In particular, $\tilde{w} = v_{\tilde{w}} u_{\tilde{w}} u'_{\tilde{w}} v'_{\tilde{w}}$ satisfies $(\cP_a)$ in Definition \ref{Z1:defi:descs&Ztau}, that is,
\begin{equation*}
    \text{$x_{[c_1-99\varepsilon, c_1+99\varepsilon)} = u_{\tilde{w}} u'_{\tilde{w}} = s^\Z_{[-99\varepsilon, 99\varepsilon)}$ for some $s \in \cW_\varepsilon$.}
\end{equation*}
Now, the condition $0 \not\in \rQ{p}(z)$ implies, by Lemma \ref{Z1:rQp<=>good_root}, that Item (2) of that lemma is not satisfied by $z$.
This implies, since $\tilde{w} = v_{\tilde{w}} u_{\tilde{w}} u'_{\tilde{w}} v'_{\tilde{w}}$ satisfies $(\cP_a)$ and since we assumed that $\per(x_{[c_0+97\varepsilon, c_1-97\varepsilon)}) \leq \varepsilon$, that $w = v_w u_w u'_w v'_w$ satisfies $(\cP_b)$.
Therefore,
\begin{equation} \label{eq:Z1:per:1}
    \text{$x_{[c_1-99\varepsilon, c_1+99\varepsilon)} = s^\Z_{[-99\varepsilon, 99\varepsilon)}$ does not occur in $x_{[c_0-500\varepsilon, c_0+500\varepsilon)}$.}
\end{equation}
But, since $c_1 - c_0 = |\tau(z_0)| \leq 401\varepsilon$, we have that $[c_1-99\varepsilon, c_1+99\varepsilon)$ is contained in $[c_0-500\varepsilon, c_1+500\varepsilon)$, and thus that $x_{[c_1-99\varepsilon, c_1+99\varepsilon)}$ occurs in $x_{[c_0-500\varepsilon, c_1+500\varepsilon)}$.
This contradicts \eqref{eq:Z1:per:1}, finishing the proof of Item (2).
\medskip

Next, we consider Item (3).
Assume that $k > d$ and that $|\tau(z_j)| \leq 401\varepsilon$ for all $j \in [0,k)$.
Then, we can use Lemma \ref{Z1:many_shorts=>nice_period} to deduce that 
\begin{enumerate}
    \item $\mroot\tau(z_j) = \mroot\tau(z_0)$ for all $j \in [0, k)$ and $|\mroot\tau(z_0)| \leq \varepsilon$;
    \item $x_{[f_0 - 99\varepsilon, f_k + 99\varepsilon)} = (\mroot\tau(z_0))^\Z_{[-99\varepsilon, |\tau(z_{[0, k)})| + 99\varepsilon)}$.
\end{enumerate}
In particular, $x_{[c_j - 99\varepsilon, c_{j+1} + 99\varepsilon)} = (\mroot\tau(z_j))^\Z_{[-99\varepsilon, |\tau(z_j)| + 99\varepsilon)}$ and $|\mroot\tau(z_j)| \leq \varepsilon$ for all $j \in [0, k)$.
We conclude that $[0,k) \subseteq \rQ{p}(z)$.

Finally, we prove Item (4).
Let $z' \in Z$ and assume that $0 \in \rQ{p}(z) \cap \rQ{p}(z')$ and that $\mroot \tau(z_0)$ is conjugate to $\mroot \tau(z'_0)$.
The condition $0 \in \rQ{p}(z) \cap \rQ{p}(z')$ permits to use Lemma \ref{Z1:rQp<=>good_root} to deduce that $\mroot \tau(z_0)$ and $\mroot \tau(z'_0)$ belong to $\cW_\varepsilon$.
Since $\mroot \tau(z_0)$ conjugate to $\mroot \tau(z'_0)$, the definition of $\cW_\varepsilon$ ensures that $\mroot \tau(z_0) = \mroot \tau(z'_0)$.
\end{proof}

\section{The second coding}
\label{sec:Z2}

We continue the proof of the main theorems.
The main result of this section is Proposition \ref{Z2:main_prop}, which describes a modification of the coding in Proposition \ref{sec:Z1}.
The principal new element in Proposition \ref{Z2:main_prop} is a period dichotomy for the words $\tau(a)$. 
This property is shared by the codings constructed in Sections \ref{sec:Z3} and \ref{sec:Z4}, so we introduce it as a definition. 

\begin{defi} \label{defi:per}
Let $(Z \subseteq \cC^\Z, \tau\colon\cC \to \cA^+)$ be a recognizable coding of the subshift $X \subseteq \cA^\Z$, $\A{C}{ap} \cup \A{C}{p}$ be a partition of $\cC$, and $\varepsilon \geq 1$.
We say that $(Z,\tau)$ has {\em dichotomous periods} w.r.t.\ $(\A{C}{ap},\A{C}{p})$ and $\varepsilon$ if for $x \in X$ and $(c, z) = \Fac_{(Z,\tau)}(x)$ the following holds:
\begin{enumerate}
    \item \label{item:defi:per:ap}
    $z_0 \in \A{C}{ap}$ implies that $\per(x_{[c_0+\varepsilon, c_1-\varepsilon)}) > \varepsilon$.
    \item \label{item:defi:per:p}
    $z_0 \in \A{C}{p}$ implies that $|\mroot\tau(z_0)| \leq \varepsilon$ and that $x_{[c_0-\varepsilon, c_1+\varepsilon)}$ is equal to $(\mroot\tau(z_0))^\Z_{[-\varepsilon, |\tau(z_0)|+\varepsilon)}$.
    \item \label{item:defi:per:conj_roots}
    If $a \in \A{C}{p}$ and $\mroot \tau(z_0)$ is conjugate to $\mroot \tau(a)$, then $\mroot \tau(z_0) = \mroot \tau(a)$.
\end{enumerate}
\end{defi}

%%%%                                                                 %%%% PROPOSITION %%%%
\begin{prop} \label{Z2:main_prop}
Let $X$ be a minimal infinite subshift, $n \geq 0$ and let $d$ be the maximum of $\lceil p_X(n)/n \rceil$, $p_X(n+1) - p_X(n)$, $\#\cA$ and $10^4$.
There exist a recognizable coding $(Z \subseteq \cC^\Z, \tau\colon\cC\to\cA^+)$ of $X$, a partition $\cC = \A{C}{ap} \cup \A{C}{p}$, and $\varepsilon \in [n/d^{2d^3+4}, n/d)$ such that:
\begin{enumerate}
    \item \label{item:Z2:cardinalities}
    $\#\A{C}{ap} \leq 2d^{3d+6}$, $\#\A{C}{p} \leq 2d^{3d+9}\pcom(X)$ and $\#\mroot\tau(\cC) \leq 3d^{3d+9}$. 
    \item \label{item:Z2:lengths}
    $|\tau(a)| \leq 10d^2n$ for $a \in \A{C}{ap}$ and $|\tau(a)| \geq 80\varepsilon$ for $a \in \cC$.
    \item \label{item:Z2:reco}
    $(Z, \tau)$ satisfies the recognizability property in Proposition \ref{Z2:special_reco}.
    \item \label{item:Z2:per}
    $(Z, \tau)$ has dichotomous periods w.r.t.\ $(\A{C}{ap},\A{C}{p})$ and $8\varepsilon$.
    \item \label{item:Z2:structure}
    The set $\A{C}{p}$ satisfies the following: if $z \in Z$, then $z_0$ and $z_1$ does not simultaneously belong to $\A{C}{p}$.
\end{enumerate}
\end{prop}

%%%%                                                                 %%%% PROPOSITION %%%%
\begin{prop} \label{Z2:special_reco}
Consider the coding described in Proposition \ref{Z3:main_prop}.
Let $x, \tilde{x} \in X$ be such that $\per(x_{[-\varepsilon, \varepsilon)}) > \varepsilon$ and $x_{[-7d^2n, 7d^2n)} = \tilde{x}_{[-7d^2n, 7d^2n)}$.
Then, $\Fac^0_{(Z,\tau)}(x)$ is equal to $\Fac^0_{(Z,\tau)}(\tilde{x})$.
\end{prop}

The strategy for proving Proposition \ref{Z2:main_prop} is as follows.
We consider the coding $(Y, \sigma)$ given by Proposition \ref{Z1:main_prop} and, for a point $y \in Y$, we glue together letters $y_i$ to form words $y_I$, where $I$ is an interval, in such a way that $y_I$ corresponds either to a maximal periodic part of $\sigma(y)$ or to an aperiodic part of $\sigma(y)$ of controlled length.
This will produce a new coding where the letters are in correspondence with the words $y_I$ and that satisfies all the properties in Proposition \ref{Z2:main_prop} except for the lower bound in Item \ref{item:Z2:lengths} for the letters associated to periodic parts $y_I$.
We solve this by slightly moving the edges of the words $\sigma(y_I)$.

We start, in Subsection \ref{subsec:Z2:eps_intervals}, by defining stable intervals, which correspond to the intervals $I$ described in the last paragraph.
The definition of the coding of Proposition \ref{Z2:main_prop} is given in Subsection \ref{subsec:Z2:defi_Z2}, together with the proof of its basic properties.
In the final subsection, we prove Propositions \ref{Z2:main_prop} and \ref{Z2:special_reco}.
\medskip

We fix the following notation for the rest of the section.
Let $X \subseteq \cA^\Z$ be a minimal infinite subshift, $n \geq 0$ and let $d$ be the maximum of $\lceil p_X(n)/n \rceil$, $p_X(n+1) - p_X(n)$, $\#\cA$ and $10^4$.
Then, Proposition \ref{Z1:main_prop} applied to $X$ and $n$ gives a recognizable coding $(Y \subseteq \cB^\Z, \sigma\colon\cB\to\cA^+)$ of $X$ and an integer $\varepsilon \in [n/d^{2d^3+4}, n/d)$.

\subsection{Stable intervals}
\label{subsec:Z2:eps_intervals}

Let $y \in Y$, $x = \sigma(y)$ and $(c, y) = \Fac_{(Y,\sigma)}(x)$.
We define
\begin{equation*}
    \rQ{short}(y) = \{i \in \Z : |\sigma(y_i)| \leq 401\varepsilon\}
    \enspace\text{and}\enspace
    \rQ{long}(y) = \Z \setminus \rQ{long}.
\end{equation*}
Let $\rQ{p}(y)$ be the set of integers $i \in \Z$ such that 
\begin{equation}
    \text{$|\mroot\sigma(y_i)| \leq \varepsilon$ and $x_{[c_i - 99\varepsilon, c_{i+1} + 99\varepsilon)} = (\mroot\sigma(y_i)^\Z)_{[-99\varepsilon, |\tau(y_i)| + 99\varepsilon)}$.}
\end{equation}
We set $\rQ{ap}(y) = \Z \setminus \rQ{p}(y)$.
Remark that the definition of $\rQ{p}(y)$ coincides with the one in Proposition \ref{Z1:per}.

%%%%                                                                 %%%% DEFINITION %%%%
\begin{defi} \label{defi:eps_interval}
A {\em stable interval} for $y$ is a finite interval $I = [i, j) \subseteq \Z$ satisfying one of the following conditions.
\begin{enumerate}
    \item $I \subseteq \rQ{p}$.
    \item $I \subseteq \rQ{ap}$, $\# (I \cap \rQ{long}) \leq 1$, and if $i \in \rQ{short}$ then $i-1 \in \rQ{p}$.
\end{enumerate}
The interval $I$ is of {\em periodic type} if it satisfies Item (1) of this definition, and of {\em aperiodic type} if it satisfies Item (2).
We say that $I$ is a {\em maximal stable interval} if for all stable interval $I' \supseteq I$ we have that $I' = I$.
\end{defi}

\begin{rema}
We stress on the fact that the previous definition does not depend just on $y$, but also on $\sigma$ and $\varepsilon$.
\end{rema}

%%%%                                                                 %%%% LEMMA %%%%
\begin{lem} \label{Z2:per_inter_props}
Let $I = [i, j)$ be a stable interval set for $y$ of periodic type.
Then:
\begin{enumerate}
    \item $\mroot\sigma(y_{[i, j)}) = \mroot\sigma(y_k)$ for all $k \in [i, j)$ and $x_{[c_i - 99\varepsilon, c_j + 99\varepsilon)}$ is equal to $(\mroot\sigma(y_i))^\Z_{[-99\varepsilon, |\sigma(y_{[i,j)})|+99\varepsilon)}$.
    \item If $I$ is maximal, $I' = [i', j')$ is a stable interval and either $i' = j$ or $j' = i$, then $I'$ is of aperiodic type.
    \item If $y' \in Y$, $0 \in \rQ{p}(y')$ and $\mroot \sigma(y_i)$ is conjugate to $\mroot \sigma(y'_0)$, then $\mroot \sigma(y_i) = \mroot \sigma(y'_0)$.
\end{enumerate}
\end{lem}
\begin{proof}
We first prove Item (1).
Let $s_k = \mroot\sigma(y_k)$. 
Being $I$ of periodic type, we have by Definition \ref{defi:eps_interval} that
\begin{equation} \label{eq:Z2:per_inter_props:1}
    \text{$|s_k| \leq \varepsilon$ and $x_{[c_k-99\varepsilon, c_{k+1}+99\varepsilon)} = (s_k^\Z)_{[-99\varepsilon, |\sigma(y_k)|+99\varepsilon)}$ for all $k \in [i, j)$.}
\end{equation}
Then, for any $k \in [i, j-1)$,
\begin{equation*}
    (s_k^\Z)_{[0, 99\varepsilon)} = 
    x_{[c_k, c_k + 99\varepsilon)} =
    x_{[c_{k+1}, c_{k+1} + 99\varepsilon)} =
    (s_{k+1}^\Z)_{[0, 99\varepsilon)}.
\end{equation*}
Combining this with the inequalities $|s_k|, |s_{k+1}| \leq \varepsilon$ and Theorem \ref{theo:fine&wilf} produces a word $t$ such that $s_k$ and $s_{k+1}$ are powers of $t$. 
Hence, as $s_k$ and $s_{k+1}$ are roots of a word, $s_k = s_{k+1} = t$ for any $k \in [i, j-1)$.
Item (1) of the lemma follows from this and \eqref{eq:Z2:per_inter_props:1}.

For Item (2), we note that if $I'$ is of periodic type then $I \cup I'$ is an interval contained in $\rQ{p}(y)$, and so $I \cup I'$ is a stable interval for $y$.
This would contradict the maximality of $I$; therefore, $I'$ is of aperiodic type.

Let us now assume that the hypothesis of Item (3) holds.
Then, since $i \in I \subseteq \rQ{p}(y)$, $0 \in \rQ{p}(y')$ and $\mroot \sigma(y_i)$ is conjugate to $\mroot \sigma(y'_0)$, the points $S^i y$ and $y'$ comply with the hypothesis of Item \ref{item:Z1:per:conj_roots} of Proposition \ref{Z1:per}.
We conclude tat $\mroot \sigma(y_i) = \mroot \sigma(y'_0)$.
\end{proof}

%%%%                                                                 %%%% LEMMA %%%%
\begin{lem} \label{Z2:aper_inter_props}
Let $y \in Y$ and $I = [i, j)$ be a stable interval for $y$ of aperiodic type.
Then, 
\begin{enumerate}
    \item $\per(x_{[c_i+97\varepsilon, c_j-97\varepsilon)}) > \varepsilon$; % OK
    \item $I$ has length at most $2d + 1$; % OK
    \item $195\varepsilon \leq |\sigma(y_I)| \leq 9d^2n$. % OK
\end{enumerate}
\end{lem}
\begin{proof}
We prove Item (1) by considering two cases.
Suppose that $i \in \rQ{long}(y)$.
Then, $i \in \rQ{long}(y) \cap \rQ{ap}(y)$ and we can use Item \ref{item:Z1:per:0ap&long=>per_int} of Proposition \ref{Z1:per} to obtain that $\per(x_{[c_i+97\varepsilon, c_{i+1}-97\varepsilon)}) > \varepsilon$. 
Assume now that $i \in \rQ{short}(y)$.
Then, Definition \ref{defi:eps_interval} ensures that $i-1 \in \rQ{p}(y)$.
Hence, Item \ref{item:Z1:per:0ap&short&edge_p=>per_int} of Proposition \ref{Z1:per} applies and so $\per(x_{[c_i+97\varepsilon, c_{i+1}-97\varepsilon)}) > \varepsilon$. 
In particular, Item (1) holds.

We prove Item (2) by contradiction.
Assume that $\#I > 2d +1$.
Then, it follows from Definition \ref{defi:eps_interval} that there exists $I' \subseteq I$ such that $\#I' > d$ and $I' \subseteq \rQ{short}$.
These conditions allow us to use Item \ref{item:Z1:per:many_short=>p} in Proposition \ref{Z1:per} and deduce that $I' \subseteq \rQ{p}(y)$, contradicting the fact that $I$ is of aperiodic type.

Finally, we consider Item (3).
Item (2) of this lemma and Item \ref{item:Z1:main_prop:lengths} produce that $|\sigma(y_I)| \leq (2d+1)\cdot3dn$, from which the upper bound in Item (3) follows.
To prove the lower bound, we consider two cases.
If there exists $k \in I \cap \rQ{long}(y)$, then $|\sigma(y_I)| \geq |\sigma(y_k)| \geq 401\varepsilon$.
Assume now that $I \cap \rQ{long}(y)$ is empty.
Then, $i \in \rQ{short}(y)$, and so Definition \ref{defi:eps_interval} indicates that $i-1 \in \rQ{p}(y)$.
This allows us to use Item \ref{item:Z1:per:0ap&short&edge_p=>per_int} of Proposition \ref{Z1:per} to obtain that $\per(x_{[c_i+97\varepsilon, c_{i+1}-97\varepsilon)}) > \varepsilon$.
In particular, $|x_{[c_i+97\varepsilon, c_{i+1}-97\varepsilon)}| > \varepsilon$; hence,
\begin{equation*}
    |\sigma(y_I)|  \geq |\sigma(y_i)| =
    |x_{[c_i+97\varepsilon, c_{i+1}-97\varepsilon)}| + 2\cdot97\varepsilon >
    195\varepsilon.
\end{equation*}
\end{proof}

%%%%                                                                 %%%% LEMMA %%%%
\begin{lem} \label{Z2:maxLen_inter}
There exists a constant $C$ depending only on $X$ such that for any $y \in Y$ and stable interval $I$ for $Y$, we have that $\#I \leq C$.
In particular, any stable interval is contained in a maximal stable interval.
\end{lem}
\begin{proof}
Let $C_0$ be the length of the longest word $w$ that occurs in some $x \in X$ such that $\per(w) \leq \varepsilon$.
We remark that $C_0$ is finite as $X$ is assumed to be minimal and infinite.
Let $C = \max\{C_0, 2d+1\}$.
We claim that for any $y \in Y$, any stable interval $I$ for $y$ has length at most $C$.
Indeed, if $I$ is of aperiodic type, then Item (3) of Lemma \ref{Z2:aper_inter_props} implies that $\#I \leq 2d+1 \leq C$, and if $I$ is of periodic type, then Item (1) of Lemma \ref{Z2:per_inter_props} ensures that $\per(\sigma(y_I)) \leq \varepsilon$, and thus that $\#I \leq |\sigma(y_I)| \leq C_0 \leq C$. 
\end{proof}

%%%%                                                                 %%%% LEMMA %%%%
\begin{lem} \label{Z2:cI_is_partition}
Let $y \in Y$.
Then, the set of all maximal stable intervals for $y$ is a partition of $\Z$.
\end{lem}
\begin{proof}
We first prove that any $k \in \Z$ is contained in a stable interval.
This would imply that any $k$ is contained in a maximal stable interval by Lemma \ref{Z2:maxLen_inter}.

Let $k \in \Z$ be arbitrary.
We consider two cases.
If $k \in \rQ{p}(y)$ or $k \in \rQ{ap}(y) \cap \rQ{long}(y)$, then $\{k\}$ is stable interval and we are finished.
Suppose now that $k \in \rQ{ap}(y) \cap \rQ{short}(y)$.
Let $i < k$ be the biggest integer such that $i \not \in \rQ{ap}(y) \cap \rQ{short}(y)$.
Note that $[i+1, k] \subseteq \rQ{ap}(y) \cap \rQ{short}(y)$.
Hence, if $i \in \rQ{ap}(y) \cap \rQ{long}(y)$, then $[i,k]$ is stable interval of aperiodic type, and if $i \in \rQ{p}$, then $[i+1,k]$ is stable interval of aperiodic type.
These are the only cases as $i \not \in \rQ{ap}(y) \cap \rQ{short}(y)$, and so we conclude that $i$ belongs to a stable interval.
\medskip

Next, we prove that for any maximal stable intervals $I, I'$, either $I = I'$ or $I \cap I' = \emptyset$.
The lemma follows from this and the fact that any $k$ is contained in a maximal stable interval.

Let $I = [i,j)$ and $I' = [i', j')$ be maximal stable intervals with nonempty intersection.
There is no loss of generality in assuming that $i \leq i' < j \leq j'$.
Note that if $i = i'$ or $j = j'$, then $I \cup I' \in \{I, I'\}$, so $I = I' = I \cup I'$ by maximality.
Hence, we may assume that $i < i' < j < j'$.
Remark that this implies that $j - 1 \in I \cap I'$.

In order to continue, we consider three cases.
\begin{enumerate}
    \item If $j - 1 \in \rQ{p}(y)$, then, as $j-1 \in I \cap I'$, Definition \ref{defi:eps_interval} implies that $I$ and $I'$ are of periodic type.
    It then follows from Definition \ref{defi:eps_interval} that $I \cup I'$ is stable interval of periodic type, and so $I = I' = I \cup I'$ by maximality.

    \item If $j -1 \in \rQ{ap}(y) \cap \rQ{long}$. Then, since $j - 1 \in I \cap I'$, Definition \ref{defi:eps_interval} ensures that $[i, j-1) \cup [j+1 , j') \subseteq \rQ{ap}(y) \cap \rQ{short}(y)$.
    Hence, $[i, j') = I \cup I'$ is a stable interval of aperiodic type, which implies that $I = I' = I \cup I'$ by maximality.
    
    \item If $j - 1 \in \rQ{ap}(y) \cap \rQ{short}(y)$. 
    Then, as $j-1 \in I \cap I'$, Definition \ref{defi:eps_interval} guarantees that $I$ and $I'$ are of aperiodic type.
    In particular, as $i'-1 \in I$, $i'-1 \in \rQ{ap}(y)$ and therefore, by Definition \ref{defi:eps_interval}, $i' \in \rQ{long}(y)$.
    We conclude, using Definition \ref{defi:eps_interval}, that $[i, i') \subseteq \rQ{ap}(y) \cap \rQ{short}(y)$, $i' \in \rQ{ap}(y) \cap \rQ{long}(y)$ and that $[i'+1,j') \subseteq \rQ{ap}(y) \cap \rQ{short}(y)$.
    Hence, $[i, j') = I \cup I'$ is a stable interval of aperiodic type and $I = I' = I\cup I'$ by maximality.
\end{enumerate}
\end{proof}

%%%%                                                                 %%%% SUBSECTION %%%%
\subsection{Construction of the second coding}
\label{subsec:Z2:defi_Z2}

The coding $(Z, \tau)$ is obtained by modifying the cut function $c$ in $\Fac_{(Y,\sigma)}(x)$ of the points $x \in X$.
We give the construction of the modified cut function as the proof of the following lemma, and we define $(Z, \tau)$ right after.

%%%%                                                                 %%%% LEMMA %%%%
\begin{lem} \label{Z2:defi_r}
Let $x \in X$ and set $(c, y) = \Fac_{(Y,\sigma)}(x)$.
There exist unique increasing sequences of integers satisfying $(k_x(j))_{j\in\Z}$ and $(r_x(j))_{j\in\Z}$ satisfying the following conditions.
\begin{enumerate}
    \item $\{[k_x(j), k_x(j+1)) : j \in \Z\}$ is the set of all maximal stable intervals of $y$.
    \item \label{item:Z2:defi_r:defi_r}
    For any $j \in \Z$,
    \begin{enumerate}
    \item if $[k_x(j), k_x(j+1))$ is of aperiodic type, then $r_x(j) = c_{k_x(j)}$.    
    \item if $[k_x(j), k_x(j+1))$ is of periodic type, then $r_x(j) = c_{k_x(j)} - |s^\ell|$, where $s = \mroot\sigma(y_{[k_x(j), k_x(j+1))})$ and $\ell = \left\lceil 80\varepsilon / |s| \right\rceil$
    \end{enumerate}
    \item \label{item:Z2:defi_r:0}
    $0$ belongs to $[r_x(0), r_x(1))$.
\end{enumerate}
Moreover, in this case, $r_x(j+1) \geq r_x(j) + 80\varepsilon$ for all $j \in \Z$.
\end{lem}
\begin{proof}
Lemma \ref{Z2:cI_is_partition} ensures that the set of all maximal stable intervals of $y$ can be described as $\{ [k(j), k(j+1)) : j \in \Z\}$ for some increasing sequence of integers $(k_j)_{j\in\Z}$.
The sequence $(k_j)_{j\in\Z}$ is unique up to an index shift.

We define $r(j)$ as follows:
\begin{enumerate}[label=(\roman*)]
    \item if $[k(j), k(j+1))$ is of aperiodic type, then $r(j) = c_{k(j)}$.    
    \item if $[k(j), k(j+1))$ is of periodic type, then $r(j) = c_{k(j)} - |s^\ell|$, where $s = \mroot\sigma(y_{[k(j), k(j+1))})$ and $\ell = \left\lceil 80\varepsilon / |s| \right\rceil$
\end{enumerate}
It is important to remark that, in case (ii), Lemma \ref{Z2:per_inter_props} ensures that $|s| \leq \varepsilon$.

We claim that
\begin{equation} \label{eq:Z2:defi_r:1}
    r(j+1) \geq r(j) + 80\varepsilon
    \enskip\text{for all $j \in \Z$.}
\end{equation}
First, we note that the definition of $r(j)$ and $r(j+1)$ guarantees that
\begin{equation} \label{eq:Z2:defi_r:2}
    k(j) - 81\varepsilon < r(j) \leq k(j)
    \enspace\text{and}\enspace
    k(j+1) - 81\varepsilon < r(j+1) \leq k(j+1).
\end{equation}
We now consider two cases.
If $[k(j), k(j+1))$ is of aperiodic type, then, by Lemma \ref{Z2:aper_inter_props}, $|\sigma(y_{[ k(j), k(j+1) )})|$ is at least $195\varepsilon$.
Combining this with \eqref{eq:Z2:defi_r:2} yields
\begin{equation*}
    |x_{[ r(j), r(j+1) )}| \geq 
    |x_{[ c_{k(j)}, c_{k(j+1)} )}| - 81\varepsilon = 
    |\sigma(y_{[ k(j), k(j+1) )})| - 81\varepsilon \geq
    80\varepsilon
\end{equation*}
Assume now that $[k(j), k(j+1))$ is of periodic type.
Then, $[k(j+1), k(j+2))$ is of aperiodic type by Lemma \ref{Z2:per_inter_props}.
In particular, $r(j+1) = c_{k(j+1)}$ by (i).
Also, since $[k(j), k(j+1))$ is of periodic type, (ii) ensures that $r(j) \leq k(j) - 80\varepsilon$.
These two things imply that $|x_{[r(j), r(j+1))}| \geq |x_{[c_{k(j)} - 80\varepsilon, c_{k(j+1)})}| \geq 80\varepsilon$, completing the proof of the claim.
\medskip

Equation \eqref{eq:Z2:defi_r:1} implies that $(r(j))_{j\in\Z}$ is increasing.
Thus, there exists a unique $\ell \in \Z$ such that $0 \in [r(\ell), r(\ell+1))$.
We define $k_x(j) = k(j + \ell)$ and $r_x(j) = r(j + \ell)$.
Then, $(k_x(j))_{j\in\Z}$ and $(r_x(j))_{j\in\Z}$ satisfy Items (1), (2) and (3) of the lemma.
Moreover, being $(r(j))_{j\in\Z}$ increasing, it is clear $\ell$ (and then also $(k_x(j))_{j\in\Z}$ and $(r_x(j))_{j\in\Z}$) is unique.
\end{proof}

We now define $(Z, \tau)$.
It follows from the recognizability property of $(Y, \sigma)$ and Lemma \ref{Z2:maxLen_inter} that the map $x \mapsto r_x(0)$ is continuous.
In particular, $U = \{x \in X : r_x(0) = 0\}$ is clopen (and nonempty).
We define $(Z \subseteq \cC^\Z, \tau\colon\cC \to \cA^+)$ as the recognizable coding of $X$ obtained from $U$ as in Proposition \ref{prop:from_clopen_to_coding}.

%%%%                                                                 %%%% SEUBSECTION %%%%
\subsection{Basic properties of the second coding}

We fix, for the rest of the section, the following notation.
Let $x$ denote an element of $X$, $(c,y) = \Fac_{(Y,\sigma)}(x)$ and $(f,z) = \Fac_{(Z,\tau)}(x)$.
We also define $(k_x(j))_{j\in\Z}$ and $(r_x(j))_{j\in\Z}$ as the sequences given by Lemma \ref{Z2:defi_r}.

%%%%                                                                 %%%% LEMMA %%%%
\begin{lem} \label{Z2:Zcuts=r}
We have that $r_x(j) = f_j$ for all $j \in \Z$.
\end{lem}
\begin{proof}
Note that $\{k \in \Z : S^k x \in U\}$ is equal to $\{r_x(j) : j \in \Z\}$.
Then, by Item (1) in Proposition \ref{prop:from_clopen_to_coding}, there exists a bijective map $g\colon \Z \to \Z$ such that $r_j(x) = f_{g(j)}$ for all $j \in \Z$.
Now, since Lemma \ref{Z2:defi_r} states that $r_x(j) < r_x(j+1)$ for all $j \in \Z$, the map $g$ is increasing.
As it is also bijective, we conclude that there exists $\ell \in \Z$ satisfying  $g(j) = j + \ell$ for all $j \in \Z$.
Finally, by Item (3) in Lemma \ref{Z2:defi_r} and the definition of $f$, we have that $0 \in [f_0, f_1)$ and $0 \in [r_x(0), r_x(1)) = [f_{g^{-1}(0)}, f_{g^{-1}(0)+1})$.
Hence, $g^{-1}(0) = 0$, $\ell = 0$ and the lemma follows.
\end{proof}

The last lemma allows us to drop the notation $r_x(j)$ and use only $f_j$.
In particular, Items (2) and (3) of Lemma \ref{Z2:defi_r} hold with $f_j$.

We define a partition $\A{C}{a} \cup \A{C}{ap}$ of $\cC$ as follows:
\begin{align*}
    \A{C}{ap} = \{a \in \cC : \per(\tau(a)) > \varepsilon\}
    \enspace\text{and}\enspace
    \A{C}{p} = \{a \in \cC : \per(\tau(a)) \leq \varepsilon\}.
\end{align*}

%%%%                                                                 %%%% LEMMA %%%%
\begin{lem} \label{Z2:tech_ext_per}
Let $j \in \Z$.
The following are equivalent:
\begin{enumerate}[label=(1.\alph*)]
    \item $z_j \in \A{C}{p}$.
    \item $[k_x(j), k_x(j+1))$ is of periodic type for $y$.
    \item Let $s = \mroot \tau(z_j)$. Then, $|s| \leq \varepsilon$, $x_{[f_j - 8\varepsilon, f_{j+1} + 8\varepsilon)} = s^\Z_{[-8\varepsilon, |\tau(z_j)| + 8\varepsilon)}$ and $s = \mroot\sigma(y_{ [k_x(j), k_x(j+1)) })$.
\end{enumerate}
The following are also equivalent:
\begin{enumerate}[label=(2.\alph*)]
    \item $z_j \in \A{C}{ap}$.
    \item $[k_x(j), k_x(j+1))$ is of aperiodic type for $y$.
    \item \label{item:Z2:tech_ext_per:inner_per>eps}
    $\per(x_{[f_j + 8\varepsilon, f_{j+1} - 8\varepsilon)}) > \varepsilon$.
\end{enumerate}
\end{lem}
\begin{proof}
We start with a general observation.
Let us write $k(j) = k_x(j)$.
Then, Item \ref{item:Z2:defi_r:defi_r} in Lemma \ref{Z2:defi_r} ensures that 
\begin{equation*}
    \text{$c_{k(j)} - 81\varepsilon < f_j \leq c_{k(j)}$ and 
    $c_{k(j+1)} - 81\varepsilon < f_{j+1} \leq c_{k(j+1)}$.}
\end{equation*}
Hence, 
\begin{equation} \label{eq:Z2:tech_ext_per:1}
    \emptyset \not= [c_{k(j)}, c_{k(j+1)} - 81\varepsilon) \subseteq
    [f_j, f_{j+1}) \subseteq 
    [c_{k(j)} - 81\varepsilon, c_{k(j+1)} + 81\varepsilon)
\end{equation}

We now prove the lemma.
Let us assume that (1.a) holds.
Then, \eqref{eq:Z2:tech_ext_per:1} implies that $\per(x_{[c_{k(j)} + 97\varepsilon, c_{k(j+1)} - 97\varepsilon)})  \leq \per(x_{[f_j, f_{j-1})}) \leq \varepsilon$.
Hence, by Lemma \ref{Z2:aper_inter_props}, $[k(j), k(j+1))$ is not of aperiodic type, that is, $[k(j), k(j+1))$ is of periodic type.

Assume next (1.b).
Then, Lemma \ref{Z2:per_inter_props} states that $s = \mroot \sigma(y_{[k(j), k(j+1))})$ satisfies $|s| \leq \varepsilon$, $s = \mroot \sigma(y_k)$ for all $k \in [k(j), k(j+1))$ and
\begin{equation} \label{eq:Z2:tech_ext_per:2}
    x_{[c_{k(j)} - 99\varepsilon, c_{k(j+1)} + 99\varepsilon)} = s^\Z_{[-99\varepsilon, c_{k(j+1)} - c_{k(j)} + 99\varepsilon)}.
\end{equation}
Moreover, Item (2) in Lemma \ref{Z2:per_inter_props} guarantees that $[k(j+1), k(j+2))$ is of aperiodic type.
Hence, by Item \ref{item:Z2:defi_r:defi_r} in Lemma \ref{Z2:defi_r}, $f_j = c_{k(j)} - |s^\ell|$ and $f_{j+1} = c_{k(j+1)}$, where $\ell = \lceil 80\varepsilon / |s| \rceil$.
We can then compute, thanks to \eqref{eq:Z2:tech_ext_per:2} and \eqref{eq:Z2:tech_ext_per:1},
\begin{multline*}
    x_{[f_j - 8\varepsilon, f_{j+1} + 8\varepsilon)} =
    s^\Z_{[-|s^\ell|-8\varepsilon, c_{k(j+1)} - c_{k(j+1)} + 8\varepsilon)} \\ =
    s^\Z_{[-8\varepsilon, |s^\ell| + c_{k(j+1)} - c_{k(j+1)} + 8\varepsilon)} = 
    s^\Z_{[-8\varepsilon, f_{j+1} - f_j + 8\varepsilon)}.
\end{multline*}
Note that the last computation also shows that $\tau(z_j) = x_{[f_j, f_{j+1})}$ is equal to $s^\Z_{[-|s^\ell|, c_{k(j+1)} - c_{k(j+1)})} = s^\Z_{[-|s^\ell|, 0)} \sigma(y_{[k(j), k(j+1))})$.
Hence, $\mroot \tau(z_j) = \mroot s = s$ and we have proved (1.c).

Observe that if (1.c) holds, then $\per(\tau(z_j)) \leq |s| \leq \varepsilon$ and $z_j \in \A{C}{p}$ by the definition of $\A{C}{p}$.

We now assume (2.a).
Then, Equation \eqref{eq:Z2:tech_ext_per:1} implies that $\per(x_{[f_j, f_{j+1})}) \geq \per(x_{[c_{k(j)}, c_{k(j+1)} - 81\varepsilon)}) > \varepsilon$.
Hence, by Lemma \ref{Z2:per_inter_props}, $[k(j), k(j+1))$ is not of periodic type, that is, $[k(j), k(j+1))$ is of aperiodic type.

Let us suppose that (2.b) holds.
In this case, Lemma \ref{Z2:aper_inter_props} and \eqref{eq:Z2:tech_ext_per:1} allows us to compute
\begin{equation*}
    \per(x_{[f_j + 8\varepsilon, f_{j+1} - 8\varepsilon)}) \geq 
    \per(x_{[c_{k(j)} + 97\varepsilon, c_{k(j+1)} - 97\varepsilon)}) > \varepsilon.
\end{equation*}

Finally, if (2.c) is satisfied, then $\per(\tau(z_j)) \geq \per(x_{[f_j + 8\varepsilon, f_{j+1} - 8\varepsilon)}) > \varepsilon$.
\end{proof}

%%%%                                                                 %%%% LEMMA %%%%
\begin{lem} \label{Z2:tech_pcom}
Suppose that $z_{-1} z_0 z_1 \in \A{C}{ap}\A{C}{p}\A{C}{ap}$.
Then, there exists a decomposition $\tau(z_{-1} z_0 z_1) = u s^m u'$ such that:
\begin{enumerate}
    \item $\varepsilon \leq |u| \leq |\tau(z_{-1})| - 2\varepsilon$ and $\varepsilon \leq |u'| \leq |\tau(z_1)| - 2\varepsilon$.
    \item $s \in \mroot \sigma(\cB)$.
    \item $s$ is not a suffix of $u$ and is not a prefix of $u'$.
\end{enumerate}
\end{lem}
\begin{proof}
Let us denote $k(j) = k_x(j)$.
We define $s = \mroot \tau(z_0)$.
Then, as $z_0 \in \A{C}{p}$, Lemma \ref{Z2:tech_ext_per} ensures that $|s| \leq \varepsilon$, $x_{[f_0 - 8\varepsilon, f_1 + 8\varepsilon)} = s^\Z_{[-8\varepsilon, f_1-f_0 + 8\varepsilon)}$, $s = \mroot \sigma(y_{[k(0), k(1))})$ and that $[k(0), k(1))$ is of periodic type in $y$.
Thus, by Lemma \ref{Z2:per_inter_props}, $s = \mroot \sigma(y_{k(0)}) \in \sigma(\cB)$.
In particular, $s$ satisfies Item (2) of this lemma.

Now, we can find an interval $I = [i, j)$ containing $[f_0 - 8\varepsilon, f_1 + 8\varepsilon)$ such that $x_I = s^\Z_{[i - f_0, j - f_0)}$ and that no other interval strictly containing $I$ satisfies the same properties.
We observe that $i \geq f_{-1} + 8\varepsilon$ as, otherwise, $\per(x_{[f_{-1} + 8\varepsilon, f_0 - 8\varepsilon)}) \leq \varepsilon$, contradicting the fact that, since $z_{-1} \in \A{C}{ap}$, Item \ref{item:Z2:tech_ext_per:inner_per>eps} Lemma \ref{Z2:tech_ext_per} holds.
Similarly, $j \leq f_2 - 8\varepsilon$.
From these two things and the fact that $I$ contains $[f_0 - 8\varepsilon, f_1 + 8\varepsilon)$ we obtain that 
\begin{equation} \label{eq:Z2:tech_pcom:1}
    \text{$f_{-1} + 8\varepsilon \leq i \leq f_0 - 8\varepsilon$ and $f_1 + 8\varepsilon \leq j \leq f_2 - 8\varepsilon$.}
\end{equation}
This allows us to write $x_{[f_{-1}, f_1)} = u s^m v$, where $|u| \in [i - f_{-1}, i - f_{-1} + |s|)$, $m \geq 0$ and $|v| \in [f_2 - j, f_2 - j + |s|)$.

We have from \eqref{eq:Z2:tech_pcom:1} that $|u| \geq i - f_{-1} \geq 8\varepsilon$ and $|v| \geq f_2 - j \geq 8\varepsilon$.
Moreover, as $|s| \leq \varepsilon$, $|u| \leq f_0-f_{-1} - 7\varepsilon = |\tau(a)| - 7\varepsilon$ and $|v| \leq f_2-f_1-7\varepsilon = |\tau(a')| - 7\varepsilon$.
This proves that Item (1) of the lemma holds.
Item (3) follows from the fact that $|s| \leq \varepsilon$, \eqref{eq:Z2:tech_pcom:1} and the maximality of $I$.
\end{proof}

%%%%                                                                 %%%% SUBSECTION %%%%
\subsection{Proof of Propositions \ref{Z2:main_prop} and \ref{Z2:special_reco}}

We are ready to prove the main results of this section.

%%%%                                                                 %%%% PROOF PROPOSITION %%%%
\begin{proof}[Proof of Proposition \ref{Z2:special_reco}]
Let $x, \tilde{x} \in X$ be such that $\per(x_{[-\varepsilon, \varepsilon)}) > \varepsilon$ and $x_{[-7d^2n, 7d^2n)} = \tilde{x}_{[-7d^2n, 7d^2n)}$.
We define $(c,y) = \Fac_{(Y,\sigma)}(x)$, $(\tilde{c}, \tilde{y}) = \Fac_{(Y,\sigma)}(\tilde{x})$, $(f,z) = \Fac_{(Z,\tau)}(x)$ and $(\tilde{f}, \tilde{z}) = \Fac_{(Z,\tau)}(\tilde{x})$.
Let $k(j) = k_x(j)$ and $\tilde{k}(j) = k_{\tilde{x}}(j)$ be the sequences from Lemma \ref{Z2:defi_r}.
With this notation, we have to prove that $f_0 = \tilde{f}_0$ and $z_0 = \tilde{z}_0$.

Note that $\per(x_{[f_0 - 8\varepsilon, f_1 + 8\varepsilon)}) \geq \per(x_{[-\varepsilon, \varepsilon)}) > \varepsilon$.
Thus, by Lemma \ref{Z2:tech_ext_per}, $[k(0), k(1))$ is of aperiodic type in $y$.

We claim that 
\begin{equation} \label{eq:Z2:special_reco:1}
    \text{$[c_{k(0)-1} - 3dn, c_{k(1)+1} + 3dn)$ is contained in $[-7d^2n, 7d^2n)$}
\end{equation}
Note that, by Items \ref{item:Z2:defi_r:defi_r} and \ref{item:Z2:defi_r:0} in Lemma \ref{Z2:defi_r}, $c_{k(0)} \leq f_0 + 81\varepsilon \leq 81\varepsilon$ and $c_{k(1)} \geq f_1 \geq 0$.
Hence, $c_{k(1)+1} \leq c_{k(0)} + (k(1)-k(0)+1)|\sigma| \leq 81\varepsilon + (k(1)-k(0)+1)|\sigma|$ and $c_{k(0)-1} \geq c_{k(1)} - (k(1)-k(0)+1)|\sigma| \geq -(k(1)-k(0)+1)|\sigma|$.
Since, by Lemma \ref{Z2:aper_inter_props}, $[k(0), k(1))$ has at most $2d+1$ elements, and since $|\sigma| \leq 3dn$ by Item \ref{item:Z1:main_prop:lengths} in Proposition \ref{Z1:main_prop}, we obtain that $c_{k(1)+1} +3dn \leq 7d^2n$ and $c_{k(0)-1} - 3dn \geq -7d^2n$.
This shows \eqref{eq:Z2:special_reco:1}.

Thanks to \eqref{eq:Z2:special_reco:1}, we can use the fact that $(Y, \sigma)$ is $3dn$-recognizable (Item \ref{item:Z1:main_prop:reco} of Proposition \ref{Z1:main_prop}) to deduce that
\begin{equation} \label{eq:Z2:special_reco:2}
    \text{$c_k = \tilde{c}_k$ and $y_k = \tilde{y}_k$ for all $k \in [k(0)-1 , k(1)+1)$.}
\end{equation}
 
We now observe that  \eqref{eq:Z2:special_reco:1} and the hypothesis guarantees that
\begin{equation} \label{eq:Z2:special_reco:2.5}
    \text{$x_{[ c_{k}-99\varepsilon, c_{k+1}+99\varepsilon )} = \tilde{x}_{[ \tilde{c}_{k}-99\varepsilon, \tilde{c}_{k+1}+99\varepsilon )}$ for every $k \in [k(0)-1, k(1)+1)$.}
\end{equation}
Thus, for any such $k$, $k \in \rQ{p}(y)$ if and only if $\rQ{p}(\tilde{y})$.
It is then not difficult to verify, using the definition of stable interval, that if $i \in \{0, 1\}$ then
\begin{enumerate}[label=(\Roman*)]
    \item $k(i) = \tilde{k}(i)$;
    \item the type of $[k(i), k(i+1))$ in $y$ and the type of $[\tilde{k}(i), \tilde{k}(i+1))$ in $\tilde{y}$ are equal.
\end{enumerate}
Then, (I) and \eqref{eq:Z2:special_reco:2.5} imply that 
\begin{equation} \label{eq:Z2:special_reco:2.75}
    x_{[k(0) - 99\varepsilon, k(1) + 99\varepsilon)} =
    \tilde{x}_{[\tilde{k}(0) - 99\varepsilon, \tilde{k}(1) + 99\varepsilon)}.
\end{equation}

We claim that
\begin{equation} \label{eq:Z2:special_reco:3}
    \text{$f_0 = \tilde{f}_0$  and $f_1 = \tilde{f}_1$.}
\end{equation}
Let $i \in \{0,1\}$.
We consider two cases.
First, we  assume that  $[k(i), k(i+1))$ is of aperiodic type in $y$.
Then, then by (II), $[\tilde{k}(i), \tilde{k}(i+1))$ is of aperiodic type in $\tilde{y}$.
Hence, by Item \ref{item:Z2:defi_r:defi_r} in Lemma \ref{Z2:defi_r}, $f_i = c_{k(i)}$ and $\tilde{f}_i = \tilde{c}_{\tilde{k}(i)}$.
This gives $f_i = \tilde{f}_i$ by (I) and \eqref{eq:Z2:special_reco:2}.

Next, we assume that $[k(i), k(i+1))$ is of periodic type in $y$.
Then, $[\tilde{k}(i), \tilde{k}(i+1))$ is of periodic type in $\tilde{y}$ by (II).
Hence, by Lemma \ref{Z2:per_inter_props}, $s = \mroot \sigma(y_{[k(i), k(i+1))})$ and $\tilde{s} = \mroot \sigma(\tilde{y}_{[\tilde{k}(i), \tilde{k}(i+1))})$ satisfy $|s|, |\tilde{s}| \leq \varepsilon$, $x_{[c_{k(i)}, c_{k(i)} + 2\varepsilon)}$ is equal to $s^\Z_{[0, 2\varepsilon)}$ and $\tilde{x}_{[\tilde{c}_{\tilde{k}(i)}, \tilde{c}_{\tilde{k}(i)} + 2\varepsilon)}$ is equal to $\tilde{s}^\Z_{[0, 2\varepsilon)}$.
In this situation, (I) and \eqref{eq:Z2:special_reco:2.5} ensures that 
\begin{equation*}
    s^\Z_{[0, 2\varepsilon)} = x_{[c_{k(i)}, c_{k(i)} + 2\varepsilon)} =
    \tilde{x}_{[\tilde{c}_{\tilde{k}(i)}, \tilde{c}_{\tilde{k}(i)} + 2\varepsilon)} =
    \tilde{s}^\Z_{[0, 2\varepsilon)}.
\end{equation*}
Since $|s|, |\tilde{s}| \leq \varepsilon$, this allows us the use of Theorem \ref{theo:fine&wilf} and deduce that $s = \tilde{s}$.
Putting this and the fact that $[k(i), k(i+1))$ and $[\tilde{k}(i), \tilde{k}(i+1))$ are of periodic type in Item \ref{item:Z2:defi_r:defi_r} of Lemma \ref{Z2:defi_r} produces $f_i = c_{k(i)} - |s^\ell|$ and $\tilde{f}_i = \tilde{c}_{\tilde{k}(i)} - |s^\ell|$, where $\ell = \left\lceil 80\varepsilon/|s|\right\rceil$.
Therefore, as $c_{k(i)} = \tilde{c}_{\tilde{k}(i)}$ by (I) and \eqref{eq:Z2:special_reco:2}, $f_i = \tilde{f}_i$.
This completes the proof of \eqref{eq:Z2:special_reco:3}.

Finally, we show that $z_0 = \tilde{z}_0$.
Item (2) in Lemma \ref{Z2:defi_r} gives that $|f_i - c_{k(i)}| \leq 81\varepsilon$.
Hence, by \eqref{eq:Z2:special_reco:2.75} and \eqref{eq:Z2:special_reco:3}, $x_{[f_0, f_1)} = \tilde{x}_{[f_0, f_1)} = \tilde{x}_{[\tilde{f}_0, \tilde{f}_1)}$.
We conclude that $\tau(z_0) = \tau(\tilde{z}_0)$, and therefore that $z_0 = \tilde{z}_0$ as $\tau$ is injective on letters by Proposition \ref{prop:from_clopen_to_coding}.
\end{proof}

We end this section with the proof of Proposition \ref{Z2:main_prop}.
%%%%                                                                 %%%% PROOF PROPOSITION %%%%
\begin{proof}[Proof of Proposition \ref{Z2:main_prop}]
Let $x \in X$, $(c, y) = \Fac_{(Y,\sigma)}(x)$ and $(f, z) = \Fac_{(Z,\tau)}(x)$.
Let $k(j) = k_x(j)$ be the sequence from Lemma \ref{Z2:defi_r}.
\medskip 

%%%%                                                                 %%%% ITEM %%%%
We start with Item (ii).
Let $a \in \A{C}{ap}$.
By minimality, there exists $j \in \Z$ such that $z_j = a$.
We compute as follows:
\begin{equation} \label{eq:Z2:main_prop:ii:1}
    |\tau(z_j)| = |f_{j+1} - f_j| \leq 
    |f_{j+1} - c_{k(j+1)}| + (k(j+1) - k(j)) |\sigma| + |f_j - c_{k(j)}|.
\end{equation}
On one hand, we have by Item (2) in Lemma \ref{Z2:defi_r} that $|f_{j+1} - c_{k(j+1)}|$ and $|f_j - c_{k(j)}|$ are at most $81\varepsilon$.
On the other hand, since $z_j \in \A{C}{ap}$, Lemma \ref{Z2:tech_ext_per} ensures that $[k(j), k(j+1))$ is of aperiodic type in $y$.
Hence, by Lemma \ref{Z2:aper_inter_props}, $\#[k(j), k(j+1)) \leq 2d + 1$.
Putting these two things in \eqref{eq:Z2:main_prop:ii:1} yields $|\tau(z_j)| \leq 2\cdot81\varepsilon + (2d+1)|\sigma| \leq 10d^2n$.

Let now $a \in \A{C}{}$ and $j \in \Z$ be such that $z_j = a$.
Then, by Lemma \ref{Z2:defi_r}, $|\tau(a)| = f_{j+1} - f_j \geq 80\varepsilon$.
\medskip

%%%%                                                                 %%%% ITEM %%%%
Next, we consider Item (i) and the inequality $\#\A{C}{ap} \leq 2d^{3d+6}$.
Observe that Lemmas \ref{Z2:per_inter_props} and \ref{Z2:aper_inter_props} ensure that
\begin{equation} \label{eq:Z2:main_prop:i:1}
    \text{$|\sigma(y_{[k(j), k(j+1))})| \geq p\varepsilon$ for all $j \in \Z$.}
\end{equation}
This allows us to define $u_j$ as the prefix of $\sigma(y_{[k(j), k(j+1))})$ of length $2\varepsilon$.

We claim that
\begin{multline} \label{eq:Z2:main_prop:i:1.5}
    \text{if $[k(j), k(j+1))$ is of periodic type, then 
    }\\\text{
    $\mroot \sigma(y_{[k(j), k(j+1))})$ is the prefix of $u_j$ of length $\per(u_j)$.}
\end{multline}
Let us suppose that $[k(j), k(j+1))$ is of periodic type.
Then, Lemma \ref{Z2:per_inter_props} states that $s = \mroot \sigma(y_{[k(j), k(j+1))})$ satisfies $|s| \leq \varepsilon$ and $x_{[c_{k(j)} - 99\varepsilon, c_{k(j+1)} + 99\varepsilon)} =  s^\Z_{[-99\varepsilon, |\sigma(y_{[k(j), k(j+1))})| + 99\varepsilon)}$.
In particular, since $u_j = x_{[c_{k(j)}, c_{k(j)} + 2\varepsilon)}$, $s^\Z_{[0, 2\varepsilon)} = u_j$.
Being $|s| \leq \varepsilon$, we obtain that $s^2$ is a prefix of $u_j$ and that $\per(u_j) \leq |s|$.
This permits to use Item \ref{item:CL:localize_per_aper:loc_per} in Proposition \ref{CL:localize_per_aper} to obtain that $\per(s^2) = \per(u_j)$.
Moreover, $|s| = \per(s^2)$ by Lemma \ref{CL:root=per}; therefore, $|s| = \per(s^2) = \per(u_j)$.
Since $s^\Z_{[0,2\varepsilon)} = u_j$, this shows that $s$ is the prefix of $u_j$ of length $\per(u_j)$, completing the proof of the claim.
\medskip

We now use \eqref{eq:Z2:main_prop:i:1.5} to prove the following:
\begin{multline} \label{eq:Z2:main_prop:i:3}
    \text{if $z_j \in \A{C}{ap}$, then $\tau(z_j)$ is uniquely determined by 
    }\\\text{
    $\sigma(y_{[k(j), k(j+1))})$ and wheter $z_{j+1}$ belongs to $\A{C}{ap}$.}
\end{multline}
Suppose that $z_j \in \A{C}{ap}$.
We consider two cases.
If $z_{j+1} \in \A{C}{ap}$, then Lemma \ref{Z2:tech_ext_per} ensures that $[k(j), k(j+1))$ and $[k(j+1), k(j+2))$ are of aperiodic type in $y$, and so, by Item \ref{item:Z2:defi_r:defi_r} in Lemma \ref{Z2:defi_r}, that $\tau(z_j) = \sigma(y_{[k(j), k(j+1))})$.
If $z_{j+1} \in \A{C}{ap}$, then Lemma  \ref{Z2:tech_ext_per}  ensures that $[k(j), k(j+1))$ is of aperiodic type and that $[k(j+1), k(j+2))$ is of periodic type.
Hence, by \ref{item:Z2:defi_r:defi_r} in Lemma \ref{Z2:defi_r}, $\tau(z_j) = x_{[c_{k(j)}, c_{k(j+1)} - |s^\ell|)}$, where $s = \mroot \sigma(y_{[k(j+1), k(j+2))})$ and $\ell = \left\lceil 80\varepsilon/|s| \right\rceil$.
Now, \eqref{eq:Z2:main_prop:i:1.5} says that $s$ is determined by $\sigma(y_{[k(j), k(j+1))})$, and the definition of $\ell$ depends only on $s$.
Therefore, $\tau(z_j) = x_{[c_{k(j)}, c_{k(j+1)} - |s^\ell|)}$ is determined $\sigma(y_{[k(j), k(j+1))})$.
The proof of \eqref{eq:Z2:main_prop:i:3} is complete.

Finally, we bound $\A{C}{ap}$.
Condition \eqref{eq:Z2:main_prop:i:3} implies that $\#\tau(\A{C}{ap})$ is at most $2$ times the number of words of the form $\sigma(y_{[k(j), k(j+1))})$, where $j \in \Z$ is such that $z_j \in \A{C}{ap}$.
Note that if $z_j \in \A{C}{ap}$ then Lemma \ref{Z2:tech_ext_per} gives that $[k(j), k(j+1))$ is of aperiodic type, and thus, by Lemma \ref{Z2:aper_inter_props}, we have that the length of $[k(j), k(j+1))$ is at most $2d+1$.
Hence, there are at most $\#\cB^{d+2}$ words $\sigma(y_{[k(j), k(j+1))})$ such that $z_j \in \A{C}{ap}$.
We conclude that $\#\tau(\A{C}{ap}) \leq 2\cdot \#\cB^{d+2}$, and therefore that $\#\A{C}{ap} \leq 2d^{3d+6}$ by Item \ref{item:Z1:main_prop:lengths} in Proposition \ref{Z1:main_prop} and the fact that $\tau$ is injective on letters

Next, we prove that $\#\mroot \tau(\cC) \leq 3d^{3d+6}$.
Since $\# \mroot \tau(\A{C}{ap}) \leq \#\A{C}{ap} \leq 2d^{3d+6}$ by what we just proved and since $\#\mroot\sigma(\cB) \leq \#\cB \leq d^3$ by Item \ref{item:Z1:main_prop:lengths} in Proposition \ref{Z1:main_prop}, it is enough to show that
\begin{equation} \label{eq:Z2:main_prop:i:claim_roots}
    \mroot \tau(\A{C}{ap}) \subseteq \{\mroot \sigma(y_0) : y \in Y, 0 \in \rQ{p}(y)\}.
\end{equation}
Let $a \in \A{C}{p}$ and $j \in \Z$ be such that $z_j = a$.
Thanks to Lemma \ref{Z2:tech_ext_per}, we have that $\mroot \tau(z_j) = \mroot \sigma(y_{[k(j), k(j+1))})$ and that $[k(j), k(j+1))$ is of periodic type in $y$.
Hence, by Lemma \ref{Z2:per_inter_props}, $\mroot \tau(z_j) = \mroot \sigma(y_{k(j)})$.
This proves \eqref{eq:Z2:main_prop:i:claim_roots} and thereby that $\#\mroot \tau(\cC) \leq 3d^{3d+6}$.

We now prove that $\#\A{C}{p} \leq 2d^{3d+9}\pcom(X)$ using Lemma \ref{Z2:tech_pcom}.
Let $\cU = \{ z_{j-1} z_j z_{j+1} : j \in \Z, z_j \in \A{C}{p}\}$.
We define the map $\pi\colon\cU \to \cup_{s \in \mroot \sigma(\cB)} \mathsf{Pow}_X(s)$ as follows.
For $aba' \in \cU$, Lemma \ref{Z2:tech_pcom} gives a decomposition $\tau(aba') = u s^m u'$.
We set $\pi(aba') = s^m$.
Observe that Item (3) in Lemma \ref{Z2:tech_pcom} ensures that $s^m \in \mathsf{Pow}_X(s)$, and, by Item (2) of the same lemma, $s \in \mroot \sigma(\cB)$.

We claim that 
\begin{equation} \label{eq:Z2:main_prop:i:4}
    \text{   if $aba', a\tilde{b}a' \in \cU$ and $\pi(aba') = \pi(a\tilde{b}a')$, then $b = \tilde{b}$.   }
\end{equation}
Let $\tau(aba') = u s^m u'$ and $\tau(a\tilde{b}a') = \tilde{u} \tilde{s}^{\tilde{m}} \tilde{u}'$ be the decompositions from the definition of $\pi$.
With this notation, the hypothesis $\pi(aba') = \pi(a\tilde{b}a')$ is equivalent to $s^m = \tilde{s}^{\tilde{m}}$.
Then, as $s = \mroot s$ and $\tilde{s} = \mroot \tilde{s}$, $s = \tilde{s}$ and $m = \tilde{m}$.

We now prove that $u = \tilde{u}$.
First, we assume without loss of generality that $|\tilde{u}| \leq |u|$.
Then, Lemma \ref{Z2:tech_pcom} ensures that
\begin{equation} \label{eq:Z2:main_prop:i:5}
    \text{$\tau(a)$ is prefix of both $u s^m$ and $\tilde{u} s^m$ and that $|\tilde{u}| \leq |u| \leq |\tau(a)| - 2\varepsilon$.}
\end{equation}
This implies that $s^\Z_{[0, |\tau(a)| - |u|)} = (S^{|u| - |\tilde{u}|} s^\Z)_{[0, |\tau(a)| - |u|)}$.
Combining this with the bound $|\tau(a)| - |u|\geq 2\varepsilon \geq 2|s|$ given by \eqref{eq:Z2:main_prop:i:5} allow us to use Item \ref{item:CL:on_sZ&tZ:equal} in Proposition \ref{CL:on_sZ&tZ} and conclude that $s^\Z = S^{|u| - |\tilde{u}|} s^\Z$.
Then, by Item \ref{item:CL:on_sZ&tZ:equal_mod} of the same proposition, $|u| = |\tilde{u}| \pmod{|s|}$.
From this and \eqref{eq:Z2:main_prop:i:5} we deduce that $u = \tilde{u} s^\ell$ for some $\ell \geq 0$.
But since, by Item (3) in Lemma \ref{Z2:tech_pcom}, $s$ is not a suffix of $u$, we must have that $\ell = 0$.
Therefore, $u = \tilde{u}$.

We can show, in a similar fashion, that $u' = \tilde{u}'$.
This allows us to conclude that $\tau(aba') = \tau(a\tilde{b}a') = u s^m u'$, and thus that $\tau(b) = \tau(\tilde{b})$.
Being $\tau$ injective on letters by Proposition \ref{prop:from_clopen_to_coding}, $b = b'$ and the claim is proved.

Condition \eqref{eq:Z2:main_prop:i:4} implies that $\#\A{C}{p} \leq \#\A{C}{ap}^2 \cdot \#(\cup_{s \in \mroot \sigma(\cB)} \mathsf{Pow}_X(s))$.
Hence, $\#\A{C}{p} \leq \#\A{C}{ap}^2\cdot \#\cB\cdot \pcom(X)$.
Since $\#\A{C}{ap} \leq 2d^{3d+6}$ and since $\#\cB \leq d^3$ by Item \ref{item:Z1:main_prop:lengths} in Proposition \ref{Z1:main_prop}, it follows that $\A{C}{p} \leq 2d^{3d+9}\pcom(X)$.
\medskip 

%%%%                                                                 %%%% ITEM %%%%
Item \ref{item:Z2:reco} is a direct consequence of Proposition \ref{Z2:special_reco}.

Let us prove Item \ref{item:Z2:per}.
Lemma \ref{Z2:tech_ext_per} ensures that $(Z,\tau)$ satisfies Items \ref{item:defi:per:ap} and \ref{item:defi:per:p} of Definition \ref{defi:per}.
Let now $x \in X$, $(c, z) = \Fac_{(Z,\tau)}(x)$ and $a \in \A{C}{p}$ be such that $\mroot \tau(z_0)$ is conjugate to $\mroot \tau(a)$.
We note that, by Lemma \ref{Z2:tech_ext_per}, $|\mroot \tau(z_0)| = |\mroot \tau(a)| \leq \varepsilon$.
Hence, $\per(x_{[c_0 + 8\varepsilon, c_1 - 8\varepsilon)}) \leq |\mroot \tau(z_0)| \leq \varepsilon$.
This implies, by Lemma \ref{Z2:tech_ext_per}, that $z_0 \in \A{C}{ap}$.
We can then use \eqref{eq:Z2:main_prop:i:claim_roots} to get that $\mroot \tau(z_0) = \mroot \sigma(y_0)$ and $\mroot \tau(a) = \mroot \sigma(y'_0)$ for certain $y, y' \in Y$ such that $0 \in \rQ{p}(y) \cap \rQ{p}(y')$.
We remark that, since $\mroot \tau(z_0)$ is conjugate to $\mroot \tau(a)$, the words $\mroot \sigma(y_0)$ and $\mroot \sigma(y'_0)$ are conjugate.
Therefore, $y$ and $y'$ satisfy the hypothesis of Item (3) of Lemma \ref{Z2:per_inter_props}.
We conclude that $\mroot \tau(z_0) = \mroot \sigma(y_0) = \mroot \sigma(y'_0)  \mroot \tau(a)$.

It is left to prove Item \ref{item:Z2:structure}.
Let $j \in \Z$.
We have, from Lemma \ref{Z2:per_inter_props}, that $[k(j), k(j+1))$ or $[k(j+1), k(j+2))$ is of aperiodic type.
Hence, by Lemma \ref{Z2:tech_ext_per}, $z_j$ or $z_{j+1}$ belongs to $\A{C}{ap}$.
\end{proof}

\section{The third coding}
\label{sec:Z3}

We continue refining the codings.
The main addition to this version is that the words $\tau(a)$ have controlled lengths.
The properties of the new coding are summarized in Proposition \ref{Z3:main_prop}.

%%%%                                                                 %%%% PROPOSITION %%%%
\begin{prop} \label{Z3:main_prop}
Let $X \subseteq \cA^\Z$ be a minimal infinite subshift, $n \geq 1$ and let $d$ be the maximum of $\lceil p_X(n)/n \rceil$, $p_X(n+1) - p_X(n)$, $\#\cA$ and $10^4$.
There exist a recognizable coding $(Z \subseteq \cC^\Z, \tau\colon \cC \to \cA^+)$ of $X$, a partition $\A{C}{ap} \cup \A{C}{sp} \cup \A{C}{wp}$ of $\cC$ and $\varepsilon \in [n/d^{2d^3+4}, n/d)$ such that:
\begin{enumerate}
    \item \label{item:Z3:main_prop:cardinalities}
    $\#\A{C}{ap} \leq 2d^{3d+6}$, $\#\A{C}{sp} \leq 3d^{3d+6}$, $\#\A{C}{wp} \leq 2d^{3d+9}\pcom(X)^4$ and $\#\mroot\tau(\cC) \leq 5d^{3d+6}$.
    \item \label{item:Z3:main_prop:lengths} 
    $20\varepsilon \leq |\tau(a)| \leq 10d^2n$ for all $a \in \cC$.
    \item \label{item:Z3:main_prop:reco}
    $(Z, \tau)$ satisfies the recognizability property described in Proposition \ref{Z3:special_reco}.
    \item \label{item:Z3:main_prop:per}
    $(Z, \tau)$ has dichotomous periods  w.r.t.\ $(\A{C}{ap},\A{C}{sp}\cup\A{C}{wp})$ and $8\varepsilon$.
    \item \label{item:Z3:main_prop:structure}
    The set $\A{C}{sp}$ satisfies the property described in Proposition \ref{Z3:structure_Csp}.
\end{enumerate}
\end{prop}

\begin{rema} \label{rema:Z3:special_per&root=per}
Assume the notation of Proposition \ref{Z3:main_prop}.
Then, $a \in \cC \setminus \A{C}{ap}$ implies that $|\mroot \tau(a)| = \per(\tau(a))$.
Indeed, Item \ref{item:Z3:main_prop:per} ensures that $|\mroot \tau(a)| \leq \varepsilon$, and Item \ref{item:Z3:main_prop:lengths} that $|\tau(a)| \geq 2\varepsilon$, therefore, by Lemma \ref{CL:root=per},  $|\mroot \tau(a)| = \per(\tau(a))$.
\end{rema}

%%%%                                                                 %%%% PROPOSITION %%%%
\begin{prop} \label{Z3:special_reco}
Consider the coding described in Proposition \ref{Z3:main_prop}.
For any $x, \tilde{x} \in X$, we have that:
\begin{enumerate}
    \item If $\per(x_{[-\varepsilon, \varepsilon)}) > \varepsilon$ and $x_{[-7d^2n, 7d^2n)} = \tilde{x}_{[-7d^2n, 7d^2n)}$, then $\Fac^0_{(Z,\tau)}(x)$ is equal to $\Fac_{(Z,\tau)}(\tilde{x})$.
    \item If $k \geq 0$, $x_{[-50d^2n, k+50d^2n)} = \tilde{x}_{[-50d^2n, k+50d^2n)}$ and $\Fac^0_{(Z,\tau)}(x)$ is equal to $\Fac_{(Z,\tau)}(\tilde{x})$, then $\Fac^0_{(Z,\tau)}(S^k x)$ is equal to $\Fac_{(Z,\tau)}(S^k \tilde{x})$.
\end{enumerate}
\end{prop}

%%%%                                                                 %%%% PROPOSITION %%%%
\begin{prop} \label{Z3:structure_Csp}
The coding of Proposition \ref{Z3:main_prop} satisfies the following.
\begin{enumerate}
    \item \label{item:Z3:structure_Csp:window}
    If $z \in Z$ and $i < j$ are integers such that $z_k \in \cC \setminus \A{C}{ap}$ for all $k \in [i, j+1)$, then $z_k = z_i \in \A{C}{sp}$ for all $k \in [i, j)$ and $\mroot \tau(z_k) = \mroot \tau(z_i)$ for all $k \in [i,j+1)$.
    
    \item \label{item:Z3:structure_Csp:power_of_2}
    If $a \in \A{C}{sp}$, then $\tau(a) = (\mroot\tau(a))^{2^r}$, where $r$ is the unique integer for which $2^r|\mroot\tau(a)|$ belongs to $[20\varepsilon, 40\varepsilon)$.
\end{enumerate}
\end{prop}

We now introduce the notation that will be used in this section.
Let $X \subseteq \cA^\Z$ be a minimal infinite subshift, $n \geq 1$ and let $d$ be the maximum of $\lceil p_X(n)/n \rceil$, $p_X(n+1) - p_X(n)$, $\#\cA$ and $10^4$.
Then, Proposition \ref{Z2:main_prop} applied to $X$ and $n$ gives a recognizable coding $(Y \subseteq \cC_Y^\Z, \sigma\colon\cC\to\cA^+)$ of $X$, a partition $\cC_Y = \A{C}{p} \cup \A{C}{ap}$, and an integer $\varepsilon \in [n/d^{2d^3+4}, n/d)$ satisfying the properties described in Proposition \ref{Z2:main_prop}.
\medskip

The strategy to prove the main proposition of this section is the following. 
The coding $(Z, \tau)$ will be obtained from $(Y, \sigma)$ by splitting the words in $\sigma(\A{C}{p})$ into subwords of carefully chosen lengths.
This will maintain most of the properties of $(Y,\sigma)$ at the same time that we gain control on the lengths of all the words $\tau(a)$.
A delicate part involves defining the splittings of the words in $\sigma(\A{C}{p})$ in such a way that $(Z,\tau)$ has the recognizability properties in Proposition \ref{Z3:special_reco}.

%%%%                                                                 %%%% SUBSECTION %%%%
\subsection{Construction of the third coding}
\label{subsec:Z3:defi_Z3}

For $s \in \mroot\sigma(\A{C}{p})$, we define $\zeta(s)$ as the unique power of two such that $\zeta(s)\cdot|s|$ lies in $[20\varepsilon, 40\varepsilon)$.
Note that, by Item \ref{item:Z2:per} of Proposition \ref{Z2:main_prop}, we have that $\zeta(s) \geq 1$.
Then, for $a \in \A{C}{p}$, we can define $p_a$ and $q_a$ as the unique integers satisfying
\begin{equation} \label{eq:Z3:defi_pa_qa}
    p_a\cdot \zeta(\mroot\sigma(a)) + q_a = \frac{|\sigma(a)|}{|\mroot\sigma(a)|}
    \enspace\text{and}\enspace
    0 < q_a \leq \zeta(\mroot\sigma(a)).
\end{equation}
It is important to remark that Item \ref{item:Z2:lengths} in Proposition \ref{Z2:main_prop} ensures that $p_a \geq 2$.

For $a \in \A{C}{p}$, let
\begin{equation*}
    \psii{sp}(a) = (\mroot\sigma(a))^{\zeta(\mroot\sigma(a))}
    \enspace\text{and}\enspace
    \psii{wp}(a) = (\mroot\sigma(a))^{\zeta(\mroot\sigma(a)) + q_a}.
\end{equation*}
Note that
\begin{equation*}
    \sigma(a) = \psii{sp}(a)^{p_a-1} \psii{wp}(a)
    \enspace\text{for all $a \in \A{C}{p}$.}
\end{equation*}
We also choose bijections 
\begin{equation*}
    \phii{sp} \colon \A{C}{sp} \to \psii{sp}(\A{C}{p})
    \enspace\text{and}\enspace
    \phii{wp} \colon \A{C}{wp} \to \psii{wp}(\A{C}{p}),
\end{equation*}
where $\A{C}{sp}$, $\A{C}{wp}$ and $\A{C}{ap}$ are pairwise disjoint.
Then, we define for $a \in \cC_Y$, 
\begin{equation} \label{eq:Z3:defi_eta}
    \eta(a) = \begin{cases}
        a & \text{if $a \in \A{C}{ap}$} \\
        \phii{sp}^{-1}(\psii{sp}(a))^{p_a-1} \phii{wp}^{-1}(\psii{wp}(a)) & \text{if $a \in \A{C}{wp}$}
    \end{cases}
\end{equation}
Let $\cC_Z = \A{C}{ap} \cup \A{C}{sp} \cup \A{C}{wp}$ and, for $a \in \cC_Z$, we set
\begin{equation} \label{eq:Z3:defi_tau}
    \tau(a) = \begin{cases}
        \sigma(a)     & \text{if $a \in \A{C}{ap}$} \\
        \phii{sp}(a)  & \text{if $a \in \A{C}{sp}$} \\
        \phii{wp}(a)  & \text{if $a \in \A{C}{wp}$}
    \end{cases}
\end{equation}
It then follows that
\begin{equation} \label{eq:Z3:sigma=tau_eta}
    \sigma = \tau \eta.
\end{equation}
Finally, we set $Z = \bigcup_{k\in\Z} S^k \eta(Y)$.

Let us comment on the definition of $\tau$.
Equation \eqref{eq:Z3:sigma=tau_eta} says that $\sigma(a) = \tau(a)$ if $a \in \A{C}{ap}$ and that $\sigma(a) = \tau(b)^{p_a -1} \tau(c)$ if $a \in \A{C}{p}$, $b = \phii{sp}^{-1}(\psii{sp}(a))$ and $c = \phii{wp}^{-1}(\psii{wp}(c))$.
In other words, $\tau$ is obtained from $\sigma$ by slicing the words $\sigma(a)$.

\subsection{Proof of Propositions \ref{Z3:main_prop}, \ref{Z3:special_reco} and \ref{Z3:structure_Csp}}

%%%%                                                                 %%%% PROOF %%%%
\begin{proof}[Proof of Proposition \ref{Z3:main_prop}]
We start with Item (1).
Item \ref{item:Z2:cardinalities} in Proposition \ref{Z2:main_prop} gives the bound $\#\A{C}{ap} \leq 2d^{3d+6}$.
Also, it follows from the definitions and Item \ref{item:Z2:cardinalities} in Proposition \ref{Z2:main_prop} that 
\begin{align*}
    \#\A{C}{sp} &= \#\psii{sp}(\A{C}{p}) = \#\mroot\sigma(\A{C}{p}) \leq 3d^{3d+6}
    \enspace\text{and} \\
    \#\A{C}{wp} &= \#\psii{wp}(\A{C}{p}) \leq \#\A{C}{p} \leq 2d^{3d+9}\cdot\pcom(X).
\end{align*}
Now, using \ref{eq:Z3:defi_tau} yields $\mroot \tau(\cC_Z \setminus \A{C}{ap}) = \mroot \sigma(\A{C}{p})$.
Therefore, $\#\mroot\tau(\cC_Z) \leq \#\A{C}{ap} + \#\mroot\sigma(\A{C}{p})$, which gives $\#\mroot\tau(\cC_Z) \leq 4 d^{3d+6}$ if we use the bounds $\#\A{C}{ap} \leq 2d^{3d+6}$ and $\#\mroot\sigma(\cC_Y) \leq 3d^{3d+6}$, where the last bound is given by Item \ref{item:Z2:cardinalities} in Proposition \ref{Z2:main_prop}.

We now consider Item (2).
Let $a \in \cC_Z$.
If $a \in \A{C}{ap}$, then by Item \ref{item:Z2:lengths} in Proposition \ref{Z2:main_prop} we have that $20\varepsilon \leq |\tau(a)| = |\sigma(a)| \leq 10d^2n$.
If $a \in \cC_Z \setminus \A{C}{ap}$ and $s = \mroot\tau(a)$, then \eqref{eq:Z3:defi_tau} implies that 
\begin{equation*}
    |\tau(a)| = \begin{cases}
        \zeta(s) |s| & \text{if $a \in \A{C}{sp}$} \\
        (\zeta(s) + q_a) |s| & \text{if $a \in \A{C}{wp}$}
    \end{cases}
\end{equation*}
Since $20\varepsilon \leq \zeta(s) |s| < 40\varepsilon$ and $0 < q_a \leq |\zeta(s)|$, we obtain that $20\varepsilon \leq |\tau(a)| \leq 80\varepsilon \leq 10d^2n$.

Next, we prove Item (3).
Let $x \in X$, $(c,y) = \Fac_{(Y,\sigma)}$ and $(f,z) = \Fac_{(Z,\tau)}(x)$.
We use Lemma \ref{lem:factorizations_and_compositions} to obtain $m \in [0, |\eta(y_0)|)$ such that 
\begin{equation} \label{eq:Z3:main_prop:1}
    \eta(y_0) = z_{[-m, -m + |\eta(z_0)|)}
    \enspace\text{and}\enspace
    c_0 = f_0 - |\tau(z_{[-m, 0)})|.
\end{equation}
We first assume that $z_0 \in \A{C}{ap}$.
Then, Equation \eqref{eq:Z3:main_prop:1} implies that $z_0$ occurs in $\eta(y_0)$, and thus, since $z_0 \in \A{C}{ap}$, Equation \eqref{eq:Z3:defi_eta} ensures that $\eta(y_0) = y_0 \in \A{C}{ap}$.
Hence, $m = 0$, $z_0 = y_0$, $c_0 = f_0$ and $c_1 = c_0 + |\sigma(y_0)| = f_0 + |\tau(z_0)| = f_1$.
Using this and Item \ref{item:Z2:per} of Proposition \ref{Z2:main_prop} with $x$ and $(c,y)$ produces
\begin{equation*}
    \per(x_{[f_0 + 8\varepsilon, f_1 - 8\varepsilon)}) =
    \per(x_{[c_0 + 8\varepsilon, c_1 - 8\varepsilon)}) > \varepsilon.
\end{equation*}
Let us now assume that $z_0 \in \cC_Z \setminus \A{C}{ap}$.
This condition and Equation \eqref{eq:Z3:defi_eta} imply that $y_0 \in \A{C}{p}$.
Hence, we can use Item \ref{item:Z2:per} in Proposition \ref{Z2:main_prop} to obtain that $\per(x_{[c_0 - 8\varepsilon, c_1 + 8\varepsilon)}) \leq \varepsilon$.
We conclude, since \eqref{eq:Z3:defi_eta} guarantees that $[f_0, f_1)$ is contained in $[c_0, c_1)$, that $\per(x_{[f_0 - 8\varepsilon, f_1 + 8\varepsilon)}) \leq \per(x_{[c_0 - 8\varepsilon, c_1 + 8\varepsilon)}) \leq \varepsilon$.

It rests to prove that $(Z,\tau)$ satisfies Item \ref{item:defi:per:conj_roots} of Definition \ref{defi:per}.
Let $a \in \cC_Z \setminus \A{C}{ap}$ be such that $\mroot \tau(z_0)$ is conjugate to $\mroot \tau(a)$.
We have from \eqref{eq:Z3:main_prop:1} that $z_0$ occurs in $\eta(y_0)$, so, by \eqref{eq:Z3:defi_eta}, $\mroot \tau(z_0) = \mroot \sigma(y_0)$.
Similarly, $a$ occurs in $\eta(b)$ and $\mroot \tau(a) = \mroot \sigma(b)$ for some $b \in \cC_Y$.
The first condition and \eqref{eq:Z3:defi_eta} imply, as $a \in \cC_Z \setminus \A{C}{ap}$, that $b \in \A{C}{p}$.
Now, the hypothesis ensures that $\mroot \sigma(y_0)$ is conjugate to $\mroot \sigma(b)$.
Therefore, as $(Y,\sigma)$ has dichotomous periods w.r.t.\ $(\A{C}{ap},\A{C}{p})$, we can use Item (3) of Definition \ref{defi:per} to obtain that $\mroot \sigma(y_0) = \mroot \sigma(b)$.
We conclude that $\mroot \tau(z_0) = \mroot \tau(a)$, completing the proof of Item (3).

Finally, for Items (4) and (5), we present the proofs of Propositions \ref{Z3:special_reco} and \ref{Z3:structure_Csp} hereafter.
\end{proof}
\medskip

%%%%                                                                 %%%% LEMMA %%%%
\begin{lem}\label{Z3:precise_structure}
Let $\A{C}{block}$ be the set of words $a^\ell b$, where $a \in \A{C}{sp}$, $b \in \A{C}{wp}$, $\ell \geq 1$ and $\mroot\tau(a) = \mroot\tau(b)$.
Then, any $z \in Z$ can be written as $z = \dots w_{-1} w_0 w_1 \dots$, where $w_j \in \A{C}{block}$ or $w_j \in \A{C}{ap}$ and $w_j w_{j+1} \not\in \A{C}{block}^2$.
\end{lem}
\begin{proof}
Let $z \in Z$ and $(c,y) = \Fac_{(Y,\sigma)}(\tau(z))$.
We set $w_j = \eta(y_j)$.
The definition of $\eta$ in \eqref{eq:Z3:defi_eta} ensures that $w_j \in \A{C}{block} \cup \A{C}{ap}$.
Moreover, Item \ref{item:Z2:structure} in Propositions \ref{Z2:main_prop} says, in this context, that $w_j w_{j+1} \not\in \A{C}{block}^2$ for all $j \in \Z$.
Finally, by Lemma \ref{lem:factorizations_and_compositions} we have that $z = S^\ell \eta(y)$ for some $\ell \in \Z$, and thus that $z = \dots w_{-1} w_0 w_1 \dots$
\end{proof}

We can now present the proof of Proposition \ref{Z3:structure_Csp}.
%%%%                                                                 %%%% PROOF %%%%
\begin{proof}[Proof of Proposition \ref{Z3:structure_Csp}]
Item \ref{item:Z3:structure_Csp:power_of_2} directly follows from the definition of $\tau$ in \eqref{eq:Z3:defi_tau}.
Let us prove Item \ref{item:Z3:structure_Csp:window}.
Let $z \in Z$ be such that $z_{[i, j+1)} \subseteq (\cC_Z \setminus \A{C}{ap})^+$.
We write $z = \dots w_{-1} w_0 w_1 \dots$ as in Lemma \ref{Z3:precise_structure} and let $\A{C}{block}$ be the set defined in Lemma \ref{Z3:precise_structure}.
Then, the hypothesis $z_{[i,j+1)} \in (\cC_Z \setminus \A{C}{ap})^+$ and the condition $w_kw_{k+1} \not\in \A{C}{block}$ imply that $z_{[i, j+1)}$ occurs in $w_k$ for certain $k \in \Z$ such that $w_k \in \A{C}{block}$.
Hence, $z_\ell = z_i \in \A{C}{sp}$ for all $\ell \in [i,j)$ and $\mroot \tau(z_k) = \mroot \tau(z_i)$ for all $k \in [i,j+1)$.
\end{proof}

%%%%                                                                 %%%% LEMMA %%%%
\begin{lem} \label{Z3:prefix_and_per=>same_root}
Let $z, \tilde{z} \in Z$ and $\ell \geq 1$ be such that $\tau(\tilde{z}_{[0,\ell)})$ is a prefix of $\tau(z_0)$.
Then, $z_0 \in \cC_Z \setminus \A{C}{ap}$ implies that $\tilde{z}_i \in \cC_Z \setminus \A{C}{ap}$ and $\mroot\tau(\tilde{z}_i) = \mroot\tau(z_0)$ for all $i \in [0,\ell)$.
\end{lem}
\begin{proof}
First, we note that, by Item \ref{item:Z3:main_prop:per} in Proposition \ref{Z3:main_prop}, $\per(\tau(\tilde{z}_i)) \leq \per(\tau(z_0)) \leq \varepsilon$ for all $i \in [0, \ell)$.
Thus, by Item \ref{item:Z3:main_prop:per} in Proposition \ref{Z3:main_prop}, $\tilde{z}_i \in \cC_Z \setminus \A{C}{ap}$ for all $i \in [0, \ell)$.

Let $s = \mroot\tau(z_0)$.
In order to continue, we claim that 
\begin{equation*}
    \text{if $a \in \cC_Z \setminus \A{C}{sp}$ and $\tau(a)$ is a prefix of $s^\infty$, then $\mroot\tau(a) = s$.}
\end{equation*}
First, we note $\tau(a)$ is a prefix of $(\mroot\tau(a))^\infty$.
Also, Item \ref{item:Z3:main_prop:per} in Proposition \ref{Z3:main_prop} guarantees that $|\mroot\tau(a)| \leq \varepsilon$ and that $|s| \leq \varepsilon$.
Hence, as $|\tau(a)| \geq 2\varepsilon$ by Item \ref{item:Z3:main_prop:lengths} in Proposition \ref{Z3:main_prop}, we can use Theorem \ref{theo:fine&wilf} to deduce that $\mroot\tau(a)$ and $s$ are powers of a common word $r$.
This implies that $s = \mroot\tau(a) = \mroot r$.

We now prove that $\mroot\tau(\tilde{z}_i) = s$ for $i \in [0, \ell)$ by induction on $i$.
If $i = 0$, then we have from the hypothesis that $\tau(\tilde{z}_0)$ is a prefix of $\tau(z_0)$ and that $\tilde{z}_0 \in \cC_Z \setminus \A{C}{ap}$.
Thus, $\mroot\tau(\tilde{z}_0) = s$ by the claim.
Let us assume now that $0 < i < \ell$ and that $\mroot\tau(\tilde{z}_j) = s$ for $j \in [0, i)$.
Then, $\tau(\tilde{z}_{[0, i)})$ is a power of $s$.
Being $\tau(\tilde{z}_{[0, i]})$ a prefix of $\tau(z_0)^\infty = s^\infty$, we deduce that $\tau(\tilde{z}_i)$ is a prefix of $s^\infty$.
This allows us to use the claim and obtain that $\mroot\tau(\tilde{z}_i) = s$.
This finishes the inductive step and the proof of the lemma.
\end{proof}

%%%%                                                                 %%%% PROOF %%%%
Finally, we prove Proposition \ref{Z3:special_reco}
\begin{proof}[Proof of Proposition \ref{Z3:special_reco}]
We fix the following notation for this proof.
Let $x, \tilde{x} \in X$, $(f,z) = \Fac_{(Z,\tau)}(x)$, $(\tilde{f},\tilde{z}) = \Fac_{(Z,\tau)}(\tilde{x})$, $(c,y) = \Fac_{(Y,\sigma)}(x)$ and $(\tilde{c},\tilde{y}) = \Fac_{(Y, \sigma)}(\tilde{x})$.

We start by proving Item (1).
Assume that $\per(x_{[-\varepsilon, \varepsilon)}) > \varepsilon$ and that $x_{[-7d^2n, 7d^2n)} = x'_{[-7d^2n, 7d^2n)}$.
We have to show that $f_0 = \tilde{f}_0$ and $z_0 = \tilde{z}_0$.

We claim that
\begin{enumerate}[label=(\roman*)]
    \item $c_0 = \tilde{c}_0$ and $y_0 = \tilde{y}_0$;
    \item $y_0 \in \A{C}{ap}$ and $\tilde{y}_0 \in \A{C}{ap}$;
    \item $z_0 = y_0$, $f_0 = c_0$, $\tilde{z}_0 = \tilde{y}_0$ and $\tilde{f}_0 = \tilde{c}_0$.
\end{enumerate}
Item (i) follows from the fact that the current hypothesis allows us to use Item \ref{item:Z2:reco} of Proposition \ref{Z2:main_prop} to get that $\Fac^0_{(Y,\sigma)}(x) = \Fac^0_{(Y,\sigma)}(\tilde{x})$, which is equivalent to (i).
For Item (ii), we note that if $y_0 \in \cC_Y \setminus \A{C}{ap} = \A{C}{p}$, then Item \ref{item:Z2:per} in Proposition \ref{Z2:main_prop} implies that $\per(x_{[c_0-8\varepsilon, c_1+8\varepsilon)}) \leq \varepsilon$.
Hence, as $[-\varepsilon, \varepsilon)$ is included in $[c_0-8\varepsilon, c_1+8\varepsilon)$, $\per(x_{[-\varepsilon, \varepsilon)}) \leq \varepsilon$, contradicting our hypothesis.
Therefore, $y_0 = \tilde{y}_0 \in \A{C}{ap}$.

To prove Item (iii), we first note that Lemma \ref{lem:factorizations_and_compositions} gives an integer $m \in [0, |\tau(z_0)|)$ such that
\begin{equation} \label{eq:Z3:special_reco:1}
    \text{$y_{[-m, -m + |\eta(z_0)|)} = \eta(z_0)$ and $-c_0 = -f_0 + |\eta(z)_{[0, m)}|$.}
\end{equation}
In particular, $y_0$ occurs in $\eta(z_0)$.
Since Item (ii) ensures that $y_0 \in \A{C}{ap}$, it follows from the definition of $\eta$ in \eqref{eq:Z3:defi_eta} that $\eta(y_0) = y_0 = z_0$.
Putting this in \eqref{eq:Z3:special_reco:1} gives that $m = 0$ and $c_0 = f_0$.
A similar argument shows that $\tilde{z}_0 = \tilde{y}_0$ and $\tilde{f}_0 = \tilde{c}_0$ as well.
This completes the proof of the claim.

Items (i) and (iii) of the claim imply that $(f_0, z_0) = (\tilde{f}_0, \tilde{z}_0)$, proving Item (1) of the proposition.
\medskip

Before proving Item (2), we claim that
\begin{equation} \label{eq:Z3:special_reco:2}
    \text{if $z_0 \in \A{C}{ap}$ and $x_{[-17d^2n, 17d^2n)} = \tilde{x}_{[-17d^2n, 17d^2n)}$, then $\Fac^0_{(Z,\tau)}(x) = \Fac^0_{(Z,\tau)}(\tilde{x})$.}
\end{equation}
To prove \eqref{eq:Z3:special_reco:2}, we start by using Item \ref{item:Z3:main_prop:per} in Proposition \ref{Z3:main_prop} to obtain that $\per(x_{[f_0+8\varepsilon, f_1-8\varepsilon)}) > \varepsilon$.
Thus, by Item \ref{item:CL:localize_per_aper:loc_aper} in Proposition \ref{CL:localize_per_aper}, there exists $j \in [f_0+8\varepsilon, f_1-8\varepsilon)$ satisfying
\begin{equation} \label{eq:Z3:special_reco:3}
    \per(x_{[j-\varepsilon, j+\varepsilon)}) > \varepsilon.
\end{equation}
Now, since $j \in [f_0+8\varepsilon, f_1-8\varepsilon)$ and $|\tau| \leq 10d^2n$, we have that $j \in [-10d^2n, 10d^2n)$.
Therefore, by the hypothesis $x_{[-17d^2n, 17d^2n)} = \tilde{x}_{[-17d^2n, 17d^2n)}$, $\tilde{x}_{[j-7d^2n, j+7d^2n)} = x_{[j-7d^2n, j+7d^2n)}$.
Combining this with \eqref{eq:Z3:special_reco:3} allows us to use Item (1) of this proposition and deduce that
\begin{equation} \label{eq:Z3:special_reco:4}
    \Fac^0_{(Z,\tau)}(S^j x) = \Fac^0_{(Z,\tau)}(S^j \tilde{x}).
\end{equation}
Observe that the condition $j \in [f_0+8\varepsilon, f_1-8\varepsilon)$ implies that $\Fac^0_{(Z,\tau)}(S^j x) = (f_0 - j, z_0)$.
Let $i$ be the integer satisfying $\tilde{f}_i \leq j < \tilde{f}_{i+1}$ and note that $\Fac^0_{(Z,\tau)}(S^j \tilde{x}) = (\tilde{f}_i - j, \tilde{z}_i)$.
Then, by \eqref{eq:Z3:special_reco:4}, $f_0 = \tilde{f}_i$ and $z_0 = \tilde{z}_i$.
In particular, $\tilde{f}_i = f_0 \leq 0 < f_1 = \tilde{f}_{i+1}$, so $i = 0$.
We conclude that $\Fac^0_{(Z,\tau)}(x) = (f_0, z_0) = (\tilde{f}_0, \tilde{z}_0) = \Fac^0_{(Z,\tau)}(\tilde{x})$.
\medskip

We now prove Item (2).
Assume that $\Fac^0_{(Z,\tau)}(x) = \Fac^0_{(Z,\tau)}(\tilde{x})$.
The is equivalent to $z_0 = \tilde{z}_0$ and $f_0 = \tilde{f}_0$, so $f_1 = \tilde{f}_1$ as well.
Hence,
\begin{equation} \label{eq:Z3:special_reco:5}
    \Fac^0_{(Z,\tau)}(S^i x) = \Fac^0_{(Z,\tau)}(S^i \tilde{x})
    \enspace\text{for all $i \in [f_0, f_1)$.}
\end{equation}
We are going to prove that
\begin{equation} \label{eq:Z3:special_reco:6}
    \text{if $x_{[-50d^2n, 50d^2n)} = \tilde{x}_{[-50d^2n, 50d^2n)}$, then 
    $\Fac^0_{(Z,\tau)}(S^{f_1} x) = \Fac^0_{(Z,\tau)}(S^{f_1} \tilde{x})$.}
\end{equation}
The lemma then follows from an inductive argument on $k$ that uses Equations \eqref{eq:Z3:special_reco:5} and \eqref{eq:Z3:special_reco:6}.
\medskip

Let us assume that $x_{[-50d^2n, 50d^2n)} = \tilde{x}_{[-50d^2n, 50d^2n)}$.
We consider two cases.
First, we assume that $z_1$ or $\tilde{z}_1$ belongs to $\A{C}{ap}$.
There is no loss of generality in assuming that $z_1$ is the one belonging to $\A{C}{ap}$.
Observe that the hypothesis and that $|\tau| \leq 10d^2n$ ensure that $x_{[f_1-7d^2n, f_1+7d^2n)} = \tilde{x}_{[f_1-7d^2n,f_1+7d^2n)}$.
This allows us to use \eqref{eq:Z3:special_reco:2} and deduce that $\Fac^0_{(Z,\tau)}(S^{f_1} x) = \Fac^0_{(Z,\tau)}(S^{f_1} \tilde{x})$.
\medskip

We now consider the case in which $z_1$ and $\tilde{z}_1$ belong to $\cC_Z \setminus \A{C}{ap}$.
Observe that, since $f_1 = \tilde{f}_1$, we have that $\Fac^0_{(Z,\tau)}(S^{f_1} x) = (0, z_1)$ and $\Fac^0_{(Z,\tau)}(S^{f_1} \tilde{x}) = (0, \tilde{z}_1)$.
Thus, it is enough to show that $z_1 = \tilde{z}_1$.

We assume without loss of generality that $|\tau(z_1)| \geq |\tau(\tilde{z}_1)|$.
Let $\ell$ be the integer satisfying $\tilde{f}_\ell \leq f_2 < \tilde{f}_{\ell+1}$.
Remark that $\ell \geq 2$ as $\tilde{f}_2 = \tilde{f}_1 + |\tau(\tilde{z}_1)| \leq f_1 + |\tau(z_1)| = f_2$.
Being $\tilde{f}_1 = f_1$ and $\tilde{f}_\ell \leq f_2$, the hypothesis $x_{[-50d^2n, 50d^2n)} = \tilde{x}_{[-50d^2n, 50d^2n)}$ and the bound $|\tau| \leq 7d^2n$ ensure that $\tau(\tilde{z}_{[1,\ell)})$ is a prefix of $\tau(z_1)$.
Hence, since we assumed that $z_1 \in \cC_Z \setminus \A{C}{sp}$, Lemma \ref{Z3:prefix_and_per=>same_root} yields that
\begin{equation} \label{eq:Z3:special_reco:7}
    \tilde{z}_i \in \cC_Z \setminus \A{C}{ap}
    \enspace\text{and}\enspace
    \mroot\tau(\tilde{z}_i) = \mroot\tau(z_1)
    \enspace\text{for all $i \in [1,\ell)$.}
\end{equation}
We set $s = \mroot\tau(z_1)$.
It then follows from \ref{eq:Z3:special_reco:7} and \ref{eq:Z3:defi_tau} that for any $i \in [1,\ell)$
\begin{equation} \label{eq:Z3:special_reco:8}
    z_1 = \begin{cases}
        s^{\zeta(s)} & \text{if $z_1 \in \A{C}{sp}$}\\
        s^{\zeta(s) + q_{z_1}} & \text{if $z_1 \in \A{C}{wp}$}
    \end{cases}
    \qquad
    \tilde{z}_i = \begin{cases}
        s^{\zeta(s)} & \text{if $\tilde{z}_i \in \A{C}{sp}$}\\
        s^{\zeta(s) + q_{\tilde{z}_i}} & \text{if $\tilde{z}_i \in \A{C}{wp}$}
    \end{cases}
\end{equation}
We can use this to prove that $|\tau(z_1)| = |\tau(\tilde{z}_1)|$ implies that $z_1 = \tilde{z}_1$.
Indeed, in the case $z_1 \in \A{C}{sp}$, it follows from \eqref{eq:Z3:special_reco:8} and the fact that $q_a > 0$ for all $a \in \A{C}{wp}$ that $\tau(\tilde{z}_1) = \tau(z_1) = s^{\zeta(s)}$, and thus that $\tilde{z}_1 = z_1 = \psii{sp}^{-1}(s^{\zeta(s)})$.
Similarly, if $z_1 \in \A{C}{wp}$, then \eqref{eq:Z3:defi_pa_qa} and the equation $\mroot \tau(z_1) = \mroot \tau(\tilde{z}_1)$ ensure that $q_{z_1}=q_{\tilde{z}_1}$, and thus from \eqref{eq:Z3:special_reco:8} we get that $\tau(\tilde{z}_1) = \tau(z_1) = s^{\zeta(s) + q_{z_1}}$.
In particular, $\tilde{z}_1 = z_1 = \psii{wp}^{-1}(s^{\zeta(s) + q_{z_1}})$.

It is left to consider the case $|\tau(z_1)| > |\tau(\tilde{z}_1)|$, so let us assume that this condition is satisfied.
Then, by \eqref{eq:Z3:special_reco:8} and the fact that $q_a > 0$ for all $a \in \A{C}{wp}$,
\begin{equation} \label{eq:Z3:special_reco:9}
    z_1 \in \A{C}{wp} \enspace\text{and}\enspace
    \tilde{z}_1 \in \A{C}{sp}.
\end{equation}
In this situation, Item \ref{item:Z3:structure_Csp:window} in Proposition \ref{Z3:structure_Csp} ensures that $z_2 \in \A{C}{ap}$.
Now, observe that $f_3 \leq 3|\tau| \leq 30d^2n$ and $f_2 \geq 0$; so the hypothesis $x_{[-50d^2n, 50d^2n)} = \tilde{x}_{[-50d^2n, 50d^2n)}$ gives that $x_{[f_2 - 7d^2n, f_3 + 7d^2n)} = \tilde{x}_{[f_2 - 7d^2n, f_3 + 7d^2n)}$.
Hence, we can use \eqref{eq:Z3:special_reco:2} to obtain that $\Fac^0_{(Z,\tau)}(S^{f_2} x)$ is equal to $\Fac^0_{(Z,\tau)}(S^{f_2} \tilde{x})$.
More precisely, since $\tilde{f}_\ell \leq f_2 < \tilde{f}_{\ell+1}$, we can write
\begin{equation*}
    (\tilde{f}_\ell - f_2, \tilde{z}_\ell) = \Fac^0_{(Z,\tau)}(S^{f_2} \tilde{x}) 
    = \Fac^0_{(Z,\tau)}(S^{f_2} x) = (0, z_2).
\end{equation*}
Hence, the equation $\Fac^0_{(Z,\tau)}(S^{f_2} x) = \Fac^0_{(Z,\tau)}(S^{f_2} \tilde{x})$ is equivalent to
\begin{equation} \label{eq:Z3:special_reco:10}
    \tilde{f}_\ell = f_2
    \enspace\text{and}\enspace
    \tilde{z}_\ell = z_2.
\end{equation}
In particular, $\tau(\tilde{z}_{[1,\ell)}) = \tau(z_1)$.

Now, since $z_1 \in \A{C}{wp}$, we have by Item \ref{item:Z3:structure_Csp:window} in Proposition \ref{Z3:structure_Csp} that $\tilde{z}_\ell = z_2 \in \A{C}{ap}$.
Therefore, Item \ref{item:Z3:structure_Csp:window} in Proposition \ref{Z3:structure_Csp} guarantees that $\tilde{z}_i \in \A{C}{sp}$ for all $i \in [0, \ell-1)$ and $\tilde{z}_{\ell-1} \in \A{C}{sp}$.
Combining this with \eqref{eq:Z3:special_reco:8} and the equality $\tau(\tilde{z}_{[1,\ell)}) = \tau(z_1)$ produces that 
\begin{equation*}
  (\zeta(s) + q_{z_1}) |s| = |\tau(z_1)| =
  |\tau(\tilde{z}_{[0,\ell)})| = (\ell \zeta(s) + q_{\tilde{z}_{\ell-1}})|s|.
\end{equation*}
Since $q_{z_1} \leq \zeta(s)$ and $q_{\tilde{z}_{\ell-1}} > 0$, we conclude that $\ell \leq 1$.
But then $\tau(z_1) = \tau(\tilde{z}_{[1,\ell)})$ is the empty word, which contradicts the definition of $\tau$.
Therefore, the case $|\tau(z_1)| > |\tau(\tilde{z}_1)|$ does not occur and the proof is complete.
\end{proof}
\section{The fourth coding}
\label{sec:Z4}

In this section, we give the final versions of the codings needed in the proof of the main theorems.
The new element of these codings is that it is possible to connect them using morphisms.

The section has two parts.
In the first one, we construct the new codings, using Proposition \ref{Z3:main_prop} and a modified higher block construction, and present their basic properties.
Then, in the second one, we show how we can connect two of these codings using the morphism described in Subsection \ref{subsubsec:Z4:morphism}.

\subsection{Construction of the fourth coding}
\label{subsec:Z4:construction}

Let $X \subseteq \cA^\Z$ be an infinite minimal subshift, $n \geq 0$ and let $d$ be the maximum of $\lceil p_X(n)/n \rceil$, $p_X(n+1) - p_X(n)$, $\#\cA$ and $10^4$.
We use Proposition \ref{Z3:main_prop} with $X$ and $n$ to obtain a recognizable coding $(Y \subseteq \cB^\Z, \sigma\colon\cB \to \cA^+)$ of $X$, a partition $\cB = \A{B}{ap}\cup\A{B}{sp}\cup\A{B}{wp}$ and an integer $\varepsilon \in [n/d^{2d^3+4}, n/d)$ satisfying Items (1) to (5) of Proposition \ref{Z3:main_prop}.

We start with the following observation.
Since $(Y, \sigma)$ is a recognizable coding of a minimal subshift, $Y$ is minimal; thus, for all $y \in Y$ there exists $k < 0$ such that $y_k \in \A{B}{ap}$.
This observation allows us to define the map $\First(y) = \max\{k < 0 : y_k \in \A{B}{ap}\}$ that returns the index of the first-to-the-left symbol in $\A{B}{ap}$.

Let $\psi_0\colon Y \to \cB^4$ be the map $y \mapsto y_{\First(y)} y_{-1} y_0 y_1$ and $\psi(y) = (\psi_0(S^j y))_{j\in\Z}$.
We treat $\psi(y)$ as a sequence over the alphabet $\cB^4$ and define $Z = \psi(Y) \subseteq (\cB^4)^\Z$.
We set
\begin{equation} \label{eq:Z4:defi_alphabet}
    \cC = \{z_0 : z \in Z\} \subseteq \cB^4.
\end{equation}
Let $\theta(a a_{-1} a_0 a_1) = a_0$ for $a a_{-1} a_0 a_1 \in \cC$ and $\tau = \sigma \theta$.
Remark that
\begin{equation} \label{eq:Z4:theta_psi=id}
    \text{$\theta\psi(y) = y$ for any $y \in Y$.}
\end{equation}
We abuse a bit of the notation and define $\First(z) = \max\{k < 0 : z_k \in \A{C}{ap}\}$ for $z \in Z$.
Note that $\First(\psi(y)) = \First(y)$ for $y \in Y$.

\subsubsection{Basic properties of the fourth coding}
\label{subsubsec:Z4:basic_lemmas}

We present here the basic properties of $(Z, \tau)$.

%%%%                                                                 %%%% LEMMA %%%%
\begin{lem} \label{Z4:lem:Z4_is_reco}
The pair $(Z, \tau)$ is a recognizable coding of $X$. 
\end{lem}
\begin{proof}
It follows from the definitions that $\psi$ commutes with the shift and that it is continuous; hence $Z = \psi(Y)$ is a subshift.
Also, by \eqref{eq:Z4:theta_psi=id}, we have that $\theta\psi(y) = y$ for any $y \in Y$, so $Y = \theta \psi(Y) = \theta(Z)$.
Therefore, $(Z, \theta)$ is a coding of $Y$.
It is easy to see from the definition of $Z$ that $(Z, \theta)$ is recognizable.
Hence, as $(Y, \sigma)$ is recognizable as well, Lemma \ref{lem:recognizability_composition} tells us that $(Z, \sigma \theta)$ is recognizable.
Being $\tau = \sigma \theta$, we conclude that $(Z, \tau)$ is recognizable.
\end{proof}

Thanks to the last lemma, $\Fac_{(Z,\tau)}(x)$ and $\Fac_{(Z,\tau)}^0(x)$ are defined for every $x \in X$.
%%%%                                                                 %%%% LEMMA %%%%
\begin{lem} \label{Z4:relate_Z3_Z4}
Let $x \in X$, $(c, y) = \Fac_{(Y,\sigma)}(x)$ and $(f, z) = \Fac_{(Z,\tau)}(x)$.
Then, $c_k = f_k$ and $\psi_0(S^k y) = z_k$ for all $k \in \Z$.
\end{lem}
\begin{proof}
On one hand, \eqref{eq:Z4:theta_psi=id} implies that $\tau\psi(y) = \sigma\theta\psi(y) = \sigma(y)$, and thus that $S^{-c_0} \tau(\psi(y))$ is equal to $S^{-c_0} \sigma(y) = x$.
On the other hand, $\sigma(y_0)$ is equal to $\tau(\psi_0(y)) = \tau(\psi(y)_0)$; so, as $(-c_0, y)$ is a $\sigma$-factorization, $[0, |\sigma(y_0)|) = [0, |\tau(\psi(y)_0)|)$ contains $-c_0$. 
From these two things, we conclude that $(-c_0, \psi(y))$ is a $\tau$-factorization of $x$.
Then, since $(Z,\tau)$ is recognizable by Lemma \ref{Z4:lem:Z4_is_reco}, $(-c_0, \psi(y))$ and $(-f_0, z)$ are the same $\tau$-factorization, that is, $c_0 = f_0$ and $\psi(y) = z$.
We use this to compute, for $j \geq 0$, 
\begin{equation*}
    f_j = -f_0 + |\tau(z_{[0,j)})| = 
    -c_0 + |\tau (\psi(y)_{[0,j)})| =
    -c_0 + |\sigma(y_{[0,j)})| = c_j,
\end{equation*}
where in the last step we used that $\tau(\psi(y)) = \sigma(y)$.
A similar computation shows that $f_j = c_j$ for $j < 0$ as well.
\end{proof}

The last lemma has the following important consequence.
For any $x \in X$, the cut functions of its $\sigma$-factorization in $Y$ and of its $\tau$-factorization in $Z$ are the same.
Therefore, we can simply write $(c, y) = \Fac_{(Y,\sigma)}(x)$ and $(c, z) = \Fac_{(Z,\tau)}(x)$.
This will be tacitly used in this subsection.

%%%%                                                                 %%%% LEMMA %%%%
\begin{lem} \label{Z4:same_z0_same_ext_window}
Let $x, \tilde{x} \in X$, $(c,z) = \Fac_{(Z,\tau)}(X)$ and $(\tilde{c}, \tilde{z}) = \Fac_{(Z,\tau)}(\tilde{x})$.
If $z_0 = \tilde{z}_0$, then $x_{[c_{\First(z)}, c_{\First(z)+1})} = \tilde{x}_{[\tilde{c}_{\First(\tilde{z})}, \tilde{c}_{\First(\tilde{z})+1})}$ and $x_{[c_{j}, c_{j+1})} = \tilde{x}_{[\tilde{c}_{j}, \tilde{c}_{j+1})}$ for $j \in [-1,1]$.
\end{lem}
\begin{proof}
Let $(c, y) = \Fac_{(Y,\sigma)}(x)$ and $(\tilde{c}, \tilde{y}) = \Fac_{(Y,\sigma)}(\tilde{x})$.
Then, by Lemma \ref{Z4:relate_Z3_Z4} and the hypothesis
\begin{equation*}
    y_{\First(y)} y_{-1} y_0 y_1 = 
    \psi_0(y) = z_0 =
    \tilde{z}_0 = \psi_0(\tilde{y}) =
    \tilde{y}_{\First(\tilde{y})} \tilde{y}_{-1} \tilde{y}_0 \tilde{y}_1.
\end{equation*}
Hence, 
\begin{equation*}
    x_{[c_{\First(z)}, c_{\First(z)+1})} =
    \sigma(y_{\First(y)}) = 
    \sigma(\tilde{y}_{\First(\tilde{y})}) =
    \tilde{x}_{[\tilde{c}_{\First(\tilde{z})}, \tilde{c}_{\First(\tilde{z})+1})}.
\end{equation*}
Similarly, $x_{[c_{-1}, c_2)} = \sigma(y_{-1}y_0y_1)$ is equal to $\tilde{x}_{[\tilde{c}_{-1}, \tilde{c}_2)} = \sigma(\tilde{y}_{-1}\tilde{y}_0\tilde{y}_1)$.
\end{proof}

Let
\begin{equation}\label{eq:Z4:main_prop:1}
    \text{$\A{C}{ap} = \theta^{-1}(\A{B}{ap})$, $\A{C}{wp} = \theta^{-1}(\A{B}{wp})$ and $\A{C}{sp} = \theta^{-1}(\A{B}{sp})$.}
\end{equation}
Note that, since $\A{B}{ap}\cup\A{B}{wp}\cup\A{B}{sp}$ is a partition of $\cB$, the sets $\A{C}{ap}$, $\A{C}{wp}$ and $\A{C}{sp}$ form a partition of $\cC = \theta^{-1}(\cB)$.
%%%%                                                                 %%%% PROPOSITION %%%%
\begin{prop} \label{Z4:main_prop}
The following conditions hold:
\begin{enumerate}
    \item \label{item:Z4:main_prop:lengths}
    $20\varepsilon \leq |\tau(a)| \leq 10d^2n$ for all $a \in \cC$.
    
    \item \label{item:Z4:main_prop:cardinalities}
    $\#\tau(\A{C}{ap}) \leq 2d^{3d+6}$, $\#(\mroot\tau(\cC)) \leq 5d^{3d+6}$ and $\#\cC \leq 7^4d^{12d+36} \pcom(X)^4$.

    \item \label{item:Z4:main_prop:per}
    $(Z, \tau)$ has dichotomous periods w.r.t. $(\A{C}{ap}, \A{C}{sp}\cup\A{C}{wp})$ and $8\varepsilon$.
\end{enumerate}
\end{prop}
\begin{proof}
We start with the proof of Items \ref{item:Z4:main_prop:cardinalities} and \ref{item:Z4:main_prop:lengths}.
We have, from the equation $\tau = \sigma\theta$ and \eqref{eq:Z4:main_prop:1}, that
\begin{equation*}
    \text{$\tau(\A{C}{ap}) = \sigma(\A{B}{ap})$ and $\tau(\cC) = \cB$.}
\end{equation*}
Thus, by Item \ref{item:Z3:main_prop:cardinalities} in Proposition \ref{Z3:main_prop}, $\#\tau(\A{C}{ap}) = \#\sigma(\A{B}{ap}) \leq 2d^{3d+6}$ and $\#(\mroot\tau(\cC)) = \#(\mroot\sigma(\cB)) \leq 5d^{3d+6}$.
Also, from the definition of $\cC$ we get that $\#\cC \leq \#\cB^4$.
Putting the bounds from Item \ref{item:Z3:main_prop:cardinalities} of Proposition \ref{Z3:main_prop} in this inequality gives that $\#\cC \leq 7^4d^{12d+36} \pcom(X)^4$.
Item \ref{item:Z4:main_prop:lengths} follows from the equation $\tau(\cC) = \cB$ and Item \ref{item:Z3:main_prop:lengths} of Proposition \ref{Z3:main_prop}.

We now prove Item \ref{item:Z4:main_prop:per}.
Let $x \in X$ be arbitrary and define $(c,z) = \Fac_{(Z,\tau)}(x)$ and $(c,y) = \Fac_{(Y,\sigma)}(x)$.
Equation \eqref{eq:Z4:main_prop:1} ensures that $y_0 \in \A{B}{ap}$ if and only if $z_0 \in \A{C}{ap}$.
We also note that, by Lemma \ref{Z4:relate_Z3_Z4}, $\mroot \tau(z_0) = \mroot \sigma(y_0)$.
Therefore, Item \ref{item:Z4:main_prop:per} of this proposition follows Item \ref{item:Z3:main_prop:per} in Proposition \ref{Z3:main_prop}.
\end{proof}

\begin{rema} \label{rema:Z4:special_per&root=per}
As was similarly observed in Remark \ref{rema:Z3:special_per&root=per}, a consequence of Items \ref{item:Z4:main_prop:per} and \ref{item:Z4:main_prop:lengths} in Proposition \ref{Z4:main_prop} is that, for all $a \in \cC \setminus \A{C}{ap}$, $|\mroot \tau(a)| = \per(\tau(a))$.
\end{rema}

%%%%                                                                 %%%% PROPOSITION %%%%
\begin{prop} \label{Z4:structure_Csp}
Let $z \in \Z$.
\begin{enumerate}
    \item \label{item:Z4:structure_Csp:window}
    If  $i < j$ are integers such that $z_k \in \cC \setminus \A{C}{ap}$ for all $k \in [i, j)$, then $\mroot \tau(z_k) = \mroot \tau(z_i)$ if $k \in [i,j)$, $z_k \in \A{C}{sp}$ if $k \in [i,j-1)$, and $z_k = z_{i+1}$ for all $i \in [i+1,j-1)$.
    \item \label{item:Z4:structure_Csp:power_of_2}
    If $z_0 \in \A{C}{sp}$, then $\tau(a) = (\mroot\tau(a))^{2^r}$, where $r$ is the unique integer for which $2^r|\mroot\tau(a)|$ belongs to $[20\varepsilon, 40\varepsilon)$.
\end{enumerate}
\end{prop}
\begin{proof}
Suppose that $i < j$ satisfy $z_k \in \cC \setminus \A{C}{ap}$ for all $k \in [i, j+1)$.
Let us denote $\Fac_{(Y,\sigma)}(\tau(z))$ by $(c,y)$.
Then, by Lemma \ref{Z4:relate_Z3_Z4}, $y_k = \theta(z_k) \in \cB \setminus \A{B}{ap}$ for all $k \in [i, j+1)$.
In this context, Item \ref{item:Z3:structure_Csp:window} of Proposition \ref{Z3:structure_Csp} ensures that $\mroot \sigma(y_k) = \mroot \sigma(y_i)$ if $k \in [i,j+1)$ and $y_k = y_i \in \A{B}{sp}$ if $k \in [i, j)$
We deduce, as $\tau(z_k) = \sigma(y_k)$, that $\mroot \tau(z_k) = \mroot \tau(z_i)$ for all $k \in [i,j+1)$.
Also, for any $k \in [i,j)$, we have that $z_k \in \theta^{-1}(y_k) \in \A{B}{sp}$.
Now, since $y_k \in \cC \setminus \A{C}{ap}$ if $[i,j+1)$, we have that $y_{\First(S^k y)} = y_{\First(S^i y)}$ for all $k \in [i,j+1)$.
Combining this with the fact that $y_k = y_i$ if $k \in [i,j)$ yields
\begin{equation*}
    \psi_0(S^k y) = y_{\First(S^k y)} y_{k-1} y_k y_{k+1} =
    y_{\First(S^i y)} y_i y_i y_i
    \enspace\text{for all $k \in [i+1,j-1)$.}
\end{equation*}
We conclude, using Lemma \ref{Z4:relate_Z3_Z4}, that $z_k = z_{i+1}$ for $k \in [i+1,j-1)$.
\end{proof}

Lemmas \ref{Z4:reco_from_aper}, \ref{Z4:reco_from_reco} and \ref{Z4:reco_from_aper_sym} will use the following notation:
\begin{equation*}
    E = 50d^2 n.
\end{equation*}
Observe that $|\sigma| \leq E$, $|\tau| \leq E$ and that $E$ is bigger than or equal to the constants $7d^2$ and $50d^2n$ appearing in Proposition \ref{Z3:special_reco}.

%%%%                                                                 %%%% LEMMA %%%%
\begin{lem} \label{Z4:reco_from_aper}
Let $x, \tilde{x} \in X$ be such that $x_{[-3E, 3E)} = \tilde{x}_{[-3E, 3E)}$ and $\per(x_{[-\varepsilon,\varepsilon)}) > \varepsilon$.
Let $(c,z) = \Fac_{(Z,\tau)}(x)$ and $(\tilde{c},\tilde{z}) = \Fac_{(Z,\tau)}(\tilde{x})$.
Then, $\tau(z_0) = \tau(\tilde{z}_0)$ and $\Fac^0_{(Z, \tau)}(S^{c_1} x) = \Fac^0_{(Z,\tau)}(S^{c_1} \tilde{x})$.
In particular, $c_0 = \tilde{c}_0$ and $c_1 = \tilde{c}_1$.
\end{lem}
\begin{proof}
We use the notation $(c,y) = \Fac_{(Y,\sigma)}(x)$ and $(\tilde{c}, \tilde{y}) = \Fac_{(Y,\sigma)}(\tilde{x})$.
Observe that the hypothesis $x_{[-3E, 3E)} = \tilde{x}_{[-3E, 3E)}$ allows us to use Item \ref{item:Z3:main_prop:reco} in Proposition \ref{Z3:main_prop} to obtain that
\begin{equation} \label{eq:Z4:reco_from_aper:1}
    \Fac^0_{(Y,\sigma)}(S^i x) = \Fac^0_{(Y,\sigma)}(S^i \tilde{x})
    \enspace\text{for all $i \in [0, 2E]$.}
\end{equation}
In particular, $c_0 = \tilde{c}_0$ and $y_0 = \tilde{y}_0$.
Thus, by Lemma \ref{Z4:relate_Z3_Z4}, $\tau(z_0) = \sigma(y_0) = \sigma(\tilde{y}_0) = \tau(\tilde{z}_0)$.
Also, for $j \in \{1,2\}$ we have that $0 \leq c_j \leq c_2 \leq 2E$, so \eqref{eq:Z4:reco_from_aper:1} can be applied to deduce that $\Fac^0_{(Y,\sigma)}(S^{c_j} x) = \Fac^0_{(Y,\sigma)}(S^{c_j} \tilde{x})$.
Then, 
\begin{equation} \label{eq:Z4:reco_from_aper:2}
    y_{[0,3)} = \tilde{y}_{[0,3)}
    \enspace\text{and}\enspace
    c_j = \tilde{c}_j\enspace\text{for all $j\in[0,3)$}.
\end{equation}

Before continuing, we prove that
\begin{equation} \label{eq:Z4:reco_from_aper:3}
    y_0, \tilde{y}_0 \in \A{C}{ap}.
\end{equation}
We note that if $y_0 \in \cC \setminus \A{C}{ap}$, then Item \ref{item:Z4:main_prop:per} in Proposition \ref{Z4:main_prop} gives that $\per(x_{[c_0-8\varepsilon, c_1+9\varepsilon)}) \leq \varepsilon$, which is impossible since we assumed that $\per(x_{[-\varepsilon, \varepsilon)}) > \varepsilon$.
Thus, $y_0 \in \A{C}{ap}$.
Similarly, $\tilde{y}_0 \in \A{C}{ap}$ as $\tilde{x}_{[-\varepsilon, \varepsilon)} = x_{[-\varepsilon, \varepsilon)}$.

We can now finish the proof.
The condition $c_1 = \tilde{c}_1$ follows from \eqref{eq:Z4:reco_from_aper:2}.
Also, we have from  Lemma \ref{Z4:relate_Z3_Z4}  that $z_1 = \psi_0(Sy) = \First(Sy) y_0 y_1 y_2$ and $\tilde{z}_1 = \psi_0(Sy) = \First(S\tilde{y}) \tilde{y}_0 \tilde{y}_1 \tilde{y}_2$.
Now, Equation \eqref{eq:Z4:reco_from_aper:3} guarantees that $\First(Sy) = y_0$ and $\First(S\tilde{y}) = \tilde{y}_0$.
Therefore, by Equation \eqref{eq:Z4:reco_from_aper:2}, $z_1 = \tilde{z}_1$.
We conclude, using that $c_1 = \tilde{c}_1$, that
\begin{equation} 
    \Fac^0_{(Z,\tau)}(S^{c_1} x) = (0, z_1) =
    (0, \tilde{z}_1) =  \Fac^0_{(Z,\tau)}(S^{c_1} \tilde{x}).
\end{equation}
\end{proof}

%%%%                                                                 %%%% LEMMA %%%%
\begin{lem} \label{Z4:reco_from_aper_sym}
Let $x, \tilde{x} \in X$, $(c,z) = \Fac_{(Z,\tau)}(x)$ and $(\tilde{c},\tilde{z}) = \Fac_{(Z,\tau)}(\tilde{x})$.
Suppose that $x_{[-4E, 4E)} = \tilde{x}_{[-4E, 4E)}$ and $z_0 \in \A{C}{ap}$.
Then, $\tau(z_0) = \tau(\tilde{z}_0)$ and $\Fac^0_{(Z, \tau)}(S^{c_1} x) = \Fac^0_{(Z,\tau)}(S^{c_1} \tilde{x})$.
In particular, $c_0 = \tilde{c}_0$ and $c_1 = \tilde{c}_1$.
\end{lem}
\begin{proof}
The condition $z_0 \in \A{C}{ap}$ implies, by Item \ref{item:Z4:main_prop:per} in Proposition \ref{Z4:main_prop}, that $\per(x_{[c_0, c_1)}) > \varepsilon$.
Thus, by Item \ref{item:CL:localize_per_aper:loc_aper} in Lemma \ref{CL:localize_per_aper}, there is $i \in [c_0, c_1)$ such that $\per(x_{[i-\varepsilon,i+\varepsilon)}) > \varepsilon$.
Now, since $|\tau| \leq E$, the hypothesis ensures that $(S^i x)_{[i-\varepsilon,i+\varepsilon)}$ is equal to $(S^i \tilde{x})_{[i-\varepsilon,i+\varepsilon)}$.
Therefore, we can use Lemma \ref{Z4:reco_from_aper} and conclude that $\tau(z_0) = \tau(\tilde{z}_0)$ and $\Fac^0_{(Z, \tau)}(S^{c_1} x) = \Fac^0_{(Z,\tau)}(S^{c_1} \tilde{x})$.
\end{proof}

%%%%                                                                 %%%% LEMMA %%%%
\begin{lem} \label{Z4:reco_from_reco}
Let $x, \tilde{x} \in X$ and $k\geq0$ be an integer.
Suppose that $\Fac^0_{(Z,\tau)}(x) = \Fac^0_{(Z,\tau)}(\tilde{x})$ and that $x_{[-3E, k+3E)} = \tilde{x}_{[-3E, k+3E)}$.
Then, $\Fac^0_{(Z,\tau)}(S^i x) = \Fac^0_{(Z,\tau)}(S^i\tilde{x})$ for all $i \in [0, k]$.
\end{lem}
\begin{proof}
We only prove that $\Fac^0_{(Z,\tau)}(S x) = \Fac^0_{(Z,\tau)}(S \tilde{x})$ if $\Fac^0_{(Z,\tau)}(x) = \Fac^0_{(Z,\tau)}(\tilde{x})$ and that $x_{[-3E, 1+3E)} = \tilde{x}_{[-3E, 1+3E)}$, as then an inductive argument on $i$ gives the lemma.

Let us write $(c,y) = \Fac_{(Y,\sigma)}(x)$ and $(\tilde{c},\tilde{y}) = \Fac_{(Y,\sigma)}(x)$.
Then, by Lemma \ref{Z4:relate_Z3_Z4}, $z = \psi(y)$ and $\tilde{z} = \psi(\tilde{y})$ satisfy $(c,z) = \Fac_{(Z,\tau)}(x)$ and $(\tilde{c},\tilde{z}) = \Fac_{(Z,\tau)}(x)$.
Hence, the hypothesis $\Fac^0_{(Z,\tau)}(x) = \Fac^0_{(Z,\tau)}(\tilde{x})$ is equivalent to $z_0 = \tilde{z}_0$ and $c_0 = \tilde{c}_0$.
This implies two things:
\begin{enumerate} [label=(\roman*)]
    \item $y_{\First(y)} y_{-1} y_0 y_1 = \psi_0(y) = z_0 = \tilde{z}_0 = \psi_0(\tilde{y}) = \tilde{y}_{\First(\tilde{y})} \tilde{y}_0 \tilde{y}_1 \tilde{y}_2$.
    \item $c_1 = c_0 + |\sigma(y_0)| = \tilde{c}_1 + |\sigma(\tilde{y}_0)| = \tilde{c}_1$ and, similarly, $c_2 = \tilde{c}_2$.
\end{enumerate}
We deduce that, for any $i \in [0, c_1)$,
\begin{equation*}
    \Fac^0_{(Z,\tau)}(S^i x) = (c_0 - i, z_0) = 
    (\tilde{c}_0 - i, \tilde{z}_0) = \Fac^0_{(Z,\tau)}(S^i \tilde{x}).
\end{equation*}
In particular, if $c_1 > 0$ then $\Fac^0_{(Z,\tau)}(S x) = (c_0 - 1, z_0)$ is equal to $\Fac^0_{(Z,\tau)}(S \tilde{x}) = (\tilde{c}_0 - 1, \tilde{z}_0)$ and the proof is complete.

We now assume that $c_1 = 1$ (so $\tilde{c}_1 = 1$ as well by (ii)).
In this case, $\Fac^0_{(Z,\tau)}(S x) = (0, z_1)$ and $\Fac^0_{(Z,\tau)}(S \tilde{x}) = (0, \tilde{z}_1)$; thus, it is enough to prove that $z_1 = \tilde{z}_1$.

We observe that, since $z_1 = \psi_0(Sy)$ and $\tilde{z}_1 = \psi_0(S\tilde{y})$,
\begin{enumerate}
    \item $z_1 = y_0 y_0 y_1 y_2$ if $y_0\in\A{B}{ap}$ and $z_1 = y_{\First(y)} y_0 y_1 y_2$ if $y_0\not\in\A{B}{ap}$, and
    \item $\tilde{z}_1 = \tilde{y}_0 \tilde{y}_0 \tilde{y}_1 \tilde{y}_2$ if $\tilde{y}_0\in\A{B}{ap}$ and $\tilde{z}_1 = \tilde{y}_{\First(\tilde{y})} \tilde{y}_0 \tilde{y}_1 \tilde{y}_2$ if $\tilde{y}_0\not\in\A{B}{ap}$.
\end{enumerate}
From these relations and (i) we deduce that
\begin{equation*}
    \text{$z_1 = \tilde{z}_1$ if and only if $y_2 = \tilde{y}_2$.}
\end{equation*}
Now, since $\Fac^0_{(Y,\sigma)}(x) = (c_0, y_0) = (\tilde{c}_0, \tilde{y}_0) = \Fac^0_{(Y,\sigma)}(\tilde{x})$ and since we assumed that $x_{[-3E, 1+3E)} = \tilde{x}_{[-3E, 1+3E)}$, we can use Item \ref{item:Z3:main_prop:reco} in Proposition \ref{Z3:main_prop} to obtain that $\Fac^0_{(Y,\sigma)}(S^i x)$ is equal to $\Fac^0_{(Y,\sigma)}(S^i \tilde{x})$ for any $i \in [0, 2E]$.
In particular, since $c_2$ satisfies $0 \leq c_2 \leq c_0 + 2|\sigma| \leq 2E$, we have that 
\begin{equation*}
    (0, y_2) = \Fac^0_{(Y,\sigma)}(S^{c_2} x) = 
    \Fac^0_{(Y,\sigma)}(S^{c_2} \tilde{x}) = 
    \Fac^0_{(Y,\sigma)}(S^{\tilde{c}_2} \tilde{x}) = (0, \tilde{y}_2),
\end{equation*}
where we used that $c_2 = \tilde{c}_2$ by (ii).
It follows that $y_2 = \tilde{y}_2$ and thus that $\Fac^0_{(Z,\tau)}(S x)$ is equal to $\Fac^0_{(Z,\tau)}(S \tilde{x})$.
\end{proof}

\subsection{Connecting two levels}
\label{subsec:Z4:two_levels}

In this subsection, we consider two of the codings constructed in Subsection \ref{subsec:Z4:construction} and prove several lemmas that relate them.
We start by fixing the necessary notation.

Let $X$ be a minimal infinite subshift, $n, n' \geq 1$ be integers and let $d$ be the maximum of $\lceil p_X(n)/n \rceil$, $\lceil p_X(n')/n' \rceil$, $p_X(n+1) - p_X(n)$, $p_X(n'+1) - p_X(n')$, $\#\cA$ and $10^4$. 
Let $E = 50d^2 n$ and $E' = 50d^2 n'$.
We will assume throughout the subsection that
\begin{equation} \label{eq:Z4:comparing_lengths_Z4_Z4'_prev}
    n' \geq d^{2d^3+4} \cdot 500d^2n.
\end{equation}
We consider the recognizable codings $(Z \subseteq \cC^\Z, \tau\colon\cC\to\cA^+)$ and $(Z' \subseteq {\cC'}^\Z, \tau'\colon\cC'\to\cA^+)$ of $X$ obtained from $n$ and $n'$ as in Subsection \ref{subsec:Z4:construction}, respectively.
Let also $\varepsilon' \in [n'/d^{2d^3+4}, n'/d)$ and $\varepsilon \in [n/d^{2d^3+4}, n/d)$ be the constants defined in Subsection \ref{subsec:Z4:construction}, and let us denote by $\A{C}{ap} \cup \A{C}{sp} \cup \A{C}{wp}$ and $\A{C}{ap}' \cup \A{C}{sp}' \cup \A{C}{wp}'$ the partitions of $\cC$ and $\cC'$ defined there.
Let $\First(z) = \max\{k < 0 : z_k \in \A{C}{ap}\}$ and $\First'(z') = \{k < 0 : z'_k \in \A{C}{ap}'\}$ for $z \in Z$ and $z' \in Z'$.

The crucial relation between $(Z',\tau')$ and $(Z,\tau)$ is the following inequality, which is a consequence of \eqref{eq:Z4:comparing_lengths_Z4_Z4'_prev}:
\begin{equation} \label{eq:Z4:comparing_lengths_Z4_Z4'}
    \varepsilon' \geq 10E.
\end{equation}

\subsubsection{Preliminary lemmas}
\label{subsubsec:Z4:pre_lemmas}

We fix, for the rest of Subsection \ref{subsubsec:Z4:pre_lemmas}, a point $x \in X$ and the notation $(c,z) = \Fac_{(Z,\tau)}(x)$ and $(c',z') = \Fac_{(Z',\tau')}(x)$.

%%%%                                                                 %%%% LEMMA %%%%
\begin{lem} \label{Z4:per<eps=>zjcj_cte}
Suppose that $\per(\tau'(z'_0)) \leq \varepsilon$.
Let $i \in [c'_0-7\varepsilon', c'_1+7\varepsilon')$ and $j$ be the integer satisfying $i \in [c_j, c_{j+1})$.
Then, $z_j = z_0 \in \A{C}{sp}$, $|\mroot\tau(z_j)| = |\mroot\tau'(z'_0)|$ and $c_j = c_0 \pmod{|\mroot\tau'(z'_0)|}$.
\end{lem}
\begin{proof}
The condition $\per(\tau'(z'_0)) \leq \varepsilon \leq \varepsilon'$ implies that $z'_0 \in \cC' \setminus \A{C}{ap}'$.
Hence, by Item \ref{item:Z4:main_prop:per} in Proposition \ref{Z4:main_prop},
\begin{equation} \label{eq:Z4:per<eps=>zjcj_cte:1}
    x_{[c'_0-8\varepsilon', c'_1+8\varepsilon')} = 
    (\mroot\tau'(z'_0))^\Z_{[-8\varepsilon', |\tau'(z'_0)|+8\varepsilon')}
\end{equation}
and
\begin{equation} \label{eq:Z4:per<eps=>zjcj_cte:2}
    |\mroot \tau'(z'_0)| = \per(\tau'(z'_0)) \leq \varepsilon.
\end{equation}
Equation \eqref{eq:Z4:per<eps=>zjcj_cte:1} implies the following:
If $k$ is the integer satisfying $c'_0-7\varepsilon' \in [c_k, c_{k+1})$ and $\ell$ is the integer satisfying $c'_0-7\varepsilon' \in [c_\ell, c_{\ell+1})$, then $\tau(z_j) \in \cC \setminus \A{C}{ap}$ for all $j \in [k,\ell)$.
Indeed, for any such $j$, we have that $\tau(z_j) = x_{[c_j, c_{j+1})}$ occurs in $x_{[c'_0-7\varepsilon'-3E, c'_1+7\varepsilon'+3E)}$, and so $\per(\tau(z_j)) \leq \varepsilon$, which implies, by Item \ref{item:Z4:main_prop:per} in Proposition \ref{Z4:main_prop}, that $z_j \in \cC \setminus \A{C}{ap}$.
We can then use Item \ref{item:Z4:structure_Csp:window} in Proposition \ref{Z4:structure_Csp} to get that
\begin{equation*}
    \text{$z_j = z_{k+1} \in \A{C}{sp}$ and $c_j = c_{k+1} \pmod{|\mroot\tau(z_{k+1})|}$ for all $j \in [k+1,\ell-1)$.}
\end{equation*}
Since $\varepsilon' \geq |\tau|$, $c_{k+1} \leq c_0 \leq c_{\ell-1}$, so we in particular have that
\begin{equation} \label{eq:Z4:per<eps=>zjcj_cte:3}
    \text{$z_j = z_0 \in \A{C}{sp}$ and $c_j = c_0 \pmod{|\mroot\tau(z_0)|}$.}
\end{equation}

We are now going to prove that $|\mroot\tau(z_0)|$ is equal to $|\mroot\tau'(z'_0)|$.
The lemma would follow from this and \eqref{eq:Z4:per<eps=>zjcj_cte:3}.

Note that $\tau'(z'_0)$ and $\tau(z_0)$ occur in $x_{[c'_0-8\varepsilon', c'_1+8\varepsilon')}$.
Also, Item \ref{item:Z4:main_prop:lengths} in Proposition \ref{Z4:main_prop} ensures that $\tau'(z'_0)$ and $\tau(z_0)$ have length at least $2\varepsilon$.
Then, as $\per(x_{[c'_0-8\varepsilon', c'_1+8\varepsilon')}) \leq \varepsilon$ by \eqref{eq:Z4:per<eps=>zjcj_cte:1} and \eqref{eq:Z4:per<eps=>zjcj_cte:2}, we can use Item \ref{item:CL:localize_per_aper:loc_per} of Proposition \ref{CL:localize_per_aper} to deduce that
\begin{equation*} 
    \per(\tau(z_0)) = \per(x_{[c'_0-8\varepsilon', c'_1+8\varepsilon')}) = \per(\tau'(z'_0)) \leq \varepsilon.
\end{equation*}
In this situation, Item \ref{item:Z4:main_prop:per} in Proposition \ref{Z4:main_prop} guarantees that $z_0 \in \cC \setminus \A{C}{ap}$ and $|\mroot \tau(z_0)| = \per(\tau(z_0)) = \per(\tau'(z'_0))$. 
Equation \eqref{eq:Z4:per<eps=>zjcj_cte:2} then yields $|\mroot\tau(z_0)| = |\mroot\tau'(z'_0)|$.
\end{proof}

%%%%                                                                 %%%% LEMMA %%%%
\begin{lem} \label{Z4:per_z'0_in_between=>reco}
Assume that $z'_0 \in \cC' \setminus \A{C}{ap}'$ and that $\per(\tau'(z'_0)) > \varepsilon$.
Let $\tilde{x} \in X$ and suppose that $x_{[-3\varepsilon', 3\varepsilon')} = \tilde{x}_{[-3\varepsilon', 3\varepsilon')}$.
Then, $\Fac^0_{(Z,\tau)}(x) = \Fac^0_{(Z,\tau)}(\tilde{x})$.
\end{lem}
\begin{proof}
The hypothesis gives that $\per(x_{[c'_0, c'_1)}) = \per(\tau'(z'_0)) > \varepsilon$, and Item \ref{item:Z4:main_prop:lengths} in Proposition \ref{Z4:main_prop} that $|x_{[c'_0, c'_1)}| \geq 2\varepsilon$.
Hence, we can use Item \ref{item:CL:localize_per_aper:loc_aper} of Proposition \ref{CL:localize_per_aper} to obtain $i_0 \in [c'_0, c'_1)$ such that
\begin{equation*}
    \per(x_{[i_0-\varepsilon,i_0+\varepsilon)}) > \varepsilon.
\end{equation*}
Now, since $z'_0 \in \cC' \setminus \A{C}{ap}'$, Item \ref{item:Z4:main_prop:per} of Proposition \ref{Z4:main_prop} applies, so
\begin{equation*}
    \per(x_{[c'_0-8\varepsilon', c'_1+8\varepsilon')}) \leq \varepsilon'.
\end{equation*}
This implies, as $\varepsilon \leq \varepsilon'$, that there exists $i \in [-2\varepsilon', -\varepsilon')$ such that $x_{[i-\varepsilon,i+\varepsilon)} = x_{[i_0-\varepsilon,i_0+\varepsilon)}$.

Our plan is to derive the lemma using Lemma \ref{Z4:reco_from_aper} with $S^i x$ and $S^i \tilde{x}$.
First, we note that
\begin{equation} \label{eq:Z4:per_z'0_in_between=>reco:1}
    \per((S^i x)_{[-\varepsilon, \varepsilon)}) =
    \per(x_{[i-\varepsilon, i+\varepsilon)}) =
    \per(x_{[i_0-\varepsilon, i_0+\varepsilon)}) >
    \varepsilon.
\end{equation}
Also, since $4E \leq \varepsilon'$ and $i \in [-2\varepsilon', -\varepsilon')$, we have that $[i-4E, i+3\varepsilon'+4E)$ is contained in $[-3\varepsilon', 3\varepsilon')$.
Thus, by the hypothesis $x_{[-3\varepsilon', 3\varepsilon')} = \tilde{x}_{[-3\varepsilon', 3\varepsilon')}$,
\begin{multline} \label{eq:Z4:per_z'0_in_between=>reco:2}
    (S^{i+E} x)_{[-4E, 3\varepsilon'+3E)} =
    x_{[i-3E, i+3\varepsilon'+4E)} \\ =
    \tilde{x}_{[i-3E, i+3\varepsilon'+4E)} =
    (S^{i+E} \tilde{x})_{[-4E,3\varepsilon'+3E)}.
\end{multline}
In particular,
\begin{equation} \label{eq:Z4:per_z'0_in_between=>reco:3}
    (S^i x)_{[-3E, 3E)} = (S^i \tilde{x})_{[-3E, 3E)}.
\end{equation}
Equations \eqref{eq:Z4:per_z'0_in_between=>reco:1} and \eqref{eq:Z4:per_z'0_in_between=>reco:3} allow us to use Lemma \ref{Z4:reco_from_aper} and deduce that
\begin{equation*}
    \Fac^0_{(Z,\tau)}(S^{i + E} x) = 
        \Fac^0_{(Z,\tau)}(S^{i + E} \tilde{x}).
\end{equation*}
Furthermore, the last equation and \eqref{eq:Z4:per_z'0_in_between=>reco:2} are the hypothesis of Lemma \ref{Z4:reco_from_reco}; hence,
\begin{equation} \label{eq:Z4:per_z'0_in_between=>reco:4}
    \Fac^0_{(Z,\tau)}(S^{i + E + k} x) = 
    \Fac^0_{(Z,\tau)}(S^{i + E + k} \tilde{x})
    \enspace\text{for any $k \in [0, 3\varepsilon')$.}
\end{equation}
Since $i \in [-2\varepsilon', -\varepsilon')$ and $E \leq \varepsilon'$, $k \coloneqq -(i+E)$ belongs to $[0, 3\varepsilon')$, so \eqref{eq:Z4:per_z'0_in_between=>reco:4} gives that $\Fac^0_{(Z,\tau)}(x) = \Fac^0_{(Z,\tau)}(\tilde{x})$.
\end{proof}

% aquí uso la power_of_two property.
%%%%                                                                 %%%% LEMMA %%%%
\begin{lem} \label{Z4:z'0_sp=>same_Fac_edges}
Suppose that $z'_0 \in \A{C}{sp}'$.
Then,
\begin{equation*}
    \Fac^0_{(Z,\tau)}(S^{c'_0+i} x) = \Fac^0_{(Z,\tau)}(S^{c'_1+i} x)
    \enspace \text{for any $i \in [-5\varepsilon', 5\varepsilon')$.}
\end{equation*}
\end{lem}
\begin{proof}
We consider two cases.
First, we assume that $\per(\tau'(z'_0)) \leq \varepsilon$.
This allows us to use Lemma \ref{Z4:per<eps=>zjcj_cte} and obtain that, if $i \in [c'_0 - 7\varepsilon, c'_1 + 7\varepsilon)$ and $j$ is the integer satisfying $i \in [c_j, c_{j+1})$, then
\begin{multline} \label{eq:Z4:z'0_sp=>same_Fac_edges:1}
    \text{$c_j = c_0 \pmod{|\mroot\tau'(z'_0)|}$, $z_j = z_0 \in \A{C}{sp}$ and }
    \\ \text{$|\mroot\tau(z_j)| = |\mroot\tau'(z'_0)|$.}
\end{multline}

Let $i \in [-5\varepsilon', 5\varepsilon')$ be arbitrary and denote by $k$ and $\ell$ the integers satisfying 
$c'_0 + i \in [c_k, c_{k+1})$ and $c'_1 + i \in [c_\ell, c_{\ell+1})$.
With this notation, $\Fac^0_{(Z,\tau)}(S^{c'_0 + i} x) = (c_k - c'_0 - i, z_p)$ and $\Fac^0_{(Z,\tau)}(S^{c'_1 + i} x) = (c_\ell - c'_1 - i, z_q)$, so we have to prove that $z_k = z_\ell$ and $c_k - c'_0 - i = c_\ell - c'_1 - i$.

We have, by \eqref{eq:Z4:z'0_sp=>same_Fac_edges:1}, that $z_k = z_\ell$.
Thus, it only rests to prove that $c_k - c'_0 - i = c_\ell - c'_1 - i$.

We note that the definition of $k$ and $\ell$ ensures that
\begin{enumerate}[label=(\roman*)]
    \item $c_k    \leq c'_0 + i < c_{k+1}    = c_k + |\tau(z_k)|$; and
    \item $c_\ell \leq c'_1 + i < c_{\ell+1} = c_\ell + |\tau(z_\ell)|$.
\end{enumerate}
If we use the equality $c_\ell = c_k + (\ell-k)|\tau(z_0)|$, which is a consequence of \eqref{eq:Z4:z'0_sp=>same_Fac_edges:1}, and that $c'_1 = c'_0 + |\tau'(z'_0)|$ to replace $c_\ell$ and $c'_1$ in (ii), we get that  
\begin{equation*}
    c_k \leq c'_0 + i + (|\tau'(z'_0)| - (\ell-k)|\tau(z_0)|) < 
        c_k + |\tau(z_0)|.
\end{equation*} 
This and (i) yield
\begin{equation} \label{eq:Z4:z'0_sp=>same_Fac_edges:3}
    \left| |\tau'(z'_0)| - (\ell-k)|\tau(z_0)| \right| < |\tau(z_0)|.
\end{equation}
Now, since $z_0 \in \A{C}{sp}$ and $|\mroot\tau(z_k)| = |\mroot\tau'(z'_0)|$ by \eqref{eq:Z4:z'0_sp=>same_Fac_edges:1} and since $z'_0 \in \A{C}{sp}'$ by the hypothesis, the definition of $\A{C}{sp}$ and $\A{C}{sp}'$ in Proposition \ref{Z4:main_prop} guarantees that $|\tau(z_0)|$ divides $|\tau'(z'_0)|$.
Therefore, the inequality in \eqref{eq:Z4:z'0_sp=>same_Fac_edges:3} is possible only if $ |\tau'(z'_0)| = (\ell-k) |\tau(z_0)|$.
We conclude, as \eqref{eq:Z4:z'0_sp=>same_Fac_edges:1} implies that $c_\ell = c_k + (\ell-k)|\tau(z_0)|$, that
\begin{equation*}
    c_\ell = c_k + (\ell-k)|\tau(z_0)| =
    c_k + |\tau'(z'_0)| = c_k + c'_1 - c'_0.
\end{equation*}
Hence, $c_k - c'_0 - i = c_\ell - c'_1 - i$ and the proof of the first case is complete.
\medskip

Next, we assume $\per(\tau'(z'_0)) > \varepsilon$.
Observe that the condition $z'_0 \in \A{C}{sp}'$ implies that $z'_0 \in \cC' \setminus \A{C}{ap}'$.
Hence, by Item \ref{item:Z4:main_prop:per} in Proposition \ref{Z4:main_prop}, $x_{[c'_0 - 8\varepsilon', c'_1 + 8\varepsilon')} = (\mroot\tau'(z'_0))^\Z_{[-8\varepsilon', |\tau'(z'_0)| + 8\varepsilon')}$ and $|\mroot\tau'(z'_0)| = \per(\tau'(z'_0)) \leq \varepsilon$.
In particular, if $i \in [-5\varepsilon', 5\varepsilon')$, then $x_{[c'_0 + i - 3\varepsilon', c'_0 + 3\varepsilon')}$ is equal to $x_{[c'_1 + i - 3\varepsilon', c'_1 + 3\varepsilon')}$.
Then, the hypothesis of Lemma \ref{Z4:per_z'0_in_between=>reco} is satisfied for $S^{c'_0+i}x$ and $S^{c'_1+i}x$, and thus we obtain that $\Fac^0_{(Z,\tau)}(S^{c'_0+i} x) = \Fac^0_{(Z,\tau)}(S^{c'_1+i} x)$ for any $i \in [-5\varepsilon', 5\varepsilon')$.
\end{proof}

%%%%                                                                 %%%% LEMMA %%%%
\begin{lem} \label{Z4:same_ap=>reco_edge}
Let $\tilde{x} \in X$ and $(\tilde{c}, \tilde{z}) = \Fac_{(Z,\tau)}(\tilde{x})$.
Suppose that $z'_0 \in \A{C}{ap}$, $k \geq 1$, and that $x_{[c_0, c_1+k+8\varepsilon')} = \tilde{x}_{[\tilde{c}_0, \tilde{c}_1+k+8\varepsilon')}$.
Then,
\begin{equation} \label{eq:Z4:same_ap=>reco_edge:0}
    \Fac_{(Z,\tau)}(S^{c'_1+i} x) = \Fac_{(Z,\tau)}(S^{\tilde{c}'_1+i} \tilde{x})
    \enspace\text{for all $i \in [-7\varepsilon', k+7\varepsilon')$.}
\end{equation}
\end{lem}
\begin{proof}
First, Item \ref{item:Z4:main_prop:per} in Proposition \ref{Z4:main_prop} ensures that $\per(x_{[c'_0+8\varepsilon', c'_1-8\varepsilon')}) > \varepsilon' \geq \varepsilon$.
Thus, Item \ref{item:CL:localize_per_aper:loc_aper} in Proposition \ref{CL:localize_per_aper} ensures that there exists an integer $m$ such that
\begin{equation} \label{eq:Z4:same_ap=>reco_edge:1}
    m \in [c'_0+8\varepsilon', c'_1-8\varepsilon')
    \enspace\text{and}\enspace
    \per(x_{[m-\varepsilon,m+\varepsilon)}) > \varepsilon.
\end{equation}
Using that $4E \leq \varepsilon'$ and $m \in [c'_0+8\varepsilon', c'_1-8\varepsilon')$ it can be checked that $[m-3E, c'_1+7\varepsilon'+3E)$ is contained in $[c'_0, c'_1 + 8\varepsilon')$.
Hence, by the hypothesis $x_{[c'_0, c'_1+8\varepsilon')} = \tilde{x}_{[\tilde{c}'_0, \tilde{c}'_1+8\varepsilon')}$,
\begin{multline} \label{eq:Z4:same_ap=>reco_edge:2}
    (S^{m+E} x)_{[-4E, c'_1+7\varepsilon'-m + 3E)} = 
    x_{[m-3E, c'_1+7\varepsilon'+4E)} \\ = 
    \tilde{x}_{[m-3E+\tilde{c}'_0-c'_0, c'_1+7\varepsilon'+4E+\tilde{c}'_0-c'_0)} \\ =
    (S^{m+\tilde{c}'_0-c'_0+E} \tilde{x})_{[-4E, c'_1+7\varepsilon'-m + 3E)}.
\end{multline}
In particular, as $c'_1-m \geq 0$ by \eqref{eq:Z4:same_ap=>reco_edge:1},
\begin{equation*}
    (S^m x)_{[-3E, 3E)} = 
    (S^{m+\tilde{c}'_0-c'_0} \tilde{x})_{[-3E, 3E)}.
\end{equation*}
This and \eqref{eq:Z4:same_ap=>reco_edge:1} allow us to use Lemma \ref{Z4:reco_from_aper} and deduce that
\begin{equation*}
    \Fac_{(Z,\tau)}^0(S^{m+E} x) = \Fac_{(Z,\tau)}^0(S^{m+\tilde{c}'_0-c'_0+E} \tilde{x}).
\end{equation*}
The last equation and \eqref{eq:Z4:same_ap=>reco_edge:2} imply that the hypothesis of Lemma \ref{Z4:reco_from_reco} holds; 
therefore, for every $j \in [0, c'_1+7\varepsilon'-m)$, $\Fac_{(Z,\tau)}^0(S^{m+E+j} x) = \Fac_{(Z,\tau)}^0(S^{m-c'_0+\tilde{c}'_0+E+j} \tilde{x})$.
Equivalently, 
\begin{equation} \label{eq:Z4:same_ap=>reco_edge:3}
    \Fac_{(Z,\tau)}^0(S^{j} x) = \Fac_{(Z,\tau)}^0(S^{-c'_0+\tilde{c}'_0+j} \tilde{x})
    \enspace\text{for all $j \in [m+E, c'_1+(r-1)\varepsilon')$.}
\end{equation}
We will derive \eqref{eq:Z4:same_ap=>reco_edge:0} from this.

Let $i \in [-7\varepsilon', 7\varepsilon')$ be arbitrary.
Then, \eqref{eq:Z4:same_ap=>reco_edge:1} and the inequality $E \leq \varepsilon'$ imply that $j \coloneqq c'_1 + i \in [m+E, c'_1+7\varepsilon')$.
Hence, Equation \eqref{eq:Z4:same_ap=>reco_edge:3} gives that
\begin{equation} \label{eq:Z4:same_ap=>reco_edge:4}
    \Fac_{(Z,\tau)}^0(S^{c'_1 + i} x) = \Fac_{(Z,\tau)}^0(S^{i+c'_1-c'_0+\tilde{c}'_0} \tilde{x}).
\end{equation}
Now, we observe that
$$ c'_1-c'_0 = |\tau(z'_0)| = |\tau(\tilde{z}'_0)| = \tilde{c}'_1-\tilde{c}'_0, $$
so $i + c'_1 - c'_0 + \tilde{c}'_0 = i + \tilde{c}_1$.
Therefore, the lemma follows from \eqref{eq:Z4:same_ap=>reco_edge:4}.
\end{proof}

%%%%                                                                 %%%% PROPOSITION %%%%
\begin{prop} \label{Z4:reco_FZ_from_FZ'}
Let $x, \tilde{x} \in X$ and suppose that $\Fac^0_{(Z',\tau')}(x) = \Fac^0_{(Z',\tau')}(\tilde{x})$.
Then,
\begin{equation} \label{eq:Z4:reco_FZ_from_FZ':0}
    \Fac^0_{(Z,\tau)}(S^i x) = \Fac^0_{(Z,\tau)}(S^i \tilde{x})
    \enspace\text{for all $i \in [c'_0-4\varepsilon', c'_2-\varepsilon')$.} 
\end{equation}
\end{prop}
\begin{proof}
The hypothesis implies that $c'_0 = \tilde{c}'_0$ and $z'_0 = \tilde{z}'_0$.
Combining this with Lemma \ref{Z4:same_z0_same_ext_window} yields
\begin{multline} \label{eq:Z4:reco_FZ_from_FZ':1}
    x_{[c'_{\First'(z')}, c'_{\First'(z')+1})} = 
    \tilde{x}_{[\tilde{c}'_{\First'(\tilde{z}')}, c_{\First'(\tilde{z}')+1})}, \enspace
    x_{[c'_{j}, c'_{j+1})} = \tilde{x}_{[c'_{j}, c'_{j+1})} \\
    \enspace\text{and}\enspace
    c'_j = \tilde{c}'_j\enspace\text{for $j \in [-1,1]$.}
\end{multline}
This and the lower bound in Item \ref{item:Z4:main_prop:lengths} of Proposition \ref{Z4:main_prop} ensure that
\begin{equation} \label{eq:Z4:reco_FZ_from_FZ':1.5}
    x_{[c'_0 - 8\varepsilon', c'_1 + 8\varepsilon')} =
    \tilde{x}_{[\tilde{c}'_0 - 8\varepsilon', \tilde{c}'_1 + 8\varepsilon')}.
\end{equation}

Next, we show that the following facts hold.
\begin{enumerate}[label=(\roman*)]
    \item $(S^{c_k}x)_{[-8\varepsilon',8\varepsilon')} = (S^{c_l}x)_{[-8\varepsilon',8\varepsilon')}$ for all $\First'(y') < k, l \leq 0$.
    \item $(S^{\tilde{c}_k} \tilde{x})_{[-8\varepsilon',8\varepsilon')} = (S^{\tilde{c}_l} \tilde{x})_{[-8\varepsilon',8\varepsilon')}$ for all $\First'(\tilde{y}') < k, l \leq 0$.
    \item $x_{[ c'_{\First'(y')}, c'_{\First'(y')+1}+8\varepsilon' )} =  \tilde{x}_{[ \tilde{c}'_{\First'(\tilde{y}')},\tilde{c}'_{\First'(\tilde{y}')+1}+8\varepsilon' )}$.
\end{enumerate}
We start with Item (i).
If $\First'(y') = -1$, then (i) is vacuously true.
We assume that $\First'(y') < -1$.
Let $k$ be such that $\First'(y') < k < 0$.
The definition of $\First'$ ensures that $z'_k \in \cC' \setminus \A{C}{ap}'$.
So, by Item \ref{item:Z4:main_prop:per} in Proposition \ref{Z4:main_prop}, $s \coloneqq \mroot\tau'(z'_k)$ satisfies $x_{[c_k-8\varepsilon',c_{k+1}+8\varepsilon')} = s^\Z_{[-8\varepsilon', |\tau'(z'_k)|+8\varepsilon')}$.
In particular, as $|\tau'(z'_0)| = 0 \pmod{|s|}$,
\begin{equation*}
    x_{[c_k-8\varepsilon', c_k+8\varepsilon')} = 
        s^\Z_{[-8\varepsilon', 8\varepsilon')} = 
        (S^{|\tau'(z'_k)|} s^\Z)_{[-8\varepsilon', 8\varepsilon')} = 
        x_{[c_{k+1}-8\varepsilon', c_{k+1}+8\varepsilon')}.
\end{equation*}
Being this valid for all $k \in [\First'(z')+1, 0)$, an inductive argument gives (i).
Fact (ii) follows analogously.
For (iii), we use \eqref{eq:Z4:reco_FZ_from_FZ':1.5}, (i) and (ii) to deduce that
\begin{equation*}
    x_{[ c'_{\First'(z')+1}, c'_{\First'(z')+1}+8\varepsilon' )} = 
    x_{[c_0-8\varepsilon', c_0+8\varepsilon')} = 
    \tilde{x}_{[\tilde{c}_0-8\varepsilon', \tilde{c}_0+8\varepsilon')} =
    \tilde{x}_{[ \tilde{c}'_{\First'(\tilde{z}')+1},\tilde{c}'_{\First'(\tilde{z}')+1}+8\varepsilon' )}.
\end{equation*}
Fact (iii) follows from this and the first equality in \eqref{eq:Z4:reco_FZ_from_FZ':1}.
\medskip

Now, since $z'_{\First'(z')} \in \A{C}{ap}'$ and since (iii) holds, the hypothesis of Lemma \ref{Z4:same_ap=>reco_edge} are satisfied.
Therefore, 
\begin{equation} \label{eq:Z4:reco_FZ_from_FZ':2}
    \Fac^0_{(Z,\tau)}(S^{c'_{\First'(z')+1} + i} x) = \Fac^0_{(Z,\tau)}(S^{\tilde{c}'_{\First'(\tilde{z}')+1} + i} \tilde{x})
    \enspace
    \text{for all $i \in [-7\varepsilon', 7\varepsilon')$.}
\end{equation}
In order to continue, we need to consider two cases.
We first assume that $\First'(z') = -1$.
Then, $z'_{-1} \in \A{C}{ap}$, so \eqref{eq:Z4:reco_FZ_from_FZ':1} and Item \ref{item:Z4:main_prop:per} in Proposition \ref{Z4:main_prop} give
\begin{equation*}
    \per(\tilde{x}_{[\tilde{c}'_{-1}+8\varepsilon', \tilde{c}'_0-8\varepsilon')}) =
    \per(x_{[c'_{-1}+8\varepsilon', c'_0-8\varepsilon')}) > \varepsilon'.
\end{equation*}
This implies, by Item \ref{item:Z4:main_prop:per} in Proposition \ref{Z4:main_prop}, that $\tilde{z}_{-1} \in \A{C}{ap}$.
Hence, $\First'(\tilde{z}) = -1$ and then \eqref{eq:Z4:reco_FZ_from_FZ':0} follows from \eqref{eq:Z4:reco_FZ_from_FZ':2}.

Next, we assume that $\First'(z') \leq -2$.
In this case, we first prove the following.
\begin{enumerate}[label=$(\alph*)$]
    \item $\Fac^0_{(Z,\tau)}(S^{c'_{\First'(y')+1} + i} x) = \Fac^0_{(Z,\tau)}(S^{c'_{-1} + i} x)$ for all $i \in [-5\varepsilon', 5\varepsilon')$. 
    \item $\Fac^0_{(Z,\tau)}(S^{\tilde{c}'_{\First'(\tilde{y}')+1} + i} \tilde{x}) = \Fac^0_{(Z,\tau)}(S^{\tilde{c}'_{-1} + i} \tilde{x})$ for all $i \in [-5\varepsilon', 5\varepsilon')$. 
\end{enumerate}
We only prove $(a)$ as $(b)$ follows from an analogous argument.
If $\First'(z') = -2$, then $(a)$ is trivially true.
Assume then that $\First'(z') \leq -3$.
The definition of $\First'$ ensures that $z'_j \in \cC' \setminus \A{C}{ap}$ for all $\First'(z')+1 \leq j \leq -1$.
Thus, by Item \ref{item:Z4:structure_Csp:window} in Proposition \ref{Z4:structure_Csp}, $z'_j \in \A{C}{sp}$ for all $\First'(z')+1 \leq j \leq -2$.
This allows us to inductively apply Lemma \ref{Z4:z'0_sp=>same_Fac_edges} and deduce that, for any $i \in [-5\varepsilon', 5\varepsilon')$,
\begin{equation*}
    \Fac^0_{(Z,\tau)}(S^{c'_{\First'(z')+1} + i} x) = 
    \Fac^0_{(Z,\tau)}(S^{c'_{\First'(z')+2} + i} x) = \dots =
    \Fac^0_{(Z,\tau)}(S^{c'_{-1} + i} x).
\end{equation*}
This shows (a).
\medskip

Now, combining Equation \eqref{eq:Z4:reco_FZ_from_FZ':2}, $(a)$ and $(b)$ produces
\begin{equation} \label{eq:Z4:reco_FZ_from_FZ':3}
    \Fac^0_{(Z,\tau)}(S^{c'_{-1} + i} x) = \Fac^0_{(Z,\tau)}(S^{\tilde{c}'_{-1} + i} \tilde{x})
    \enspace
    \text{for all $i \in [-5\varepsilon', 5\varepsilon')$.}
\end{equation}
We are going to derive \eqref{eq:Z4:reco_FZ_from_FZ':0} from this and \eqref{eq:Z4:reco_FZ_from_FZ':1}.

Let $i \in [c'_0-4\varepsilon', c'_2-\varepsilon')$ be arbitrary.
We note that \eqref{eq:Z4:reco_FZ_from_FZ':3} gives, in particular, that $x_{[c_{-1} - 5\varepsilon', c_{-1} + 5\varepsilon')} = \tilde{x}_{[\tilde{c}_{-1} - 5\varepsilon', \tilde{c}_{-1} + 5\varepsilon')}$.
From this, \eqref{eq:Z4:reco_FZ_from_FZ':1} we get that
\begin{equation} \label{eq:Z4:reco_FZ_from_FZ':4}
    x_{[c_{-1} - 5\varepsilon', c_2)} = 
        \tilde{x}_{[\tilde{c}_{-1} - 5\varepsilon', \tilde{c}_2)}.
\end{equation}
In view of Equations \eqref{eq:Z4:reco_FZ_from_FZ':3} and \eqref{eq:Z4:reco_FZ_from_FZ':4} and of $3E \leq \varepsilon'$, the hypothesis of Lemma \ref{Z4:reco_from_reco} holds; hence, 
\begin{equation*}
    \text{$\Fac^0_{(Z,\tau)}(S^{c'_{-1} + j} x) = \Fac^0_{(Z,\tau)}(S^{\tilde{c}'_{-1} + j} \tilde{x})$
        for all $j \in [-4\varepsilon', c'_2-c'_{-1}-\varepsilon')$.}
\end{equation*}
We set $j = i - c'_{-1}$ and note that $j \in [-4\varepsilon', c'_2-c'_{-1}-\varepsilon')$.
Therefore, the last equation can be used to obtain that
\begin{equation*}
    \Fac^0_{(Z,\tau)}(S^i x) = 
    \Fac^0_{(Z,\tau)}(S^{c'_{-1} + j} x) =     
    \Fac^0_{(Z,\tau)}(S^{\tilde{c}'_{-1} + j} \tilde{x}) = 
    \Fac^0_{(Z,\tau)}(S^{i + \tilde{c}'_{-1} - c'_{-1} } \tilde{x}).
\end{equation*}
Being $\tilde{c}'_{-1} = c'_{-1}$ by \eqref{eq:Z4:reco_FZ_from_FZ':1}, we deduce that \eqref{eq:Z4:reco_FZ_from_FZ':0} holds.
\end{proof}

\subsubsection{The connecting morphism}
\label{subsubsec:Z4:morphism}

In this subsection, we build a morphism $\gamma$ that connects $(Z', \tau')$ with $(Z, \tau)$.
We start by introducing the auxiliary map $r\colon Z' \to \Z$ and proving some properties for it.
The crucial Proposition \ref{Z4:pre_def_gamma} will allow us to define the connecting morphism $\gamma$.
We finish the section with Propositions \ref{Z4:gamma:ap}, \ref{Z4:gamma:p&big} and \ref{Z4:gamma:p&small}, which will be crucial for proving $(\cP_1)$ in Theorems \ref{theo:main_sublinear} and \ref{theo:main_nonsuperlinear}.

For $z' \in Z'$ and $(c, z) = \Fac_{(Z,\tau)}(\tau'(z'))$, let 
\begin{equation} \label{eq:Z4:defi_r}
    r(z') = \begin{cases}
    0 & \text{if } \per(\tau'(z'_0)) \leq \varepsilon \\
    \min\{i \geq 0 : z_i \in \A{C}{ap}\} & \text{if } \per(\tau'(z'_0)) > \varepsilon
    \end{cases}
\end{equation}

%%%%                                                                 %%%% LEMMA %%%%
\begin{lem} \label{Z4:bound_r}
Let $z' \in Z'$ and $(c,z) = \Fac_{(Z,\tau)}(\tau'(z'))$.
\begin{enumerate}
    \item \label{Z4:bound_r:ap}
    If $z'_0 \in \A{C}{ap}'$, then $c_{r(z')} \in [-\varepsilon', |\tau'(z'_0)| - 8\varepsilon')$.
    \item \label{Z4:bound_r:p&big}
    If $\varepsilon < |\mroot\tau'(z'_0)| \leq \varepsilon'$ and $i$ is the integer satisfying $|\tau'(z'_0)| \in [c_i, c_{i+1})$, then $c_{r(z')} \in [-\varepsilon', \varepsilon')$ and $c_{i + r(Sz')} \in [|\tau'(z'_0)| - \varepsilon', |\tau'(z'_0)| + \varepsilon')$.
    \item \label{Z4:bound_r:p&small}
    If $|\mroot\tau'(z'_0)| \leq \varepsilon$, then $c_{r(z')} \in [-\varepsilon', \varepsilon')$.
\end{enumerate}
\end{lem}
\begin{proof}
We start with Item \eqref{Z4:bound_r:ap}.
Being $r(z')$ nonnegative by the definition of $r$, we have that $c_{r(z')} \geq c_0$.
Hence, $c_{r(z')} \geq -|\tau| \geq -\varepsilon'$.
To prove the other inequality, we note that the condition $z'_0 \in \A{C}{ap}'$ implies, by Item \ref{item:Z4:main_prop:per} in Proposition \ref{Z4:main_prop}, that $\per(\tau'(z')_{[8\varepsilon', |\tau'(z'_0)| - 8\varepsilon')}) > \varepsilon$.
Using Item \ref{item:CL:localize_per_aper:loc_aper} of Lemma \ref{CL:localize_per_aper}, we get $k \in [8\varepsilon', |\tau'(z'_0)| - 8\varepsilon')$ satisfying $\per(\tau'(z')_{[k - \varepsilon', k + \varepsilon')}) > \varepsilon$.
Let $j$ be the integer satisfying $k \in [c_j, c_{j+1})$.
Then, $\per(\tau'(z')_{[c_j - 8\varepsilon, c_{j+1} + 8\varepsilon)}) \geq \per(\tau'(z')_{[k-\varepsilon, k+\varepsilon)}) > \varepsilon$, so $z_j \in \A{C}{ap}$ by Item \ref{item:Z4:main_prop:per} in Proposition \ref{Z4:main_prop}.
Also, since $c_{j+1} \geq k \geq 0$, we have that $j \geq 0$.
We conclude, by the minimality condition in the definition of $r$, that $r(z') \leq j$.
Therefore, $c_{r(z')} \leq c_j \leq k \leq |\tau'(z'_0)| - 8\varepsilon'$.
\medskip

We now consider Item \eqref{Z4:bound_r:p&big}.
Let $s = \mroot \tau'(z'_0)$.
Since $|s| \leq \varepsilon'$, Item \ref{item:Z4:main_prop:per} in Proposition \ref{Z4:main_prop} ensures that $\per(\tau'(z'_0)) = |s| \in (\varepsilon, \varepsilon']$.
Thus, by Item \ref{item:CL:localize_per_aper:loc_aper} of Lemma \ref{CL:localize_per_aper}, there is $k \in [0, |\tau'(z'_0)|)$ such that $\per(\tau'(z')_{[k - \varepsilon', k + \varepsilon')}) > \varepsilon$.
Moreover, the condition $|s| \leq \varepsilon'$ implies, by Item \ref{item:Z4:main_prop:per} in Proposition \ref{Z4:main_prop}, that $\per(\tau'(z')_{[-8\varepsilon', |\tau'(z'_0)| + 8\varepsilon')})$ is at most $\varepsilon'$.
Therefore, we can find $k_0 \in [0, \varepsilon')$ and $k_1 \in [|\tau'(z'_0)|, |\tau'(z'_0)| + \varepsilon')$ satisfying $\tau'(z')_{[k - \varepsilon, k + \varepsilon)} = \tau'(z')_{[k_0 - \varepsilon, k_0 + \varepsilon)} = \tau'(z')_{[k_1 - \varepsilon, k_1 + \varepsilon)}$.
In particular,
\begin{equation} \label{eq:Z4:bound_r:1}
    \per(\tau'(z')_{[k - \varepsilon, k + \varepsilon)}) = 
    \per(\tau'(z')_{[k_0 - \varepsilon, k_0 + \varepsilon)}) =
    \per(\tau'(z')_{[k_1 - \varepsilon, k_1 + \varepsilon)}) > \varepsilon.
\end{equation}
Let $j_0, j_1 \in \Z$ be the integers satisfying $k_0 \in [c_{j_0}, c_{j_0+1})$ and $k_1 \in [c_{j_1}, c_{j_1+1})$.
Observe that, by \eqref{eq:Z4:bound_r:1} and Item \ref{item:Z4:main_prop:per} in Proposition \ref{Z4:main_prop}, $z_{j_0}$ and $z_{j_1}$ belong to $\A{C}{ap}$.
Also, since $c_{j_0+1} \geq k_0 \geq 0$ and $c_{j_1+1} \geq k_1 \geq |\tau'(z'_0)|$ (where $i$ is the element defined in the statement of the lemma), we have that $j_0 \geq 0$ and $j_1 \geq i$.
We conclude, from the definition of $r$, that $r(z') \leq j_0$ and $i + r(Sz') \leq j_1$.
Therefore, $c_{r(z')} \leq c_{j_0} \leq k_0 \leq \varepsilon'$ and $c_{i + r(Sz')} \leq c_{j_1} \leq |\tau'(z'_0)| + \varepsilon'$.
Finally, \eqref{eq:Z4:defi_r} ensures that $r(z') \geq 0$ and $i + r(Sz') \geq i$, so $c_{r(z')} \geq -|\tau| \geq -\varepsilon'$ and $c_{i + r(Sz')} \geq |\tau'(z'_0)| \geq |\tau'(z'_0)| - \varepsilon'$.
This completes the proof of Item \eqref{Z4:bound_r:p&big}.
\medskip 

For  Item \eqref{Z4:bound_r:p&small}, we note that the condition $|s| \leq \varepsilon$ implies that $\per(\tau'(z'_0))$.
Hence, $r(z') = 0$ and $c_{r(z')} \in [-|\tau|, 0] \subseteq [-\varepsilon', \varepsilon')$.
\end{proof}

%%%%                                                                 %%%% PROPOSITION %%%%
\begin{prop} \label{Z4:pre_def_gamma}
Let $z', \tilde{z}' \in Z'$, $(c,z) = \Fac_{(Z,\tau)}(\tau'(z'))$ and $(\tilde{c}, \tilde{z}) = \Fac_{(Z,\tau)}(\tau'(\tilde{z}'))$.
We define $i$ and $j$ as the integers satisfying $|\tau'(z'_0)| \in [c_i, c_{i+1})$ and $|\tau'(\tilde{z}'_0)| \in [\tilde{c}_j, \tilde{c}_{j+1})$.
If $z'_0 = \tilde{z}'_0$, then $c_{r(z')} = \tilde{c}_{r(\tilde{z}')}$, $c_{i + r(Sz')} = \tilde{c}_{j + r(S\tilde{z}')}$ and $z_{[r(z'), i + r(Sz'))}$ is equal to $\tilde{z}_{[r(\tilde{z}'), j + r(S\tilde{z}'))}$.
\end{prop}
\begin{proof}
We start with some observations that will be used throughout the proof.
Since $z'_0 = \tilde{z}'_0$, Lemma \ref{Z4:same_z0_same_ext_window} gives that
\begin{equation*}
    \tau'(z'_{-1}) = \tau'(\tilde{z}'_{-1}), 
    \tau'(z'_0) = \tau'(\tilde{z}'_0)
    \enskip\text{and}\enskip
    \tau'(z'_1) = \tau'(\tilde{z}'_1).
\end{equation*}
In particular,
\begin{equation} \label{eq:Z4:pre_def_gamma:eq_on_window}
    \tau'(z'_0)_{[-|\tau'(z'_{-1})|, |\tau'(z'_0 z'_1)|)} = 
    \tau'(\tilde{z}'_0)_{[-|\tau'(z'_{-1})|, |\tau'(z'_0 z'_1)|)}.
\end{equation}
Also, since $z'_0 = \tilde{z}'_0$, we have from Proposition \ref{Z4:reco_FZ_from_FZ'} that 
\begin{equation} \label{eq:Z4:pre_def_gamma:same_Fac}
    \Fac^0_{(Z,\tau)}(S^k \tau'(z')) = 
    \Fac^0_{(Z,\tau)}(S^k \tau'(\tilde{z}'))
    \enskip\text{for all $k \in[-4\varepsilon', |\tau'(z'_0)| + 4\varepsilon')$.}
\end{equation}

We now prove that
\begin{equation} \label{eq:Z4:pre_def_gamma:cr&zr_are_eq}
    c_{r(z')} = \tilde{c}_{r(\tilde{z}')}
    \enskip\text{and}\enskip
    z_{r(z')} = \tilde{z}_{r(\tilde{z}')}.
\end{equation}
Note that Lemma \ref{Z4:bound_r} ensures that 
\begin{equation} \label{eq:Z4:pre_def_gamma:cr_in_interval}
    c_{r(z')}, \tilde{c}_{r(\tilde{z}')} \in [-\varepsilon', |\tau'(z'_0)| - 8\varepsilon').
\end{equation}
Hence, from \eqref{eq:Z4:pre_def_gamma:same_Fac} we get that $\Fac^0_{(Z,\tau)}(S^{c_{r(z')}} \tau'(z')) = \Fac^0_{(Z,\tau)}(S^{c_{r(z')}} \tau'(\tilde{z}'))$.
This implies the following: If $\ell$ is the integer satisfying $c_{r(z')} \in [\tilde{c}_\ell, \tilde{c}_{\ell+1})$, then
\begin{equation} \label{eq:Z4:pre_def_gamma:1}
    \text{$c_{r(z')} = \tilde{c}_\ell$ and $z_{r(z')} = \tilde{z}_\ell$.}
\end{equation}
Note that $\ell \geq 0$ (as $\tilde{c}_{\ell+1} = c_{r(z')+1} \geq 0$).
Being $\tau'(z'_0) = \tau'(\tilde{z}'_0)$, we get, from \eqref{eq:Z4:defi_r}, that $r(\tilde{z}') \leq \ell$.
In particular, $\tilde{c}_{r(\tilde{z}')} \leq \tilde{c}_\ell = c_{r(z')}$.
A symmetric argument shows that $c_{r(z')} \leq \tilde{c}_{r(\tilde{z}')}$, which allows us to conclude that $\tilde{c}_{r(\tilde{z})} = c_{r(z')}$.
Then, it follows from \eqref{eq:Z4:pre_def_gamma:1} that $\tilde{c}_{r(\tilde{z})} = \tilde{c}_\ell$.
Therefore, $r(\tilde{z}) = \ell$, and thus $z_{r(z')} = \tilde{z}_\ell = \tilde{z}_{r(\tilde{z})}$ by \eqref{eq:Z4:pre_def_gamma:1}.
This proves \eqref{eq:Z4:pre_def_gamma:cr&zr_are_eq}.
\medskip

Observe that \eqref{eq:Z4:pre_def_gamma:cr&zr_are_eq} implies that $\Fac_{(Z,\tau)}(S^{c_{r(z')}} \tau'(z'))$ is equal to $\Fac_{(Z,\tau)}(S^{c_{r(z')}} \tau'(\tilde{z}'))$.
This, \eqref{eq:Z4:pre_def_gamma:eq_on_window} and \eqref{eq:Z4:pre_def_gamma:cr_in_interval} permit to use Lemma \ref{Z4:reco_from_reco} and obtain that
\begin{equation} \label{eq:Z4:pre_def_gamma:Faci_equal}
    \Fac^0_{(Z,\tau)}( S^k \tau'(z') ) = 
    \Fac^0_{(Z,\tau)}( S^k \tau'(\tilde{z}') )
    \enskip\text{for all $k\in[c_{r(z')},|\tau'(z'_0 z'_1)|-\varepsilon')$.}
\end{equation}
Then, since $c_{r(z')} = \tilde{c}_{r(\tilde{z}')}$, we have, for any $k \in \Z$ such that $c_k \in [c_{r(z')}, |\tau'(z'_0 z'_1)|-\varepsilon')$, that
\begin{equation} \label{eq:Z4:pre_def_gamma:ci_equal}
    c_k = \tilde{c}_{k - r(z') + r(\tilde{z}')}
    \enskip\text{and}\enskip
    z_k = \tilde{z}_{k - r(z') + r(\tilde{z}')}.
\end{equation}
To continue, we consider two cases.
Assume that $\per(\tau'(z'_1)) \leq \varepsilon$.
We note that, since $\tau'(z'_1) = \tau'(\tilde{z}'_1)$, $\per(\tau'(\tilde{z}'_1)) \leq \varepsilon$.
Hence, by \eqref{eq:Z4:defi_r}, $r(Sz') = r(S\tilde{z}') = 0$.
Now, the definition of $i$ and $j$ and \eqref{eq:Z4:pre_def_gamma:ci_equal} imply that $c_i = \tilde{c}_j$ and $z_i = \tilde{z}_j$.
Therefore, $c_{i + r(Sz')} = \tilde{c}_{j + r(S\tilde{z}')}$ and $z_{i + r(Sz')} = \tilde{z}_{j + r(S\tilde{z}')}$.
This completes the proof in this case.

Let us now assume that $\per(\tau'(z'_1)) > \varepsilon$.
We are going to argue as in the proof of \eqref{eq:Z4:pre_def_gamma:cr&zr_are_eq}.
Being $\tau'(z'_1) = \tau'(\tilde{z}'_1)$, we have that $\per(\tau'(\tilde{z}'_1)) > \varepsilon$.
Hence, by the definition of $r$, $z_{i + r(Sz')}$ and $\tilde{z}_{j + r(S\tilde{z}')}$ belong to  $\A{C}{ap}$.
Then, since $c_{i + r(Sz')} \in [|\tau'(z'_0)-\varepsilon', |\tau'(z'_0 z'_1)| - 8\varepsilon')$ by Lemma \ref{Z4:bound_r}, it follows from \eqref{eq:Z4:pre_def_gamma:Faci_equal} that
\begin{equation*}
    \Fac^0_{(Z,\tau)}( S^{c_{i + r(Sz')}} \tau'(z') ) =
    \Fac^0_{(Z,\tau)}( S^{c_{i + r(Sz')}} \tau'(\tilde{z}') ).
\end{equation*}
Therefore, if $k$ is the integer satisfying $c_{i + r(Sz')} \in [\tilde{c}_k, \tilde{c}_{k+1})$, then $c_{i + r(Sz')} = \tilde{c}_k$ and $z_{i + r(Sz')} = \tilde{z}_k$.
As $\per(\tau'(z'_1)) > \varepsilon$, we have that $\tilde{z}_k = z_{i+r(Sz')} \in \A{C}{ap}$.
Also, since $|\tau'(z'_0)| = |\tau'(\tilde{z}'_0)|$, $\tilde{c}_{k+1} = c_{i +r(Sz')+1} \geq |\tau'(\tilde{z}'_0)|$, so $k \geq j$.
The last two things imply, by the definition of $r(S\tilde{z}')$, that $\tilde{c}_{j + r(S\tilde{z}')} \leq \tilde{c}_k = c_{i + r(Sz')}$.
Similarly, $c_{i + r(Sz')} \leq \tilde{c}_{j + r(S\tilde{z}')}$.
We conclude that $k = j + r(S\tilde{z}')$, $\tilde{c}_{j + r(S\tilde{z}')} = c_{i + r(Sz')}$ and that $\tilde{z}_{j + r(S\tilde{z}')} = z_{i + r(Sz')}$.
\end{proof}

%%%%                                                                 %%%% DEFINITION %%%%
\begin{defi} \label{Z4:defi_gamma}
The last proposition allows us to define $\gamma \colon \cC' \to \cC^+$ in such a way that, if $z' \in Z'$, $(c,z) = \Fac_{(Z,\tau)}(\tau'(z'))$ and $i$ is the integer satisfying $|\tau'(z'_0)| \in [c_i, c_{i+1})$, then
\begin{equation} \label{eq:Z4:defi_gamma:defi}
    \gamma(z'_0) = z_{[r(z'), i + r(Sz'))}.
\end{equation}
We call $\gamma$ the {\em connecting morphism} from $(Z',\tau')$ to $(Z, \tau)$.
\end{defi}

\begin{rema}
Let $z' \in Z'$ and $(c, z) = \Fac_{(Z,\tau)}(\tau'(z'))$.
Then, \eqref{eq:Z4:defi_gamma:defi} ensures that $r(z') + |\gamma(z'_0)| = i + r(Sz')$, where $i$ is the integer satisfying $|\tau'(z'_0)| \in [c_i, c_{i+1})$.
This relation will be freely used throughout this subsection.
\end{rema}

The rest of this section is devoted to prove the main properties of $\gamma$.
We first introduce some notation.
Let $\rho(a') = \tau'(a')$ if $a' \in \A{C}{ap}'$ and $\rho(a') = \mroot \tau'(a')$ if $a' \in \cC' \setminus \A{C}{ap}'$.
We define $\psi(z') = (\rho(z'_{-1}), \tau'(z'_0), \rho(z'_1))$ if $z' \in Z'$.
Let $\rho(a)$ and $\psi(z)$ be defined analogously for $a \in \cC$ and $z \in Z$.

We fix, for the rest of the section, points $z', \tilde{z}' \in Z'$ and the notation $(c,z) = \Fac_{(Z,\tau)}(\tau'(z'))$ and $(\tilde{c}, \tilde{z}) = \Fac_{(Z,\tau)}(\tau'(\tilde{z}'))$.

%%%%                                                                 %%%% LEMMA %%%%
\begin{lem} \label{Z4:reco_for_psi}
Let $x, \tilde{x} \in X$, $(c,z) = \Fac_{(Z,\tau)}(x)$ and $(\tilde{c}, \tilde{z}) = \Fac_{(Z,\tau)}(\tilde{x})$.
If $x_{[-\varepsilon', \varepsilon')} = \tilde{x}_{[-\varepsilon', \varepsilon')}$ and $z_0 \in \A{C}{ap}$, then $\psi(z) = \psi(\tilde{z})$.
\end{lem}
\begin{proof}
The hypothesis implies that $(S^{c_0} x)_{[-E, E)} = (S^{c_0} \tilde{x})_{[-E, E)}$.
Then, as $z_0 \in \A{C}{ap}$, we can use Lemma \ref{Z4:reco_from_aper_sym} to deduce that $c_0 = \tilde{c}_0$ and $\tau(z_0) = \tau(\tilde{z}_0)$.
It is left to show that $\rho(z_{-1}) = \rho(\tilde{z}_{-1})$ and $\rho(z_1) = \rho(\tilde{z}_1)$.
We will only prove the first equality as the other follows from a similar argument.

There are three cases.
Assume first that $z_{-1} \in \A{C}{ap}$.
Then, the hypothesis ensures that $(S^{c_0-1} x)_{[-E, E)} = (S^{c_0-1} \tilde{x})_{[-E, E)}$.
Since $z_{-1} \in \A{C}{ap}$, this permits using Lemma \ref{Z4:reco_from_aper_sym} with $S^{c_0-1} x$ and $S^{c_0-1} \tilde{x}$ to deduce that $\tau(z_{-1}) = \tau(\tilde{z}_{-1})$.
The case $\tilde{z}_{-1} \in \A{C}{ap}$ is analogous.

Let us now assume that $z_{-1}, \tilde{z}_{-1} \in \cC \setminus \A{C}{ap}$.
We define $s = \mroot \tau(z_{-1})$ and $\tilde{s} = \mroot \tau(\tilde{z}_{-1})$.
We have to prove that $s = \tilde{s}$.
Observe that, by Item \ref{item:Z4:main_prop:per} in Proposition \ref{Z4:main_prop}, $s^\Z_{[-8\varepsilon, 0)} = x_{[c_0-8\varepsilon, c_0)}$ and $\tilde{s}^\Z_{[-8\varepsilon, 0)} = \tilde{x}_{[\tilde{c}_0-8\varepsilon, \tilde{c}_0)}$.
Being $c_0$ equal to $\tilde{c}_0$ and since $x_{[-\varepsilon', \varepsilon')} = \tilde{x}_{[-\varepsilon', \varepsilon')}$, we deduce that $s^\Z_{[-8\varepsilon', 0)} = \tilde{s}^\Z_{[-8\varepsilon', 0)}$.
Then, by Theorem \ref{theo:fine&wilf}, $s$ and $\tilde{s}$ are power of a common word, which implies that $s = \tilde{s}$.
\end{proof}

%%%%                                                                 %%%% LEMMA %%%%
\begin{lem} \label{Z4:reco_from_aper_r}
Suppose that $\rho(z'_0) = \rho(\tilde{z}'_0)$ and that $\per(\tau'(z'_0)) > \varepsilon$.
Assume that $\rho(z'_{-1}) = \rho(\tilde{z}'_{-1})$ or that $\per(\tau'(z'_0)) \leq \varepsilon'$.
Then, $c_{r(z')} = \tilde{c}_{r(\tilde{z}')}$ and $\psi(S^{r(z')} z) = \psi(S^{r(\tilde{z}')} \tilde{z})$.
\end{lem}
\begin{proof}
We first prove the lemma in the case $\per(\tau'(z'_0)) \leq \varepsilon'$.
Note that the condition $\rho(z'_0) = \rho(\tilde{z}'_0)$ and Items \ref{item:Z4:main_prop:lengths} and \ref{item:Z4:main_prop:per} in Proposition \ref{Z4:main_prop} guarantee that $\per(\tau'(\tilde{z}'_0)) = \per(\tau'(z'_0)) \leq \varepsilon'$ and $z'_0, \tilde{z}'_0 \in \cC' \setminus \A{C}{ap}'$.
The first thing and Lemma \ref{Z4:bound_r} give
\begin{equation} \label{eq:Z4:reco_from_aper_r:1}
    c_{r(z')}, \tilde{c}_{r(\tilde{z})} \in [-\varepsilon', \varepsilon').
\end{equation}
The second thing and the hypothesis $\rho(z'_0) = \rho(\tilde{z}'_0)$ imply that 
\begin{equation} \label{eq:Z4:reco_from_aper_r:2}
    \tau'(z')_{[-8\varepsilon', 8\varepsilon')} = 
    \tau'(\tilde{z}')_{[-8\varepsilon', 8\varepsilon')}.
\end{equation}
Now, we know from the definition of $r$ and the condition $\per(\tau'(z'_0)) = \per(\tau'(\tilde{z}'_0)) > \varepsilon$ that $z_{r(z')}, \tilde{z}_{r(\tilde{z}')} \in \A{C}{ap}$.
Hence, by Equations \eqref{eq:Z4:reco_from_aper_r:1} and \eqref{eq:Z4:reco_from_aper_r:2}, we can use Lemma \ref{Z4:reco_from_aper_sym} and deduce the following:
If $i$ and $j$ are the integers satisfying $c_{r(z')} \in [\tilde{c}_i, \tilde{c}_{i+1})$ and $\tilde{c}_{r(\tilde{z}')} \in [c_j, c_{j+1})$, then $c_{r(z')} = \tilde{c}_i$ and $\tilde{c}_{r(\tilde{z}')} = c_j$.
Therefore, by the definition of $r$, that $\tilde{c}_{r(\tilde{z}')} \leq \tilde{c}_i = c_{r(z')}$ and $c_{r(z')} \leq c_j = \tilde{c}_{r(\tilde{z}')}$.
We conclude that $c_{r(z')} = \tilde{c}_{r(\tilde{z}')}$.
This and Equations \eqref{eq:Z4:reco_from_aper_r:1} and \eqref{eq:Z4:reco_from_aper_r:2} allow us to use Lemma \ref{Z4:reco_for_psi}, yielding $\psi(S^{r(z')} z) = \psi(S^{r(\tilde{z}')} \tilde{z})$.

We now assume that $\rho(z'_{-1}) = \rho(\tilde{z}'_{-1})$ and that $\per(\tau'(z'_0)) > \varepsilon'$.
Then, $z'_0 \in \A{C}{ap}'$, so, since $\rho(z'_0) = \rho(\tilde{z}'_0)$, we have that $\tau'(z'_0) = \tau'(\tilde{z}'_0)$.
Combining this with the equation $\rho(z'_{-1}) = \rho(\tilde{z}'_{-1})$ and Item \ref{item:Z4:main_prop:per} of Proposition \ref{Z4:main_prop} produces
\begin{equation} \label{eq:Z4:reco_from_aper_r:3}
    \tau'(z')_{[-8\varepsilon', |\tau'(z'_0)| + 8\varepsilon')} = 
    \tau'(\tilde{z}')_{[-8\varepsilon', |\tau'(z'_0)| + 8\varepsilon')}.
\end{equation}
Now, by Lemma \ref{Z4:bound_r},
\begin{equation} \label{eq:Z4:reco_from_aper_r:4}
    c_{r(z')}, \tilde{c}_{r(\tilde{z}')} \in 
    [-\varepsilon', |\tau'(z'_0)| - 8\varepsilon').
\end{equation}
Equations \eqref{eq:Z4:reco_from_aper_r:3} and \eqref{eq:Z4:reco_from_aper_r:4} imply that 
\begin{multline}
    (S^{c_{r(z')}} \tau'(z'))_{[-\varepsilon', \varepsilon')} =
    (S^{c_{r(z')}} \tau'(\tilde{z}'))_{[-\varepsilon', \varepsilon')} \\
    \enskip\text{and}\enskip
    (S^{\tilde{c}_{r(\tilde{z}')}} \tau'(z'))_{[-\varepsilon', \varepsilon')} =
    (S^{\tilde{c}_{r(\tilde{z}')}} \tau'(\tilde{z}'))_{[-\varepsilon', \varepsilon')}.
\end{multline}
Since $z_{r(z')}, \tilde{z}_{r(\tilde{z}')} \in \A{C}{ap}$ by \eqref{eq:Z4:defi_r}, we can use Lemma \ref{Z4:reco_from_aper_sym} to deduce the following:
If $i$ and $j$ are the integers satisfying $c_{r(z')} \in [\tilde{c}_i, \tilde{c}_{i+1})$ and $\tilde{c}_{r(\tilde{z})} \in [c_j, c_{j+1})$, then $c_{r(z')} = \tilde{c}_i$ and $\tilde{c}_{r(\tilde{z})} = c_j$.
We can then argue as in the first case to conclude that $c_{r(z')} = \tilde{c}_{r(\tilde{z}')}$ and $\psi(S^{r(z')} z) = \psi(S^{r(\tilde{z}')} \tilde{z})$.
\end{proof}

%%%%                                                                 %%%% PROPOSITION %%%%
\begin{prop} \label{Z4:gamma:ap} % APERIODIC CASE
Suppose that  $z'_0, \tilde{z}'_0 \in \A{C}{ap}'$ and $\psi(z') = \psi(\tilde{z}')$.
Then:
\begin{enumerate}
    \item $c_{r(z')} = \tilde{c}_{r(\tilde{z}')}$ and $c_{r(z') + |\gamma(z'_0)|} = \tilde{c}_{r(\tilde{z}') + |\gamma(\tilde{z}'_0)|} $.
    \item $\psi(S^{r(z')} z) = \psi(S^{r(\tilde{z}')} \tilde{z})$ and $\psi(S^{r(z') + |\gamma(z'_0)|} z) = \psi(S^{r(\tilde{z}') + |\gamma(\tilde{z}'_0)|} \tilde{z})$.
\end{enumerate}
\end{prop}
\begin{proof}
We have, from the condition $\psi(z') = \psi(\tilde{z}')$, that $\rho(z'_{-1}) = \rho(\tilde{z}'_{-1})$ and $\rho(z'_0) = \rho(\tilde{z}'_0)$.
Also, since $z'_0 \in \A{C}{ap}'$, we have that $\per(\tau'(z'_0)) > \varepsilon$.
Hence, we can use Lemma \ref{Z4:reco_from_aper_r} to deduce that 
\begin{equation} \label{eq:Z4:gamma:ap:1}
    c_{r(z')} = \tilde{c}_{r(\tilde{z}')}
    \enskip\text{and}\enskip
    \psi(S^{r(z')} z) = \psi(S^{r(\tilde{z}')} \tilde{z}).
\end{equation}
Now, from Lemma \ref{Z4:bound_r} we have that $c_{r(z')} \in [-\varepsilon', |\tau'(z'_0)| - 8\varepsilon')$.
Also, the hypothesis and Item \ref{item:Z4:main_prop:lengths} in Proposition \ref{Z4:main_prop} give
\begin{equation} \label{eq:Z4:gamma:ap:1.5}
    \tau'(z')_{[-8\varepsilon', |\tau'(z'_0)| + 8\varepsilon')} =
    \tau'(\tilde{z}')_{[-8\varepsilon', |\tau'(z'_0)| + 8\varepsilon')}
\end{equation}
These two things, together with the fact that $z_{r(z')} \in \A{C}{ap}$, allow us to use Lemma \ref{Z4:reco_from_aper_sym} and deduce that 
\begin{equation} \label{eq:Z4:gamma:ap:2}
    \Fac^0_{(Z,\tau)}(S^{c_{r(z')+1}} \tau'(z')) = 
    \Fac^0_{(Z,\tau)}(S^{c_{r(z')+1}} \tau'(z')).
\end{equation}
In particular, $c_{r(z')+1} = \tilde{c}_{r(\tilde{z}')+1}$.
\medskip

To continue, we have to consider two cases.
We first assume that $z'_1 \in \A{C}{ap}'$.
Then, since $\psi(z') = \psi(\tilde{z}')$, we can use Lemma \ref{Z4:reco_from_aper_r} to obtain that $c_{r(z') + |\gamma(z'_0)|} = \tilde{c}_{r(\tilde{z}') + |\gamma(\tilde{z}'_0)|}$ and $\psi(S^{r(z') + |\gamma(z'_0)|} z) = \psi(S^{r(\tilde{z}') + |\gamma(\tilde{z}'_0)|} \tilde{z})$.

It rests to consider the case $z'_1 \in \cC' \setminus \A{C}{ap}'$.
Equations \eqref{eq:Z4:gamma:ap:1.5} and \eqref{eq:Z4:gamma:ap:2} enable us to use Lemma \ref{Z4:reco_from_reco} and deduce that $\Fac^0_{(Z,\tau)}(S^k \tau'(z'))$ is equal to $\Fac^0_{(Z,\tau)}(S^k \tau'(\tilde{z}'))$ for all $k \in [c_{r(z')+1}, |\tau'(z'_0)| + 7\varepsilon')$.
Since $c_{r(z')+1} \leq |\tau'(z'_0)| - 7\varepsilon'$, we in particular have that
\begin{equation} \label{eq:Z4:gamma:ap:3}
    \Fac^0_{(Z,\tau)}(S^k \tau'(z')) =
    \Fac^0_{(Z,\tau)}(S^k \tau'(\tilde{z}'))
    \enskip\text{for all $k \in [|\tau'(z'_0)| - 7\varepsilon', |\tau'(z'_0)| + 7\varepsilon')$.}
\end{equation}
Now, the condition $\psi(z') = \psi(\tilde{z}')$ implies that $\tilde{z}_1 \in \cC' \setminus \A{C}{ap}'$.
Thus, by Lemma \ref{Z4:bound_r},
\begin{equation*}
    c_{r(z') + |\gamma(z'_0)|}, \tilde{c}_{r(\tilde{z}') + |\gamma(\tilde{z}'_0)|}
    \in [|\tau'(z'_0)| - \varepsilon', |\tau'(z'_0)| + \varepsilon').
\end{equation*}
We conclude, using \eqref{eq:Z4:gamma:ap:3}, that $c_{r(z') + |\gamma(z'_0)|} = \tilde{c}_{r(\tilde{z}') + |\gamma(\tilde{z}'_0)}$ and $z_{r(z') + |\gamma(z'_0)|} = \tilde{z}_{r(\tilde{z}') + |\gamma(\tilde{z}'_0)}$.
The lemma follows.
\end{proof}

%%%%                                                                 %%%% PROPOSITION %%%%
\begin{prop} \label{Z4:gamma:p&big} % PERIODIC WITH BIG PERIOD
Suppose that $\rho(z^n_0) = \rho(\tilde{z}^n_0)$ and $\varepsilon < |\mroot \tau'(z'_0)| \leq \varepsilon'$.
Then:
\begin{enumerate}
    \item $c_{r(z')} = c_{r_1(z') + |\gamma(z'_0)|} - |\tau'(z'_0)| = \tilde{c}_{r(\tilde{z}')} = \tilde{c}_{r(\tilde{z}') + |\gamma(\tilde{z}'_0)|} - |\tau'(\tilde{z}'_0)|$.
    \item $z_{r(z')} = z_{r(z') + |\gamma(z'_0)|} = \tilde{z}_{r(\tilde{z}')} = \tilde{z}_{r(\tilde{z}') + |\gamma(\tilde{z}'_0)|}$.
\end{enumerate}
\end{prop}
\begin{proof}
Note that, by Item \ref{item:Z4:main_prop:per} in Proposition \ref{Z4:main_prop}, $\per(\tau'(z'_0)) = |\mroot \tau'(z'_0)| \in (\varepsilon, \varepsilon']$.
This and the condition $\rho(z^n_0) = \rho(\tilde{z}^n_0)$ permit to use Lemma \ref{Z4:reco_from_aper_r} to obtain that $c_{r(z')} = \tilde{c}_{r(\tilde{z}')}$.
Then, since $z'_0 \in \cC' \setminus \A{C}{ap}'$, from Lemma \ref{Z4:bound_r} we have that
\begin{equation*}
    c_{r(z')} = \tilde{c}_{r(\tilde{z}')} \in 
    [-\varepsilon', \varepsilon').
\end{equation*}
Now, the hypothesis allows us to use Lemma \ref{Z4:per_z'0_in_between=>reco} and deduce that $\Fac^0_{(Z,\tau)}(S^{c_{r(z')}} \tau'(z')) = \Fac^0_{(Z,\tau)}(S^{c_{r(z')}} \tau'(\tilde{z}'))$.
Since $c_{r(z')} = \tilde{c}_{r(\tilde{z}')}$, we get that
\begin{equation*}
    z_{r(z')} = \tilde{z}_{r(\tilde{z}')}.
\end{equation*}
We use again Lemma \ref{Z4:per_z'0_in_between=>reco} to obtain that $\Fac^0_{(Z,\tau)}(S^k \tau'(z')) = \Fac^0_{(Z,\tau)}(S^{k + |\tau'(z'_0)|} \tau'(\tilde{z}'))$ for all $k \in [-4\varepsilon', 4\varepsilon')$.
This implies, since $c_{r(z')} \in [-\varepsilon', \varepsilon')$, that
\begin{equation*}
    c_{r(z') + |\gamma(z'_0)|} = c_{r(z')} + |\tau'(z'_0)|.
\end{equation*}
Similarly, $\tilde{c}_{r(\tilde{z}') + |\gamma(\tilde{z}'_0)|} = \tilde{c}_{r(\tilde{z}')} + |\tau'(z'_0)|$.
\end{proof}

%%%%                                                                 %%%% PROPOSITION %%%%
\begin{prop} \label{Z4:gamma:p&small}  % PERIODIC WITH SMALL PERIOD
Let $k, \ell \geq 1$ and $s$ be such that $|s| \leq \varepsilon$ and $s = \mroot \tau'(z'_i) = \mroot \tau'(\tilde{z}'_j)$ for all $i \in [0,k)$ and $j \in [0,\ell)$.
Then:
\begin{enumerate}
    \item There is $t$ such that $|t| = |s|$ and  $t = \mroot \tau(z_i) = \mroot \tau(\tilde{z}_j)$ for all $i \in [r(z'), r(z') + |\gamma(z'_{[0,k)})|)$ and $j \in [r(\tilde{z}'), r(\tilde{z}') + |\gamma(\tilde{z}'_{[0,\ell)})|)$.
    \item $c_{r(z')} = \tilde{c}_{r(\tilde{z}')} = c_{r(z') + |\gamma(z'_{[0,k)})|} = \tilde{c}_{r(\tilde{z}') + |\gamma(\tilde{z}'_{[0,\ell)})} \pmod{|s|}$.
    \item $\psi(S^{r(z')} z) = \psi(S^{r(\tilde{z}')} \tilde{z})$ and, if $\rho(z'_k) = \rho(\tilde{z}'_\ell)$, then $\psi(S^{r(z') + |\gamma(z'_{[0,k)})|} z) = \psi(S^{r(\tilde{z}') + |\gamma(\tilde{z}'_{[0,\ell)})|} \tilde{z})$.
\end{enumerate}
\end{prop}
\begin{proof}
We note that, since $|s| \leq \varepsilon$, Lemma \ref{Z4:bound_r} implies that 
\begin{equation} \label{eq:Z4:gamma:p&small:1}
    c_{r(z')}, \tilde{c}_{r(\tilde{z}')} \in [-\varepsilon', \varepsilon').
\end{equation}
Hence, by Lemma \ref{Z4:per<eps=>zjcj_cte}, every $i \in \Z$ such that $c_i \in [-4\varepsilon', |\tau'(z'_{[0,k)})| + 4\varepsilon')$ satisfies
\begin{equation} \label{eq:Z4:gamma:p&small:2}
    t \coloneqq = \mroot \tau(z_{r(z')}) = \mroot \tau(z_i), 
    |t| = |s|
    \enskip\text{and}\enskip
    c_i = c_{r(z')} \pmod{|s|}
\end{equation}
Similarly, for all $j \in \Z$ such that $\tilde{c}_j \in [-4\varepsilon', |\tau'(\tilde{z}'_{[0,\ell)})| + 4\varepsilon')$,
\begin{equation} \label{eq:Z4:gamma:p&small:3}
    \tilde{t} \coloneqq = \mroot \tau(\tilde{z}_{r(\tilde{z}')}) = \mroot \tau(\tilde{z}_j), 
    |\tilde{t}| = |s|
    \enskip\text{and}\enskip
    \tilde{c}_j = \tilde{c}_{r(\tilde{z}')} \pmod{|s|}
\end{equation}
We will use these relations to prove the following:
\begin{equation} \label{eq:Z4:gamma:p&small:4}
    c_{r(z')} = \tilde{c}_{r(\tilde{z}')} \pmod{|s|}
    \enskip\text{and}\enskip
    t = \tilde{t}.
\end{equation}
Since $|s| \leq \varepsilon \leq \varepsilon'$, we can use Item \ref{item:Z4:main_prop:per} in Proposition \ref{Z4:main_prop} and \eqref{eq:Z4:gamma:p&small:1} to get that $s^\Z_{[c_{r(z')}, c_{r(z')} + 8\varepsilon)} = \tau'(z')_{[c_{r(z')}, c_{r(z')} + 8\varepsilon)} = t^\Z_{[0, 8\varepsilon)}$.
As $|s| = |t| \leq \varepsilon$, Item \ref{item:CL:on_sZ&tZ:equal} of Lemma \ref{CL:on_sZ&tZ} gives that $S^{c_{r(z')}} s^\Z = t^\Z$.
Similarly,  $S^{\tilde{c}_{r(\tilde{z}')}} s^\Z = \tilde{t}^\Z$.
We conclude that 
\begin{equation} \label{eq:Z4:gamma:p&small:5}
    S^{-c_{r(z')}} t^\Z = S^{-\tilde{c}_{r(\tilde{z}')}} \tilde{t}^\Z.
\end{equation}
Since $|\mroot \tau(z_{r(z')})| = |\mroot \tau(\tilde{z}_{r(\tilde{z}')})| = |s|$, we deduce that $t$ and $\tilde{t}$ are conjugate.
Therefore, by Item \ref{item:Z4:main_prop:per} in Proposition \ref{Z4:main_prop}, $t = \tilde{t}$.
Putting this in \eqref{eq:Z4:gamma:p&small:5} and then using Item \ref{item:CL:on_sZ&tZ:equal_mod} of Lemma \ref{CL:on_sZ&tZ} yields $c_{r(z')} = \tilde{c}_{r(\tilde{z}')} \pmod{|s|}$.
This completes the proof of \eqref{eq:Z4:gamma:p&small:4}.
\medskip

Let $\alpha$ be the integer satisfying $|\tau'(z'_{[0,k)})| \in [c_\alpha, c_{\alpha+1})$.
We have, by \eqref{eq:Z4:gamma:p&small:2}, that $z_i \in \cC \setminus \A{C}{ap}$ for all $i \in [r(z'), \alpha)$.
Also, by the definition of $r$, we have that $z_i \in \cC \setminus \A{C}{ap}$ for all $i \in [\alpha, \alpha + r(S^k z'))$.
Hence, by \ref{item:Z4:structure_Csp:window} in Proposition \ref{Z4:structure_Csp}, $\mroot \tau(z_i) = \mroot \tau(z_{r(z')}) = t$ for every $i \in [r(z'), \alpha + r(S^k z'))$.
In particular, $c_{\alpha + r(S^k z')} = c_{r(z')} \pmod{|s|}$.
Since $\alpha + r(S^kz') = r(z') + |\gamma(z'_{[0,k)})|$, we get that $c_{r(z') + |\gamma(z'_{[0,k)})|} = c_{r(z')} \pmod{|s|}$ and that $\mroot \tau(z_i) = t$ for every $i \in [r(z'), r(z') + |\gamma(z'_{[0,k)})|)$.
We can prove in a similar way that $\tilde{c}_{r(\tilde{z}') + |\gamma(\tilde{z}'_{[0,\ell)})|} = \tilde{c}_{r(\tilde{z}')} \pmod{|s|}$ and that $\mroot \tau(\tilde{z}_j) = \tilde{t}$ for every $j \in [r(\tilde{z}'), r(\tilde{z}') + |\gamma(\tilde{z}'_{[0,\ell)})|)$.
Being $t$ equal to $\tilde{t}$, we obtain Item (1).
Moreover, since $c_{r(z')} = \tilde{c}_{r(\tilde{z}')} \pmod{|s}$, we also have Item (2).

It is left to prove Item (3).
We note that, since $t = \tilde{t}$, Equations \eqref{eq:Z4:gamma:p&small:2} and \eqref{eq:Z4:gamma:p&small:3} imply that $\psi(S^{r(z')} z) = \psi(S^{r(\tilde{z}')} \tilde{z}) = (t,t,t)$.
Let us now assume that $\rho(z'_k) = \rho(\tilde{z}'_\ell)$.
There are two cases.
First, we assume that $|\mroot \tau'(z'_k)| \leq \varepsilon$.
Then, by Lemma \ref{Z4:bound_r}, $c_{r(z') + |\gamma(z'_{[0,k)})|} \in [|\tau'(z'_{[0,k)})|-\varepsilon', |\tau'(z'_{[0,k)})|+\varepsilon')$.
We get, using \eqref{eq:Z4:gamma:p&small:2}, that $\psi(S^{r(z') + |\gamma(z'_{[0,k)})|} z) = (t, t, t)$.
Now, since $\rho(z'_k) = \rho(\tilde{z}'_\ell)$, we have that $|\mroot \tau'(\tilde{z}'_\ell)| \leq \varepsilon$.
Hence, a similar argument shows that $\psi(S^{r(\tilde{z}) + |\gamma(\tilde{z}'_{[0,\ell)})|} \tilde{z}) = (t, t, t) = \psi(S^{r(z') + |\gamma(z'_{[0,k)})|} z)$.
Next, we assume that $|\mroot \tau'(z'_k)| > \varepsilon$.
Then, as $\rho(z'_k) = \rho(\tilde{z}'_\ell)$ and $\rho(z'_{k-1}) = \rho(\tilde{z}'_{\ell-1}) = s$, we can use Lemma \ref{Z4:reco_from_aper_r} with $z'_k$ and $\tilde{z}'_\ell$ to deduce that $\psi(S^{r(\tilde{z}) + |\gamma(\tilde{z}'_{[0,\ell)})|} \tilde{z}) = \psi(S^{r(z') + |\gamma(z'_{[0,k)})|} z)$.
\end{proof}

\section{Main Theorems}
\label{sec:MT}

We now complete the proofs of Theorems \ref{theo:main_sublinear} and \ref{theo:main_nonsuperlinear}.
The part of the proof in which we have to obtain a complexity restriction from the $\cS$-adic structures can be done without difficulties with Lemma \ref{MT:pX_bound}.
For the other part, we first present in Theorem \ref{theo:main_abstract} sufficient condition under which an $\cS$-adic structure as the ones in Theorems \ref{theo:main_sublinear} and \ref{theo:main_nonsuperlinear} can be obtained.
Then, we check that linear-growth and nonsuperlinear-growth complexity subshifts satisfy these conditions using Lemmas \ref{preli:find_good_s_sublinear} and \ref{preli:find_good_s_nonsuperlinear}.

\subsection{A set of sufficient conditions}

This subsection is devoted to prove the following theorem.

\begin{theo} \label{theo:main_abstract}
Let $X \subseteq \cA^\Z$ be an infinite minimal subshift.
Let $(\ell_n)_{n\geq0}$ be an increasing sequence of positive integers and $d \geq \max\{10^4, \#\cA\}$.
Suppose that for every $n \geq 0$
\begin{equation} \label{eq:theo:main_abstract:1}
    \text{$p_X(\ell_n) \leq d$, $p_X(\ell_n+1)-p_X(\ell_n) \leq d$, and $\frac{\ell_{n+1}}{\ell_n} \geq 10^4d^{2d^3+6}$.}
\end{equation}
Then, there exists a recognizable $\cS$-adic sequence $\bsigma = (\sigma_n\colon\cA_{n+1}\to\cA_n^+)_{n\geq0}$ generating $X$ such that for all $n \geq 1$:
\begin{enumerate}[label=$(\cP_{\arabic*})$]
    \item $\#(\mroot\sigma_{[0,n)}(\cA_n)) \leq 35 d^{12d+24}$ and $\#\cA_n \leq 7^4d^{12d+36} \cdot \pcom(X)^4$.
    \item $|\sigma_{[0,n)}(a)| \leq 4d^{d^3+6}\cdot |\sigma_{[0,n)}(b)|$ for every $a, b \in \cA_n$. 
    \item $|\sigma_{n-1}(a)| \leq 40d^{2d^3+8}\cdot \frac{\ell_n}{\ell_{n-1}}$ for every $a \in \cA_n$.
\end{enumerate}
\end{theo}

The proof is presented as a series of lemmas.

We fix an infinite minimal subshift $X \subseteq \cA^\Z$, an increasing sequence $(\ell_n)_{n\geq0}$ of positive integers and $d \geq \max\{10^4, \#\cA\}$ such that \eqref{eq:theo:main_abstract:1} holds for every $n \geq 0$.

We start by defining $\bsigma$.
Let $(Z_n\subseteq\cC_n^\Z, \tau_n\colon\cC_n\to\cA^+)$ be the coding constructed in Subsection \ref{subsec:Z4:construction} using $\ell_n$, and let $\varepsilon_n \in [\ell_n/d^{2d^3+4}, \ell_n/d)$ and $\cC_n = \cC_{n, \mathrm{ap}} \cup \cC_{n, \mathrm{wp}} \cup \cC_{n, \mathrm{sp}}$ be the constant and the partition that appear in this construction.
In this context, Proposition \ref{Z4:main_prop} states the following: 
\begin{enumerate}
    \item[(i)] $(\#\mroot\tau_n(\cC_n)) \leq 5d^{3d+6}$, $\#\cC_n \leq  7^4d^{12d+36} \pcom(X)^4$ and $\#\tau(\cC_{n, \mathrm{ap}}) \leq 2d^{3d+6}$.
    \item[(ii)] $2\varepsilon_n \leq |\tau_n(a)| \leq 10d^2\ell_n$ for all $a \in \cC_n$.
\end{enumerate}
Also, the definition of $\cC_n$ in \eqref{eq:Z4:defi_alphabet} guarantees that 
\begin{enumerate}
    \item[(iii)] For all $a \in \cC_n$ there is $z \in Z_n$ such that $z_0 = a$.
\end{enumerate}
Moreover, \eqref{eq:theo:main_abstract:1} implies that $500d^2 \ell_n \leq \ell_{n+1}/d^{2d^3+4}$, so the results from Subsection \ref{subsec:Z4:two_levels} can be used with $(Z_{n+1}, \ell_{n+1})$ and $(Z_n, \tau_n)$.
In particular, 
\begin{enumerate}
    \item[(iv)] Propositions \ref{Z4:pre_def_gamma}, \ref{Z4:gamma:ap}, \ref{Z4:gamma:p&big} and \ref{Z4:gamma:p&small} can be used with $(Z_{n+1}, \ell_{n+1})$ and $(Z_n, \tau_n)$.
\end{enumerate}
We define the map $r_n$ as follows.
If $z' \in Z_{n+1}$ and $(c,z) = \Fac_{(Z_n,\tau_n)}(\tau_{n+1}(z'))$, then
\begin{equation} \label{eq:MT:defi_rn}
    r_n(z') = \begin{cases}
    0 & \text{if } \per(\tau_{n+1}(z'_0)) \leq \varepsilon \\
    \min\{i \geq 0 : z_i \in \cC_{n,\mathrm{ap}}\} & \text{if } \per(\tau_{n+1}(z'_0)) > \varepsilon
    \end{cases}
\end{equation}
Note that this is analogous to the definition of $r$ in \eqref{eq:Z4:defi_r}.
Therefore, Proposition \ref{Z4:pre_def_gamma} ensures that the connecting morphism $\sigma_n\colon\cC_{n+1}\to\cC_n^+$ described in Definition \ref{Z4:defi_gamma} is well-defined.
The morphism $\sigma_n$ satisfies the following:
If $z' \in Z_{n+1}$, $(c,z) = \Fac_{(Z_n,\tau_n)}(\tau_{n+1}(z'))$ and $i$ is the integer satisfying $|\tau_{n+1}(z'_0)| \in [c_i, c_{i+1})$, then
\begin{equation} \label{eq:MT:defi_sigman}
    \sigma_n(z'_0) = z_{[r_n(z'), i + r_n(Sz'))}.
\end{equation}
We set $\sigma_0 = \tau_0$ and $\bsigma = (\sigma_n)_{n\geq0}$.
\medskip

Next, we describe $\sigma_{[0,n)}(z'_0)$ in terms of $\tau_n(z'_0)$ and the auxiliary functions $q_{j,n}$ that we now define.
For $z' \in Z_n$, we set $q_{n,n}(z') = 0$ and then inductively define, for $0\leq j < n$,
\begin{equation} \label{eq:theo:main_abstract:defi_q}
    q_{j,n}(z') = q_{j+1,n}(z') + c_{r_j(z)},
\end{equation}
where $(c,z) = \Fac_{(Z_{j+1}, \tau_{j+1})} (S^{q_{j+1,n}(z')} \tau_n(z'))$.
An inductive use of \eqref{eq:theo:main_abstract:defi_q} yields the formula
\begin{equation} \label{eq:theo:main_abstract:q_formula:partial}
    \sigma_{[0,n)}(z'_0) = \tau_n(z')_{[q_{0,n}(z'), |\tau_n(z'_0)|+q_{0,n}(S z'))}.
\end{equation}
In particular,
\begin{equation} \label{eq:theo:main_abstract:q_formula:full}
    \sigma_{[0,n)}(z') = S^{q_{0,n}(z')} \tau_n(z')
    \enspace \text{for all $n\geq1$ and $z' \in Z_n$.}
\end{equation}

We now prove that $\bsigma$ satisfies all the conditions in Theorem \ref{theo:main_abstract}.

\begin{lem} \label{MT:tech_Sadic_generator}
Let $\btau = (\tau_n\colon\cA_{n+1}\to\cA_n^+)_{n\geq0}$ be an $\cS$-adic sequence.
Suppose there are subshifts $Z_n \subseteq \cA_n^\Z$ satisfying $\cA_n \subseteq \cL(Z_n)$.
Then, for every $x \in X_{\btau}$ there are sequences $(n_\ell)_{\ell\geq0}$ and $x_\ell \in \cup_{k\in\Z} S^k \tau_{[0,n_\ell)}(Z_{n_\ell})$ such that $x$ is the limit of $(x_\ell)_{\ell\geq0}$.
\end{lem}
\begin{proof}
Let $x \in X_{\btau}$.
Then, for all $\ell \geq 0$ there exist $n_\ell \geq 0$ and $a_\ell \in \cA_{n_\ell}$ for which $x_{[-\ell, \ell)}$ occurs in $\tau_{[0,n_\ell)}(a_\ell)$.
The hypothesis permits to find $z_\ell \in Z_{n_\ell}$ such that $(z_\ell)_0 = a_\ell$.
Being $x_{[-\ell, \ell)}$ a subword of $\tau_{[0,n_\ell)}((z_\ell)_0)$, there is a point of the form $x_\ell = S^{k_\ell} \tau_{[0,n_\ell)}(z_\ell)$ satisfying $(x_\ell)_{[-\ell, \ell)} = x_{[-\ell, \ell)}$.
Then, $x$ is the limit of $(x_\ell)_{\ell\geq0}$.
The lemma follows.
\end{proof}

\begin{lem} \label{MT:generates&reco}
The $\cS$-adic sequence $\bsigma$ is recognizable and generates $X$.
\end{lem}
\begin{proof}
First, we show that $\bsigma$ generates $X$.
Note that \eqref{eq:theo:main_abstract:q_formula:full} ensures that
\begin{equation*}
    \bigcup_{k\in\Z} S^k \sigma_{[0,n)}(Z_n) \subseteq \bigcup_{k\in\Z} S^k \tau_n(Z_n) = X.
\end{equation*}
Now, thanks to Condition (iii), we can use Lemma \ref{MT:tech_Sadic_generator}, so any $x \in X_{\bsigma}$ is an adherent point of a sequence $x_n \in \cup_{k\in\Z} S^k \sigma_{[0,n)}(Z_n) = X$.
Therefore, $X_{\bsigma} \subseteq X$.
We conclude that $X_{\bsigma} = X$ by the minimality of $X$.

It rests to prove that $(Z_n,\sigma_{[0,n)})$ is recognizable.
Let $(k,z)$ and $(\tilde{k}, \tilde{z})$ be two $\sigma_{[0,n)}$-factorizations in $Z_n$ of $x \in X$.
Then, Equation \eqref{eq:theo:main_abstract:q_formula:full} implies that 
\begin{equation*}
    S^{k+q_{0,n}(z)} \tau_n(z) = S^k \sigma_{[0,n)}(z) = 
    S^{\tilde{k}} \sigma_{[0,n)}(\tilde{z}) =
    S^{\tilde{k} + q_{0,n}(\tilde{z})} \tau_n(\tilde{z}).
\end{equation*}
In particular, $S^\ell \tau_n(z) = \tau_n(\tilde{z})$ where $\ell = k+q_{0,n}(z) - \tilde{k} - q_{0,n}(\tilde{z})$.
Without loss of generality, we assume that $\ell \geq 0$.
We can find $i \geq 0$ such that $(\ell-|\tau_n(z_{[0,i)})|, S^i z)$ is a $\tau_n$-factorization of $S^\ell \tau_n(z)$ in $Z_n$.
Then, as $(0,\tilde{z})$ is a $\tau_n$-factorization of $S^\ell \tau_n(z)$ in $Z_n$, we deduce from the recognizability property of $(Z_n, \tau_n)$ that
\begin{equation} \label{eq:theo:main_abstract:5}
    \ell = k+q_{0,n}(z) - \tilde{k} - q_{0,n}(\tilde{z}) = |\tau_n(z_{[0,i)})|
    \enspace\text{and}\enspace
    S^i z = \tilde{z}
\end{equation}
Using this and the fact that $(k,z)$ and $(\tilde{k},\tilde{z})$ are $\sigma_{[0,n)}$-factorizations of $x$, we can write
\begin{multline*}
    \sigma_{[0,n)}(z) = 
    S^{\tilde{k} - k} \sigma_{[0,n)}(\tilde{z}) =
    S^{\tilde{k} - k} \sigma_{[0,n)}(S^i z) \\ =
    S^{q_{0,n}(z) - q_{0,n}(\tilde{z}) - |\tau_n(z_{[0,i)})| + |\sigma_{[0,n)}(z_{[0,i)})|} \sigma_{[0,n)}(z).
\end{multline*}
Being $Z_n$ aperiodic (as $X$ is aperiodic), we get that $q_{0,n}(z) - q_{0,n}(\tilde{z}) + |\sigma_{[0,n)}(z_{[0,i)})|$ is equal to $|\tau_n(z_{[0,i)})|$.
Putting this in \eqref{eq:theo:main_abstract:5} produces $|\sigma_{[0,n)}(z_{[0,i)})| = k - \tilde{k}$.
Since $k \in [0, |\sigma_{[0,n)}(z_0)|)$ and $\tilde{k} \in [0, |\sigma_{[0,n)}(\tilde{z}_0)|)$, we obtain that 
\begin{equation*}
    |\sigma_{[0,n)}(z_{[0,i)})| \leq k < |\sigma_{[0,n)}(z_0)|.
\end{equation*}
We deduce that $i = 0$, and then, from \eqref{eq:theo:main_abstract:5}, that $z = \tilde{z}$ and $k = \tilde{k}$.
\end{proof}

Before continuing, we give some bounds for $q_{0,n}$.
\begin{lem} \label{MT:bounds_q}
Let $n \geq 1$ and $z' \in Z_n$.
Then,
\begin{equation} \label{eq:theo:main_abstract:bounds_q}
    -2\varepsilon_n \leq q_{0,n}(z') \leq |\tau_n(z'_0)| - 7\varepsilon_n.
\end{equation}
Moreover, if $z'_0 \in \cC_n \setminus \cC_{n,\mathrm{ap}}$, then
\begin{equation} \label{eq:theo:main_abstract:bounds_q_per}
    -2\varepsilon_n \leq q_{0,n}(z') \leq 2\varepsilon_n.
\end{equation}
\end{lem}
\begin{proof}
Lemma \ref{Z4:bound_r} gives the bound $-\varepsilon_j \leq c_{r(z)} \leq |\tau_{j+1}(z_0)| - 8\varepsilon_j$ for all $0 \leq j < n$ and $z \in Z_{j+1}$.
Thus, from \eqref{eq:theo:main_abstract:q_formula:partial}, $-\varepsilon_n \leq q_{n-1,n}(z') \leq |\tau_n(z'_0)| - 8\varepsilon_n$ and $q_{j+1,n}(z') - |\tau_{j+1}| \leq q_{j,n}(z') \leq q_{j+1,n}(z') + |\tau_{j+1}|.$.
We obtain that
\begin{equation} \label{eq:theo:main_abstract:proof_bounds_q}
    -\varepsilon_n - \sum_{j=0}^{n-2}|\tau_{j+1}| \leq 
    q_{0,n}(z') \leq 
    |\tau_n(z'_0)| - 8\varepsilon_n + \sum_{j=0}^{n-2}|\tau_{j+1}|.
\end{equation}
Now, since $|\tau_j| \leq 10d^2\ell_j$ and $500d^{2d^3+6}\cdot\ell_j \leq \ell_{j+1}$, we have the bound $d^{n-2-j} |\tau_j| \leq \ell_{n-1}$ for every $j \in [0,n-1)$.
Therefore,
\begin{equation*}
    \sum_{j=0}^{n-2}|\tau_{j+1}| \leq
    \sum_{j=0}^{n-2} \frac{1}{d^{n-2-j}} \ell_{n-1} \leq
    2d \ell_{n-1} \leq \varepsilon_n.
\end{equation*}
Putting this in \eqref{eq:theo:main_abstract:proof_bounds_q} yields \eqref{eq:theo:main_abstract:bounds_q}.
Moreover, if $z'_0 \in \cC_n \setminus \cC_{n,\mathrm{ap}}$, then Lemma \ref{Z4:bound_r} gives that $q_{n-1,n}(z') \in [-\varepsilon_n, \varepsilon_n)$.
So, the previous argument shows, in this case, that $q_{0,n}(z') \in [-2\varepsilon_n, 2\varepsilon_n)$.
\end{proof}

\begin{lem} \label{MT:cP2&cP3}
For every $n\geq 1$ and $z' \in Z_n$,
\begin{equation} \label{eq:theo:main_abstract:raw_bounds_sigma}
    \frac{5}{d^{d^3+4}} \ell_n \leq |\sigma_{[0,n)}(z'_0)| \leq 20d^2 \ell_n.
\end{equation}
In particular, $\bsigma$ satisfies Items $(\cP_2)$ and $(\cP_3)$ of Theorem \ref{theo:main_abstract}.
\end{lem}
\begin{proof}
We first show that \eqref{eq:theo:main_abstract:raw_bounds_sigma} implies that $\bsigma$ satisfies Items $(\cP_1)$ and $(\cP_2)$ of Theorem \ref{theo:main_abstract}.
Observe that, by Condition (iii), \eqref{eq:theo:main_abstract:raw_bounds_sigma} gives, for every $n \geq 1$ and $a, b \in \cC_n$, that
\begin{equation} \label{eq:MT:bounds_sigma0n:P2}
    |\sigma_{[0,n)}(a)| \leq 20d^2 \ell_n \leq 4d^{d^3+6}\cdot|\sigma_{[0,n)}(b)|.
\end{equation}
Thus, $(\cP_2)$ is satisfied.
For Item $(\cP_3)$, we note that, for any pair of morphisms $\xi$ and $\xi'$ such that $\xi \xi'$ is defined, we have that $|\xi \xi'| \geq \langle\xi\rangle |\xi'|$.
Therefore,
\begin{equation*}
    |\sigma_{[0,n+1)}| \geq \langle\sigma_{[0,n)}\rangle |\sigma_n|.
\end{equation*}
Then, by Item \eqref{eq:MT:bounds_sigma0n:P2}, 
\begin{equation*}
    |\sigma_n| \leq \frac{|\sigma_{[0,n+1)}|}{\langle\sigma_{[0,n)}\rangle} \leq 
        \frac{10d^2\ell_{n+1}}{1/4d^{2d^3+6}\ell_n} = 40d^{2d^3+8} \frac{\ell_{n+1}}{\ell_n}.
\end{equation*}

We now prove \eqref{eq:theo:main_abstract:raw_bounds_sigma}.
Let $n \geq 0$ and $z' \in Z_n$ be arbitrary.
On one hand, from \eqref{eq:theo:main_abstract:q_formula:partial} we have that
\begin{equation*}
    |\sigma_{[0,n)}(z'_0)| = 
    |\tau_n(z'_0)| - q_{0,n}(z') + q_{0,n}(Sz')
    \enspace\text{for any $z' \in Z_n$.}
\end{equation*}
Hence, by \eqref{eq:theo:main_abstract:bounds_q} and Condition (ii),
\begin{equation*}
    |\sigma_{[0,n)}(z'_0)| = 
    |\tau_n(z'_0)| - q_{0,n}(z') + q_{0,n}(Sz') \leq 
    |\tau_n(z'_0)| + |\tau_n(z'_1)| - 5\varepsilon_n
    \leq 20d^2\ell_n.
\end{equation*}
Similarly,
\begin{equation*}
    |\sigma_{[0,n)}(z'_0)| = 
    |\tau_n(z'_0)| - q_{0,n}(z') + q_{0,n}(Sz') \geq
    5\varepsilon_n \geq \frac{5}{d^{d^3+4}} \ell_n.
\end{equation*}
\end{proof}

We now introduce some notation.
Let $\rho(a) = a$ if $a$ belongs to $\cC_{n, \mathrm{ap}}$ for some $n \geq 1$ and let $\rho(a) = \mroot \tau_n(a)$ if $n \geq 1$ and $a \in \cC_n \setminus \cC_{n, \mathrm{ap}}$.
We set $\psi(z) = (\rho(z_{-1}), \rho(z_0), \rho(z_1))$ for $n \geq 1$ and $z \in Z_n$.
Note that these definitions are consistent with the ones in Subsection \ref{subsubsec:Z4:morphism}.

The proof of the following lemma will be postponed until the end of the subsection.
\begin{lem} \label{MT:tech_core_roots}
Let $z, \tilde{z} \in Z_n$ be such that $\psi(z) = \psi(\tilde{z})$.
Then, $\mroot\sigma_{[0,n)}(z_0) = \mroot\sigma_{[0,n)}(\tilde{z}_0)$.
\end{lem}

\begin{lem} \label{MT:cP1}
Item $(\cP_1)$ of Theorem \ref{theo:main_abstract} is satisfied by $\bsigma$.
\end{lem}
\begin{proof}
The inequality $\#\cC_n \leq 7^4d^{12d+36} \pcom(X)^4$ in Item $(\cP_1)$ follows from Condition (i).
To prove the other inequality, we note that Lemma \ref{MT:tech_core_roots} implies that 
\begin{equation*}
    \#\mroot\sigma_{[0,n)}(\cC_n) \leq 
    \#\psi(Z_n) \cdot \#\mroot\tau_n(\cC_n).
\end{equation*}
Now, it follows from the definition of $\psi$ that $\#\psi(Z_n)$ is at most $(\#\mroot \tau_n(\cC_n) + \#\tau_n(\cC_{n,\mathrm{ap}}))^3$.
Combining this with the bounds given by Condition (i) yields
\begin{multline*}
    \#\mroot\sigma_{[0,n)}(\cC_n) \leq 
    \#\psi(Z_n) \cdot \#\mroot\tau_n(\cC_n) \\ \leq
    (5d^{3d+6} + 2d^{3d+6})^3 \cdot 5d^{3d+6} \leq 35 d^{12d+24}.
\end{multline*}
\end{proof}

It only rests to prove Lemma \ref{MT:tech_core_roots}.
We start by fixing some notation.
Let $z^n, \tilde{z}^n \in Z_n$ be such that $\psi(z^n) = \psi(\tilde{z}^n)$.
We set $s = \mroot\tau_n(z^n_0) = \mroot\tau_n(\tilde{z}^n_0)$.
For $j \in [0,n)$, we inductively define $z^j = \sigma_j(z^{j+1})$ and $\tilde{z}^j = \sigma_j(\tilde{z}^{j+1})$.
Let $(c^j, y^j) = \Fac_{(Z_j,\tau_j)}(\tau_{j+1}(z^{j+1}))$ and $(\tilde{c}^j, \tilde{y}^j) = \Fac_{(Z_j,\tau_j)}(\tau_{j+1}(\tilde{z}^{j+1}))$

With the notation introduced, we have, for every $j \in [0,n)$, that
\begin{equation} \label{eq:MT:yj_to_zj}
    S^{r_j(z^{j+1})} y^j = z^j
\end{equation}
and that
\begin{equation} \label{eq:MT:zn_to_zj}
    z^j = \sigma_{[j,n)}(z^n).
\end{equation}
We can also write, thanks to \eqref{eq:theo:main_abstract:defi_q},
\begin{multline}
    q_{j,n}(z^n) = q_{j+1,n}(z^n) + c^j_{r_j(z^{j+1})} \\
    \enskip\text{and}\enskip
    q_{j,n}(S z^n) = q_{j+1,n}(S z^n) + 
    c^j_{r_j(z^{j+1}) + |\sigma_{[j,n)}(z^n_0)|}
\end{multline}
for every $j \in [0,n)$.
Similar relations hold for $\tilde{z}^n$.
\medskip

The next three lemmas are the core of the proof of Lemma \ref{MT:tech_core_roots}.

\begin{lem} \label{MT:claim:ap}
Suppose that $\psi(z^n) = \psi(\tilde{z}^n)$ and that $\varepsilon_n < |s|$.
Then, for every $j \in [0, n]$, the following holds:
\begin{enumerate}[label=(\alph*)]
    \item $q_{j,n}(z^n) = q_{j,n}(\tilde{z}^n)$ and $q_{j,n}(Sz^n) = q_{j,n}(S\tilde{z}^n)$.
    \item $\psi(z^j) = \psi(\tilde{z}^j)$ and $\psi(S^{|\sigma_{[j,n)}(z^n)|} z^j) = \psi(S^{|\sigma_{[j,n)}(\tilde{z}^n)|} \tilde{z}^j)$.
    \item $z^j_0, \tilde{z}^j_0, z^j_{|\sigma_{[j,n)}(z^n)|}, \tilde{z}^j_{|\sigma_{[j,n)}(\tilde{z}^n)|} \in \cC_{j,\mathrm{ap}}$.
\end{enumerate}
\end{lem}
\begin{proof}
We prove the claim by induction on $j$.
The case $j = n$ is a direct consequence of the hypothesis.
Assume that $j \in [0, n)$ and that the claim is true for $j+1$.
The inductive hypothesis gives that $z^{j+1}_0$ and $\tilde{z}^{j+1}_0$ belong to $\cC_{j+1,\mathrm{ap}}$.
In particular, $\per(\tau_{j+1}(z^{j+1}_0))$ and $\per(\tau_{j+1}(\tilde{z}^{j+1}_0))$ are greater than $\varepsilon_{j+1}$.
Hence, by the definition of $r_j$ and \eqref{eq:MT:yj_to_zj}, $z^j_0 = y^j_{r_j(z^{j+1})} \in \cC_{j,\mathrm{ap}}$ and $\tilde{z}^j_0 \in \cC_{j,\mathrm{ap}}$.

Next, by Items (b) and (c) of the induction hypothesis, $z^j$ and $\tilde{z}^j$ satisfy the hypothesis of Lemma \ref{Z4:gamma:ap}.
Hence, by \eqref{eq:MT:yj_to_zj},
\begin{enumerate}
    \item $c^j_{r_j(z^{j+1})} = \tilde{c}^j_{r_j(\tilde{z}^{j+1})}$, and
    \item $\psi(z^j) = \psi(S^{r_j(z^{j+1})} y^j) = \psi(S^{r_j(\tilde{y}^{j+1})} \tilde{z}^j) = \psi(\tilde{z}^j)$.
\end{enumerate}
Putting the first equation and Item (a) of the induction hypothesis in the definition of $q_{j,n}$ yields
\begin{equation*}
    q_{j,n}(z^n) = q_{j+1,n}(z^n) + c^j_{r_j(z^{j+1})} =
    q_{j+1,n}(\tilde{z}^n) + \tilde{c}^j_{r_j(\tilde{z}^{j+1})} = q_{j,n}(\tilde{z}^n).
\end{equation*}
The rest of the inductive step follows from similar arguments.
\end{proof}

\begin{lem} \label{MT:claim:p&big}
Suppose that $\rho(z^n_0) = \rho(\tilde{z}^n_0)$ and $\varepsilon_{n-1} < |s| \leq \varepsilon_n$.
Then, for every $j \in [0, n)$, the following holds:
\begin{enumerate}[label=(\alph*)]
    \item $q_{j,n}(z^n) = q_{j,n}(\tilde{z}^n) = q_{j,n}(Sz^n) = q_{j,n}(S\tilde{z}^n)$.
    \item $\psi(z^j) = \psi(\tilde{z}^j) = \psi(S^{|\sigma_{[j,n)}(z^n)|} z^j) = \psi(S^{|\sigma_{[j,n)}(\tilde{z}^n)|} \tilde{z}^j)$.
    \item $z^j_0$, $\tilde{z}^j_0$, $z^j_{|\sigma_{[j,n)}(z^n)|}$ and $\tilde{z}^j_{|\sigma_{[j,n)}(\tilde{z}^n)|}$ belong to $\cC_{j,\mathrm{ap}}$.
\end{enumerate}
\end{lem}
\begin{proof}
We first assume that $j = n-1$.
Let us write $r = r_{n-1}$, $c_j = c^{n-1}_j$, $z = z^{n-1}$, $y = y^{n-1}$, etc.
Since $\varepsilon_{n-1} < |s| \leq \varepsilon_n$ and $\rho(z^n_0) = \rho(\tilde{z}^n_0)$, we can use Lemma \ref{Z4:gamma:p&big} with $z^n_0$ and $\tilde{z}^n_0$ to deduce the following:
\begin{enumerate}[label=(\alph*')]
    \item $c_{r(z^n)} = \tilde{c}_{r(\tilde{z}^n)} = c_{r(z^n) + |\sigma_{n-1}(z^n_0)|} - |\tau_n(z^n_0)| = \tilde{c}_{r(\tilde{z}^n) + |\sigma_{n-1}(\tilde{z}^n_0)|} - |\tau_n(\tilde{z}^n_0)|$.
    \item $y_{r(z^n)} = \tilde{y}_{r(\tilde{z}^n)} = y_{{r(z^n)} + |\sigma_{n-1}(z^n_0)|} = \tilde{y}_{r(\tilde{z}^n) + |\sigma_{n-1}(\tilde{z}^n_0)|}$.
\end{enumerate}
Item (b') implies, by \eqref{eq:MT:yj_to_zj}, that Item (b) of the claim holds for $j = n-1$.
Also, since $|s| > \varepsilon_{n-1}$, the definition of $r$ ensures that $z_0 = y_{r(z^n)} \in \cC_{n-1,\mathrm{ap}}$, so Item (c) of the claim holds.
For Item (a), we note that, since $q_{n,n} \equiv 0$, the definition of $q_{n-1,n}$ ensures that $q_{n-1,n}(z^n) = c_{r(z^n)}$, $q_{n-1,n}(\tilde{z}^n) = \tilde{c}_{r(\tilde{c}^n)}$, $q_{n-1,n}(Sz^n) = c_{r(z^n) + |\sigma_{n-1}(z^n_0)|} - |\tau_n(z^n_0)|$ and $q_{n-1,n}(S\tilde{z}^n) = \tilde{c}_{r(\tilde{z}^n) + |\sigma_{n-1}(\tilde{z}^n_0)|} - |\tau_n(\tilde{z}^n_0)|$.
Therefore, Item (a) of the claim follows from Item (a').
\medskip 

We now assume that $j \in [0, n-1)$ and that the claim holds for $j+1$.
Item (c') of the induction hypothesis ensures that
\begin{equation} \label{eq:MT:claim:p&big:1}
    \text{$z^{j+1}_0, \tilde{z}^{j+1}_0 \in \cC_{j+1,\mathrm{ap}}$ and $\psi(z^{j+1}) = \psi(\tilde{z}^{j+1})$.}
\end{equation}
Then, by the definition of $r_j$ and \eqref{eq:MT:yj_to_zj}, $z^j_0 = y^j_{r_j(z^{j+1})} \in \cC_{j,\mathrm{ap}}$ and $\tilde{z}^j_0 \in \cC_{j,\mathrm{ap}}$.
Equation \eqref{eq:MT:claim:p&big:1} also allows us to use Lemma \ref{Z4:gamma:ap} with $z^{j+1}$ and $\tilde{z}^{j+1}$ and deduce the following:
\begin{enumerate}[label=(\alph*')]
    \item $c^j_{r_j(z^{j+1})} = \tilde{c}^j_{r_j(\tilde{z}^{j+1})}$.
    \item $\psi(S^{r_j(z^{j+1})} y^j) = \psi(S^{r_j(\tilde{z}^{j+1})} \tilde{y}^j)$.
\end{enumerate}
Equation \eqref{eq:MT:yj_to_zj} ensures that Item (b') is equivalent to $\psi(z^j) = \psi(\tilde{z}^j)$.
Now, putting Item (a') and Item (a) of the induction hypothesis in the definition of $q_{j,n}$ yields
\begin{equation*}
    q_{j,n}(z^n) = q_{j+1,n}(z^n_0) + c^j_{r_j(z^{j+1})} =
    q_{j+1,n}(\tilde{z}^n_0) + \tilde{c}^j_{r_j(\tilde{z}^{j+1})} = q_{j,n}(\tilde{z}^n).
\end{equation*}
Similar arguments, which rely on using Lemma \ref{Z4:gamma:ap} with $S^{|\sigma_{[j+1,n)}(z^n_0)|} z^{j+1}$ and $S^{|\sigma_{[j+1,n)}(\tilde{z}^n_0)|} \tilde{z}^{j+1}$, show that $z^j_{|\sigma_{[j,n)}(z^n)|}, \tilde{z}^j_{|\sigma_{[j,n)}(\tilde{z}^n)|} \in \cC_{j,\mathrm{ap}}$, $\psi(S^{|\sigma_{[j,n)}(z^n)|} z^j) = \psi(S^{|\sigma_{[j,n)}(\tilde{z}^n)|} \tilde{z}^j)$ and $q_{j,n}(Sz^n) = q_{j,n}(S\tilde{z}^n)$.

To complete the proof, it is enough to show that $\psi(z^j) = \psi(S^{|\sigma_{[j,n)}(z^n)|} z^j)$ and $q_{j,n}(z^n) = q_{j,n}(Sz^n)$.
We observe that Item (b) of the inductive hypothesis guarantees that $\psi(z^{j+1}) = \psi(S^{|\sigma_{[j+1,n)}(z^n_0)|} z^{j+1})$.
Since we know that $z^{j+1}_0$ and $z^{j+1}_{|\sigma_{[j+1,n)}(z^n_0)|}$ belong to $\cC_{j+1,\mathrm{ap}}$, we can use Lemma \ref{Z4:gamma:ap} with $z^{j+1}$ and $S^{|\sigma_{[j+1,n)}(z^n_0)|} z^{j+1}$ to obtain the following:
\begin{enumerate}[label=(\alph*'')]
    \item $c^j_{r_j(z^{j+1})} = c^j_{r_j(z^{j+1}) + |\sigma_{[j,n)}(z^n_0)|} - |\tau_{j+1}\sigma_{[j+1,n)}(z^n_0)|$.
    \item $\psi(S^{r_j(z^{j+1})} y^j) = \psi(S^{r_j(z^{j+1}) + |\sigma_{[j,n)}(z^n_0)|} y^j)$.
\end{enumerate}
Item (b'') implies, by \eqref{eq:MT:yj_to_zj}, that $\psi(z^j) = \psi(S^{|\sigma_{[j,n)}(z^n_0)|} z^j)$.
Also, using the definition of $q_{j+1,n}$ and Item (a'') we can write
\begin{multline*}
    q_{j,n}(S z^n) - q_{j+1,n}(S z^n) =
    c^j_{r_j(z^{j+1}) + |\sigma_{[j,n)}(z^n_0)|} - |\tau_{j+1}\sigma_{[j+1,n)}(z^n_0)| \\ =
    c^j_{r_j(z^{j+1})} = 
    q_{j,n}(z^n) - q_{j+1,n}(z^n).
\end{multline*}
This and Item (a) of the induction hypothesis gives that $q_{j,n}(S z^n) = q_{j,n}(z^n)$.
\end{proof}

\begin{lem} \label{MT:claim:p&small}
Suppose that $\rho(z^n_0) = \rho(\tilde{z}^n_0)$ and that $|s| \leq \varepsilon_{n-1}$.
Let $j_0 \in [0,n)$ be the least element satisfying $|s| \leq \varepsilon_{j_0}$.
Then, for every $j \in [j_0, n]$, the following holds:
\begin{enumerate}[label=(\alph*)]
    \item There is $s_j$ such that $|s_j| = |s|$ and $s_j = \mroot\tau_j(z^j_k) = \mroot\tau_j(\tilde{z}^j_\ell)$ for all $k \in [0, |\sigma_{[j,n)}(z^n_0)|)$ and $\ell \in [0, |\sigma_{[j,n)}(\tilde{z}^n_0)|)$.
    \item $q_{j,n}(z^n) = q_{j,n}(\tilde{z}^n) = q_{j,n}(Sz^n) = q_{j,n}(S\tilde{z}^n) \pmod{|s|}$.
    \item $\psi(z^j) = \psi(\tilde{z}^j)$ and $\psi(S^{|\sigma_{[j,n)}(z^n)|} z^j) = \psi(S^{|\sigma_{[j,n)}(\tilde{z}^n)|} \tilde{z}^j)$.
\end{enumerate}
\end{lem}
\begin{proof}
The case $j = n$ follows directly from the hypothesis.
Assume that $j \in [j_0,n)$ and that the claim holds for $j+1$.
We observe that, by Items (a) and (c) of the induction hypothesis, $z^{j+1}_{[0, |\sigma_{[j,n)}(z^n_0)|)}$ and $\tilde{z}^{j+1}_{[0, |\sigma_{[j,n)}(\tilde{z}^n_0)|)}$ comply with the hypothesis of Lemma \ref{Z4:gamma:p&small}.
Therefore, by \eqref{eq:MT:yj_to_zj}, Items (a) and (c) hold for $j$.
Moreover, we have that
\begin{equation} \label{eq:MT:claim:p&small:1}
    c^j_{r_j(y^{j+1})} = \tilde{c}^j_{r_j(\tilde{y}^{j+1})} = 
    c^j_{r_j(y^{j+1}) + |\sigma_{[j,n)}(z^n_0)|} = 
    \tilde{c}^j_{r_j(\tilde{y}^{j+1}) + |\sigma_{[j,n)}(\tilde{z}^n_0)|} \pmod{|s|}.
\end{equation}
Now, Item (b) of the induction hypothesis gives that $q_{j+1,n}(z^n) = q_{j+1,n}(\tilde{z}^n) = q_{j+1,n}(Sz^n) = q_{j+1,n}(S\tilde{z}^n) \pmod{|s|}$.
Hence, by the definition of $q_{j,n}$,
\begin{equation*} 
    q_{j,n}(z^n) =
    q_{j+1,n}(z^n) + c^j_{r_j(y^{j+1})} =
    q_{j+1,n}(\tilde{z}^n) + c^j_{r_j(\tilde{y}^{j+1})} 
    = q_{j,n}(\tilde{z}^n) \pmod{|s|}.
\end{equation*}
We note that, since Item (a) holds for $j$, we have that $|\tau(z^j_{[0, |\sigma_{[j,n)}(z^n_0)|)})| = 0 \pmod{|s|}$.
Hence, by the definition of $q_{j,n}$,
\begin{multline*}
    q_{j,n}(S z^n) - q_{j+1,n}(S z^n) = 
    c^j_{r_j(y^{j+1}) + |\sigma_{[j,n)}(z^n_0)|} - |\sigma_{[j,n)}(z^n_0)| \\ =
    c^j_{r_j(y^{j+1}) + |\sigma_{[j,n)}(z^n_0)|} \pmod{|s|}
\end{multline*}
Thus, by \eqref{eq:MT:claim:p&small:1} and Item (b) of the induction hypothesis, $q_{j,n}(S z^n) = q_{j,n}(z^n)$.
Similarly, $q_{j,n}(S \tilde{z}^n) = q_{j,n}(\tilde{z}^n)$.
We conclude that Item (b) holds for $j$.
\end{proof}

The last ingredient for the proof of Lemma \ref{MT:tech_core_roots} is the following lemma.

\begin{lem} \label{MT:from_q_to_root}
Suppose that $\psi(z^n) = \psi(\tilde{z}^n)$.
\begin{enumerate}
    \item If $z^n_0 \in \cC_{n,\mathrm{ap}}$, $q_{0,n}(z^n) = q_{0,n}(\tilde{z}^n)$ and $q_{0,n}(Sz^n) = q_{0,n}(S\tilde{z}^n)$, then $\sigma_{[0,n)}(z^n_0) = \sigma_{[0,n)}(\tilde{z}^n_0)$.
    \item Let $s = \mroot \tau_n(z^n_0)$ and suppose that $|s| \leq \varepsilon_n$ and $q_{0,n}(z^n) = q_{0,n}(\tilde{z}^n) = q_{0,n}(z^n) = q_{0,n}(\tilde{z}^n) \pmod{|s|}$. Then, $\mroot \sigma_{[0,n)}(z^n_0) = \mroot \sigma_{[0,n)}(\tilde{z}^n_0)$.
\end{enumerate}
\end{lem}
\begin{proof}
Assume that the hypothesis of Item (1) holds.
We also assume, without loss of generality, that $|\tau_n(z^n_1)| \leq |\tau_n(\tilde{z}^n_1)|$.
We start by noticing that, since $z^n_0 \in \cC_{n,\mathrm{ap}}$ and $\psi(z^n) = \psi(\tilde{z}^n)$, we have that $\tau_n(z^n_0)$ is equal to $\tau_n(\tilde{z}^n_0)$.
Furthermore, by Condition (ii) we have that 
\begin{equation} \label{eq:MT:from_q_to_root:1}
    \tau_n(z^n)_{[-8\varepsilon_n, |\tau_n(z^n_0 z^n_1)|)} = 
    \tau_n(\tilde{z}^n)_{[-8\varepsilon_n, |\tau_n(z^n_0 z^n_1)|)}.
\end{equation}
Now, from Lemma \ref{MT:bounds_q} and the hypothesis we get that 
\begin{multline*}
    q_{0,n}(\tilde{z}^n) = q_{0,n}(z^n) \in [-2\varepsilon_n, |\tau_n(z^n_0)| - 7\varepsilon_n) \\
    \enskip\text{and}\enskip
    q_{0,n}(S\tilde{z}^n) = q_{0,n}(Sz^n) \in [|\tau_n(z^n_0)|-2\varepsilon_n, |\tau_n(z^n_0 z^n_1)| - 7\varepsilon_n).
\end{multline*}
We conclude, using \eqref{eq:MT:from_q_to_root:1}, that 
\begin{multline*}
    \sigma_{[0,n)}(z^n_0) = \tau_n(z^n)_{[q_{0,n}(z^n), |\tau_n(z^n_0)| + q_{0,n}(Sz^n))} \\
    = \tau_n(\tilde{z}^n)_{[q_{0,n}(\tilde{z}^n), |\tau_n(\tilde{z}^n_0)| + q_{0,n}(S\tilde{z}^n))}
    = \sigma_{[0,n)}(\tilde{z}^n_0).
\end{multline*}

Next, we assume that the hypothesis of Item (2) holds.
% We remark that, since $\psi(z^n) = \psi(\tilde{z}^n)$, we have that $s = \mroot \tau(z^n_0) = \mroot \tau(\tilde{z}^n_0)$.
The condition $|s| \leq \varepsilon_n$ enables us to use \eqref{eq:theo:main_abstract:bounds_q_per} from Lemma \ref{MT:bounds_q}, so
\begin{equation*}
    q_{0,n}(z^n) \in [-2\varepsilon_n, 2\varepsilon_n).
\end{equation*}
Also, since $|s| \leq \varepsilon_n$, Item \ref{item:Z4:main_prop:per} in Proposition \ref{Z4:main_prop} guarantees that 
\begin{equation*}
    \tau_n(z^n)_{[-8\varepsilon_n, |\tau_n(z^n_0)| + 8\varepsilon_n)} =
    s^\Z_{[-8\varepsilon_n, |\tau_n(z^n_0)| + 8\varepsilon_n)}.
\end{equation*}
Therefore, 
\begin{equation} \label{eq:MT:from_q_to_root:2}
    \sigma_{[0,n)}(z^n_0) = s^\Z_{[q_{0,n}(z^n), |\tau_n(z^n_0)| + q_{0,n}(S z^n))}.
\end{equation}
Now, by the hypothesis, there is $k \in \Z$ such that $k = q_{0,n}(z^n) = q_{0,n}(S z^n) \pmod{|s|}$.
We deduce from \eqref{eq:MT:from_q_to_root:2} that
\begin{equation*}
    \mroot \sigma_{[0,n)}(z^n_0) = s_{[k, k+|s|)}.
\end{equation*}
Being $\psi(z^n)$ equal to $\psi(\tilde{z}^n)$, we have that $s = \mroot \tau_n(\tilde{z}^n_0)$.
Hence, we can give similar arguments to prove that $\mroot \sigma_{[0,n)}(\tilde{z}^n_0) = s_{[\tilde{k}, \tilde{k}+|s|)}$, where $\tilde{k} = q_{0,n}(\tilde{z}^n) = q_{0,n}(S \tilde{z}^n) \pmod{|s|}$.
We conclude, as the hypothesis ensures that $k = \tilde{k} \pmod{|s|}$, that $\mroot \sigma_{[0,n)}(z^n_0) = \mroot \sigma_{[0,n)}(\tilde{z}^n_0)$.
\end{proof}

We have all the necessary elements to prove Lemma \ref{MT:tech_core_roots}.
\begin{proof}[Proof of Lemma \ref{MT:tech_core_roots}]
Let  $z', \tilde{z}' \in Z_n$ be such that $\psi(z') = \psi(\tilde{z}')$ and let $s = \mroot \tau(z'_0) = \mroot \tau(\tilde{z}'_0)$.
We split the proof into two cases.
Let us first assume that $|s| > \varepsilon_n$.
Then, we can use Lemma \ref{MT:claim:ap} and deduce that $q_{0,n}(z^n) = q_{0,n}(\tilde{z}^n)$ and $q_{0,n}(Sz^n) = q_{0,n}(S\tilde{z}^n)$.
Thus, by Lemma \ref{MT:from_q_to_root}, $\sigma_{[0,n)}(z^n_0) = \sigma_{[0,n)}(\tilde{z}^n_0)$, which implies that $\mroot \sigma_{[0,n)}(z^n_0) = \mroot \sigma_{[0,n)}(\tilde{z}^n_0)$.

Next, we assume that $|s| \leq \varepsilon_n$.
Let $j \in [0,n]$ be the least element satisfying $|s| \leq \varepsilon_j$.
We claim that the following is true:
\begin{enumerate}[label=(\alph*)]
    \item If $j < n$, then $\mroot \tau_j(z^j)$ is equal to $\mroot \tau_j(S^{|\sigma_{[j,n)}(z^n)|-1} z^j)$ and has length $|s|$.
    \item $q_{j,n}(z^n) = q_{j,n}(\tilde{z}^n) = q_{j,n}(Sz^n) = q_{j,n}(S\tilde{z}^n) \pmod{|s|}$.
    \item $\psi(z^j) = \psi(\tilde{z}^j)$ and $\psi(S^{|\sigma_{[j,n)}(z^n)|} z^j) = \psi(S^{|\sigma_{[j,n)}(\tilde{z}^n)|} \tilde{z}^j)$.
\end{enumerate}
If $j = n$, then the claim is equivalent to the hypothesis $\psi(z') = \psi(\tilde{z}')$.
We assume that $j < n$.
Then, $|s| \leq \varepsilon_{n-1}$, which permits to use Lemma \ref{MT:claim:p&small} and conclude that Items (b) and (c) of the claim hold.
Moreover, Lemma \ref{MT:claim:p&small} also states that there is $t$ such that $|t| = |s|$ and $t = \mroot \tau_j(z^j_k) = \mroot \tau_j(\tilde{z}^j_\ell)$ for all $k \in [0, |\sigma_{[j,n)}(z^n_0)|)$ and $\ell \in [0, |\sigma_{[j,n)}(\tilde{z}^n_0)|)$.
In particular, Item (a) holds.
This completes the proof of the claim.

Next, we now prove that 
\begin{equation} \label{eq:MT:tech_core_roots:1}
    q_{0,n}(z^n) = q_{0,n}(\tilde{z}^n) = q_{0,n}(Sz^n) = q_{0,n}(S\tilde{z}^n) \pmod{|s|}
\end{equation}
If $j = 0$, then \eqref{eq:MT:tech_core_roots:1} follows from the claim.
Let us assume that $j > 0$.
Then, $\varepsilon_{j-1} < |s| \leq \varepsilon_j$.
This and Items (a) and (c) of the claim allow us to use Lemma \ref{MT:claim:p&big} twice, first with $z^j$ and $\tilde{z}^j$, and then with $S^{|\sigma_{[j,n)}(z^n)|-1} z^j$ and $S^{|\sigma_{[j,n)}(\tilde{z}^n)|-1} \tilde{z}^j$.
We get
\begin{equation*}
    q_{0,j}(z^j) = q_{0,j}(\tilde{z}^j).
    \enskip\text{and}\enskip
    q_{0,j}(S^{|\sigma_{[j,n)}(z^n)|} z^j) = q_{0,j}(S^{|\sigma_{[j,n)}(\tilde{z}^n)|} \tilde{z}^j).
\end{equation*}
Moreover, since $\mroot \tau_j(z^j)$ is equal to $\mroot \tau_j(S^{|\sigma_{[j,n)}(z^n)|-1} z^j)$ and has length $|s| \leq \varepsilon_j$, we can use Lemma \ref{MT:claim:p&big} with $z^j$ and $S^{|\sigma_{[j,n)}(z^n)|-1} z^j$ to obtain that $q_{0,j}(z^j) = q_{0,j}(S^{|\sigma_{[j,n)}(z^n)|} z^j)$.
Therefore,
\begin{equation} \label{eq:MT:tech_core_roots:2}
    q_{0,j}(z^j) = q_{0,j}(\tilde{z}^j) = 
    q_{0,j}(S^{|\sigma_{[j,n)}(z^n)|} z^j) = q_{0,j}(S^{|\sigma_{[j,n)}(\tilde{z}^n)|} \tilde{z}^j).
\end{equation}
Now, from the definition of $q_{0,n}$ we have that $q_{0,n}(z^n) = q_{j,n}(z^n) + q_{0,j}(z^j)$ and $q_{0,n}(\tilde{z}^n) = q_{j,n}(\tilde{z}^n) + q_{0,j}(\tilde{z}^j)$.
Putting \eqref{eq:MT:tech_core_roots:2} and Item (b) of the claim in this relation produces $q_{0,n}(z^n) = q_{0,n}(\tilde{z}^n) \pmod{|s|}$.
The rest of the equalities in \eqref{eq:MT:tech_core_roots:1} follow from \eqref{eq:MT:tech_core_roots:2} and Item (b) of the claim in the same way.
The proof of \eqref{eq:MT:tech_core_roots:1} is complete.
\medskip

We recall that we assumed that $\psi(z^n) = \psi(\tilde{z}^n)$ and $|s| \leq \varepsilon_n$.
These two things and \eqref{eq:MT:tech_core_roots:1} permit to use Lemma \ref{MT:from_q_to_root} and conclude that $\mroot \sigma_{[0,n)}(z^n_0) = \mroot \sigma_{[0,n)}(\tilde{z}^n_0)$.
\end{proof}

\subsection{Proof of the main theorems}

\begin{lem} \label{MT:pX_bound}
Let $X$ be a subshift and $\cW$ a set of words such that $X \subseteq \bigcup_{k\in\Z} S^k \cW^\Z$.
Then, $p_X(\langle\cW\rangle) \leq |\cW|\cdot\#(\mroot\cW)^2$.
\end{lem}
\begin{proof}
The hypothesis implies that any $w$ of length $\langle\cW\rangle$ occurring in some $x \in X$ occurs in a word of the form $uv$, where $u, v \in \cW$.
In particular, $w$ occurs in $(\mroot u)^{|\cW|} (\mroot v)^{|\cW|}$.
There are at most $|\cW|\cdot\#(\mroot\cW)^2$ words satisfying this condition, so $p_X(\langle\cW\rangle) \leq |\cW|\cdot\#(\mroot\cW)^2$.
\end{proof}

\begin{theo} \label{theo:main_sublinear}
A minimal subshift $X$ has linear-growth complexity {\em i.e.}
$$ \limsup_{n\to+\infty} p_X(n) / n < +\infty, $$
if and only if there exist $d \geq 1$ and an $\cS$-adic sequence $\bsigma = (\sigma_n\colon\cA_{n+1}\to\cA_n^+)_{n\geq0}$ such that for every $n \geq 0$:
\begin{enumerate}[label=$(\cP_{\arabic*})$]
    \item $\#(\mroot\sigma_{[0,n)}(\cA_n)) \leq d$.
    \item $|\sigma_{[0,n)}(a)| \leq d\cdot|\sigma_{[0,n)}(b)|$ for every $a, b \in \cA_n$.
    \item $|\sigma_{n-1}(a)| \leq d$ for every $a \in \cA_n$.
\end{enumerate}
If $X$ is infinite and has linear-growth complexity, then $\bsigma$ can be chosen to be recognizable and satisfying $\#\cA_n \leq d\cdot \pcom(X)^4$ for all $n \geq 0$.
\end{theo}

\begin{theo} \label{theo:main_nonsuperlinear}
A minimal subshift $X$ has nonsuperlinear-growth complexity {\em i.e.}
$$ \liminf_{n\to+\infty} p_X(n) / n < +\infty, $$
if and only if there exists $d \geq 1$ and an $\cS$-adic sequence $\bsigma = (\sigma_n\colon\cA_{n+1}\to\cA_n^+)_{n\geq0}$ such that for every $n \geq 0$
\begin{enumerate}[label=$(\cP_{\arabic*})$]
    \item $\#(\mroot\sigma_{[0,n)}(\cA_n)) \leq d$.
    \item $|\sigma_{[0,n)}(a)| \leq d\cdot |\sigma_{[0,n)}(b)|$ for every $a, b \in \cA_n$.
\end{enumerate}
If $X$ is infinite and has nonsuperlinear-growth complexity, then $\bsigma$ can be chosen to be recognizable and satisfying $\#\cA_n \leq d\cdot \pcom(X)^4$ for all $n \geq 0$.
\end{theo}

We prove Theorems \ref{theo:main_sublinear} and \ref{theo:main_nonsuperlinear} simultaneously.

\begin{proof}[Proof of Theorems \ref{theo:main_sublinear} and \ref{theo:main_nonsuperlinear}]
Let $d = \liminf_{k\to+\infty} p_X(k)/k$ and $d' = \sup_{k\geq0} p_X(k)/k$.

We first assume that $p_X$ has nonsuperlinear- or linear-growth and show that there exists an $\cS$-adic sequence as the ones in Theorems \ref{theo:main_sublinear} and \ref{theo:main_nonsuperlinear}, respectively.

If $p_X$ has nonsuperlinear-growth, then $d$ is finite and so, using Lemma \ref{preli:find_good_s_nonsuperlinear}, we obtain a sequence $(\ell_n)_{n\geq0}$ such that for all $n \geq 0$
\begin{equation} \label{eq:theo:two_mains:1}
    \text{$\ell_{n+1} \geq d \ell_n$, $p_X(\ell_n) \leq d \ell_n$ and $p_X(\ell_n+1)-p_X(\ell_n) \leq d$.}
\end{equation}
If $p_X$ has linear growth, then $d'$ is finite and using Lemma \ref{preli:find_good_s_sublinear} we get a sequence $(\ell_n)_{n\geq0}$ that satisfies \eqref{eq:theo:two_mains:1} and
\begin{equation} \label{eq:theo:two_mains:2}
    \ell_{n+1} \leq d' \ell_n
    \enspace\text{for every $n\geq0$.}
\end{equation}
We use Theorem \ref{theo:main_abstract} with the sequence $(\ell_n)_{n\geq0}$.
This produces a recognizable $\cS$-adic sequence $\bsigma = (\sigma_n\colon\cA_{n+1}\to\cA_n^+)_{n\geq0}$ generating $X$ such that for every $n \geq 1$:
\begin{enumerate}[label=$(\cP_{\arabic*}')$]
    \item $\#(\mroot\sigma_{[0,n)}(\cA_n)) \leq d$ and $\#\cA_n \leq \pcom(X)$.
    \item $|\sigma_{[0,n)}(a)| \leq d\cdot |\sigma_{[0,n)}(b)|$ for every $a, b \in \cA_n$.
    \item $|\sigma_{n-1}(a)| \leq d \ell_n/\ell_{n-1}$ for every $a \in \cA_n$.
\end{enumerate}
In particular, the conclusion of Theorem \ref{theo:main_nonsuperlinear} holds.
Moreover, if $p_X$ has linear growth, then Equation \eqref{eq:theo:two_mains:2} holds, so we also have the bound $|\sigma_{n-1}(a)| \leq d d'$ for every $n \geq 1$ and $a \in \cA_n$.
Therefore, in this case, $\sigma$ satisfies the conclusion of Theorem \ref{theo:main_sublinear}.
\medskip

We now assume that there exists an $\cS$-adic sequence $\bsigma = (\sigma_n\colon\cA_{n+1}\to\cA_n^+)_{n\geq0}$ satisfying the conclusion of Theorem \ref{theo:main_nonsuperlinear} or the one of Theorem \ref{theo:main_sublinear}.

Note that since $\bsigma$ generates $X$, we have that 
\begin{equation*}
    X \subseteq 
    \bigcup_{k\in\Z} S^k \sigma_{[0,n)}(\cA_n^\Z)
\end{equation*}
for any $n \geq 1$.
Thus, by Lemma \ref{MT:pX_bound}, $p_X(\langle\sigma_{[0,n)}\rangle)$ is at most $|\sigma_{[0,n)}|\cdot \#(\mroot\sigma_{[0,n)}(\cA_n))^2$.
Items $(\cP_1)$ and $(\cP_2)$ of Theorems \ref{theo:main_nonsuperlinear} and \ref{theo:main_sublinear} then imply that
\begin{equation} \label{eq:theo:two_mains:3}
    p_X(\langle\sigma_{[0,n)}\rangle) \leq d^3 \langle\sigma_{[0,n)}\rangle
    \enspace\text{for all $n \geq 1$.}
\end{equation}
This proves that $X$ has nonsuperlinear-growth complexity.

It rests to prove that $X$ has linear-growth complexity when the conclusion of Theorem \ref{theo:main_sublinear} holds.
We assume that $\bsigma$ satisfies Items $(\cP_1)$, $(\cP_2)$ and $(\cP_3)$ of Theorem \ref{theo:main_sublinear}.
Let $k \geq 1$ be arbitrary and let $n \geq 1$ be the biggest integer such that $\langle\sigma_{[0,n)}\rangle \leq k$.
Then, by the maximality of $n$ and Items $(\cP_2)$ and $(\cP_3)$ of Theorem \ref{theo:main_sublinear}, we can compute
\begin{equation*}
    k < \langle\sigma_{[0,n+1)}\rangle \leq
    |\sigma_{[0,n)}|\cdot |\sigma_n| \leq 
    d^2 \langle\sigma_{[0,n)}\rangle \leq d^2 k.
\end{equation*}
Combining this with \eqref{eq:theo:two_mains:3} yields $p_X(k) \leq d^5 k$.
This proves that $X$ has linear-growth complexity.
\end{proof}

\section{Bounded alphabet structures}
\label{sec:alph_rank}

Theorems \ref{theo:main_sublinear} and \ref{theo:main_nonsuperlinear} provide $\cS$-adic structures for linear-growth complexity subshifts and nonsuperlinear-growth complexity subshifts.
These are not the only representations known for these classes: for instance, in \cite{DDPM20_interplay} it is proved that if $X$ has nonsuperlinear-growth complexity, then $X$ is generated by a recognizable, proper and primitive $\cS$-adic sequence $\bsigma = (\sigma_n\colon\cA_{n+1}\to\cA_n^+)_{n\geq0}$ such that $\#\cA_n$ is uniformly bounded.
The last condition is known as the {\em bounded alphabet property} and it is considered natural in the low complexity setting; see \cite{ferenczi96, DLR13_towards, E22}.
Note that the representations given by Theorems \ref{theo:main_sublinear} and \ref{theo:main_nonsuperlinear} do not necessarily satisfy this property.
In fact, our construction gives a bonded alphabet $\cS$-adic sequence if and only if the subshift has finite power complexity.
Thus, it is natural to ask whether it is possible to modify Theorems \ref{theo:main_sublinear} and \ref{theo:main_nonsuperlinear} so that they give bounded alphabet $\cS$-adic sequences.
In this section, we show that such a strengthening is not possible for Theorem \ref{theo:main_sublinear}.
More precisely, we prove the following:

\begin{theo} \label{theo:negative_result}
There exists a minimal subshift $X$ such that:
\begin{enumerate}
    \item $X$ has linear-growth complexity.
    \item If $\bsigma = (\sigma_n\colon\cA_{n+1}\to\cA_n^+)_{n\geq0}$ is an $\cS$-adic sequence generating $X$ and satisfying Items (1), (2) and (3) of Theorem \ref{theo:main_sublinear}, then $\sup_{n\geq1} \#\cA_n = +\infty$.
\end{enumerate}
\end{theo}

We were not able to obtain an analogous result for Theorem \ref{theo:main_nonsuperlinear}, so we leave this as an open question.

It is interesting to compare Theorem \ref{theo:negative_result} with the main result of \cite{Ler_1°diff}, which describes bounded alphabet $\cS$-adic representations of minimal subshifts whose complexity function satisfies $p_X(n+1)-p_X(n) \leq 2$ for all $n \geq 1$.
We are led to ask the following.
\begin{ques}
How small can $\sup_{n\geq1} p_X(n+1) - p_X(n)$ be made in Theorem \ref{theo:negative_result}?
\end{ques}

We now turn into proving Theorem \ref{theo:negative_result}.
We start with some technical lemmas.

Let $n, n_0, d, \ell \geq 1$.
We define $P(n,n_0,\ell)$ as the set of integer sequences $(p_1,\dots,p_\ell)$ such that $p_j n_0 \in [8^jn, 2\cdot8^jn)$.
Let $K(n,d,\ell)$ be the set of integer sequences $(k_1,\dots,k_\ell) \in P(n,n_0,\ell)$ for which there exists $E \subseteq [d^{-1}n, dn)$, with at most $d$ elements, such that every $k_j$ can be written as $\sum_{e \in E} \alpha_e e$, where $\alpha_e \in \Z_{\geq0}$.
%%%%                                                                 %%%% LEMMA %%%%
\begin{lem} \label{alph_rank:tech_counting}
Suppose that $n > n_0^3$, $n > \ell^{18\ell^2}$ and $\ell > 3d > 8$.
Then, $P(n,n_0,\ell) \setminus K(n,d,\ell)$ is nonempty.
\end{lem}
\begin{proof}
We will show that $\#P(n,n_0,\ell) > K(n,d,\ell)$, which implies that $P(n,n_0,\ell) \setminus K(n,d,\ell)$ is nonempty.
We first estimate $\#P(n,n_0,\ell)$.
Note that there are at least $(2\cdot8^jn - 8^jn)/n_0$ ways of choosing $p_j$ in a sequence $(p_1,\dots,p_\ell) \in P(n,n_0,\ell)$.
Thus,
\begin{equation} \label{eq:alph_rank:tech_counting:1}
    \# P(n,n_0,\ell) \geq \prod_{j=1}^\ell 8^j n/n_0 \geq (n/n_0)^\ell.
\end{equation}
Next, we estimate $\# K(n,d,\ell)$.
Choosing a set $E \subseteq [d^{-1}n, dn)$ with at most $d$ elements can be done in no more than $(dn-d^{-1}n)^d \leq (dn)^d$ different ways.
For each set $E$ and $j \in [1,d]$, there are at most $(2\cdot8^jn / d^{-1}n)^d$ numbers $\sum_{e \in E} \alpha_e e$ in $[8^jn, 2\cdot 8^jn)$.
Thus, each $E$ generates at most $(2d\cdot8^1)^d\cdot (2d\cdot8^2)^d \dots (2d\cdot8^\ell)^d$ sequences $(k_1,\dots,k_\ell) \in K(n,d,\ell)$.
Therefore,
\begin{multline} \label{eq:alph_rank:tech_counting:2}
    \# K(n,d,\ell) \leq 
    (dn)^d\cdot (2d\cdot8^1)\cdot (2d\cdot8^2) \dots (2d\cdot8^\ell) 
    \\ \leq
    d^{2d} n^d 2^\ell 8^{(\ell+2)^2} \leq
    n^d \ell^{6\ell^2},
\end{multline}
where we used that $\ell > 3d > 8$.

Now, since we assumed that $n > n_0^3$, we have that $(n/n_0)^\ell > n^{2\ell/3}$.
Hence, as the hypothesis ensures that $\ell > 3d$ and $n > \ell^{18\ell^2}$,
\begin{equation*}
    (n/n_0)^\ell > n^d n^{\ell/3} >  n^d \ell^{6\ell^2}.
\end{equation*}
This and Equations \eqref{eq:alph_rank:tech_counting:1} and \eqref{eq:alph_rank:tech_counting:2} imply that $P(n,n_0,\ell) \setminus K(n,d,\ell)$ is nonempty.
\end{proof}

%%%%                                                                 %%%% LEMMA %%%%
\begin{lem} \label{alph_rank:construction}
Let $(\ell_n)_{n\geq0}$ be a sequence of positive integers.
We consider, for each $n \geq 0$, $k_n \geq 1$ and a sequence $(p^n_1,\dots,p^n_{\ell_n})$ such that $p^n_j \in [8^j k_n, 2\cdot8^j k_n)$.
For $a \in \{0,1\}$ and $\bar{a} = 1 - a$, we define
\begin{equation} \label{alph_rank:defi:construction}
    \tau_n(a) = a^{p^n_1} \bar{a}^{p^n_1}
    a^{p^n_2} \bar{a}^{p^n_2} \dots 
    a^{p^n_{\ell_n}} \bar{a}^{p^n_{\ell_n}}.
\end{equation}
and $\btau = (\tau_n)_{n\geq0}$.
Then:
\begin{enumerate}
    \item $X_{\btau}$ is infinite, minimal and with linear-growth complexity.
    \item $|\tau_{[0,n)}(0)| = |\tau_{[0,n)}(1)|$.
    \item $(X_{\btau}^{(n)}, \tau_{[0,n)})$ is $|\tau_{[0,n)}|$-recognizable.
    \item $1 0^{p^n_j} 1 \in \cL(X_{\btau}^{(n)})$ for all $j \in [1, \ell_n]$.
\end{enumerate}
\end{lem}
\begin{proof}
Let $\cA = \{0,1\}$.
For $a \in \cA$, $n \geq 0$ and $j \in [1,\ell_n]$, we use the notation $w_{n,j}(a) = \tau_{[0,n)}(a)^{p^n_j}$ and
\begin{equation} \label{alph_rank:construction:1}
    W_{n,j}(a) = 
    w_{n,1}(a) w_{n,1}(\bar{a})\dots 
    w_{n,2}(a) w_{n,2}(\bar{a})\dots 
    w_{n,j}(a) w_{n,j}(\bar{a}). 
\end{equation}
Remark that $W_{n,\ell_n}(a) = \tau_{[0,n+1)}(a)$.

We start by proving the following properties of the morphisms $\tau_n$.
\begin{enumerate}[label=(\roman*)]
    \item $|w_{n,j}(0)| = |w_{n,j}(1)|$ and $|W_{n,j}(0)| = |W_{n,j}(1)|$ for all $n \geq 0$ and $j \in [1, \ell_n]$.
    \item $8^jk_n |\tau_{[0,n)}| \leq |w_{n,j}(a)| \leq 2\cdot8^jk_n |\tau_{[0,n)}|$ and $2\cdot8^jk_n |\tau_{[0,n)}| \leq |W_{n,j}(a)| < 8^{j+1}k_n |\tau_{[0,n)}|$.
    \item If $n\geq0$, $a,b \in \cA$ and $t$ is a word such that $|t| \geq |\tau_n|/2$, $t$ is a prefix of $\tau_n(a)$ and $t$ is a suffix of $\tau_n(b)$, then $\tau_n(a) = \tau_n(b) = t$.
\end{enumerate}
Item (i) directly follows from \eqref{alph_rank:defi:construction}.
In Item (ii), the first inequality is a consequence of the equality $|w_{n,j}(a)| = p^n_j|\tau_{[0,n)}|$ and that, by the hypothesis, $p^n_j \in [8^jk_n, 2\cdot8^jk_n)$.
We can use this to compute $|W_{n,j}(a)| \geq 2|w_{n,j}(a)| \geq 2\cdot 8^j k_n |\tau_{[0,n)}|$ and
\begin{equation*}
    |W_{n,j}(a)| = 
    \sum_{i=1}^j |w_{n,j}(a) w_{n,j}(\bar{a})| \leq
    2 \sum_{i=1}^j 2\cdot 8^i k_n |\tau_{[0,n)}| < 
    8^{j+1} k_n |\tau_{[0,n)}|,
\end{equation*}
which shows the second inequality in (ii).
Finally, we prove Item (iii).
We note that \eqref{alph_rank:construction:1} implies that $|\tau_{[0,n+1)}| = |W_{n,\ell_n}(0)| \geq 2 |w_{n,\ell_n}(0)|$.
Hence, as (i) ensures that $|\tau_{[0,n)}(0)| = |\tau_{[0,n)}(1)|$, 
\begin{equation*}
    |\tau_n| = 
    |\tau_{[0,n+1)}| / |\tau_{[0,n)}| \geq 
    2 |w_{n,\ell_n}(a)| / |\tau_{[0,n)}| = 
    2|a^{p^n_{\ell_n}}|,
\end{equation*}
which allows us to bound $|t| \geq |\tau_n|/2 \geq |a^{p^n_{\ell_n}}|$.
Being $t$ a suffix of $\tau_n(a)$, this implies that $a^{p^n_{\ell_n}}$ is a suffix of $t$.
Moreover, since $t$ is a prefix of $\tau_n(b)$, $a^{p^n_{\ell_n}}$ occurs in $\tau_n(b)$.
But \eqref{alph_rank:defi:construction} guarantees that $a^{p^n_{\ell_n}}$ occurs in $\tau_n(b)$ only as a suffix, so we must have that $t = \tau_n(b)$.
Therefore, $\tau_n(a) = \tau_n(b) = t$.
\medskip

We now prove that $\btau$ satisfies the properties of the lemma.
The morphisms $\tau_n$ are positive, so $X$ is minimal.
It follows from \eqref{alph_rank:defi:construction} that $0^{p^n_1}1$ and $0^{p^n_0}0$ belong to $\cL(X_{\btau}^{(n)})$ for all $n \geq 0$, so $\tau_{[0,n)}(0)1$ and $\tau_{[0,n)}(0)0$ are elements of $\cL(X)$.
This shows that $X$ has infinitely many right-special words, and thus that $X$ is infinite.
To prove that $X$ has linear-growth complexity, we will show that $p_X(k) \leq 1024 k$ for all $k \geq 1$.
Let $k \geq 1$ be arbitrary.
We take $n \geq 0$ such that $|\tau_{[0,n)}| \leq k < |\tau_{[0,n+1)}|$.
We consider three cases.
Assume first that $k < |w_{n,1}(0)|$.
Then, from \eqref{alph_rank:defi:construction} we have that any $w \in \cL(X) \cap \cA^k$ occurs in a word of the form $w_{n,1}(a) w_{n,1}(b)$ for some $a, b \in \cA$.
This implies, since $|w| \geq |\tau_{[0,n)}(a)|$ and $w_{n,1}(a) = \tau_{[0,n)}(a)^{p^n_1}$ for any $a \in \cA$, that $p_X(k) \leq \#\cA^2\cdot k = 4k$.

Let us now assume that $|w_{n,1}(0)| \leq k < |W_{n,\ell_n-1}(0)|$.
Let $j \in [1,\ell_n-1]$ be the least integer satisfying $k < |W_{n,j}(0)|$.
Then, by (ii), 
\begin{equation*}
    k \leq |w_{i,n}(a)| = |\tau_{[0,n)}(0)^{p^n_i}|
    \enspace \text{for all $i \in [j+1,\ell_n]$ and $a \in \cA$.}
\end{equation*}
Using this, \eqref{alph_rank:defi:construction} and the definition of $w_{n,i}(a)$ we deduce that any $w \in \cL(X) \cap \cA^k$ occurs in a word having either the form $W_{n,j}(a) \tau_{[0,n)}(b)^{p^n_{j+1}}$ or the form $\tau_{[0,n)}(a)^{p^n_{j+1}} \tau_{[0,n)}(b)^{p^n_{j+1}}$, where $a, b \in \cA$.
Therefore,
\begin{equation*}
    p_X(k) \leq 
    \#\cA^2\cdot(|W_{n,j}(0) \tau_{[0,n)}(0)^{p^n_{j+1}}| + |\tau_{[0,n)}(0)^{p^n_{j+1}} \tau_{[0,n)}(0)^{p^n_{j+1}}|).
\end{equation*}
Putting that $|W_{n,j}(0)| \leq |\tau_{[0,n)}(0)^{p^n_{j+1}}|$ in the last inequality yields
\begin{equation*}
    p_X(k) \leq 16p^n_{j+1}|\tau_{[0,n)}(0)| = 16|w_{n,j+1}(0)|.
\end{equation*}
Now, if $j = 1$, then $16|w_{n,2}(0)| \leq 16\cdot16|w_{n,1}(0)| \leq 256k$ by (ii), and so $p_X(k) \leq 256k$.
If $j > 1$, then the minimality of $j$ implies that $k \geq |W_{n,j-1}(0)|$.
Hence, by (ii), $16|w_{n,j+1}(0)| \leq 16\cdot8^2|W_{n,j}(0)| \leq 1024k$, and therefore $p_X(k) \leq 1024k$.

Finally, we assume that $k \geq |W_{n,\ell_n-1}(0)|$.
Since $k < |\tau_{[0,n+1)}|$, we have that any $w \in \cL(X) \cap \cA^k$ occurs in a word of the form $\tau_{[0,n+1)}(a) \tau_{[0,n+1)}(b)$, where $a, b \in \cA$.
Thus, $p_X(k) \leq 4|\tau_{[0,n+1)}|$.
Moreover, being $|W_{n,\ell_n-1}(0)|$ at most $k$, we have from (ii) that $4|\tau_{[0,n+1)}| = 4 |W_{n,\ell_n}(0)| \leq 8^2 |W_{n,\ell_n-1}(0)| \leq 8^2k$; therefore, $p_X(k) \leq 64k$.

We conclude that $p_X(k) \leq 1024 k$ for every $k \geq 1$ and that $X$ has linear-growth complexity.
\medskip

Items (2) and (4) of the lemma follow from \eqref{alph_rank:defi:construction}.
Thus, it only rests to prove Item (3).
We note that, since $|\tau_{[0,n)}| = |\tau_0| \cdots |\tau_{n-1}|$, it is enough, by Lemma \ref{lem:recognizability_composition}, to prove that $(\cA^\Z, \tau_n)$ is $|\tau_n|$-recognizable for all $n \geq 0$.
Let $x, \tilde{x} \in \cA_n$ and $(k,y), (\tilde{k},\tilde{y})$ be $\tau_n$-factorizations of $x,\tilde{x}$ in $\cA^\Z$, respectively, and assume that $x_{[-|\tau_n|, |\tau_n|)}$ is equal to $\tilde{x}_{[-|\tau_n|, |\tau_n|)}$.
We assume with no loss of generality that $k \leq \tilde{k}$.
There are two cases. 
If $\tilde{k} - k \leq |\tau_n|/2$, then $x_{[-k, -\tilde{k} + |\tau_n|)}$ has length at least $|\tau_n|/2$ and is a suffix of $\tau_n(y_0)$.
As $x_{[-|\tau_n|, |\tau_n|)} = \tilde{x}_{[-|\tau_n|, |\tau_n|)}$, we also have that $x_{[-k, -\tilde{k} + |\tau_n|)}$ is a prefix of $\tau_n(\tilde{y}_1)$.
We deduce, using (iii), that $\tau_n(y_0) = \tau_n(\tilde{y}_0) = x_{[-k, -\tilde{k} + |\tau_n|)}$, and thus that $y_0 = \tilde{y}_0$.
Moreover, since $k,\tilde{k} \in [0, |\tau_n|)$, $k = \tilde{k}$.
Let us now suppose that $\tilde{k} - k \geq |\tau_n|/2$.
Then, $|x_{[-\tilde{k}, -k)}| \geq |\tau_n|/2$ and $x_{[-\tilde{k}, -k)}$ is both a suffix of $\tau_n(\tilde{y}_{-1})$ and a prefix of $\tau_n(y_0)$.
Hence, by (iii), $\tau_n(\tilde{y}_{-1}) = \tau_n(y_0) = x_{[-k, -\tilde{k})}$, which is impossible as $k, \tilde{k} \in [0, |\tau_n|)$.
We conclude that $(X_{\btau}^{(n)}, \tau_n)$ is $|\tau_n|$-recognizable.
\end{proof}

%%%%                                                                 %%%% MAIN PROOF %%%%
We can now prove Theorem \ref{theo:negative_result}.

\begin{proof}[Proof of Theorem \ref{theo:negative_result}]
Let $(\ell_n)_{n\geq0}$ and $(d_n)_{n\geq0}$ be nondecreasing diverging sequences of integers with $\ell_n > 3d_n > 8$.
We inductively define $M_n$, $m_n$, $(p^n_1,\dots,p^n_{\ell_n})$ and $\tau_n$ as follows.
Let $m_0 = 1$ and $M_0$ be such that $M_0 > m_0^3$ and $M_0 > \ell_0^{18\ell_0^2}$.
Then, we can use Lemma \ref{alph_rank:tech_counting} to find
\begin{equation*}
    (p^0_1, p^0_2, \dots, p^0_{\ell_0}) \in 
    P(M_0, m_0, \ell_0) \setminus K(M_0, d_0, \ell_0).
\end{equation*}
We define $\tau_0$ using $(p^0_1,\dots, p^0_{\ell_0})$ as in \eqref{alph_rank:defi:construction}.
Suppose now that $M_n$, $m_n$, $(p^n_1,\dots,p^n_{\ell_n})$ and $\tau_n$ are defined.
We set $m_{n+1} = |\tau_{[0,n+1)}|$ and take $M_{n+1}$ so that $M_{n+1} > m_{n+1}^3$ and $M_{n+1} > \ell_{n+1}^{18\ell_{n+1}^2}$.
Then, we can use Lemma \ref{alph_rank:defi:construction} to find
\begin{equation} \label{eq:theo:negative_result:1}
    (p^{n+1}_0, p^{n+1}_1, \dots, p^{n+1}_{\ell_n})_{j=1}^{\ell_n} \in 
    P(M_{n+1}, m_{n+1}, \ell_{n+1}) \setminus K(M_{n+1}, d_{n+1}, \ell_{n+1})
\end{equation}
and define $\tau_{n+1}$ using $(p^{n+1}_1,\dots,p^{n+1}_{\ell_{n+1}})$ as in \eqref{alph_rank:defi:construction}.
\medskip

We set $\btau = (\tau_n)_{n\geq0}$.
Then, Items (1) to (4) in Lemma \ref{alph_rank:construction} hold.
In particular, $X_{\btau}$ is minimal and has linear-growth complexity.
We prove that $X_{\btau}$ satisfies the conclusion of the theorem by contradiction.
Suppose that there exist $d$ and $\bsigma = (\sigma_n\colon\cA_{n+1} \to \cA_n^+)_{n\geq0}$ satisfying Items (1), (2) and (3) of Theorem \ref{theo:main_sublinear} and $\#\cA_n \leq d$ for all $n \geq 1$.

We claim that there exists $n, n' \geq 0$ such that 
\begin{equation} \label{eq:theo:negative_result:2}
    \text{$d_n \geq 2d^6+d$, $2m_n \leq \frac{1}{6d^2}M_n \leq \langle\sigma_{[0,n')}\rangle$ and $|\sigma_{[0,n')}| \leq \frac{1}{6}M_n$.}
\end{equation}
We take $n \geq 0$ big enough so that $d_n \geq 12d^4+d$ and $m_n \geq 12d^2$.
Let $n' \geq 0$ be the integer satisfying $|\sigma_{[0,n')}| \leq \frac{1}{6} M_n < |\sigma_{[0,n'+1)}|$.
Then, by Items (2) and (3) in Theorem \ref{theo:main_sublinear},
\begin{equation*}
    \frac{1}{6} M_n <
    |\sigma_{[0,n'+1)}| \leq 
    d^2\langle\sigma_{[0,n')}\rangle. 
\end{equation*}
Also, since we chose $M_n$ and $n$ so that $M_n > m_n^3$ and $m_n \geq 12d^2$, we have that $m_n \leq \frac{1}{12d^2}M_n$.
This completes the proof of the claim.
\medskip

Let $w_a = \tau_{[0,n)}(a)$ for $a \in \{0,1\}$.
Then, by Item (4) in Lemma \ref{alph_rank:construction}, $1 0^{p^n_j} 1\in \cL(X_{\btau}^{(n)})$ for every $j \in [1, p^n_{\ell_n}]$.
Being $X_{\btau}$ generated by $\bsigma$, there exist $u_j \in \cA^+$ such that $w_1 w_0^{p^n_j} w_1$ occurs in $\sigma_{[0,n')}(u_j)$.
Moreover, we can take $u_j$ so that the following condition holds: If $a_j$ is the first letter of $u_j$ and $b_j$ is the last letter of $u_j$, then there exists a prefix $s_j$ of $\sigma_{[0,n')}(a_j)$ and a suffix $t_j$ of $\sigma_{[0,n')}(b_j)$ such that $s_j w_1 w_0^{p^n_j} w_1 t_j = \sigma_{[0,n')}(u_j)$.
Observe that $|\sigma_{[0,n')}(u_j)| \geq |w_0^{p^n_1}| = M_n$, so
\begin{equation} \label{eq:theo:negative_result:3}
    |u_j| \geq |\sigma_{[0,n')}(u_j)| / |\sigma_{[0,n')}| \geq 
    M_n / |\sigma_{[0,n')}| \geq 6.
\end{equation}
We define $a_j a'_j a''_j$ as the first three letters of $u_j$ and $b''_j b'_j b_j$ as the last three letters of $u_j$.

We claim that
\begin{equation}
    \text{if $i, j \in [1, \ell_n]$ and $a_i a'_i a''_i b''_i b'_i b_i = a_j a'_j a''_j b''_j b'_j b_j$, then $s_i = s_j$ and $t_i = t_j$.}
\end{equation}
The inequality $|w_1| = m_n \leq |\sigma_{[0,n')}|$ and \eqref{eq:theo:negative_result:3} ensure that $\sigma_{[0,n)}(a_i)$ is a prefix of $s_i w_1 w_0^\infty$.
Hence, as $s_i$ is a prefix of $\sigma_{[0,n)}(a_i)$, we can write $\sigma_{[0,n)}(a_i a'_i) = s_i w_1 w_0^{q_i} r_i$, where $q_i \geq 1$ and $r_i$ is a prefix of $w_0$ different from $w_0$.
Similarly, $\sigma_{[0,n)}(a_j a'_j) = s_j w_1 w_0^{q_j} r_j$ for some $q_j \geq 2$ and prefix $r_j$ of $w_0$ different from $w_0$.
Then, as the hypothesis implies that $\sigma_{[0,n)}(a_i a'_i a''_i) = \sigma_{[0,n)}(a_j a'_j a''_j)$, Item (4) in Lemma \ref{alph_rank:construction} can be used to obtain $r_i = r_j$.
We obtain that
\begin{equation*}
    s_i w_1 w_0^{q_i} r_i = 
    \sigma_{[0,n)}(a_i a'_i) = \sigma_{[0,n)}(a_j a'_j) =
    s_j w_1 w_0^{q_j} r_j =
    s_j w_1 w_0^{q_j} r_i.
\end{equation*}
Therefore, if $q_i \not= q_j$ then $w_0 = w_1$, which contradicts the fact that $X$ is infinite.
We conclude that $q_i = q_j$, and thus that $s_i = s_j$.
A similar argument shows that $t_i = t_j$, and the claim follows.
\medskip

Thanks to the claim, we have that the set 
\begin{multline*} 
    S = \{|\sigma_{[0,n')}(a_k a'_k)|-|s_k w_1|, |\sigma_{[0,n')}(b'_kb_k)|-|w_1t_k| : k \in [1,\ell_n]\} 
    \\ \cup \{|\sigma_{[0,n')}(a)| : a \in \cA_{n'}\}
\end{multline*}
has no more than $2\#\cA_{n'}^6 + \#\cA_{n'}$ elements.
Thus, by the choice of $d_n$,
\begin{equation}  \label{eq:theo:negative_result:4}
    \# S \leq 2d^6 + d \leq d_n.
\end{equation}
Also, by \eqref{eq:theo:negative_result:2}, $\max S \leq 2|\sigma_{[0,n')}| \leq  M_n$.
Moreover, as $|s_jw_1| \leq |\sigma_{[0,n')}(a_j)|+m_n$ and $|w_1t_j| \leq |\sigma_{[0,n')}(b_j)|+m_n$, we have that $\min S \geq \langle\sigma_{[0,n')}\rangle - m_n \geq \frac{1}{12d^2} M_n$, where in the last step we used \eqref{eq:theo:negative_result:2}.
Therefore, as $d_n \geq 2d^6+d$,
\begin{equation} \label{eq:theo:negative_result:5}
    S \subseteq [\frac{1}{d_n} M_n, d_n M_n].
\end{equation}
Let us write $u_j = a_ja'_j v_j b'_j b_j$ for certain $v_j$.
Then, the equation $\sigma_{[0,n)}(u_j) = s_i w_1 w_0^{q_i} r_i$ implies that
\begin{equation*}
    p^n_j m_n = |w_0^{p^n_j}| = 
    (\sigma_{[0,n')}(a_ja'_j) - |s_jw_1|) + |\sigma_{[0,n')}(v_j)| + (\sigma_{[0,n')}(b'_jb_j) - |w_1t_j|).
\end{equation*}
This shows that 
\begin{equation}  \label{eq:theo:negative_result:6}
    \text{every $p^n_j m_n$ can be written as $\sum_{e\in S} \gamma_e e$ for certain $\gamma_e \in \Z_{\geq0}$.}
\end{equation}
We conclude, from Equations \eqref{eq:theo:negative_result:4}, \eqref{eq:theo:negative_result:5} and \eqref{eq:theo:negative_result:6} that $(p^n_1,\dots,p^n_{\ell_n})$ belongs to $K(M_n, d_n, \ell_n)$.
This contradicts \eqref{eq:theo:negative_result:1}.
\end{proof}
\section{Applications}
\label{sec:apps}

We present in this section new and simpler proofs, based on Theorems \ref{theo:main_sublinear} and \ref{theo:main_nonsuperlinear}, of known results about linear-growth and nonsuperlinear-growth complexity subshifts.

\subsection{Cassaigne's Theorem }

A classic result on linear-growth complexity subshift is Cassaigne's Theorem \cite{Cassaigne95}, which states that, for any transitive subshift $X$ in this complexity class, $p_X(n+1)-p_X(n)$ is uniformly bounded.
We show in this subsection how to use Theorem \ref{theo:main_sublinear} to give a different proof of this result, in the case in which $X$ is minimal.

We start with a lemma containing the technical core of our approach.
\begin{lem} \label{apps:tech_pre_cassaigne}
Let $x, y \in \cA^\Z$, $p_1, \dots, p_n \in \Z$ be a collection of different integers and $\ell_1, \dots, \ell_n \geq 1$.
Suppose that:
\begin{enumerate}
    \item $x_{[p_j, p_j + \ell_j)} = y_{[0, \ell_j)}$ for all $j \in [1,n]$.
    \item $|p_j - p_i| \leq \frac{1}{2}\ell_k$ for all $i,j,k \in [1,n]$.
\end{enumerate}
Then, there exists $w \in \cA^+$ such that, for all $i, j \in [1,n]$, $x_{[p_i, p_j)}$ is a power of $w$ and $x_{[p_i, p_j + \min(\ell_i, \ell_j))}$ is a prefix of $w^\infty$.
\end{lem}
\begin{proof}
Being the $p_j$ different, there is no loss of generality in assuming that $p_1 < p_2 < \dots < p_n$.
We define, for $i,j \in [1,n]$ with $i < j$, $w_{i,j} = \mroot x_{[p_i, p_j)}$ and $\ell_{i,j} = \min(\ell_i, \ell_j)$.
Then, Item (1) in the statement of the lemma ensures that $x_{[p_i, p_i + \ell_{i,j})} = x_{[p_j, p_j + \ell_{i,j})}$, and thus that $x_{[p_i, p_j + \ell_{i,j})}$ is a prefix of $w_{i,j}^\infty$.
In particular, as $p_i < p_j$ and $w_{i,j} = \mroot x_{[p_i, j_j)}$, we have that for all $i,j \in [1,n]$ with $i < j$,
\begin{equation} \label{eq:apps:tech_pre_cassaigne:1}
    \text{$x_{[p_i, p_i + \ell_{i,j}]}$ and $x_{[p_j, p_j + \ell_{i,j})}$ are prefixes of $w_{i,j}^\infty$.}
\end{equation}
Therefore, it is enough to find $w$ such that $w = w_{i,j}$ for all $i < j$.

First, we show that
\begin{equation}  \label{eq:apps:tech_pre_cassaigne:2}
    \text{$w_{i,k} = w_{j,k}$ for all $i,j,k \in [1,n]$ with $i,j < k$.}
\end{equation}
Observe that, in this situation, we have from \eqref{eq:apps:tech_pre_cassaigne:1} that if $m = \min(\ell_{i,k}, \ell_{j,k})$, then $x_{[p_k, p_k + m)}$ is a prefix of $w_{i,k}^\infty$ and $x_{[p_k, p_k + m)}$ is a prefix of $w_{j,k}^\infty$.
Thus,
\begin{equation} \label{eq:apps:tech_pre_cassaigne:3}
    (w_{i,k}^\Z)_{[0, m)} = 
    (w_{j,k}^\Z)_{[0, m)}.
\end{equation}
Now, by Item (2) in the statement of the lemma and the definition of $w_{i,k}$ and $w_{j,k}$, $w_{i,k}$ and $w_{j,k}$ have length at most $m/2$.
This and \eqref{eq:apps:tech_pre_cassaigne:3} permits to use Theorem \ref{theo:fine&wilf} and obtain that $w_{i,k}$ and $w_{j,k}$ are powers of a common word.
This implies, since $w_{i,k}$ and $w_{j,k}$ are defined as roots, that $w_{i,k} = w_{j,k}$.

We now note that if $i,j \in [1,n]$ and $i < j$, then \eqref{eq:apps:tech_pre_cassaigne:2} ensures that $w_{1,j} = w_{i,j}$.
Hence, as $x_{[p_1, p_j)} = x_{[p_1, p_i)} x_{[p_i, p_j)}$, $w_{1,i} = w_{1,j} = w_{i,j}$.
Being $i,j$ arbitrary, this implies that $w_{1,2} = w_{1,j} = w_{i,j}$.
Therefore, the lemma follows from defining $w \coloneqq w_{1,2}$.
\end{proof}

The proposition below uses Lemma \ref{apps:tech_pre_cassaigne} to give a bound for $p_X(n+1)-p_X(n)$ in a very general context.

\begin{prop} \label{apps:pX_bound_from_coding}
Let $\cW \subseteq \cA^+$ and $X \subseteq \bigcup_{k\in\Z} S^k\cW^\Z$.
Then, for any $\ell < \langle\cW\rangle$,
\begin{equation*}
    p_X(\ell+1) - p_X(\ell) \leq 
    256\#\cA \cdot\#(\mroot\cW)^2 |\cW|^2 / \ell^2.
\end{equation*}
\end{prop}
\begin{proof}
We prove the proposition by contradiction.
Suppose that $\ell < \langle\cW\rangle$ and that $p_X(\ell+1) - p_X(\ell) \geq 256\#\cA \cdot\#(\mroot\cW)^2 |\cW|^2 / \ell^2$.
Then, by Proposition \ref{prop:props_special_words}, we can find at least $256\#(\mroot\cW)^2 |\cW|^2 / \ell^2$ right-special words $\{u_i : i \in I\}$ of length $\ell$ in $X$.
Let $u_i a_{i,0}$ and $u_i a_{i,1}$ be two different right extensions for $u_i$ in $X$.
We are going to prove that $a_{i,0} = a_{i,1}$ for some $i \in I$, contradicting the fact that $u_i a_{i,0}$ and $u_i a_{i,1}$ are different.
\medskip

Let $X' = \{ \dots vvv.v'v'v'\dots : v,v' \in \mroot\cW\} \subseteq \cA^\Z$.
Then, it is not difficult to check that:
\begin{enumerate}[label=(\alph*)]
    \item every $w \in \cL(X)$ of length at most $\langle\cW\rangle$ occurs in some $x \in X'$.
    \item $\#X' \leq \#(\mroot \cW)^2$.
\end{enumerate}
In particular, each $u_i a_{i,j}$ occurs in some $x_{i,j} = \dots v_{i,j}.v'_{i,j}\dots \in X'$, so

\begin{equation} \label{eq:apps:pX_bound_from_coding:1}
    \text{$u_i a_{i,j} = (x_{i,j})_{[\beta_{i,j}, \beta_{i,j} + \ell]}$ for some $\beta_{i,j} \in [-|\cW|, |\cW|)$.}
\end{equation}
We use (b) and the Pigeonhole principle to obtain a set $I' \subseteq I$ and $x_0, x_1 \in X'$ such that $\#I' \geq 256 |\cW|^2 / \ell^2$ and $x_j = x_{i,j}$ for all $i \in I'$ and $j \in \{0,1\}$.
We use again the Pigeonhole principle to find $I'' \subseteq I'$ satisfying $\#I'' \geq \#I'/(8|\cW|/\ell)^2 \geq 4$ and
\begin{equation} \label{eq:apps:pX_bound_from_coding:2}
    \text{$|\beta_{i,j} - \beta_{i',j}| \leq 2|\cW| / (8 |\cW| / \ell) = \ell / 4$ for all $i,i' \in I''$ and $j \in \{0,1\}$.}
\end{equation}

Let $\beta = \max\{\beta_{i,1} : i \in I''\}$ and, for $i \in I''$, $\gamma_i = \beta_{i,0} -\beta_{i,1} + \beta$.
We claim that for all $i \in I''$,
\begin{enumerate}[label=(\roman*)]
    \item $\beta_{i,1} \leq \beta \leq \beta_{i,1} + \ell/4$ and $\beta_{i,0} \leq \gamma_i \leq \beta_{i,0} + \ell/4$;
    \item $u_i$ has a suffix $u'_i$ of length at least $\frac{3}{4}\ell$ such that $u'_i a_{i,0} = (x_0)_{[\gamma_i, \gamma_i + |u'_i|]}$ and $u'_i a_{i,1} = (x_1)_{[\beta, \beta + |u'_i|]}$;
    \item if $p, q \in I''$ are different, then $\gamma_p + |u'_p| \not= \gamma_q + |u'_q|$ and $\gamma_p \not= \gamma_q$.
\end{enumerate}
Item (i) follows from \eqref{eq:apps:pX_bound_from_coding:2}.
For Item (ii), we first note that the definition of $\gamma_i$ ensures that $m_i \coloneqq \ell+\beta_{i,1}-\beta = \ell+\beta_{i,0}-\gamma_i$, and that (i) gives $\frac{3}{4}\ell \leq m_i \leq \ell$.
Thus, $u_i$ has a suffix $u'_i$ of length $m_i \geq \frac{3}{4}\ell$ such that, by \eqref{eq:apps:pX_bound_from_coding:1}, satisfies Item (ii).

It is left to prove (iii).
Assume that $p, q \in I''$ and $\gamma_p + |u'_p| = \gamma_q + |u'_q|$.
Then, $\beta_{p,0} = \gamma_p + |u'_p| - \ell = \gamma_q + |u'_q| - \ell = \beta_{q,0}$ and hence $u_p = u_q$, which implies that $p = q$.
Let us now suppose that $\gamma_p = \gamma_q$.
Note that if $|u'_p| = |u'_q|$ then $\gamma_p + |u'_p| = \gamma_q + |u'_q|$, and so $p = q$ by what we just proved.
Thus, there is no loss of generality in assuming that $|u'_p| < |u'_q|$.
Then, (ii) allows us to write $u'_p a_{p,0} = (x_0)_{[\gamma_p, \gamma_p + |u'_p|]}$ and $u'_q = (x_0)_{[\gamma_p, \gamma_p + |u'_q|]}$.
In particular, $u'_p a_{p,0}$ is a prefix of $u'_q$.
Similarly, (ii) implies that $u'_p a_{p,1} = (x_1)_{[\beta, \beta + |u'_p|]}$ is a prefix of $u'_q = (x_1)_{[\beta, \beta + |u'_q|)}$.
Therefore, $u'_p a_{p,0} = u'_p a_{p,1}$, which contradicts the definition of $a_{p,0}$ and $a_{p,1}$.
This shows that the case $|u'_p| < |u'_q|$ does not occur, so $|u'_p| = |u'_q|$ and $p = q$.
This completes the proof of the claim.
\medskip

Thanks to the claim, we have that $(x_0)_{[\gamma_i, \gamma_i + |u'_i|)} = (S^\beta x_1)_{[0, |u'_i|)}$ and $|\gamma_i - \gamma_{i'}| \leq \frac{1}{2}|u'_i|$ for all $i, i' \in I''$. 
Moreover, all the $\gamma_i$ are different by (iii).
Therefore, we can use Lemma \ref{apps:tech_pre_cassaigne} and deduce that there exists $w \in \cA^+$ such that for any $p, q \in I''$, 
\begin{multline} \label{eq:apps:pX_bound_from_coding:3}
    \text{$(x_0)_{[\gamma_p, \gamma_q)}$ is a power of $w$ 
    }\\\text{
    and $(x_0)_{[ \gamma_p, \gamma_p + \min(|u'_p|, |u'_q|) )}$ is a prefix of $w^\infty$.}
\end{multline}
We use Item (iii) of the claim and that $\#I'' \geq 4$ to find $p, q \in I''$ such that $|u'_t| < |u'_p| < |u'_q|$ for all $t \in I''\setminus\{p,q\}$.
Furthermore, (iii) allows us to find $r, s \in I''\setminus \{p,q\}$ such that $\gamma_r + |u'_r| < \gamma_s + |u'_s|$.

We observe that, since $|u'_s| < |u'_p|$, the second part of \eqref{eq:apps:pX_bound_from_coding:3} ensures that that $u'_s = (x_0)_{[ \gamma_s, \gamma_s + |u'_s| )} $ is a prefix of $w^\infty$.
Then, by the first part of \eqref{eq:apps:pX_bound_from_coding:3}, $(x_0)_{[ \gamma_r, \gamma_s + |u'_s| )}$ is a prefix of $w^\infty$.
Since $\gamma_r + |u'_r| < \gamma_s + |u'_s|$, we get that
\begin{equation} \label{eq:apps:pX_bound_from_coding:4}
    \text{$u'_r a_{r,0} = (x_0)_{[\gamma_r, \gamma_r + |u'_r|]}$ is a prefix of $w^\infty$.}
\end{equation}
Now, the definition of $r$ and $p$ guarantees that $|u'_r| < |u'_p|$, so, by \eqref{eq:apps:pX_bound_from_coding:1}, $u'_r a_{r,1} = (x_1)_{[\beta, \beta + |u'_r|])}$ is a prefix of $(x_0)_{[\gamma_p, \gamma_p + |u'_p|)} = (x_1)_{[\beta, \beta + |u'_p|)}$.
Moreover, as $|u'_p| < |u'_q|$, the second part of \eqref{eq:apps:pX_bound_from_coding:3} gives that $(x_0)_{[\gamma_p, \gamma_p + |u'_p|)}$ is a prefix of $w^\infty$.
We conclude that $u'_r a_{r,1}$ is a prefix of $w^\infty$.
But then \eqref{eq:apps:pX_bound_from_coding:4} implies that $u'_r a_{r,1} = u'_r a_{r,0}$, contradicting our assumptions.
\end{proof}

\begin{theo}[\cite{Cassaigne95}]
\label{apps:theo:cassaigne}
Let $X$ be a minimal linear-growth complexity subshift. 
Then, $p_X(\ell+1) - p_X(\ell)$ is uniformly bounded.
\end{theo}
\begin{proof}
Let $\ell \geq 1$ be arbitrary.
The theorem is trivial if $X$ is finite, so we assume that $X$ is infinite.
Then, we can use Theorem \ref{theo:main_sublinear} to obtain an $\cS$-adic sequence $\bsigma = (\sigma_n\colon\cA_{n+1}^+\to\cA_n^+)_{n\geq1}$ generating $X$ and $d \geq 1$ such that the conditions (1), (2) and (3) of Theorem \ref{theo:main_sublinear} hold.

Let $n \geq 1$ be the least integer such that $\langle \sigma_{[0,n)} \rangle > \ell$.
Then, $X$ is a subset of $\bigcup_{k \in \Z} S^k \sigma_{[0,n)}(\cA_n^\Z)$, so Proposition \ref{apps:pX_bound_from_coding} and the conditions in Theorem \ref{theo:main_sublinear} give the bounds:
\begin{align*}
    p_X(\ell+1) - p_X(\ell) &\leq 
    256 \#\cA_0 \#(\mroot\sigma_{[0,n)}(\cA_n))^2 |\sigma_{[0,n)}|^2 / \ell^2 \\ & \leq 
    256 \#\cA_0\cdot d^2\cdot |\sigma_{[0,n)}|^2 / \ell^2.
\end{align*}
Now, the minimality of $n$ ensures that $\langle\sigma_{[0,n-1)}\rangle \leq \ell$, so by Items (2) and (3) in Theorem \ref{theo:main_sublinear},
\begin{equation*}
|\sigma_{[0,n)}| \leq 
|\sigma_{n-1}|\cdot|\sigma_{[0,n-1)}| \leq
d^2\langle\sigma_{[0,n-1)}\rangle \leq d^2 \ell.
\end{equation*}
Therefore, $p_X(\ell+1) - p_X(\ell) \leq  256 \#\cA_0\cdot d^6$ and $p_X(\ell+1) - p_X(\ell)$ is uniformly bounded.
\end{proof}

\subsection{A theorem of Cassaigne, Frid, Puzynina and Zamboni}

The following result was proven in \cite{CFPZ18}.
\begin{theo} \label{apps:theo:CFPZ19_original}
Let $x \in \cA^\N$ be an infinite sequence.
The following conditions are equivalent:
\begin{enumerate}
    \item $x$ has linear-word complexity.
    \item There exists $S \subseteq \cA^*$ such that $S^2 \supseteq \cL(x)$ and $\sup_{n\geq1} p_S(n) < +\infty$.
\end{enumerate}
\end{theo}
In this subsection, we give a different proof of Theorem \ref{apps:theo:CFPZ19_original} for the case of minimal subshifts.
We start by proving the following corollary of Theorem \ref{theo:main_sublinear}.

\begin{prop} \label{apps:tech_contract_CFPZ19}
Let $X$ be an infinite minimal subshift of linear-growth complexity.
There exists $d \geq 1$ such that for any $d' \geq 2$ we can find $\btau = (\tau_n\colon\cA_{n+1}\to\cA_n^+)_{n\geq0}$ generating $X$ such that: 
\begin{enumerate}
    \item $\#(\mroot\tau_{[0,n)}(\cA_n)) \leq d$.
    \item $|\tau_{[0,n)}(a)| \leq d\cdot |\tau_{[0,n)}(b)|$ for all $a, b \in \cA_n$.
    \item $d' \leq |\tau_{n-1}(a)| \leq d^{6\log_2 d'}$ for all $a \in \cA_n$.
\end{enumerate}
\end{prop}
\begin{proof}
Let $\btau' = (\tau_n\colon\cA_{n+1}\to\cA_n^+)_{n\geq0}$ and $d$ be the elements given by Theorem \ref{theo:main_sublinear} when it is applied with $X$, and let $d' \geq 1$ be arbitrary.
We will construct $\btau$ by carefully contracting $\btau'$.

Let $n_0 = 0$ and inductively define $n_{k+1}$ as the smallest integer such that $n_{k+1} > n_k$ and $\langle\tau_{[n_k, n_{k+1})}\rangle \geq 2$.
We observe that, since $\langle\tau_{[n_k, n_{k+1}-1)}\rangle = 1$ by the minimality of $n_{k+1}$, we have that $\langle\tau_{[0,n_{k+1}-1)}\rangle \leq |\tau_{[0,n_k)}|$.
Then, by Items (2) and (3) in Theorem \ref{theo:main_sublinear}, we can bound
\begin{equation*}
    |\tau_{[0, n_{k+1})}| \leq
    |\tau_{[0, n_{k+1}-1)}| |\tau_{n_{k+1}-1}| \leq
    d \langle\tau_{[0, n_{k+1}-1)}\rangle \cdot d \leq 
    d^2 |\tau_{[0,n_k)}| \leq d^3 \langle\tau_{[0,n_k)}\rangle.
\end{equation*}
We note now that for any pair of morphisms $\sigma$ and $\sigma'$ for which $\sigma\sigma'$ is defined we have that $|\sigma \sigma'| \geq \langle\sigma\rangle |\sigma'|$.
Therefore,
\begin{equation} \label{eq:apps:tech_contract_CFPZ19:1}
    |\tau_{[n_k, n_{k+1})}| \leq
    \frac{ |\tau_{[0, n_{k+1})}| }{ \langle\tau_{[0, n_k)}\rangle } 
    \leq d^3.
\end{equation}
We set $\ell = \lceil \log_2 d' \rceil$ and consider the contraction $\btau = (\tau_{[n_{\ell k}, n_{\ell(k+1)})})_{k\geq0}$.
It follows from the definition of $n_k$ that
$ \langle \tau_{[n_{\ell k}, n_{\ell(k+1)})} \rangle \geq 2^{\ell} \geq d', $
and, from \eqref{eq:apps:tech_contract_CFPZ19:1}, that
$|\tau_{[n_{\ell k}, n_{\ell(k+1)})}| \leq d^{3\ell} \leq d^{6\log_2 d'}$.
Thus, $\btau$ satisfies Item (3) of this proposition. 
Moreover, since $\btau'$ satisfies Item (1) and (2) in Theorem \ref{theo:main_sublinear} and since $\btau$ is a contraction of $\btau'$, Items (1) and (2) of this proposition hold.
\end{proof}

\begin{lem} \label{apps:tech_w_CFPZ19}
Let $w \in \cA^+$ and $\ell \leq |w|$.
There exists a set of words $\cV$ such that:
\begin{enumerate}
    \item $\#\cV \cap \cA^n \leq 2^5|w|/\ell$ for all $n \geq 1$.
    \item $\langle\cV\rangle \geq \ell$, $|\cV| \leq |w|$.
    \item For any $u$ occurring in $w$ of length $|u| \geq 2^6 \ell$ we have that $u \in \cV^2$.
\end{enumerate}
\end{lem}
\begin{proof}
For $i \geq 0$ and $j \in [0,7]$, let $w = u_{i,j}(1) u_{i,j}(2) \dots u_{i,j}(2^i)$ be the (unique) decomposition of $w$ into $2^i$ words $u_{i,j}(k) \in \cA^*$ such that $|u_{i,j}(1) \dots u_{i,j}(k)| = \lfloor (8k+j)|w|/2^{i+3}\rfloor$ for all $k \in [1, 2^i]$.
We define $\cV_i$ as the set of words that are a prefix or a suffix of length at least $\ell$ of some $u_{i,j}(k)$.
Set $\cV \coloneqq \cup_{0 \leq i < \log_2 (|w|/\ell)} \cV_i$.

It follows from the definition of $\cV$ that $\langle\cV\rangle \geq \ell$ and that $|\cV| \leq |w|$, so Item (2) holds.
For Item (1), we note that if $n \geq 1$, then each $u_{i,j}(k)$ has at most one prefix of length $n$ and at most one suffix of length $n$.
Hence, $\#\cV_i \cap \cA^n$ is bounded by above by $2\cdot 8\cdot 2^i = 2^{i+4}$.
Therefore, $\#\cV \cap \cA^n \leq \sum_{0 \leq i < \log_2|w|/\ell} 2^{i+4} \leq 2^5 |w|/\ell$.

We now prove Item (3).
Let $u$ be a word of length $|u| \geq 2^6 \ell$ that occurs in $w$.
Let us write $w = t u s$, where $t, s \in \cA^*$, and take $i \geq 0$ such that $|w|/2^{i+1} < |u| \leq |w|/2^i$.
Remark that $i < \log_2 (|w|/\ell)$ as $|u| \geq 2^6 \ell$.
We also consider the unique pair $(k,j) \in [1,2^i] \times [0,7]$ such that
\begin{equation*}
    (8k + j)|w| / 2^{i+3} \leq |t| + |w|/2 < (8k + j + 1)|w| / 2^{i+3}.
\end{equation*}
Then, we can write $u = u' u''$ in such a way that $|tu'| = \lfloor (8k + j)|w| / 2^{i+3} \rfloor$.
It is not difficult to check that $u'$ is a suffix of $u_{i,j}(k)$ and that $u''$ is a prefix of $u_{i,j}(k+1)$.
Moreover, $|u'|, |u''| \geq |w|/2^i \geq \ell$, so $u', u'' \in \cV_i$ and $u \in \cV^2$.
\end{proof}

\begin{theo} \label{apps:theo:CFPZ19}
Let $X \subseteq \cA^\Z$ be an minimal subshift.
The following conditions are equivalent:
\begin{enumerate}
    \item $X$ has linear-word complexity.
    \item There exists $S \subseteq \cA^*$ such that $S^2 \supseteq \cL(X)$ and $\sup_{n\geq1} p_S(n) < +\infty$.
\end{enumerate}
\end{theo}
\begin{proof}
We first suppose that $X$ satisfies Condition (2) and define $d = \sup_{n\geq1} p_S(n)$.
Then $\cL(X) \cap \cA^n \subseteq \{s s' : s,s' \in S, |s| = n - |s'|\}$.
Hence, 
$$ p_X(n) \leq \sum_{k=0}^n p_S(k) p_S(n-k) \leq (n+1)d^2 $$
and $p_X$ has linear growth.

Let us now suppose that $X$ has linear-growth complexity.
The case in which $X$ is finite is trivial; hence, we assume that $X$ is infinite.
Using Proposition \ref{apps:tech_contract_CFPZ19} with $d' = 2$, we can find a constant $d$ and an $\cS$-adic sequence $\btau = (\tau_n\colon\cA_{n+1}\to\cA_n^+)_{n\geq0}$ generating $X$ such that Items (1), (2) and (3) in Proposition \ref{apps:tech_contract_CFPZ19} hold.

We define $S$ as follows.
Let $n \geq 1$.
For $u \in \mroot \tau_{[0,n)}(\cA_{n})$, we take $p_{n,u} \geq 1$ such that $|\tau_{[0,n)}| \leq |u^{p_{n,u}}| < 2 |\tau_{[0,n)}|$.
We define $\cW_n = \{ u^{p_{n,u}} v^{p_{n,v}} : u, v \in \mroot \tau_{[0,n)}(\cA_n) \}$.
For each $w \in \cW_n$, we use Lemma \ref{apps:tech_w_CFPZ19} with $\ell = \langle\tau_{[0,n-1)}\rangle/2^6$ to obtain a set $\cV_{n,w}$ satisfying the following:
\begin{enumerate} [label=(\alph*)]
    \item $\# \cV_{n,w} \cap \cA^k \leq 2^{11} |w| / \langle\tau_{[0,n-1)}\rangle$ for all $k \geq 1$.
    \item $\langle\cV_{n,w}\rangle \geq \langle\tau_{[0,n-1)}\rangle/2^6$, $|\cV_{n,w}| \leq |w|$.
    \item if $u$ occurs in $w$ and $|u| \geq \langle\tau_{[0,n-1)}\rangle$, then $w \in \cV_{n,w}^2$.
\end{enumerate}
We set $S = \cup_{n \geq 1} \cup_{w \in \cW_n} \cV_{n,w}$.

Before continuing, we make some observations about the definitions.
It follows from the definition of $\cW_n$ and Item (1) in Proposition \ref{apps:tech_contract_CFPZ19} that 
\begin{equation} \label{eq:apps:theo:CFPZ19:1}
    \#\cW_n \leq \#(\mroot \tau_{[0,n)}(\cA_n))^2 \leq d^2.
\end{equation}
We also have that
\begin{equation} \label{eq:apps:theo:CFPZ19:2}
    \text{if $a, b \in \cA_n$ and $v$ occurs in $\tau_{[0,n)}(ab)$, then $v$ occurs in some $w \in \cW_n$.}
\end{equation}
Note that since $|w| \leq 2|\tau_{[0,n)}|$ for all $w \in \cW_n$, (a) and (b) imply that
\begin{equation} \label{eq:apps:theo:CFPZ19:3}
    |\cV_{n,w}| \leq 2|\tau_{[0,n)}|
    \enspace\text{and}\enspace
    \#\cV_{n,w} \leq 2^{12} |\tau_{[0,n)}| / \langle\tau_{[0,n-1)}\rangle \leq 2^{12} d^2,
\end{equation}
where in the last step we used Items (2) and (3) of Proposition \ref{apps:tech_contract_CFPZ19}.

We now prove that $S$ satisfies the desired properties.
Let us start by showing that $\cL(X) \subseteq S^2$.
Let $u \in \cL(X)$ and let $n \geq 1$ be the biggest integer such that $|u| \geq \langle\tau_{[0,n)}\rangle$.
Then, $|u| < \langle\tau_{[0,n+1)}\rangle$ and, thus, as $\btau$ generates $X$, there exists $a,b \in \cA_{n+1}$ such that $u$ occurs in $\tau_{[0,n+1)}(a b)$.
Hence, by \eqref{eq:apps:theo:CFPZ19:2}, $u$ occurs in some $w \in \cW_{n+1}$, which implies, by (c), that $u \in \cV_{n+1,w}^2 \subseteq S^2$.

It remains to prove that $p_S$ is uniformly bounded.
Let $S_n = \cup_{w \in \cW_n} \cV_{n,w}$.
We claim the following:
\begin{enumerate}[label=(\roman*)]
    \item $\# p_{S_n}(k) \leq 2^{12} d^4$ for all $n \geq 0$ and $k \geq 0$.
    \item for any $k \geq 0$, there are at most $\log_2(d)+7$ integers $n$ such that $S_n \cap \cA^k$ is not empty.
\end{enumerate}
Observe that (i) and (ii) allow us to write
\begin{equation*}
    p_S(k) \leq 
    \sum_{n : S_n \cap \cA^k \not= \emptyset} p_{S_n}(k) \leq
    (\log_2 d + 7)\cdot 2^{12} d^4,
\end{equation*}
which would show that $p_S$ is uniformly bounded and would complete the proof.

Let us first prove (i).
The definition of $S_n$ ensures that $p_{S_n}(k) \leq \#\cW_n\cdot\max\{\#\cV_{n,w} \cap \cA^k : w\in \cW_n\}$.
Hence, by \eqref{eq:apps:theo:CFPZ19:1} and \eqref{eq:apps:theo:CFPZ19:3}, $p_{S_n}(k) \leq d^2\cdot 2^{12} d^2$. 

Next, we prove (ii) by contradiction.
Assume that there are more than $\log_2 d + 7$ integers $n$ such that $S_n \cap \cA^k \not= \emptyset$.
Then, we can find $n$ and $m$ such that $S_n \cap \cA^k \not= \emptyset$, $S_m \cap \cA^k \not= \emptyset$ and $m > n + \log_2 d + 7$.
We have, on one hand, that the definition of $S_m$ ensures that $k \geq \min_{w \in \cW_n} \langle\cV_{n,w}\rangle$.
Hence, by (b) and Item (3) in Proposition \ref{apps:tech_contract_CFPZ19},
\begin{equation} \label{eq:apps:theo:CFPZ19:4}
    k \geq \langle\tau_{[0,m-1)}\rangle/2^6 \geq 
    2^{m-n-7} \langle\tau_{[0,n)}\rangle.
\end{equation}
On the other hand, the definition of $S_n$ guarantees that $k \leq \max_{w \in \cW_n} |\cV_{n,w}|$.
Combining this with \eqref{eq:apps:theo:CFPZ19:3} and Item (2) in Proposition \ref{apps:tech_contract_CFPZ19} produces
\begin{equation} \label{eq:apps:theo:CFPZ19:5}
    k \leq 2|\tau_{[0,n)}| \leq 
    2d \langle\tau_{[0,n)}\rangle.
\end{equation}
Equations \eqref{eq:apps:theo:CFPZ19:4} and \eqref{eq:apps:theo:CFPZ19:5} are incompatible as $m - n - 7 > \log_2 d$.
This contradiction proves (ii) and completes the proof of the theorem.
\end{proof}

\subsection{Topological rank}

The topological rank of a minimal subshift $X$ is the least element $k \in [1,+\infty]$ such that there exists a recognizable $\cS$-adic sequence $\btau = (\tau_n\colon\cA_{n+1}\to\cA_n^+)_{n\geq0}$ satisfying, for every $n \geq 1$, that $\#\cA_n \leq k$ and that $\tau_n$ is positive and proper.
The class of finite topological rank subshifts satisfies several rigidity properties, and many tools have been developed to handle it; a non-exhaustive list includes \cite{BKMS13,eigenvalues_jems,E22,EM21,DM08,BSTY19,HPS92}.

It was proved in \cite{DDPM20_interplay} that a minimal subshift of nonsuperlinear-growth complexity has finite topological rank, and thus that the aforementioned rigidity properties hold for this class. 
We present in this subsection a new proof of this fact based in Theorem \ref{theo:main_nonsuperlinear}.

\begin{theo}[\cite{DDPM20_interplay}, Theorem 5.5] \label{apps:theo:nonsuperlinear_have_FTR}
Let $X$ be a minimal subshift of nonsuperlinear-growth complexity.
Then, $X$ has finite topological rank.
\end{theo}
\begin{proof}
The case in which $X$ is finite is trivial, and so we may assume that $X$ is infinite.
Then, Theorem \ref{theo:main_nonsuperlinear} gives $d$ and a recognizable $\cS$-adic sequence $\bsigma = (\sigma_n\colon\cA_{n+1} \to \cA_n^+)_{n\geq0}$ generating $X$ such that Items (1) and (2) of Theorem \ref{theo:main_nonsuperlinear} hold.
In particular,
\begin{equation} \label{eq:apps:theo:nonsuperlinear_have_FTR:1}
    X \subseteq 
    \bigcup_{k \in \Z} S^k (\mroot\sigma_{[0,n)}(\cA_n))^\Z
    \enspace\text{for all $n \geq 1$.}
\end{equation}
Now, since $|\sigma_{[0,n)}|$ goes to $+\infty$ as $n \to +\infty$ and $d|\sigma_{[0,n)}| \leq \langle \sigma_{[0,n)} \rangle$ by Item (2) in Theorem \ref{theo:main_nonsuperlinear}, we have that $\langle \sigma_{[0,n)} \rangle$ diverges to $+\infty$ as $n \to +\infty$.
Hence, since $X$ is aperiodic,
\begin{equation*}
    \lim_{n\to+\infty} \langle \mroot\sigma_{[0,n)}(\cA_n) \rangle = +\infty.
\end{equation*}
This and \eqref{eq:apps:theo:nonsuperlinear_have_FTR:1} allow us to use \cite[Corollary 1.4]{E22} or \cite[Theorem 4.3]{DDPM20_interplay} to conclude that $X$ has finite topological rank.
\end{proof}

% \subsection{Strong mixing}

% The following result was proved in CITE.
% \begin{theo} \label{apps:theo:strong_mixing}
% Let $X$ be a subshift of nonsuperlinear-growth complexity and $\mu$ an ergodic measure.
% Then, $(X, S, \mu)$ is not strongly mixing.
% \end{theo}
% Understanding mixing properties from an $\cS$-adic structure is not easy in general, but there are some techniques in the literature that work if the $\cS$-adic structure is good enough; see \cite{mixing_two_letters05, creutz22, danilenko16}.
% In particular, in \cite{BKMS13} it was proved that if an $\cS$-adic sequence satisfies the so-called {\em exact rank property}, then the system that it generates cannot be strongly mixing.
% In this subsection, we consider minimal subshifts and adapt the technique of \cite{BKMS13} to give a proof of Theorem \ref{apps:theo:strong_mixing} based on Theorem \ref{theo:main_nonsuperlinear}.
% In fact, this approach permits to prove that the system is {\em partially rigid}, but due to space limitations we decided not to include this proof.

\section*{Acknowledgement}
This work was supported by Doctoral Felowship CONICYT-PFCHA/Doctorado Nacional/2020-21202229.

\sloppy\printbibliography

\end{document}